\documentclass[12pt,reqno]{amsart}
\usepackage{amsfonts,amsrefs,latexsym,amsmath,amssymb,mathrsfs,verbatim}
\usepackage{url,color}
\usepackage{pifont}
\usepackage{upgreek}
\usepackage{fancyhdr}
\usepackage{hyperref}
\usepackage{calligra}
\usepackage{marvosym}
\usepackage[percent]{overpic}
\usepackage{pict2e}
\usepackage[hypcap=false]{caption}
\usepackage{xcolor}
\usepackage{wasysym}
\usepackage{accents}
\usepackage{scalerel}

\topmargin - .23 in

\newtheorem{theorem}{Theorem}[section]
\newtheorem{proposition}[theorem]{Proposition}
\newtheorem{lemma}[theorem]{Lemma}
\newtheorem{corollary}[theorem]{Corollary}

\theoremstyle{definition}
\newtheorem{definition}[theorem]{Definition}
\newtheorem{remark}[theorem]{Remark}
\newtheorem{notation}[theorem]{Notation}

\DeclareMathAlphabet{\mathcalligra}{T1}{calligra}{m}{n}
\DeclareFontShape{T1}{calligra}{m}{n}{<->s*[2.2]callig15}{}

\newcommand{\slow}{w}
\newcommand{\altslow}{v}
\newcommand{\bigslow}{\vec{W}}
\newcommand{\altbigslow}{\vec{V}}

\newcommand{\Tboot}{T_{(Boot)}}

\newcommand{\Trandatasize}[1]{\mathring{\updelta}}

\newcommand{\Psiep}{\mathring{\upalpha}}
\newcommand{\TrandatasizeWithFactor}{\mathring{\updelta}_*}

\newcommand{\TranminusdatasizeWithFactor}{\mathring{\updelta}_*}

\newcommand{\gt}{\underline{g}}
\newcommand{\gtinverse}{\underline{g}^{-1}}
\newcommand{\gsphere}{g \mkern-8.5mu / }
\newcommand{\ginversesphere}{\gsphere^{-1}}

\newcommand{\gtancomp}{\upsilon}

\newcommand{\mytr}{{\mbox{\upshape{tr}}_{\mkern-2mu \gsphere}}}

\newcommand{\D}{\mathscr{D}}
\newcommand{\angD}{ {\nabla \mkern-14mu / \,} }

\newcommand{\angdiv}{\mbox{\upshape{div} $\mkern-17mu /$\,}}
\newcommand{\angLap}{ {\Delta \mkern-12mu / \, } }

\newcommand{\GdVar}{\upgamma}
\newcommand{\BadVar}{\underline{\upgamma}}

\newcommand{\Fullset}{\mathscr{Z}}
\newcommand{\Tanset}{\mathscr{P}}

\newcommand{\Singletan}{P}

\newcommand{\angdiff}{ {{d \mkern-9mu /} }}

\newcommand{\angdiffuparg}[1]{ {d \mkern-9mu /}^{#1} }

\newcommand{\Lie}{\mathcal{L}}
\newcommand{\SigmatLie}{\underline{\mathcal{L}}}
\newcommand{\angLie}{ { \mathcal{L} \mkern-10mu / } }

\newcommand{\Lgeo}{L_{(Geo)}}
\newcommand{\Lunit}{L}
\newcommand{\covL}{H}

\newcommand{\uLgood}{\breve{\underline{L}}}
\newcommand{\Radunit}{X}
\newcommand{\Rad}{\breve{X}}

\newcommand{\NonRadialRad}{\Xi}
\newcommand{\XiCoordComp}{\upxi}

\newcommand{\CoordAng}{\Theta}
\newcommand{\Timenormal}{N}
\newcommand{\Mult}{T}
\newcommand{\GeoAng}{Y}

\newcommand{\GeoAngFlatRadComponent}{\uprho}

\newcommand{\angpi}{ { \pi \mkern-10mu / }}

\newcommand{\angk}{ { {k \mkern-10mu /} \, } }

\newcommand{\angkdoublearg}[2]{ {{k \mkern-10mu /}_{#1 #2} \, } }

\newcommand{\angG}{ {{G \mkern-12mu /} \, }}
\newcommand{\angGarg}[1]{ {{G \mkern-12mu /}_{\mkern 1mu #1} \, }}

\newcommand{\angGprime}{ {{ {G'} \mkern-16mu /} \, \, }}
\newcommand{\angGprimearg}[1]{ {{ {G'} \mkern-16mu /}_{\mkern 1mu #1} \, }}

\newcommand{\angGmixedarg}[2]{ {{G \mkern-12mu /}_{#1}^{\ #2} \, }}
\newcommand{\angGnospacemixedarg}[2]{ {{G \mkern-12mu /}_{#1}^{\mkern 2mu #2}}}

\newcommand{\angxi}{ { {\xi \mkern-9mu /}  \, }}

\newcommand{\angxiarg}[1]{ {{\xi \mkern-9mu /}_{#1}  \, }}

\newcommand{\deform}[1]{{^{(#1)} \mkern-1mu \pi}}
\newcommand{\deformarg}[3]{{^{(#1)} \mkern-1mu \pi_{#2 #3}}}

\newcommand{\angdeformarg}[3]{{^{(#1)} \mkern-2mu \angpi_{#2 #3}}}

\newcommand{\Lineproject}{{\Pi \mkern-12mu / } \, }
\newcommand{\Sigmatproject}{\underline{\Pi}}

\newcommand{\vol}{\varpi}
\newcommand{\tvol}{\underline{\varpi}}
\newcommand{\conevol}{\overline{\varpi}}

\newcommand{\spherevol}{\uplambda_{{g \mkern-8.5mu /}}}
\newcommand{\argspherevol}[1]{\uplambda_{{g \mkern-8.5mu /}#1}}

\newcommand{\enzero}{\mathbb{E}_{(Fast)}}
\newcommand{\flzero}{\mathbb{F}_{(Fast)}}
\newcommand{\coercivespacetime}{\mathbb{K}}
\newcommand{\slowen}{\mathbb{E}_{(Slow)}}
\newcommand{\slowfl}{\mathbb{F}_{(Slow)}}

\newcommand{\totmax}[1]{\mathbb{Q}_{#1}}
\newcommand{\totTanmax}[1]{\mathbb{Q}_{#1}}

\newcommand{\slowtotTanmax}[1]{\mathbb{W}_{#1}}
\newcommand{\coerciveTanspacetimemax}[1]{\mathbb{K}_{#1}}

\newcommand{\upchifullmodarg}[1]{{^{(#1)} \mkern-4mu \mathscr{X}}}

\newcommand{\waveinhom}{\mathfrak{F}}

\newcommand{\Vplus}[2]{{^{(+)} \mkern-.5mu  \mathcal{V}_{#1}^{#2}}}
\newcommand{\Vminus}[2]{{^{(-)} \mkern-.5mu  \mathcal{V}_{#1}^{#2}}}
\newcommand{\Sigmaplus}[3]{{^{(+)} \mkern-.5mu  \Sigma_{#1;#2}^{#3}}}
\newcommand{\Sigmaminus}[3]{{^{(-)} \mkern-.5mu  \Sigma_{#1;#2}^{#3}}}

\newcommand{\Contwo}{B}

\newcommand{\LateTimeLUnitMu}{\upkappa}

\newcommand{\smoothfunction}{\mathrm{f}}

\newcommand{\slowbasicenergyerror}{\mathfrak{W}}
\newcommand{\basicenergyerrorarg}[2]{{^{(#1)} \mkern-.5mu \mathfrak{P}}_{(#2)}}

\newcommand{\enmomtensor}{Q}

\newcommand{\myarray}[2][]{\left(
		\begin{array}{lr}
    	 #1 \\
    	 #2 
     \end{array} \right)}

\newcommand{\threemyarray}[3][]{\left(
		\begin{array}{lr}
    	 #1 \\
    	 #2 \\
    	 #3
     \end{array} \right)}

\newcommand{\mydiam}{{\mkern-1mu \scaleobj{.75}{\blacklozenge}}}

\numberwithin{equation}{subsection}

\begin{document}
\title{Shock formation for $2D$ quasilinear wave systems featuring multiple speeds: 
Blowup for the fastest wave, with non-trivial interactions up to the singularity}
\author[JS]{Jared Speck$^{* \dagger}$}

\thanks{$^{\dagger}$JS gratefully acknowledges support from NSF grant \# DMS-1162211,
from NSF CAREER grant \# DMS-1454419,
from a Sloan Research Fellowship provided by the Alfred P. Sloan foundation,
and from a Solomon Buchsbaum grant administered by the Massachusetts Institute of Technology.
}

\thanks{$^{*}$Massachusetts Institute of Technology, Cambridge, MA, USA.
\texttt{jspeck@math.mit.edu}}

\begin{abstract}
We prove a stable shock formation result for a large class of systems
of quasilinear wave equations in two spatial dimensions. We give a precise
description of the dynamics all the way up to the singularity.
Our main theorem applies to systems of two wave equations featuring two distinct wave speeds
and various quasilinear and semilinear nonlinearities,
while the solutions under study 
are (non-symmetric) perturbations of simple outgoing plane symmetric waves.
The two waves are allowed to interact all the way up to the singularity.
Our approach is robust and could be used 
to prove shock formation results for other related systems 
with many unknowns and multiple speeds, 
in various solution regimes,
and in higher spatial dimensions.
However, a fundamental aspect of our framework
is that it applies only to solutions in which the ``fastest wave'' forms a shock
while the remaining solution variables do not, even though they can be non-zero
at the fastest wave's singularity.

Our approach is based on an extended version of the geometric vectorfield
method developed by D.~Christodoulou in his study of shock formation for scalar wave equations,
as well as the framework that we developed in our joint work with J.~Luk, 
in which we proved a shock formation result
for a quasilinear wave-transport system featuring a single wave operator.
A key new difficulty that we encounter is that 
the geometric vectorfields that we use to commute the equations are, 
by necessity, adapted to the wave operator of the (shock-forming) fast wave and therefore
exhibit very poor commutation properties with the slow wave operator,
much worse than their commutation properties with a transport operator.
In fact, commuting the vectorfields all the way through the slow wave operator 
would create uncontrollable error terms. To overcome this difficulty, we rely on 
a first-order reformulation of the slow wave equation, 
which, though somewhat limiting in the precision it affords,
allows us to avoid uncontrollable commutator terms.

\bigskip

\bigskip

\end{abstract}

\maketitle

\centerline{\today}

\tableofcontents
\setcounter{tocdepth}{2}

\newpage

\section{Introduction}
\label{S:INTRO}
Our main goal in this paper is to develop flexible techniques
that advance the theory of shock formation  
in initially regular solutions to quasilinear hyperbolic PDE systems featuring \emph{multiple speeds} of propagation.
Our techniques apply in \emph{more than one spatial dimension},
a setting in which one is forced to complement the method
of characteristics with an exceptionally technical ingredient: energy estimates that hold all the way up to the shock singularity. 
We recall that shock singularities are tied to the intersection of a family of characteristics
and are such that the (initially smooth) solution remains bounded but some derivative
of it blows up in finite time, a phenomenon also known as wave breaking.
Our approach has robust features and could be used to prove shock formation
for a large class of systems; see 
Subsect.\ \ref{SS:EXTENDINGRESULTORELATEDSYSTEMS} for discussion on various types of
systems that could be treated with our approach.
However, for convenience, we study in detail 
only systems of pure wave\footnote{For quasilinear
wave equations, either the first- or second-order 
Cartesian coordinate partial derivatives of the solution
blow up in finite time, depending on whether the quasilinear terms are of type 
$\Phi \cdot \partial^2 \Phi$ or $\partial \Phi \cdot \partial^2 \Phi$.} 
type in the present article.

Specifically, our main result is a sharp proof of finite-time shock formation
for an open set of nearly plane symmetric solutions to equations \eqref{E:FASTWAVE}-\eqref{E:SLOWWAVE},
which form a system of two quasilinear wave equations in two spatial dimensions
featuring two distinct (dynamic) wave speeds; see Theorem~\ref{T:ROUGHMAINTHM} on pg.~\pageref{T:ROUGHMAINTHM}
for a rough summary of our results,
Subsect.\ \ref{SS:SYSTEMSUNDERSTUDY} for our assumptions on the nonlinearities,
and Theorem~\ref{T:MAINTHEOREM} in Sect.\ \ref{S:MAINTHEOREM} 
for the precise statements. Here we provide a very
rough summary.

\begin{theorem}[\textbf{Main theorem} (very rough statement)]	
	\label{T:VERYROUGHMAINTHM}
	In two spatial dimensions, 
	under suitable assumptions on the nonlinearities\footnote{To ensure that shocks form,
	we make a genuine nonlinearity-type assumption,
	which results in the presence of Riccati-type terms that drive the blowup; 
	see Remark~\ref{R:GENUINELYNONLINEAR}.}
	(described in Subsect.\ \ref{SS:SYSTEMSUNDERSTUDY}),
	there exists an open\footnote{By open, we mean open with respect to a suitable Sobolev topology.} 
	set of regular, approximately plane symmetric\footnote{By plane symmetric 
	initial data, we mean data that are functions of the Cartesian coordinate $x^1$.} 
	initial data for the quasilinear wave equation system \eqref{E:FASTWAVE}-\eqref{E:SLOWWAVE}
	such that the solution variables exhibit the following behavior: 
	some first-order Cartesian coordinate partial derivative
	of one of the solution variables blows up in finite time while
	the first-order Cartesian coordinate partial derivatives of the other solution variable
	remain uniformly bounded, all the way up to the singularity 
	in the first solution variable's derivatives.
\end{theorem}

Our proof of Theorem~\ref{T:MAINTHEOREM} is not by contradiction
but is instead based on giving a \emph{complete description of the dynamics},
all the way up to the first singularity, 
which is tied to the intersection of a family of outgoing\footnote{Throughout, ``outgoing'' roughly means right-moving,
that is, along the positive $x^1$ axis, as is shown in Figure~\ref{F:FRAME} on pg.~\pageref{F:FRAME}.} 
approximately plane symmetric characteristics;\footnote{The characteristics that intersect are co-dimension one and have the topology $\mathbb{I} \times \mathbb{T}$, where 
$\mathbb{I}$ is an interval of time and $\mathbb{T}$ is the standard torus.}
see Figure~\ref{F:FRAME} on pg.~\pageref{F:FRAME} for a picture of the setup, in
which those characteristics are about to intersect to form a shock.

It is important to note that our approach here 
relies on our assumption that the shock-forming wave variable satisfies 
a scalar equation whose principal part depends \emph{only on the shock-forming variable itself};
see equation \eqref{E:FASTWAVE}.\footnote{In contrast, the principal part of the wave equation \eqref{E:SLOWWAVE} of the non-shock-forming solution variable is allowed to depend on both solution variables.}
Especially for this reason (and for others as well),
one would need new ideas to prove shock formation results in more than one spatial dimension
for more general second-order quasilinear hyperbolic systems,
where the coupling of all of the unknowns can occur in the principal part
of all of the equations.
These more complicated kinds of systems arise, 
for example, in the study of elasticity and crystal optics,
where the unknowns are the scalar components $\Phi^A$ 
of the ``array-valued'' solution $\Phi$
and the evolution equations can be written in the form\footnote{The principal coefficients $H_{AB}^{\alpha \beta}(\partial \Phi)$ 
must, of course, verify appropriate technical assumptions to ensure
the hyperbolicity (and local well-posedness) of the system.}
$H_{AB}^{\alpha \beta}(\partial \Phi)
\partial_{\alpha} \partial_{\beta} \Phi^B
= 
0
$
(with implied summation over $\alpha$, $\beta$, and $B$).

Prior to this paper, the only constructive proof of shock formation
for a quasilinear hyperbolic system in more than one spatial dimension featuring multiple speeds
was our work \cites{jLjS2016a,jLjS2016b}, joint with J.~Luk,
in which we studied a system of
quasilinear wave equations featuring a \emph{single wave operator}
coupled to a quasilinear transport equation.\footnote{More precisely, the equations 
studied in \cites{jLjS2016a,jLjS2016b} were a formulation of the compressible
Euler equations with vorticity.} 
As we explain below, 
the following basic features of the system studied in \cites{jLjS2016a,jLjS2016b}
played a crucial role in the analysis:
\textbf{i)} there was precisely one wave operator in the system
and 
\textbf{ii)} transport operators are \emph{first-order}.
Thus, the main contribution of the present article
is that we allow for additional (second-order)
quasilinear wave operators in the system.
This requires new geometric and analytic ideas since,
as we explain two paragraphs below,
a naive approach to analyzing the additional
wave operators would lead to commutator error terms
that are uncontrollable near the shock.

The phenomenon of shock formation is ubiquitous in the study of quasilinear hyperbolic PDEs 
in the sense that many systems without special structure\footnote{Some systems in one spatial dimension with special structure,
such as \emph{totally linearly degenerate} quasilinear hyperbolic systems, are not expected to admit any shock-forming solutions
that arise from smooth initial data.} 
are known,
in the case of one spatial dimension,
to admit shock-forming solutions, at least for some regular initial conditions.
In fact, the theory of solutions to quasilinear hyperbolic PDE systems
in one spatial dimension is rather advanced in that it
incorporates the formation of shocks starting from regular initial conditions
as well as their subsequent interactions, at least for solutions
with small total variation. Indeed, a key reason behind the advanced status 
of the $1D$ theory is the availability of estimates within the class of
functions of bounded variation; readers can consult the monograph \cite{cD2010} for a detailed account of the one-dimensional theory.
In more than one spatial dimension, the theory of solutions to quasilinear hyperbolic PDEs 
(without symmetry assumptions) is much less developed,
owing in part to the fact that bounded variation estimates for hyperbolic systems 
typically fail in this setting \cite{jR1986}.
In fact, in more than one spatial dimension,
there are very few works even on the formation
of a shock\footnote{One of course needs to make assumptions on the structure of the nonlinearities
in order to ensure that shocks can form.} 
starting from smooth initial conditions, let alone
the subsequent behavior of the shock wave\footnote{We note, however, that
Majda has solved \cites{aM1981,aM1983a,aM1983b}, in appropriate Sobolev spaces, 
the \emph{shock front problem}. That is, he proved
a local existence result starting from an initial discontinuity 
given across a smooth hypersurface contained in the Cauchy hypersurface.
The data must verify suitable jump conditions, 
entropy conditions, and higher-order compatibility conditions.
Moreover, as we describe in Subsect.\ \ref{SS:PRIORWORKSANDSUMMARY}, 
Christodoulou recently solved \cite{dC2017}
the restricted shock development problem.}
or the interaction of multiple shock waves; our work here concerns the first of these problems.
Specifically our result builds on the body of work 
\cites{sA1995,sA1999a,sA1999b,dC2007,jS2016b,dCsM2014,gHsKjSwW2016,sM2016,sMpY2017,jLjS2016a,jLjS2016b}
on shock formation in more than one spatial dimension,
the new feature being that in the present article,
we have treated wave systems with multiple wave speeds
such that all solution variables are allowed to be non-zero at the first singularity.
All prior shock formation results in more than one spatial dimension, 
with the exception of the aforementioned works \cites{jLjS2016a,jLjS2016b}, 
concern scalar quasilinear wave equations, 
which enjoy the following fundamental property: 
there is only one wave speed,
which is tied to the characteristics of the principal part of the equation. 
As we mentioned earlier, the methods of \cites{jLjS2016a,jLjS2016b}
yield similar shock formation results for wave-transport systems in which
there is precisely one wave operator.

As we mentioned above,
many quasilinear hyperbolic systems of mathematical and physical interest
have a principal part that is more complicated
than that of the scalar wave equations treated in
\cites{sA1995,sA1999a,sA1999b,dC2007,jS2016b,dCsM2014,gHsKjSwW2016,sM2016,sMpY2017}
and the wave-transport systems 
treated in \cites{jLjS2016a,jLjS2016b},
which feature precisely one wave operator.
It is of interest to understand whether or not 
shock formation also occurs in solutions to such more complicated systems 
in more than one spatial dimension. 
We now explain why our proof of shock formation for quasilinear wave systems with multiple wave speeds,
though they are not the most general type of second-order hyperbolic systems of interest,
requires new ideas compared to \cites{jLjS2016a,jLjS2016b}.
Like all prior works on shock formation in more than one spatial dimension, 
our approach in \cites{jLjS2016a,jLjS2016b} 
was fundamentally based on the construction of geometric
vectorfields adapted to the single wave operator. 
The following idea, originating in \cites{sA1995,sA1999a,sA1999b,dC2007},
lied at the heart of our analysis of \cites{jLjS2016a,jLjS2016b}: because the vectorfields were adapted to the wave operator, 
we were able to commute them \emph{all the way through it} while generating only controllable error terms.
Moreover, commuting all the way through the wave operator seems like an essential aspect of
the proof since the special cancellations that one relies on to control error terms
seem to be visible only under a covariant second-order formulation of the wave equation;
see also Remark~\ref{R:NEEDSECONDORDER}.
To handle the presence of a transport operator in the system of \cites{jLjS2016a,jLjS2016b},
we relied on the following key insight: upon introducing a geometric weight,\footnote{Specifically, the weight is 
the inverse foliation density
$\upmu$, which we describe later on in detail (see Def.~\ref{D:FIRSTUPMU}).}
one can commute the same geometric vectorfields
through an essentially arbitrary,\footnote{More precisely, the transport operator must
be transversal to the null hypersurfaces corresponding to the wave operator.} 
solution-dependent, \emph{first-order} (transport) operator;
thanks to the weight, one encounters only error terms that can be controlled
all the way up to the shock. However, as we explain at the end of Subsect.\ \ref{SS:PRIORWORKSANDSUMMARY},
it seems that commuting the geometric vectorfields through
a typical second-order differential operator,
such as a wave operator
with a speed different from the one to which the vectorfields are adapted,
leads to uncontrollable error terms;
see the end of Subsect.\ \ref{SS:PRIORWORKSANDSUMMARY}
for further discussion on this point.
For this reason, our treatment of systems featuring two wave operators 
with strictly different speeds\footnote{By strictly different speeds, we mean that the characteristics corresponding 
to the two wave operators are strictly separated; see Subsubsect.\ \ref{SSS:WAVESPEEDASSUMPTIONS} for our precise assumptions.}
is based on the following fundamental strategy,
which we discuss in more detail in Subsect.\ \ref{SS:PRIORWORKSANDSUMMARY}:
\begin{quote}
We use a \emph{first-order reformulation} of one of the wave equations
(corresponding to the solution variable that does \emph{not} form a shock),
which, though somewhat limiting in the precision it affords,
\emph{allows us to commute the equations with the geometric vectorfields while 
avoiding uncontrollable error terms}.
\end{quote}

For the solutions under consideration,
the ``fast wave variable,'' denoted by $\Psi$ throughout, is the one that forms a shock.
That is, $\Psi$ remains bounded but some of its first-order partial derivatives with respect to the 
Cartesian coordinates blow up. In contrast, the ``slow wave variable,''
denoted by $\slow$ throughout, is more regular in that its 
first-order partial derivatives with respect to the 
Cartesian coordinates remain \emph{uniformly bounded}, all the way up to the shock.
We can draw an analogy to the little that is known about
shock formation in elasticity, a 
second-order hyperbolic theory in which longitudinal waves
propagate at a faster speed than transverse waves \cite{sTZ1998}.
In spherical symmetry, there are no transverse elastic waves, 
and the equations of elasticity reduce to a single scalar quasilinear wave equation that 
governs the propagation of longitudinal waves.
In this setting, under an appropriate genuinely nonlinear assumption,
John proved \cite{fJ1984} a small-data finite-time shock formation result.
Thus, for elastic waves in the simplified setting of spherical symmetry,
it is precisely the fastest wave that forms a shock
(while the ``transverse wave part'' of the solution remains trivial).
In our work here, the slow moving wave $\slow$ is allowed to be non-zero at the singularity.
For this reason, our proof of the more regular behavior of $\slow$ 
is non-trivial and relies on our first-order formulation
of the evolution equations for $\slow$.
As will become clear, one would likely need new ideas to treat data 
such that the slow wave variable forms a shock.
This is what one expects to happen, for example, 
in various solution regimes 
for the Euler-Einstein equations of cosmology 
and for the Euler-Maxwell equations of plasma dynamics,
where one expects the slow-moving sound waves to be able to drive finite-time shock formation
in the ``fluid part'' of the system (for appropriate data). Roughly, the reason that one would need new ideas 
to prove shock formation for the slow wave is that in the present article, 
to close the energy estimates, we crucially rely on the fact that the characteristic hypersurfaces 
of the fast wave operator 
(which are also known as null hypersurfaces in view of their connection to the Lorentzian notion of a null vectorfield)
are \emph{spacelike} relative to the slow wave operator;
indeed, this essentially defines what it means for the slow wave operator to be ``slow.''
We denote these fast wave characteristic hypersurfaces
by $\mathcal{P}_u^t$ when they are truncated at time $t$;
see Figure~\ref{F:FRAME} on pg.~\pageref{F:FRAME} for a picture of the $\mathcal{P}_u^t$. 
Analytically, the fact that the $\mathcal{P}_u^t$ are spacelike
relative to the slow wave operator is reflected in the coerciveness estimate
\eqref{E:SLOWNULLFLUXCOERCIVENESS}, which shows that energy
for the slow wave variable $\slow$ along $\mathcal{P}_u^t$ is positive definite in $\slow$ and \emph{all} of its
partial derivatives with respect to the Cartesian coordinates and \emph{does not feature a degenerate $\upmu$ weight}.
This non-degenerate coerciveness along $\mathcal{P}_u^t$
appears to be essential for closing the energy estimates.
We remark that in one spatial dimension, 
one can rely exclusively on the method of characteristics (without energy estimates)
and thus there are many results in which the slow wave can blow up, 
\cites{fJ1974,pL1964,dCdRP2016,aRfS2008} to name just a few.

\begin{remark}[\textbf{We do not use a first-order formulation for the shock-forming wave equation}]
\label{R:NEEDSECONDORDER}
As we describe in
Subsect.\ \ref{SS:CONTEXTANDPRIORRESULTS},
it does not seem possible to treat the wave equation for $\Psi$ using
a first-order formulation in the spirit of the one
that we use for $\slow$; such a formulation of 
the evolution equations for $\Psi$ would 
not exhibit the special geometric cancellations 
found in our second-order formulation of it
(see equation \eqref{E:FASTWAVE} and the discussion below it).
\end{remark}

\subsection{A more precise summary of the main results}
\label{SS:SUMMARYOFMAINRESULTS}
In this paper, we consider data for the system of wave equations given on 
the union of a portion of a null hypersurface $\mathcal{P}_0$
and a portion of the two-dimensional spacelike Cauchy hypersurface
$\Sigma_0 := \lbrace 0 \rbrace \times \Sigma \simeq \Sigma$, where
the ``space manifold'' $\Sigma$ is
\begin{align} \label{E:SPACEMANIFOLD}
\Sigma
:= 
\mathbb{R} \times \mathbb{T}.
\end{align}
Here and throughout,
$\mathbb{T}$ is the standard one-dimensional torus 
(that is, the interval $[0,1]$ with the endpoints identified and equipped with a standard smooth orientation)
while ``null'' means null with respect to the Lorentzian metric $g$
corresponding to the wave equation for the shock-forming variable $\Psi$.
In fact, 
\begin{quote}
\emph{We tailor all geometric constructions to the fast wave metric $g$}. 
\end{quote}
See Figure~\ref{F:FRAME} on pg.~\pageref{F:FRAME} for a picture of
the setup. We allow for non-trivial data on $\mathcal{P}_0$ because 
this setup would be convenient, in principle, for proving that, at least for some of our solutions,
$\Psi$ and $\slow$ are both non-zero at the singularity
(roughly because one could place non-zero data on $\mathcal{P}_0$ near the shock). 
However, we do not explicitly
exhibit any data such that both solution variables are 
guaranteed to be non-zero at the singularity.
Our assumption of two spatial dimensions
is for technical convenience only;
similar results could be proved in three or more spatial dimensions.
This assumption allows us to avoid the technical issue of deriving
elliptic estimates for the foliations that we use in our analysis, 
which are needed in three or more spatial dimensions. The elliptic 
estimates would be somewhat lengthy 
to derive but have been well-understood in the context of shock formation starting from
\cite{dC2007}. Our proof of shock formation could also be 
extended to different spatial topologies, though changing
the spatial topology might alter the kinds of data to which
our methods apply.\footnote{The formation of the shock
is local in nature. Thus, given any spatial manifold, our approach
could be used to exhibit an open set of data on it such that the
solution forms a shock in finite time. Roughly, there is no
obstacle to proving large-data shock formation on general manifolds.} 

We recall that the shock-forming variable $\Psi$
corresponds to the ``fast speed'' while
the less singular variable $\slow$ corresponds to the ``slow speed.''
We will study solutions with initial data
that are (non-symmetric) perturbations of
the initial data corresponding to
simple outgoing plane symmetric waves 
with $\slow \equiv 0$.
We discuss the notion of a simple outgoing plane symmetric solution
in more detail in Subsubsect.\ \ref{SSS:NEARLYSIMPLEWAVES}, in which, for illustration,
we provide a proof of our main results for plane symmetric solutions.
By plane symmetric solutions, we mean solutions that depend only on the Cartesian coordinates $t$ and $x^1$.
In the context of our main results, 
the factor of $\mathbb{T}$ in the space manifold  
$
\Sigma
= 
\mathbb{R} \times \mathbb{T}
$
corresponds to perturbations away from plane symmetry.
The advantage of studying (asymmetric) perturbations of simple outgoing plane symmetric waves
is that it allows us to focus our attention on the singularity formation
without having to confront additional evolutionary phenomena such 
as dispersion, which is exhibited, for example, in the \emph{initial}
evolutionary phase of small-data solutions 
with initial data given on $\mathbb{R}^2$.
We studied similar nearly plane symmetric solution regimes in our joint work \cite{jSgHjLwW2016}
on shock formation for scalar quasilinear wave equations as well as our
joint work \cite{jLjS2016b} on the compressible Euler equations with vorticity,
which we further describe below.

We now give a slightly more precise statement our main results;
see Theorem~\ref{T:MAINTHEOREM} for the precise statements.

\begin{theorem}[\textbf{Main theorem} (slightly more precise statement)]	
	\label{T:ROUGHMAINTHM}
	Under suitable assumptions on the nonlinearities
	(described in Subsect.\ \ref{SS:SYSTEMSUNDERSTUDY}),
	there exists an open set of regular data (belonging to an appropriate Sobolev space)
	for the system \eqref{E:FASTWAVE}-\eqref{E:SLOWWAVE}
	such that $\max_{\alpha=0,1,2} |\partial_{\alpha} \Psi|$ blows up in finite time due to the intersection 
	the ``fast'' characteristics $\mathcal{P}_u$,
	while $|\Psi|$, $|\slow|$, and $\max_{\alpha=0,1,2} |\partial_{\alpha} \slow|$ 
	remain uniformly bounded, 
	all the way up to the singularity.
	More precisely, we allow the data for $\Psi$ to be large or small, 
	but they must be close to the data 
	corresponding to a \underline{simple outgoing plane wave}.
	We furthermore assume that the data for $\slow$ are relatively small
	compared to the data for $\Psi$.
\end{theorem}

\begin{remark}[\textbf{The need for geometric coordinates}]
\label{R:NEEDFORGEOMETRICCOORDS}
Although Theorem~\ref{T:ROUGHMAINTHM} refers to the
Cartesian coordinates, as in 
all of the prior proofs of shock formation in more than one spatial dimension,
we need very sharp estimates, tied to a system of geometric coordinates,
to close our proof; 
the Cartesian coordinates are 
not adequate for measuring the regularity and boundedness 
properties of the solutions nor for tracking the intersection of characteristics.
\end{remark}
 
\subsection{Paper outline and remarks}
\label{SS:PAPEROUTLINE}
\begin{itemize}
	\item In the rest of Sect.\ \ref{S:INTRO}, we place our main result in context,
		construct some of the basic geometric objects that we use in our analysis,
		and provide an overview of the proof.
	\item In the present paper, 
		we often cite identities and estimates proved in the work \cite{jSgHjLwW2016},
		which yielded similar shock formation results for scalar wave equations
		(see, however, Remark~\ref{R:NEWPARAMETER}).
	\item In Sect.\ \ref{S:GEOMETRICSETUP}, we construct the remaining 
		geo-analytic objects that we use in our proof and exhibit their main
		properties.
	\item In Sect.\ \ref{S:NORMSANDBOOTSTRAP}, we define the norms that we use
		in our analysis, state our assumptions on the data,
		and formulate bootstrap assumptions.
	\item In Sect.\ \ref{S:ENERGIES}, we define our energies and provide the
		energy identities that we use in our $L^2$ analysis.
	\item In Sects.\ \ref{S:PRELIMINARYPOINTWISE}-\ref{S:POINTWISEESTIMATES},
		we derive a priori $L^{\infty}$ and pointwise estimates for the solution.
	\item In Sect.\ \ref{S:ENERGYESTIMATES}, we derive a priori energy estimates.
		This is the most important and technically demanding section of the paper
		and it relies on all of the geometric constructions and estimates
		from the prior sections.
	\item In Sect.\ \ref{S:MAINTHEOREM}, we provide our main theorem.
		This section is short because the estimates of the preceding sections
		essentially allow us to quote the proof of \cite{jSgHjLwW2016}*{Theorem~15.1}.
\end{itemize}

\subsection{Further context and prior results}
\label{SS:CONTEXTANDPRIORRESULTS}
In his foundational work \cite{bR1860} in which he invented the Riemann invariants,
Riemann initiated the rigorous study of finite-time shock formation in 
initially regular solutions to quasilinear hyperbolic systems
in one spatial dimension, specifically the compressible Euler equations of fluid mechanics.
An abundance of shock formation results in one spatial dimension
were proved in the aftermath of Riemann's work,
with important contributions 
by Lax \cite{pL1964},
John \cite{fJ1974},
Klainerman--Majda \cite{sKaM1980},
and many others, continuing up until the present day
\cite{dCdRP2016}. 
The first constructive results on finite-time shock formation
in more than one spatial dimension (without symmetry assumptions)
were proved by Alinhac \cites{sA1995,sA1999a,sA1999b}, who 
treated scalar quasilinear wave equations in two and three spatial dimensions
that fail to satisfy the null condition. In the case of the 
quasilinear wave equations of irrotational relativistic fluid mechanics (which are scalar equations),
Alinhac's results were remarkably sharpened by Christodoulou \cite{dC2007},
whose fully geometric framework subsequently led to further shock formation results for scalar
quasilinear wave equations \cites{jS2016b,dCsM2014,sMpY2017,jSgHjLwW2016} 
as well as the compressible Euler equations with vorticity \cites{jLjS2016a,jLjS2016b}.
In Subsect.\ \ref{SS:PRIORWORKSANDSUMMARY},
we describe some of these works in more detail.

As we mentioned at the beginning, our main goal in this paper is to develop techniques
for studying shock formation in solutions to quasilinear hyperbolic systems 
in more than one spatial dimension
whose principal part exhibits multiple speeds of propagation.
In any attempt to carry out such a program,
one must grapple with following difficulty: the known approaches
to proving shock formation in more than one spatial dimension for scalar wave equations
are based on geo-analytic constructions that are fully adapted to 
the principal part of the scalar equation
(which corresponds to a dynamic Lorentzian metric), 
or, more precisely, to a family of characteristic
hypersurfaces corresponding to the Lorentzian metric.
One might say that in prior works, all coordinate/gauge freedoms 
for the domain were exhausted in order to
understand the intersection of the characteristics corresponding to the scalar equation
and the relationship of the intersection to the blowup of the solution's derivatives.
Therefore, those works left open the question of how to prove shock formation for systems featuring multiple 
unknowns and a more complicated principal part with distinct speeds of propagation;
roughly speaking, to treat such more complicated systems,
one has to understand how geometric objects adapted to one 
part of the principal operator (that is, to one wave speed) interact with 
remaining part of the principal operator
(corresponding to the other speeds).
For systems with appropriate structure, one could skirt this difficulty by
considering only a subset of initial conditions that lead to 
the following behavior: only one solution variable is non-zero at the first singularity.
For such solutions, the analysis effectively reduces to the study of a scalar equation.
A simple but important example of this approach is given by
John's aforementioned study \cite{fJ1984} of shock formation in spherically symmetric solutions to the equations of elasticity,
in which case the equations of motion, 
which generally have a complicated principal part with multiple speeds of propagation, 
reduce to a spherically symmetric \emph{scalar} quasilinear wave equation.
However, such a drastically simplified setup is mathematically and physically unsatisfying
in that it is typically not stable against nontrivial perturbations with very large spatial support, 
no matter how small they might be.

In our joint works \cites{jLjS2016a,jLjS2016b}, 
we proved the first shock formation result for a system with more than 
one speed in which all solution variables can be active (non-zero) 
at the first singularity. Specifically, the system (which is actually a new formulation of the compressible Euler equations) 
featured one wave operator and one transport operator,
and we showed that for suitable data, a family of outgoing \emph{wave characteristics} 
intersect in finite time and cause a singularity in the 
first-order Cartesian partial derivatives of the wave variables but the not transport variable.\footnote{In \cites{jLjS2016a,jLjS2016b},
the transport variable was the specific vorticity, defined to be vorticity divided by density.}
With the result \cite{jLjS2016b} in mind,
one might say the main new contribution of this article is 
to upgrade the framework established
in \cite{jLjS2016b} in order to accommodate an additional \emph{second-order}
quasilinear hyperbolic scalar equation 
(see Subsect.\ \ref{SS:EXTENDINGRESULTORELATEDSYSTEMS} for remarks on how to extend our result to additional systems). 


As we mentioned earlier, our main results apply to initial data such that
the ``fast wave'' variable $\Psi$
(corresponding to the strictly faster of the two speeds of propagation)
is the one that forms a shock. 
That is, $\Psi$ remains bounded but some first-order Cartesian coordinate partial derivative $\partial_{\alpha} \Psi$ 
blows up in finite time. 
Like all prior works on shock formation, 
in the equations that we study in this article,
the blowup of $\partial \Psi$
is driven by the presence of Riccati-type self-interaction 
inhomogeneous terms
$\partial \Psi \cdot \partial \Psi$ in the wave equation for $\Psi$.
As we have already stressed, the ``slow wave'' variable $\slow$ exhibits much less singular behavior,
even though its wave equation is also allowed to contain
(when expressed relative to standard Cartesian coordinates)
Riccati-type self-interaction terms $\partial \slow \cdot \partial \slow$.
Our ability to track the different behaviors of $\Psi$ and $\slow$
requires new geometric and analytic insights, 
notably the advantages that arise from
our first-order reformulation of the slow wave equation.

We now outline why, for the solutions under study, 
$\partial \Psi$ blows up while $\partial \slow$ does not,
even though the wave equations for $\Psi$ and $\slow$ can have a similar structure.\footnote{In fact, the wave equation
for $\slow$ is allowed to be even more complicated than that of $\Psi$ since we allow the principal operator
for $\slow$ to depend on $\Psi$, $\slow$, and $\partial \slow$, while the principal operator for
$\Psi$ is allowed to depend only on $\Psi$;
see \eqref{E:FASTWAVE}-\eqref{E:SLOWWAVE}.}
The main idea is that we consider initial data that lead to 
\emph{nearly simple outgoing plane wave solutions} in which
$\Psi$ has small initial dependence 
(in appropriate norms)
on the Cartesian torus coordinate $x^2$ and in which 
$\slow$ and $\partial \slow$ are initially relatively small and remain so throughout the evolution.
By an ``outgoing simple wave,'' we roughly mean a solution such that
$\slow \equiv 0$
and such that the dynamics of $\Psi$ 
are dominated by a right-moving (along the $x^1$ axis) wave, 
as opposed to a combination of right- and left-moving waves.
Our results show that for (non-symmetric) 
perturbations of such solutions,
$\partial \Psi$ blows up before the small terms $\partial \slow \cdot \partial \slow$ 
are able to drive the blowup of $\partial \slow$.
In particular, in the solution regime under consideration,
$\slow$ and $\partial \slow$ remain small in $L^{\infty}$ all the way up to the 
first singularity in $\partial \Psi$.
This is a subtle effect in that the wave equation for $\slow$ is allowed
to contain source terms (see RHS~\eqref{E:SLOWWAVE})
that are linear (but not quadratic or higher order!) 
in the tensorial component of $\partial \Psi$ that 
blows up.
Therefore, our work entails the study of 
non-trivial interactions between a wave that 
forms a singularity with another wave that, 
as we must prove, exhibits less singular behavior,
even though it is coupled to the ``singular part'' of the singular wave.

The set of initial data that we treat in our main results
is motivated in part by John's aforementioned work
\cite{fJ1974}, in which he proved a blowup result
for first-order quasilinear hyperbolic systems 
with multiple speeds
in one spatial dimension 
whose small-data solutions behave like simple waves\footnote{That is, in \cite{fJ1974},
one non-zero solution component is dominant near the singularity.} 
near the singularity.
John's result was recently sharpened 
\cite{dCdRP2016} by Christodoulou--Perez using
extensions of the framework developed by Christodoulou in \cite{dC2007}.
In Subsubsect.\ \ref{SSS:NEARLYSIMPLEWAVES}, to illustrate some of the main ideas,
we provide a proof of our main results in the special case that the initial data have exact plane
symmetry, that is, for data that are \emph{independent} of the Cartesian coordinate $x^2$. We caution, however, 
that the assumption of plane symmetry represents a drastic simplification of the full problem;
in plane symmetry, we are able to avoid deriving energy estimates, which is the main technical
difficulty that one encounters in the general case.
As we will explain, our analytic approach, in particular our approach to energy estimates, 
is based on geometric decompositions adapted to the characteristics corresponding to principal part of 
$\Psi$'s wave equation together with a first-order reformulation of the 
wave equation for $\slow$ that is compatible with these geometric decompositions.
This allows us to track the distinct behavior of the two waves all the way up to the shock.
As we mentioned above, one would likely need new ideas to treat data 
such that slow wave $\partial \slow$ is expected to form the first singularity and, 
at the same time, to interact with the fast wave $\Psi$ and its derivatives.

\subsection{Additional details concerning the most relevant prior works}
\label{SS:PRIORWORKSANDSUMMARY}
Our work here builds upon the outstanding contributions of 
Alinhac \cites{sA1995,sA1999a,sA1999b},
who was the first to prove small-data shock formation 
for solutions to quasilinear wave equations in more than one spatial dimension. 
Specifically, Alinhac studied scalar quasilinear wave equations
of the form 
\begin{align} \label{E:ALWAVE}
	(g^{-1})^{\alpha \beta}(\partial \Phi) \partial_{\alpha} \partial_{\beta} \Phi = 0,
\end{align}
on $\mathbb{R}^{1+n}$ for $n=2,3$.
For all such equations that fail to satisfy the null condition,
he identified a set of small, regular, compactly supported initial data 
such that the corresponding solution forms a shock in finite time.
The set of initial data to which his main results apply
were such that the constant-time hypersurface of first blowup
contains \underline{exactly one point} at which $\partial^2 \Phi$ blows up.\footnote{For equations of type
\eqref{E:ALWAVE}, a shock singularity 
is such that $\Phi$ and $\partial \Phi$ remain bounded while $\partial^2 \Phi$ blows up.}

The most important ingredient in Alinhac's approach was a dynamically constructed system of 
geometric coordinates tied to an eikonal function
(see Subsubsect.\ \ref{SSS:GEOMETRICINGREDIENTS} for a precise definition), 
whose level sets are $g$-null hypersurfaces (that is, characteristics). 
The main idea behind his approach was as follows: 
show that relative to the geometric coordinates, 
the solution remains regular up to the singularity, except possibly at the very
high (geometric) derivative levels. This enables one to approach the problem of shock formation from a more traditional
perspective in which one derives long-time-existence-type estimates.
It turns out that this approach, while viable, is 
extremely technically demanding to implement. The reason is that the best estimates known
allow for the possibility that the high-order geometric energies blow up, which 
makes it difficult (though possible)
to prove that the solution remains regular relative to the geometric
coordinates at the lower derivative levels.
After one has obtained regular estimates 
(at the lower derivative levels)
relative to the geometric coordinates,
one can easily carry out the rest of the proof, 
namely deriving the singularity formation
relative to the Cartesian coordinates, which one obtains by
showing that the geometric coordinates degenerate relative to the Cartesian ones.
Roughly, both the formation of the shock and the blowup of the solution's Cartesian coordinate partial derivatives
are caused by the intersection of the level sets of the eikonal 
function. The shock-generating initial conditions 
identified by Alinhac form a set of ``non-degenerate'' initial data, which 
can be thought of as generic inside the set of all smooth, small, compactly supported initial data.
Although Alinhac's use of an eikonal function allowed him to provide
a sharp description of the first singularity,
his approach to deriving energy estimates 
was based on a Nash--Moser iteration scheme featuring a free boundary.
His iteration scheme fundamentally relied on his non-degeneracy assumptions
on the data, which led to solutions whose first singularity is \emph{isolated} 
in the constant-time hypersurface of first blowup. His reliance on a Nash--Moser scheme is tied to the fact that
the regularity theory for the eikonal function is very difficult at the top order.

Our work also builds upon Christodoulou's groundbreaking sharpening \cite{dC2007}
of Alinhac's results for the subclass of (scalar) Euler-Lagrange wave equations
corresponding to the irrotational relativistic Euler equations in three spatial dimensions.
For these wave equations, Christodoulou proved the following main results:
\textbf{i)} He showed that for solutions generated by small,\footnote{In particular,
unlike Alinhac's proof, Christodoulou's yields global information about solutions corresponding to
an open set of data that contains
the trivial data in its interior.} 
regular, compactly supported data, shocks are the only possible singularities. 
The formation of a shock exactly corresponds to
the intersection of the characteristics, or equivalently, the 
vanishing of the inverse foliation density of the characteristics,
denoted by $\upmu$ (see Def.~\ref{D:FIRSTUPMU}). Put differently, Christodoulou proved that 
if the characteristics never intersect, then the solution exists globally.\footnote{More precisely, he
studied the solution only in a region that is trapped in between an inner null cone and an outer null cone.}
\textbf{ii)} He exhibited an open set of data for which shocks do in fact form. Unlike Alinhac's data, Christodoulou's
do not have to lead to a solution with an isolated first singularity.
Most importantly, \textbf{iii)} Christodoulou gave a complete description of a portion
of the maximal development\footnote{Roughly, the maximal development 
is the largest possible classical solution that is uniquely determined by the data. Readers can consult
\cites{jSb2016,wW2013} for further discussion.} 
of the data, all the way up to the boundary.
Although for brevity we have not given a complete description of the maximal development in this article,
it is likely that our sharp estimates could be used to obtain such a description,
by invoking arguments along the lines of \cite{dC2007}*{Chapter~15}.
Like Alinhac's approach and the approach of the present article, Christodoulou's framework
was based on an eikonal function.
However, unlike Alinhac,
Christodoulou did not use Nash--Moser estimates when deriving energy estimates.
Instead, he used a sharper, more geometric approach that required the full strength of his framework.
In particular, to avoid derivative loss,
Christodoulou derived sharp information regarding the tensorial regularity properties 
of eikonal functions in the context of shock formation, 
in which their level sets intersect. In this endeavor, he was undoubtedly aided by 
the experience he gained from his celebrated joint proof with Klainerman of the stability of Minkowski spacetime \cite{dCsK1993}.
In that work, the authors also had to deeply understand the high-order regularity properties of eikonal functions,
though in a less degenerate context in which they remain regular.
Christodoulou's sharp description of the maximal development, though
of interest in itself, is also important for another reason:
it is an essential ingredient for properly setting up the shock development problem, 
which is the problem of weakly continuing the solution
to the relativistic Euler equations past the first singularity under
appropriate selection criteria in the form of jump conditions; see
\cite{dC2007} for further discussion. 
We note that the shock development problem was recently solved in spherical symmetry \cite{dCaL2016}.
Moreover, in a recent breakthrough work \cite{dC2017}, 
for the non-relativistic compressible Euler equations without symmetry assumptions,
Christodoulou solved the shock development problem in a restricted case
(known as the restricted shock development problem)
such that the jump in entropy across the shock hypersurface was ignored.

Christodoulou's sharp, geometric approach has led to further advancements on shock formation
in solutions starting from smooth initial conditions, 
including extensions of his results to larger classes of equations and new types of initial conditions
\cites{jS2016b,dCsM2014,gHsKjSwW2016,sM2016,sMpY2017,jLjS2016a,jLjS2016b};
see the survey article \cite{gHsKjSwW2016} for an in-depth discussion of some of these results.
However, a crucial feature of both Alinhac's and Christodoulou's frameworks is that they are tailored
precisely to a single family of characteristics -- the family whose intersection
corresponds to a shock singularity. Thus, all prior works left open the question of whether or not
these approaches can be adapted to systems featuring multiple speeds of propagation.
As we mentioned above, first affirmative result in this direction was provided by our
joint works \cites{jLjS2016a,jLjS2016b} with J.~Luk,
in which we discovered some remarkable geo-analytic structures in the compressible Euler equations
with vorticity. Inspired by these structures,
we developed an extended version of Christodoulou's framework,
and we used it to prove shock formation for solutions to
a quasilinear \emph{wave-transport} system. More precisely, 
the wave-transport system that we studied in \cites{jLjS2016a,jLjS2016b}
was a new formulation of the compressible Euler equations, 
where the velocity components and density satisfied a system of covariant wave equations, 
\emph{all with the same covariant wave operator $\square_g$}
(corresponding to a single Lorentzian metric $g$),
and the vorticity satisfied a (first-order) transport equation.
There were two speeds in the system: the speed of sound, corresponding to sound wave propagation,
and the speed associated to the transporting of vorticity. 
A particularly remarkable aspect of the equations studied in \cites{jLjS2016a,jLjS2016b}, which is central to the proofs, 
is that the inhomogeneous terms had a good null structure
that did not interfere with the shock formation processes.
The null structures, which are fully nonlinear in nature, 
are a tensorial generalization of the good null structure enjoyed by the
standard null form $\mathcal{Q}^g(\partial \Psi,\partial \Psi)$,
which is an admissible term in our systems \eqref{E:FASTWAVE}-\eqref{E:SLOWWAVE}
(see just below those equations for further discussion on this point).

In our works \cites{jLjS2016a,jLjS2016b}, to control the 
wave-transport solution's derivatives, 
we followed the approach of \cite{dC2007} 
(itself inspired by the Christodoulou--Klainerman proof \cite{dCsK1993} of the stability of Minkowski spacetime)
and constructed a family of dynamic geometric objects, including geometric vectorfields, 
adapted to the characteristics\footnote{More precisely, as we describe in 
Subsect.\ \ref{SS:OVERVIEWOFPROOF}, the vectorfields were adapted to an eikonal function corresponding to $g$,
whose level sets are $g$-null hypersurfaces.} 
of $g$. A seemingly unavoidable aspect of our approach in \cites{jLjS2016a,jLjS2016b} 
was that, due to the coupled nature of the system, we were forced to
commute the transport equation with the \emph{same geometric vectorfields}
in order to obtain estimates for the solution's derivatives.
In general, one might expect to encounter crippling error terms
from this procedure, since the geometric vectorfields are not
adapted to the transport operator. What allowed our proof to go through
are the following facts: 
\textbf{i)} transport operators are first-order
and \textbf{ii)} the operators $\upmu \partial_{\alpha}$
exhibit good commutation properties with the geometric vectorfields,
where $\partial_{\alpha}$ is a Cartesian coordinate partial derivative vectorfield
and $\upmu > 0$, mentioned above, is the inverse foliation density of the wave characteristics
(which we rigorously define in Subsubsect.\ \ref{SSS:GEOMETRICINGREDIENTS} 
since $\upmu$ plays a critical role in the present work as well).
Therefore, since the transport operator was just a 
(solution-dependent) linear combination of the $\partial_{\alpha}$,
upon multiplying the transport equation by $\upmu$
and commuting it with the geometric vectorfields,
we were able to completely avoid the worst imaginable commutator error terms, 
which enabled us to close the proof. 

We now stress that our approach in \cites{jLjS2016a,jLjS2016b}
does not allow one to commute the geometric vectorfields through
typical second-order operators
$\upmu \partial_{\alpha} \partial_{\beta}$; this would typically generate
crippling commutator error terms\footnote{Using a weight with a different power of $\upmu$, such as
$\upmu^2 \partial_{\alpha} \partial_{\beta}$, also seems to lead to insurmountable difficulties.} 
featuring a factor of $1/\upmu$,
which blows up as $\upmu \to 0$ and obstructs the goal of deriving regular estimates.
In particular, in itself, the approach of 
\cites{jLjS2016a,jLjS2016b}
does not manifestly allow one to couple an additional
quasilinear wave equation with a ``new'' metric $h$ that is different from $g$.
Here it makes sense to clarify the following point: the crippling error terms
do \emph{not} arise when commuting the geometric vectorfields through $\square_g$
(since the vectorfields are adapted to the characteristics of $g$), but they \emph{do} arise
when commuting them through a typical second-order differential operator.
Thus, the works \cites{jLjS2016a,jLjS2016b} left open the question of how to 
prove shock formation for solutions to second-order quasilinear 
systems featuring more than one wave operator.
As we have mentioned, in the present article, we prove the first shock formation results 
for systems of this type. The following key idea, mentioned earlier,
lies at the heart of our approach here.
\begin{quote}
\emph{It is possible to formulate the wave equation for the 
non-shock-forming slow variable as a first-order system
that can be treated using an extension of the approach 
of \cite{jLjS2016b}}; see equations \eqref{E:SLOW0EVOLUTION}-\eqref{E:SYMMETRYOFMIXEDPARTIALS}.
\end{quote}



\subsection{Remarks on the nonlinear terms and extending the results to related systems}
\label{SS:EXTENDINGRESULTORELATEDSYSTEMS}
The formation of shocks exhibited by Theorem~\ref{T:ROUGHMAINTHM}
is of course tied to our structural assumption on the nonlinearities,
which we precisely describe in Subsect.\ \ref{SS:SYSTEMSUNDERSTUDY}.
As we mentioned earlier,
the blowup of $\Psi$ is driven by the presence of
a Riccati-type interaction term in its wave equation,
which is captured by our assumption \eqref{E:NONVANISHINGNONLINEARCOEFFICIENT} below.
For this reason, the wave equation of $\Psi$ can be caricatured as\footnote{Throughout, if $V$ is a vectorfield
and $f$ is a scalar function, 
then $V f := V^{\alpha} \partial_{\alpha} f$ 
denotes the $V$-directional derivative of $f.$}
$\Lunit_{(Flat)} \partial_1 \Psi \sim (\partial_1 \Psi)^2 + \mbox{\upshape Error}$,
where $\Lunit_{(Flat)} := \partial_t + \partial_1$ and
$\mbox{\upshape Error}$ depends on $\Psi$, $\slow$, and their derivatives
(and in particular $\mbox{\upshape Error}$ contains the quasilinear interaction terms).
Although this caricature wave equation suggests that $\partial_1 \Psi$
should blow up in finite time along the integral curves of 
$\Lunit_{(Flat)}$, this is not how our proof works.
It seems that in order to close the energy estimates and to show that
error terms do not interfere with the blowup, one needs to derive
very sharp estimates tailored to the family of characteristics
corresponding to $\Psi$, which are in turn influenced by $\Psi$
in view of the quasilinear nature\footnote{The metric $g$ in the wave equation
for $\Psi$ is such that $g=g(\Psi)$. In particular,
$g$ does not depend on $\slow$ and thus the characteristics corresponding
to $g$ are not directly influenced by $\slow$.} 
of the equation. 
We also stress that, as in prior shock formation results, 
our proof is more sensitive to perturbations
of the equations than typical proofs of global existence.
This is not surprising in view of the fact that adding terms
of the form, say $\pm (\partial_1 \Psi)^3$, to the RHS of the
above caricature equation can drastically alter the global behavior
of its solutions. In contrast, since $\slow$ and $\partial_{\alpha} \slow$
remain bounded up to the shock, our approach is able to accommodate
essentially arbitrary semilinear terms comprised of products of these variables;
see Subsect.\ \ref{SS:SYSTEMSUNDERSTUDY} for our precise assumptions on the nonlinearities.

For convenience, we have chosen not to treat the most general
type of system to which our approach applies.
Our approach is flexible in the sense
that it could be used to treat systems featuring additional
wave equations, transport equations, or
symmetric hyperbolic equations. However, the following assumptions play a critical role in our analysis.
\begin{itemize}
	\item The Lorentzian metric $g$ corresponding to the principal part of the wave equation
		of $\Psi$ depends only on $\Psi$. That is, in the wave equation \eqref{E:FASTWAVE} below,
		$g = g(\Psi)$.
		More generally, we could allow for 
		$g = g(\Psi_1,\cdots,\Psi_m)$, as long
		as the \emph{same metric} $g$ corresponds to the principal part of the 
		wave equation of $\Psi_i$ for $1 \leq i \leq m$.
		This assumption is needed to control the top-order derivatives of the 
		eikonal function corresponding to $g$
		(see the discussion of modified quantities in Subsubsect.\ \ref{SSS:ENERGYESTIMATES}
		for more details on this point).
	\item The shock-forming variable $\Psi$ corresponds to 
		the fastest speed (in the strict sense)
		in the system. This assumption implies that the null hypersurfaces $\mathcal{P}_u$
		corresponding to the metric $g(\Psi)$ are \emph{spacelike}
		from the perspective of the principal parts of the remaining equations 
		in the system. This is important because our proof requires the availability 
		of positive definite energies for the slow wave solution variables
		along $g$-null hypersurfaces. In the present article, the positive definiteness of the energies
		for the slow wave $\slow$ along these hypersurfaces is guaranteed by 
		the estimates in equation \eqref{E:SLOWNULLFLUXCOERCIVENESS}.
	\item See Remark~\ref{R:DIFFERENTFASTEQUATION} below for a discussion of other types of
		``fast'' wave equations for which we could prove a stable shock formation result.
\end{itemize}

\begin{remark}
In Subsect.\ \ref{SS:SYSTEMSUNDERSTUDY} 
we will make further assumptions on the nonlinearities
and quantify the assumption that $\Psi$ is the fast wave.
\end{remark}

\subsection{Basic notational and index conventions}
\label{SS:NOTATIONANDINDEXCONVENTIONS}
We now summarize some of our notation. Some of the concepts referred to here
are defined later in the article.
Throughout, $\lbrace x^{\alpha} \rbrace_{\alpha =0,1,2}$
denote the standard Cartesian coordinates
on the spacetime $\mathbb{R} \times \Sigma$,
where $x^0 \in \mathbb{R}$ is the time variable and 
$(x^1,x^2) \in \Sigma = \mathbb{R} \times \mathbb{T}$ are the space variables,
chosen such that $\partial_2$ is positively oriented.
We denote the corresponding partial derivative vectorfields by
$
\displaystyle
\partial_{\alpha}
=:
\frac{\partial}{\partial x^{\alpha}}
$
(which are globally defined and smooth even though $x^2$ is only locally defined),
and we often use the alternate notation $t := x^0$ and $\partial_t := \partial_0$.

\begin{itemize}
	\item Lowercase Greek spacetime indices 
	$\alpha$, $\beta$, etc.\
	correspond to the Cartesian spacetime coordinates 
	and vary over $0,1,2$.
	Lowercase Latin spatial indices
	$a$,$b$, etc.\ 
	correspond to the Cartesian spatial coordinates and vary over $1,2$.
	We use tilded indices such as $\widetilde{\alpha}$ in the same way that
	we use their non-tilded counterparts.
	All lowercase Greek indices are lowered and raised with the fast wave spacetime metric
	$g$ and its inverse $g^{-1}$, and \emph{not with the Minkowski metric}.
\item We use Einstein's summation convention in that repeated indices are summed over their respective ranges.
\item We sometimes use $\cdot$ to denote the natural contraction between two tensors
		(and thus raising or lowering indices with a metric is not relevant for this contraction). 
		For example, if $\xi$ is a spacetime one-form and $V$ is a 
		spacetime vectorfield,
		then $\xi \cdot V := \xi_{\alpha} V^{\alpha}$.
\item If $\xi$ is a one-form and $V$ is a vectorfield, then
	$\xi_V := \xi_{\alpha} V^{\alpha}$. 
	Similarly, if $W$ is a vectorfield, then
	$W_V := W_{\alpha} V^{\alpha} = g(W,V)$.
\item If $\xi$ is an $\ell_{t,u}$-tangent one-form
	(as defined in Subsect.\ \ref{SS:PROJECTIONTENSORFIELDANDPROJECTEDLIEDERIVATIVES}),
	then $\xi^{\#}$ denotes its $\gsphere$-dual vectorfield,
	where $\gsphere$ is the Riemannian metric induced on $\ell_{t,u}$ by $g$.
	Similarly, if $\xi$ is a symmetric type $\binom{0}{2}$ $\ell_{t,u}$-tangent tensor, 
	then $\xi^{\#}$ denotes the type $\binom{1}{1}$ $\ell_{t,u}$-tangent tensor formed by raising one index with $\ginversesphere$
	and $\xi^{\# \#}$ denotes the type $\binom{2}{0}$ $\ell_{t,u}$-tangent tensor formed by raising both indices with $\ginversesphere$.
\item If $\xi$ is an $\ell_{t,u}$-tangent tensor, then 
	the norm $|\xi|$ is defined relative to the Riemannian metric $\gsphere$, 
	as we make precise in Def.~\ref{D:POINTWISENORM}.
\item Unless otherwise indicated, 
	all quantities in our estimates that are not explicitly under
	an integral are viewed as functions of 
	the geometric coordinates $(t,u,\vartheta)$
	of Def.~\ref{D:GEOMETRICCOORDINATES}.
	Unless otherwise indicated, 
	integrands have the functional dependence 
	established below in
	Def.~\ref{D:NONDEGENERATEVOLUMEFORMS}.
\item If $Q_1$ and $Q_2$ are two operators, then
	$[Q_1,Q_2] = Q_1 Q_2 - Q_2 Q_1$ denotes their commutator.
\item $A \lesssim B$ means that there exists $C > 0$ such that $A \leq C B$.
\item $A \approx B$ means that $A \lesssim B$ and $B \lesssim A$.
\item $A = \mathcal{O}(B)$ means that $|A| \lesssim |B|$.
\item Constants such as $C$ and $c$ are free to vary from line to line.
	\textbf{These constants, as well as implicit constants, 
	are allowed to depend in an increasing, 
	continuous fashion on the data-size parameters 
	$\mathring{\updelta}$
	and 
	$\TranminusdatasizeWithFactor^{-1}$
	from
	Subsect.\ \ref{SS:DATAASSUMPTIONS}.
	However, the constants can be chosen to be 
	independent of the parameters 
	$\Psiep$,
	$\mathring{\upepsilon}$,
	and $\varepsilon$ whenever the following conditions hold:
	\textbf{i)}
	$\mathring{\upepsilon}$
	and $\varepsilon$
	are sufficiently small relative to 
	$1$,
	sufficiently small relative to
	$\mathring{\updelta}^{-1}$,
	and sufficiently small relative to $\TranminusdatasizeWithFactor$,
	and 
	\textbf{ii)}
	$\Psiep$ is sufficiently small relative to $1$},
	in the sense described in Subsect.\ \ref{SS:SMALLNESSASSUMPTIONS}.
\item Constants $C_{\mydiam}$ are also allowed to vary from line to line, but
		unlike $C$ and $c$, the $C_{\mydiam}$ are
		\textbf{universal 
		in that, as long as $\Psiep$, $\mathring{\upepsilon}$, and $\varepsilon$ are sufficiently small relative to $1$,
		they do not depend on
		$\varepsilon$,
		$\mathring{\upepsilon}$,
		$\mathring{\updelta}$,
		or $\TranminusdatasizeWithFactor$}.
\item $A = \mathcal{O}_{\mydiam}(B)$ means that $|A| \leq C_{\mydiam} |B|$,
	with $C_{\mydiam}$ as above.
\item For example, $\TranminusdatasizeWithFactor^{-2} = \mathcal{O}(1)$,
		$2 + \Psiep + \Psiep^2 = \mathcal{O}_{\mydiam}(1)$,
		$\Psiep \varepsilon = \mathcal{O}(\varepsilon)$,
		$C_{\mydiam} \Psiep^2 =  \mathcal{O}_{\mydiam}(\Psiep)$,
		and $C \Psiep = \mathcal{O}(1)$; some of these examples are non-optimal.
\item $\lfloor \cdot \rfloor$
	and $\lceil \cdot \rceil$
	respectively denote the standard floor and ceiling functions. 
\end{itemize}

\subsection{The systems under study}
\label{SS:SYSTEMSUNDERSTUDY}
\subsubsection{Statement of the equations}
\label{SSS:STATEMENTOFEQUATIONS}
For notational convenience, we introduce the following array associated to the slow wave:
\begin{align} \label{E:SLOWWAVEVARIABLES}
	\bigslow
	& := (\slow,\slow_0,\slow_1,\slow_2),
	&
	\slow_{\alpha} 
	&:= \partial_{\alpha} w,
	&& (\alpha = 0,1,2),
\end{align}
where we again stress that $\partial_{\alpha}$ denotes a Cartesian coordinate partial derivative vectorfield.

\begin{remark}[\textbf{Remark on the pointwise norm of the array $\bigslow$}]
	\label{R:POINTWISENORMOFSLOWVARIABLEARRAY}
	Throughout the article, we view $\bigslow$ 
	to be an array of scalar functions without tensorial structure. Thus, 
	there should be no danger of confusing the definition $| \bigslow |^2 := \slow^2 + \sum_{\alpha = 0}^2 \slow_{\alpha}^2$
	with the definition \eqref{E:POINTWISENORM} below for the pointwise norm 
	$|\cdot|$
	of an $\ell_{t,u}$-tangent tensor.
\end{remark}

\begin{center}
	{\large \underline{\textbf{The system of wave equations under study}}}
\end{center}
Our main results concern the following system of two wave equations:
\begin{subequations}
\begin{align}
	\square_{g(\Psi)} \Psi 
	& = 
		\mathfrak{M}(\Psi,\bigslow) \mathcal{Q}^g(\partial \Psi,\partial \Psi)
		+
		\mathfrak{N}_1^{\alpha}(\Psi,\bigslow) \partial_{\alpha} \Psi
		+ 
		\mathfrak{N}_2(\Psi,\bigslow),
		\label{E:FASTWAVE} \\
	(h^{-1})^{\alpha \beta}(\Psi,\bigslow)  \partial_{\alpha} \partial_{\beta} \slow 
	& = 
		\widetilde{\mathfrak{M}}(\Psi,\bigslow)  \mathcal{Q}^g(\partial \Psi,\partial \Psi)
		+
		\widetilde{\mathfrak{N}}_1^{\alpha}(\Psi,\bigslow) \partial_{\alpha} \Psi
		+ 
		\widetilde{\mathfrak{N}}_2(\Psi,\bigslow).
		\label{E:SLOWWAVE}
	\end{align}
\end{subequations}

Above and throughout, 
$g$ and $h$ are, by assumption, Lorentzian metrics for small values of their arguments
(see below for our precise assumptions),
$\square_{g(\Psi)}$
is the covariant wave operator\footnote{Relative to arbitrary coordinates,
$\square_g f
= 
\frac{1}{\mbox{$\sqrt{|\mbox{\upshape det} g|}$}}
\partial_{\alpha}\left(\sqrt{|\mbox{\upshape det} g|} (g^{-1})^{\alpha \beta} \partial_{\beta} f \right)$.
\label{FN:COVWAVEOPARBITRARYCOORDS}}  
of $g(\Psi)$,
$\mathfrak{M}$,
$\mathfrak{N}_1^{\alpha}$, 
$\cdots$, and $\widetilde{\mathfrak{N}}_2$ 
are smooth nonlinear terms described below,
and
\begin{align} \label{E:STANDARDNULLFORM}
	 \mathcal{Q}^g(\partial \Psi,\partial \Psi)
	 := (g^{-1})^{\alpha \beta}(\Psi)	\partial_{\alpha} \Psi \partial_{\beta} \Psi
\end{align}
is the standard null form associated to $g$. 
It is important for our proof that the wave operator
of the shock-forming variable 
$\Psi$ is covariant, the reason being that the geometric
vectorfields that we construct exhibit good commutation properties with\footnote{More precisely, they exhibit good
commutation properties with $\upmu \square_g$, where we define $\upmu$ in Def.~\ref{D:FIRSTUPMU}.} 
the operator $\square_g$.
We stress that
$\mathcal{Q}^g(\partial \Psi,\partial \Psi)$ is,
from the point of view of closing our estimates,
the only allowable term on RHSs~\eqref{E:FASTWAVE}-\eqref{E:SLOWWAVE}
that is quadratic in $\partial \Psi$. The reason is that 
$\mathcal{Q}^g(\partial \Psi,\partial \Psi)$ has the following special nonlinear structure:
it is \emph{linear} in the tensorial component of $\partial \Psi$ that blows up;
see \eqref{E:UPMUTIMESNULLFORMSSCHEMATIC} for the geo-analytic statement of this fact.
More precisely, upon decomposing $\mathcal{Q}^g(\partial \Psi,\partial \Psi)$ relative to an appropriate frame,
we find find that it is linear in a derivative of $\Psi$ in 
a direction that is transversal to the $g$-characteristics.
The key point is that it is precisely the transversal derivative of $\Psi$ that blows up,
while the derivatives of $\Psi$ in directions tangential to the $g$-characteristics remain
uniformly bounded\footnote{Except possibly at the high derivative levels, due to the degenerate
high-order energy estimates that we derive; see Subsubsect.\ \ref{SSS:ENERGYESTIMATES}.} 
all the way up to the singularity.
For this reason, such terms have only a negligible effect on the dynamics
all the way up to the shock, at least compared to the Riccati-type term 
that is quadratic in the transversal derivatives of $\Psi$ and that
drives the singularity formation. Note that this Riccati-type term
becomes visible only if we expand the expression $\square_{g(\Psi)} \Psi$ on LHS~\eqref{E:FASTWAVE}
relative to Cartesian coordinates.
We refer readers to \cite{jLjS2016a} for further discussion of these issues,
noting only that the good structure of 
$\mathcal{Q}^g(\partial \Psi,\partial \Psi)$ is referred to as 
\emph{the strong null condition} (relative to $g$) in \cite{jLjS2016a}.
Note that inhomogeneous terms that are quadratic or higher-order in $\partial \Psi$
typically have the following property:
they are at least quadratic in the derivatives of $\Psi$ in directions transversal to the $g$-characteristics.
Such terms are too singular to be included on RHSs \eqref{E:FASTWAVE}-\eqref{E:SLOWWAVE}
within our framework and in fact, might introduce instabilities that 
prevent a shock from forming or, alternatively, 
that generate a completely different kind of blowup.
In particular, like all prior works on shock formation for wave equations, our proof is unstable against
the addition of cubic terms $(\partial \Psi)^3$ to the equations, and similarly for
terms that are higher-order in $\partial \Psi$.

\begin{remark}[\textbf{Extending the result to a different type of fast wave equation}]
	\label{R:DIFFERENTFASTEQUATION}
	Instead of studying equation \eqref{E:FASTWAVE}, we could alternatively 
	prove a shock-formation result for ``fast''
	non-covariant quasilinear wave equations of the form
	\begin{align} \label{E:NONCOVWAVEEQUATION}
	(g^{-1})^{\alpha \beta}(\partial \Phi)  \partial_{\alpha} \partial_{\beta} \Phi
		= 
		\mathfrak{N}(\Phi,\partial \Phi,\slow),
	\end{align}
	where 
	$\mathfrak{N}(\cdot)$ is a smooth function of its arguments such that
	$
	\displaystyle
	\mathfrak{N}(\Phi,\partial \Phi,0) 
	= 
	0
	$
	for $\nu = 0,1,2$.
	As is explained in \cite{jS2016b}, 
	to treat equations of type \eqref{E:NONCOVWAVEEQUATION},
	one could first differentiate equation\footnote{More generally, 
	our approach could be extended
	to allow for $(g^{-1})^{\alpha \beta}= (g^{-1})^{\alpha \beta}(\Phi,\partial \Phi)$
	in equation \eqref{E:NONCOVWAVEEQUATION}.} 
	\eqref{E:NONCOVWAVEEQUATION}
	with the Cartesian coordinate partial derivatives $\partial_{\nu}$
	to obtain a system of type
	\eqref{E:FASTWAVE}-\eqref{E:SLOWWAVE}
	in the unknowns 
	$\Phi$,
	$\vec{\Psi}$,
	and
	$\bigslow$ that obey the semilinear inhomogeneous term assumptions stated in Subsubsect.\ \ref{SSS:ASSUMPTIONSONREMAINNGNONLINEARITIES},
	where $\vec{\Psi} := (\Psi_0,\Psi_1,\Psi_2) := (\partial_0 \Phi, \partial_1 \Phi, \partial_2 \Phi)$
	and $g = g(\vec{\Psi})$. More precisely, in \cite{jS2016b}, we showed that the 
	scalar functions $\Psi_{\nu}$ satisfy a system of covariant wave equations of type
	\eqref{E:FASTWAVE}, where the terms that are quadratic in 
	$\partial \vec{\Psi}$ exhibit the same kind of good null structure as the 
	standard $g$-null form \eqref{E:STANDARDNULLFORM}.
	The assumption
	$
	\displaystyle
	\mathfrak{N}(\Phi,\partial \Phi,0) 
	= 
	0
	$
	guarantees that the system admits simple outgoing plane wave solutions in which $\slow \equiv 0$
	(see Subsubsect.\ \ref{SSS:ASSUMPTIONSONREMAINNGNONLINEARITIES} for further discussion). This assumption is convenient for our analysis. It
	could be weakened to allow for a larger class of semilinear terms $\mathfrak{N}$ such that
	the system no longer admits exact simple outgoing plane wave solutions. However, 
	compared to our main theorem,
	we generally would have to make different assumptions on the initial data 
	(adapted to $\mathfrak{N}$)
	to guarantee that a shock forms in finite time;
	see the discussion below equation 
	\eqref{E:SOMENONINEARITIESARELINEAR} for related remarks.
\end{remark}

\subsubsection{Assumptions on the remaining nonlinearities}
\label{SSS:ASSUMPTIONSONREMAINNGNONLINEARITIES}
We assume that relative to the Cartesian coordinates,
the nonlinearities 
$g_{\alpha \beta}(\cdot)$, 
$h_{\alpha \beta}(\cdot)$,
$\mathfrak{M}(\cdot)$, 
$\mathfrak{N}_1^{\alpha}(\cdot)$,
$\cdots$, 
$\widetilde{\mathfrak{N}}_2(\cdot)$ 
in the system \eqref{E:FASTWAVE}-\eqref{E:SLOWWAVE}
are given smooth functions of their arguments (for $|\Psi|$ and $|\bigslow|$ sufficiently small)
and that
\begin{align}  \label{E:SOMENONINEARITIESARELINEAR}
	\mathfrak{N}_1^{\alpha}(\Psi,0)
	& =
	\mathfrak{N}_2(\Psi,0)
	=
	\widetilde{\mathfrak{N}}_1^{\alpha}(\Psi,0)
	=
	\widetilde{\mathfrak{N}}_2(\Psi,0)
	= 0,
	&& (\alpha = 0,1,2).
\end{align}
That is, we assume that the semilinear terms in \eqref{E:SOMENONINEARITIESARELINEAR} 
vanish when $\bigslow = 0$.
The assumptions \eqref{E:SOMENONINEARITIESARELINEAR}
are such that the system
\eqref{E:FASTWAVE}-\eqref{E:SLOWWAVE}
admits simple outgoing plane wave solutions in which
$\slow \equiv 0$, $\Psi = \Psi(t,x^1)$, and
$\Psi$ is a ``right-moving'' wave, as opposed to 
being a combination of left- and right-moving waves;
see Subsect.\ \ref{SS:EXISTENCEOFDATA} for further discussion on this point.
Our main theorem concerns perturbations (without symmetry assumptions)
of these simple outgoing plane waves. Our results could be extended to allow
for additional kinds of semilinear terms on
RHSs \eqref{E:FASTWAVE}-\eqref{E:SLOWWAVE}, such as a Klein-Gordon term 
(i.e., a constant multiple of $\Psi$ on RHS~\eqref{E:FASTWAVE}) 
or products with the schematic structure $\Psi \partial \Psi$. However, in the presence of
these semilinear terms, the equations no longer admit 
simple outgoing plane wave solutions (aside from the trivial zero solution).
Consequently, our assumptions on the initial data 
(see Subsects.\ \ref{SS:DATAASSUMPTIONS} and \ref{SS:SMALLNESSASSUMPTIONS})
that lead to shock formation
would generally have to be adjusted to accommodate such new types of semilinear terms.\footnote{A good model equation
for understanding the subtleties in this analysis is
the inhomogeneous Burgers' equation $\partial_t \Psi + \Psi \partial_1 \Psi = \Psi^2$.
Roughly, for data such that $\Psi$ is initially small while $\partial_1 \Psi$ is initially large
in some region, the solution is such that $\partial_1 \Psi$ blows up along the characteristics
while $\Psi$ remains bounded (at least up to the first singularity in $\partial_1 \Psi$), 
much like in the case of the homogeneous Burgers' equation. However,
unlike the homogeneous Burger's equation,
the inhomogeneous equation also admits the $T$-parameterized family of ODE-type blowup 
solutions 
$
\displaystyle
\Psi_T(t) := \frac{1}{T-t}
$, whose singularity is at the level of $\Psi$ itself. \label{FN:INHOMOGENEOUSBURGERS}}
This would lengthen the article and obscure the new ideas that we aim to highlight here;
for this reason, we limit our study to semilinear terms that verify
\eqref{E:SOMENONINEARITIESARELINEAR}.

Regarding the fast wave metric $g$, we assume that 
\begin{align} \label{E:LITTLEGDECOMPOSED}
	g_{\alpha \beta} 
	= g_{\alpha \beta}(\Psi)
	& := m_{\alpha \beta} 
		+ g_{\alpha \beta}^{(Small)}(\Psi),
	&& (\alpha, \beta = 0,1,2),
\end{align}
where 
$m_{\alpha \beta} = \mbox{diag}(-1,1,1)$ is the standard Minkowski metric on $\mathbb{R} \times \Sigma$
(where $\Sigma$ is defined in \eqref{E:SPACEMANIFOLD})
and the Cartesian components $g_{\alpha \beta}^{(Small)}(\Psi)$ are given smooth functions of $\Psi$
such that
\begin{align} \label{E:METRICPERTURBATIONFUNCTION}
	g_{\alpha \beta}^{(Small)}(\Psi = 0)
	& = 0.
\end{align}
We also introduce the scalar functions
\begin{align} \label{E:BIGGDEF}
	G_{\alpha \beta}
	= G_{\alpha \beta}(\Psi)
	& := \frac{d}{d \Psi} g_{\alpha \beta}(\Psi),
	&&
	G_{\alpha \beta}'
	= G_{\alpha \beta}'(\Psi)
	:= \frac{d^2}{d \Psi^2} g_{\alpha \beta}(\Psi),
\end{align}
which appear throughout our analysis.
In order to ensure that shocks can form in solutions, 
including plane symmetric ones that depend only on $t$ and $x^1$,
we assume that
\begin{align} \label{E:NONVANISHINGNONLINEARCOEFFICIENT}
	G_{\alpha \beta}(\Psi = 0) \Lunit_{(Flat)}^{\alpha} \Lunit_{(Flat)}^{\beta} \neq 0,
\end{align}
where
\begin{align} \label{E:LFLAT}
	\Lunit_{(Flat)} : = \partial_t + \partial_1.
\end{align}
As is explained in \cite{jSgHjLwW2016},
these assumptions are essentially equivalent to the assumption
that the null condition 
\emph{fails to hold} for plane symmetric solutions
to the wave equation for $\Psi$.
Roughly, these assumptions 
ensure that for the solutions under study, the coefficient
of the main terms driving the blowup is non-zero;
as will become clear,
the main term is the first product on RHS~\eqref{E:UPMUFIRSTTRANSPORT}.

\begin{remark}[\textbf{Genuinely nonlinear systems}]
\label{R:GENUINELYNONLINEAR}
	Our assumption that the vectorfield \eqref{E:LFLAT}
	verifies \eqref{E:NONVANISHINGNONLINEARCOEFFICIENT}
	is similar to the well-known genuine nonlinearity condition for first-order strictly hyperbolic systems.
	In particular, for plane symmetric solutions with $\Psi$ sufficiently small, 
	the assumption \eqref{E:NONVANISHINGNONLINEARCOEFFICIENT}
	ensures that there are quadratic 
	Riccati-type terms in the fast wave equation \eqref{E:FASTWAVE}, 
	which become visible if one expands the LHS relative to the Cartesian coordinates.
	The Riccati-type terms provide essentially the same blowup-mechanism
	as the one that drives the blowup in solutions to $2 \times 2$
	genuinely nonlinear strictly hyperbolic systems, 
	which Lax studied in his well-known work \cite{pL1964}.
\end{remark}

As we mentioned above, a fundamental aspect of our proof is 
that we reformulate the slow wave equation \eqref{E:SLOWWAVE}
as a first-order system, 
which allows us to avoid certain top-order commutator error terms that we would have no means to control.
Specifically, we study the following first-order system 
which, under the assumption \eqref{E:ZEROZEROISMINUSONE} below, is easily seen to be a consequence of \eqref{E:SLOWWAVE},
$(i,j =1,2)$:
\begin{subequations}
\begin{align}
	\partial_t \slow_0
	& = 
		(h^{-1})^{ab}(\Psi,\bigslow) \partial_a \slow_b
		+ 
		2 (h^{-1})^{0a}(\Psi,\bigslow) \partial_a \slow_0
			\label{E:SLOW0EVOLUTION} \\
		& \ \
		- 
		\widetilde{\mathfrak{M}}(\Psi,\bigslow)  \mathcal{Q}^g(\partial \Psi,\partial \Psi)
		-
		\widetilde{\mathfrak{N}}_1^{\alpha}(\Psi,\bigslow) \partial_{\alpha} \Psi
		- 
		\widetilde{\mathfrak{N}}_2(\Psi,\bigslow),
		 \notag \\
	\partial_t \slow_i
	& = \partial_i \slow_0,
		\label{E:SLOWIEVOLUTION} \\
	\partial_t \slow
	& = \slow_0,
	\label{E:SLOWEVOLUTION}
		\\
	\partial_i \slow_j
	& = \partial_j \slow_i.
		\label{E:SYMMETRYOFMIXEDPARTIALS}
\end{align}
\end{subequations}
Note that \eqref{E:SYMMETRYOFMIXEDPARTIALS} can be viewed as a constraint 
representing the symmetry of the mixed partial derivatives of $\slow$ with respect to the Cartesian coordinates.
It is easy to check that the constraint \eqref{E:SYMMETRYOFMIXEDPARTIALS}, if verified at time $0$,
is propagated by the flow of equation \eqref{E:SLOWIEVOLUTION}.

\subsubsection{Assumptions tied to the wave speeds}
\label{SSS:WAVESPEEDASSUMPTIONS}
We now quantify our assumption that $\Psi$ is the fast wave and
$\slow$ is the slow wave.
\begin{center}
	\underline{\textbf{Assumption on the wave speeds}}
\end{center}

We assume that the following holds for non-zero vectors $V$
whenever $|\Psi| + |\bigslow|$ is sufficiently small:
\begin{align} \label{E:VECTORSHCAUSALIMPLIESGTIMELIKE}
	h_{\alpha \beta}
	V^{\alpha} V^{\beta} \leq 0
	\implies 
	g_{\alpha \beta} V^{\alpha} V^{\beta} < 0.
\end{align}

Note that \eqref{E:VECTORSHCAUSALIMPLIESGTIMELIKE} is equivalent
to the following implication, valid for non-zero co-vectors $\omega$:
\begin{align} \label{E:VECTORSGCAUSALIMPLIESHTIMELIKE}
	(g^{-1})^{\alpha \beta}
	\omega_{\alpha} \omega_{\beta} \leq 0
	\implies 
	(h^{-1})^{\alpha \beta}
	\omega_{\alpha} \omega_{\beta} < 0.
\end{align}

From \eqref{E:VECTORSGCAUSALIMPLIESHTIMELIKE} and the fact that
$(g^{-1})^{\alpha \beta}(\Psi = 0) 
= (m^{-1})^{\alpha \beta} 
= \mbox{\upshape diag}(-1,1,1)$
(see \eqref{E:LITTLEGDECOMPOSED}-\eqref{E:METRICPERTURBATIONFUNCTION}),
it follows that $(h^{-1})^{00}(\Psi,\bigslow) < 0$ 
whenever $|\Psi|$ and $|\bigslow|$ are sufficiently small.
For convenience, we rescale the metrics and equations by a positive conformal factor
so that the following holds
relative to the Cartesian coordinates:
\begin{align} \label{E:ZEROZEROISMINUSONE}
	(g^{-1})^{00}(\Psi)
	& = (h^{-1})^{00}(\Psi,\bigslow) 
	\equiv - 1.
\end{align}
The identities assumed in \eqref{E:ZEROZEROISMINUSONE} 
simplify many calculations but are in no way essential.
Note that in view of the definition of a covariant wave operator,
rescaling the metric $g$ introduces an additional semilinear inhomogeneous null form term
of the form $\mathfrak{M}(\Psi) \mathcal{Q}^g(\partial \Psi,\partial \Psi)$
on RHS~\eqref{E:FASTWAVE}. 
It turns out that due to its good null structure, 
this term does not have a substantial influence of the dynamics of the solutions
that we study in our main theorem. Note also that this new term already falls 
under the scope of the allowable terms 
on RHS~\eqref{E:FASTWAVE}.

%

\subsection{Overview of the proof of the main result}
\label{SS:OVERVIEWOFPROOF}
In this subsection, we provide an overview of the proof of 
our main result, Theorem~\ref{T:MAINTHEOREM}.
Our basic geometric setup is similar to the one 
pioneered by Christodoulou in his study of shock formation
in irrotational relativistic fluid mechanics \cite{dC2007}.

\subsubsection{Basic geometric ingredients}
\label{SSS:GEOMETRICINGREDIENTS}
As in all prior works on shock formation in more than one spatial dimension, 
to follow the solution all the way to the singularity in $\max_{\alpha=0,1,2} |\partial_{\alpha} \Psi|$,
we construct an eikonal function adapted to the metric $g(\Psi)$.
\begin{definition}[\textbf{Eikonal function}]
\label{D:INTROEIKONAL}
The eikonal function $u$ 
solves the eikonal equation initial value problem
\begin{subequations}
\begin{align} \label{E:INTROEIKONAL}
	(g^{-1})^{\alpha \beta}(\Psi)
	\partial_{\alpha} u \partial_{\beta} u
	& = 0, 
	\qquad \partial_t u > 0,
		\\
	u|_{\Sigma_0}
	& = 1 - x^1,
	\label{E:INTROEIKONALINITIALVALUE}
\end{align}
\end{subequations}
where $\Sigma_0 \simeq \mathbb{R} \times \mathbb{T}$ is the hypersurface of constant Cartesian time $0$.
\end{definition}

Our choice of initial conditions in \eqref{E:INTROEIKONALINITIALVALUE}
is adapted to the approximate plane symmetry of the data that we will consider.
The level sets of $u$ are $g$-null hypersurfaces, 
which we denote by $\mathcal{P}_u$ (see Def.~\ref{D:HYPERSURFACESANDCONICALREGIONS})
and which we often refer to as the \emph{characteristics}.
See Figure~\ref{F:FRAME} on pg.~\pageref{F:FRAME} 
for a depiction of the characteristics, where the characteristics
$\mathcal{P}_u^t$ in the figure have been truncated at time $t$.
We clarify that even though the system
\eqref{E:FASTWAVE}
+
\eqref{E:SLOW0EVOLUTION}-\eqref{E:SYMMETRYOFMIXEDPARTIALS}
features multiple speeds of propagation,
we study only the characteristic family $\lbrace \mathcal{P}_u \rbrace_{u \in [0,1]}$ in detail since,
for the data under consideration,
the intersection of distinct members of this family corresponds to the formation of a shock.

Using $u$, we will construct a collection of geometric
objects that can be used to derive sharp information about the solution.
The most important of these is the \emph{inverse foliation density} $\upmu$. Its vanishing
corresponds to the intersection of the characteristics
and, as it turns out 
(see Subsubsect.\ \ref{SSS:FORMATIONOFSHOCK}),
the formation of a singularity in $\max_{\alpha=0,1,2} |\partial_{\alpha} \Psi|$.

\begin{definition}[\textbf{Inverse foliation density}]
\label{D:FIRSTUPMU}
We define $\upmu > 0$ as follows:
\begin{align} \label{E:FIRSTUPMU}
	\upmu 
	& := \frac{-1}{(g^{-1})^{\alpha \beta}(\Psi) \partial_{\alpha} t \partial_{\beta} u},
\end{align}
where $t$ is the Cartesian time coordinate.
\end{definition}
Note that by \eqref{E:LITTLEGDECOMPOSED}-\eqref{E:METRICPERTURBATIONFUNCTION}
and the initial conditions \eqref{E:INTROEIKONALINITIALVALUE} for $u$,
we have $\upmu|_{\Sigma_0} = 1 + \mathcal{O}_{\mydiam}(\Psi)$
(see Subsect.\ \ref{SS:NOTATIONANDINDEXCONVENTIONS} regarding our use of the notation $\mathcal{O}_{\mydiam}(\cdot)$).
Thus, for the data that we consider in this article, in which
$|\Psi|$ is initially small, it follows that $\upmu$ is initially near unity.
In short, our main goal in the article is to exhibit 
an open set of data such that $\upmu$ vanishes in finite time,
to show that its vanishing is tied to 
the blowup of $\max_{\alpha=0,1,2} |\partial_{\alpha} \Psi|$,
and to show that $|\slow|$ and 
 $\max_{\alpha=0,1,2} |\partial_{\alpha} \slow|$
remain bounded.

The following spacetime subsets are tied to $u$ and
play a fundamental role in our analysis.
\begin{definition} [\textbf{Subsets of spacetime}]
\label{D:HYPERSURFACESANDCONICALREGIONS}
We define the following subsets of spacetime:
\begin{subequations}
\begin{align}
	\Sigma_{t'} & := \lbrace (t,x^1,x^2) \in \mathbb{R} \times \mathbb{R} \times \mathbb{T}  
		\ | \ t = t' \rbrace, 
		\label{E:SIGMAT} \\
	\Sigma_{t'}^{u'} & := \lbrace (t,x^1,x^2) \in \mathbb{R} \times \mathbb{R} \times \mathbb{T} 
		 \ | \ t = t', \ 0 \leq u(t,x^1,x^2) \leq u' \rbrace, 
		\label{E:SIGMATU} 
		\\
	\mathcal{P}_{u'}
	& := 
		\lbrace (t,x^1,x^2) \in \mathbb{R} \times \mathbb{R} \times \mathbb{T} 
			\ | \ u(t,x^1,x^2) = u' 
		\rbrace, 
		\label{E:PU} \\
	\mathcal{P}_{u'}^{t'} 
	& := 
		\lbrace (t,x^1,x^2) \in \mathbb{R} \times \mathbb{R} \times \mathbb{T} 
			\ | \ 0 \leq t \leq t', \ u(t,x^1,x^2) = u' 
		\rbrace, 
		\label{E:PUT} \\
	\ell_{t',u'} 
		&:= \mathcal{P}_{u'}^{t'} \cap \Sigma_{t'}^{u'}
		= \lbrace (t,x^1,x^2) \in \mathbb{R} \times \mathbb{R} \times \mathbb{T} 
			\ | \ t = t', \ u(t,x^1,x^2) = u' \rbrace, 
			\label{E:LTU} \\
	\mathcal{M}_{t',u'} & := \cup_{u \in [0,u']} \mathcal{P}_u^{t'} \cap \lbrace (t,x^1,x^2) 
		\in \mathbb{R} \times \mathbb{R} \times \mathbb{T}  \ | \ 0 \leq t < t' \rbrace.
		\label{E:MTUDEF}
\end{align}
\end{subequations}
\end{definition}
We refer to the $\Sigma_t$ and $\Sigma_t^u$ as ``constant time slices,'' 
the $\mathcal{P}_u$ and
$\mathcal{P}_u^t$ as ``characteristics'' or ``null hypersurfaces,''
and the $\ell_{t,u}$ as ``tori.'' 
Note that $\mathcal{M}_{t,u}$ is ``open-at-the-top'' by construction.

To study the solution, we complement $t$ and $u$ with a 
\emph{geometric torus coordinate} $\vartheta$ 
to form a geometric coordinate system $(t,u,\vartheta)$
with corresponding partial derivative vectorfields
$
\displaystyle
\left\lbrace
	\frac{\partial}{\partial t}, \frac{\partial}{\partial u}, \CoordAng := \frac{\partial}{\partial \vartheta}
\right\rbrace
$.
To differentiate the equations and obtain estimates for the solution's derivatives,
we also construct a related vectorfield frame 
\begin{align} \label{E:INTROGOODFRAME}
	\lbrace \Lunit, \Rad, \GeoAng \rbrace,
\end{align}
which spans the tangent space at each point where $\upmu > 0$.
The vectorfield $\Lunit$ verifies
$
\displaystyle
\Lunit = \frac{\partial}{\partial t}
$
and is null with respect to $g$, while
$\Rad$ and $\GeoAng$ are, respectively, replacements for 
$
\displaystyle
\frac{\partial}{\partial u}
$
and
$
\CoordAng
$
with better regularity properties that are needed 
to close the top-order energy estimates; see Subsect.\ \ref{SS:EIKONALFUNCTIONANDRELATED} 
for the details behind the construction of $\vartheta$ and the vectorfields.
We will prove that for the solutions under study, 
the vectorfields $\Lunit$ and $\GeoAng$ remain close to their background values,
which are respectively $\partial_t + \partial_1$ and $\partial_2$. In contrast, 
$\Rad$ behaves like $- \upmu \partial_1$
and thus shrinks as the shock forms. Moreover, 
$
\displaystyle
\Radunit 
:= 
\frac{1}{\upmu} \Rad
$
remains close to $- \partial_1$ all the way up to the shock.
See Figure~\ref{F:FRAME} on pg.~\pageref{F:FRAME} 
for a schematic depiction
of the vectorfields $\lbrace \Lunit, \Rad, \GeoAng \rbrace$
and note in particular that $|\Rad|$ is smaller in the region where $\upmu$ is small.
Note also that we have displayed the (outgoing) characteristics $\mathcal{P}_u$ of $g(\Psi)$ in the figure
but we have not displayed any characteristics of $h$ since they do not play a role in our analysis.
Moreover, $\Lunit$ and $\GeoAng$ are tangential to the characteristics $\mathcal{P}_u$
while $\Rad$ is transversal to them.
A key aspect of our proof is that
we will be able to derive \emph{uniform bounds}
for the $\Lunit$, $\GeoAng$, and $\Rad$ derivatives of the solution all the way up to the shock,
except near the top derivative level; as we describe in the discussion surrounding
\eqref{E:INTROTOPENERGY}-\eqref{E:INTROLOWESTENERGY},
our high-order geometric energies are allowed to blow up as the shock forms.
The fact that we can derive non-singular estimates for the low-level
$\Rad$ derivatives of the solution is fundamentally tied to the fact that $|\Rad|$
shrinks like $\upmu$ as $\upmu \to 0$. Note that this is compatible with the formation
of a singularity in $\max_{\alpha=0,1,2} |\partial_{\alpha} \Psi|$.
More precisely, 
our main theorem yields that
$
\displaystyle
|\Rad \Psi|
\gtrsim 1
$
near points where $\upmu$ is small and thus the derivative of $\Psi$
with respect to the order-unity vectorfield
$
\displaystyle
\Radunit 
:=
\frac{1}{\upmu}
\Rad
$
blows up precisely when $\upmu$ vanishes;
see Subsubsect.\ \ref{SSS:FORMATIONOFSHOCK} 
for a more detailed overview of this aspect of the proof.

\begin{center}
\begin{overpic}[scale=.35]{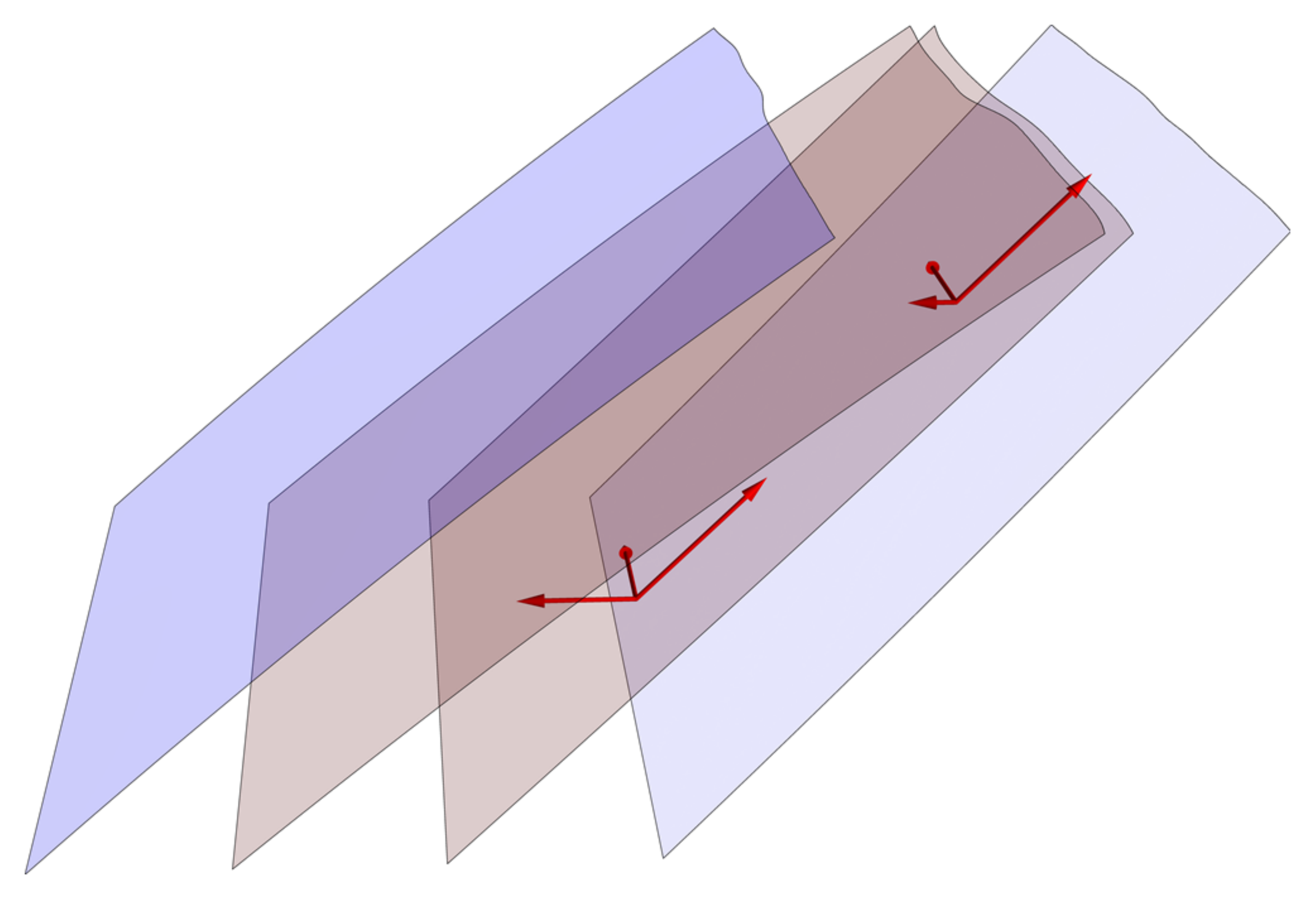} 
\put (83,55.5) {\large$\displaystyle \Lunit$}
\put (65.2,44.2) {\large$\displaystyle \Rad$}
\put (69.5,49.5) {\large$\displaystyle \GeoAng$}
\put (59,31) {\large$\displaystyle \Lunit$}
\put (35.8,21.5) {\large$\displaystyle \Rad$}
\put (46.2,28) {\large$\displaystyle \GeoAng$}
\put (51,13) {\large$\displaystyle \mathcal{P}_0^t$}
\put (37,13) {\large$\displaystyle \mathcal{P}_u^t$}
\put (7,13) {\large$\displaystyle \mathcal{P}_1^t$}
\put (22,26) {\large$\displaystyle \upmu \approx 1$}
\put (64,68) {\large$\displaystyle \upmu \ \mbox{\upshape small}$}
\end{overpic}
\captionof{figure}{The vectorfield frame from \eqref{E:INTROGOODFRAME} at two distinct points in $\mathcal{P}_u$}
\label{F:FRAME}
\end{center}

\subsubsection{The spacetime regions under study}
\label{SSS:SPACETIMEREGION}
For convenience, we study only the future portion of the solution 
that is completely determined by the data lying in the subset
$\Sigma_0^{U_0} \subset \Sigma_0$
of thickness $U_0$
and on a portion of the characteristic
$\mathcal{P}_0$,
where
\begin{align} \label{E:FIXEDPARAMETER}
0 < U_0 \leq 1 
\end{align}
is a parameter, fixed until Theorem~\ref{T:MAINTHEOREM}; see Figure~\ref{F:REGION}.
We will study spacetime regions such that $0 \leq u \leq U_0$,
where $u$ is the eikonal function from Def.~\ref{D:INTROEIKONAL}.
We have introduced the parameter $U_0$ because one would need to allow $U_0$ to vary 
in order to study the behavior of the solution up to the boundary of the maximal development,
as Christodoulou did in \cite{dC2007}*{Chapter 15}.
For brevity, we do not pursue this issue in the present article.

\begin{center}
\begin{overpic}[scale=.2]{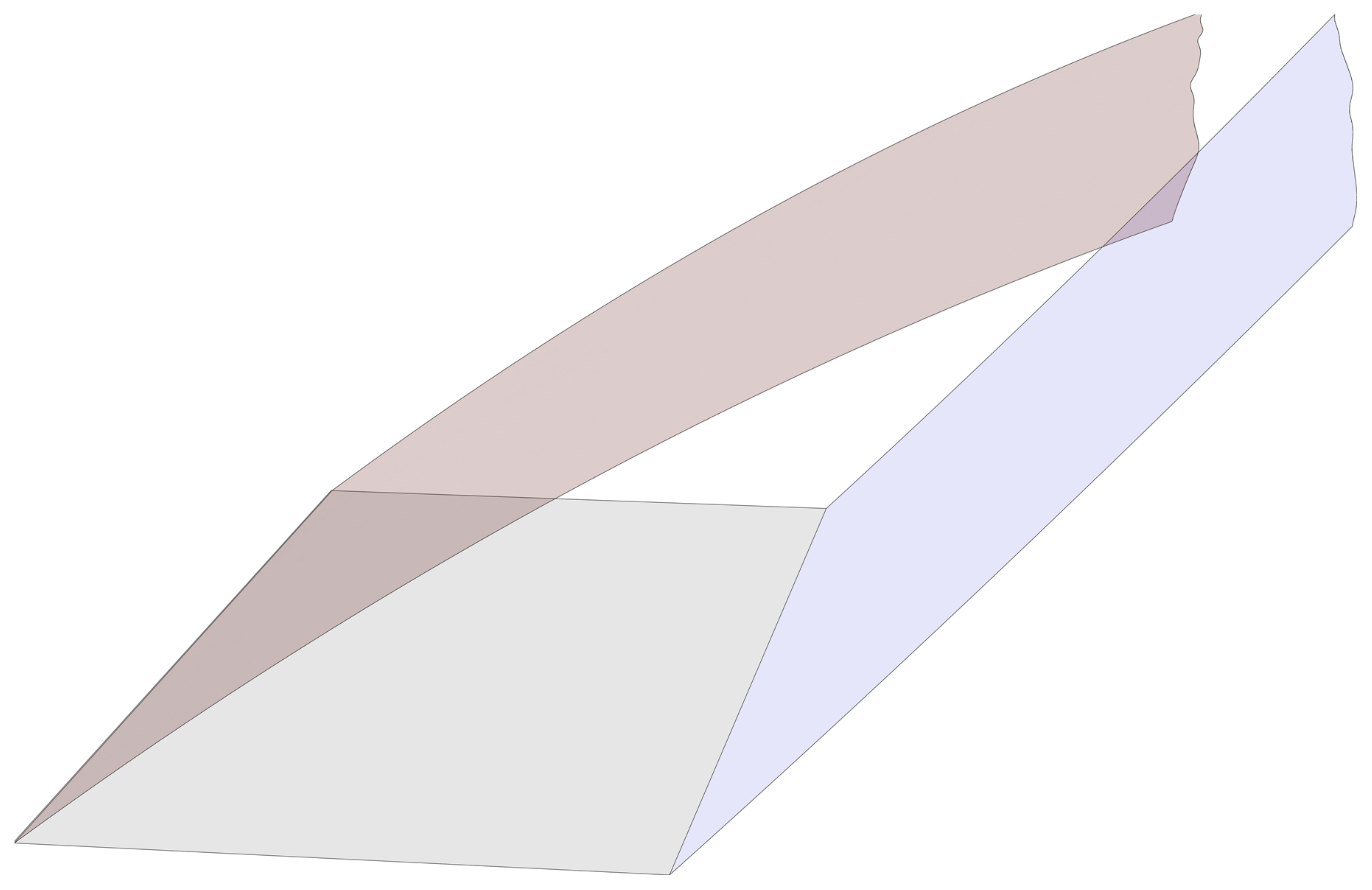} 
\put (37,33) {\large$\displaystyle \mathcal{P}_{U_0}^t$}
\put (74,33) {\large$\displaystyle \mathcal{P}_0^t$}
\put (18,5) {\large$\displaystyle \mbox{``interesting'' data}$}
\put (77,38) {\large \rotatebox{45}{$\displaystyle \mbox{very small data}$}}
\put (35,18) {\large$\Sigma_0^{U_0}$}
\put (31,10) {\large$\displaystyle U_0$}
\put (-.8,16) {\large$\displaystyle x^2 \in \mathbb{T}$}
\put (24,-3.2) {\large$\displaystyle x^1 \in \mathbb{R}$}
\thicklines
\put (-1.1,3){\vector(.9,1){22}}
\put (.5,1.8){\vector(100,-4.5){48}}
\put (10.5,13.9){\line(.9,1){2}}
\put (53.5,11.9){\line(.43,1){1}}
\put (11.5,15){\line(100,-4.5){42.5}}
\end{overpic}
\captionof{figure}{The spacetime region under study}
\label{F:REGION}
\end{center}

In our analysis, we will restrict our attention to times $t$
verifying $0 \leq t < 2 \TranminusdatasizeWithFactor^{-1}$, where
$\TranminusdatasizeWithFactor > 0$ is the data-dependent parameter
defined by
\begin{align} \label{E:INTROCRITICALBLOWUPTIMEFACTOR}
		\TranminusdatasizeWithFactor
		& := \frac{1}{2} 
		\sup_{\Sigma_0^1} 
		\left[G_{\Lunit \Lunit} \Rad \Psi \right]_-.
\end{align}
The quantity \eqref{E:INTROCRITICALBLOWUPTIMEFACTOR} is essentially the main term in the transport equation
for $\upmu$ (see \eqref{E:UPMUFIRSTTRANSPORT}) that drives $\upmu$ to $0$ in finite time.
In \eqref{E:INTROCRITICALBLOWUPTIMEFACTOR},
$G_{\Lunit \Lunit} := G_{\alpha \beta} \Lunit^{\alpha} \Lunit^{\beta}$,
where $G_{\alpha \beta}$ is defined in \eqref{E:BIGGDEF}
and $\Lunit$ is the $g$-null vectorfield mentioned in Subsubsect.\ \ref{SSS:GEOMETRICINGREDIENTS}
(see Def.~\ref{D:LUNITDEF} for the precise definition).
In our analysis, we take into account only the portion
of the data lying in the subset $\mathcal{P}_0^{2 \TranminusdatasizeWithFactor^{-1}}$
of the characteristic $\mathcal{P}_0$ since, by domain
of dependence considerations, only this portion can influence
the solution in the regions under study.
The parameter $\TranminusdatasizeWithFactor$ is important because
under certain assumptions described below, the time of first shock formation is
a small perturbation\footnote{For $\Psiep$ and $\mathring{\upepsilon}$ sufficiently small, the time of first shock formation
is $\lbrace 1 + \mathcal{O}_{\mydiam}(\Psiep) + \mathcal{O}(\mathring{\upepsilon}) \rbrace \TranminusdatasizeWithFactor^{-1}$,
where $\Psiep$ and $\mathring{\upepsilon}$ are the data-size parameters 
described in Subsubsect.\ \ref{SSS:DATAASSUMPTIONSANDBOOTSTRAPASSUMPTIONS}.} 
of $\TranminusdatasizeWithFactor^{-1}$.
We will clarify the connection between $\TranminusdatasizeWithFactor$ and the time of first shock formation in
Subsubsect.\ \ref{SSS:FORMATIONOFSHOCK}. 
Moreover, in view of the above remarks, we see that to close our bootstrap argument 
(which we briefly overview in Subsubsect.\ \ref{SSS:DATAASSUMPTIONSANDBOOTSTRAPASSUMPTIONS}), 
it is sufficient to control the solution
for times up to $2 \TranminusdatasizeWithFactor^{-1}$, 
which is plenty of time for the shock to form.

\subsubsection{A model problem: shock formation for nearly simple outgoing waves under the assumption of plane symmetry}
\label{SSS:NEARLYSIMPLEWAVES}
In this subsubsection, we illustrate some of the main ideas behind
our analysis by sketching a proof of our main results for
plane symmetric solutions, that is, solutions that depend only on
$t$ and $x^1$. For such solutions, we are able to 
rely exclusively on the method of characteristics
when deriving estimates. In particular, we can avoid energy estimates,
which drastically simplifies the proof. Our analysis in this subsubsection 
can be viewed as a sharpening of the approach of John \cite{fJ1974},
in the spirit of the recent work \cite{dCdRP2016}.

For convenience, we consider only the case in which the fast wave metric perturbation function
from \eqref{E:LITTLEGDECOMPOSED} takes the simple form
\begin{align} \label{E:MODELMETRICPERT}
g_{\alpha \beta}^{(Small)}(\Psi) 
= 
\left\lbrace
	(1 + \Psi)^2 - 1 
\right\rbrace
\delta_{\alpha}^1 \delta_{\beta}^1,
\end{align}
where $\delta_{\alpha}^{\beta}$ is the standard Kronecker delta.
Moreover, in this subsubsection only, we use, in addition to the vectorfield 
$\Lunit$, the vectorfield $\uLgood$ defined by
\begin{align} \label{E:ULGOODDEF}
	\uLgood 
	& := 
	\upmu \Lunit
	+ 
	2 \Rad.
\end{align}
It is easy to check that $g(\uLgood,\uLgood) = 0$ 
(that is, that $\uLgood$ is $g$-null)
and that the following relations hold
(these relations follow easily from Lemma~\ref{L:BASICPROPERTIESOFFRAME}):
\begin{align} \label{E:NULLVECGEOMETRICCOMP}
	\Lunit t 
	= 1,
		\qquad
	\Lunit u 
	= 0,
		\qquad
	\uLgood t
	= \upmu,
		\qquad
	\uLgood u
	= 2.
\end{align}
From the point of view of the estimates derived in this subsubsection, the vectorfield
$\uLgood$ plays a role similar to the one played by the
$\mathcal{P}_u$-transversal
vectorfield $\Rad$ that we use in the rest of the paper.
The advantage of $\uLgood$ in this subsubsection is that it is $g$-null
and thus the principal part of the fast wave equation takes a simple
form in plane symmetry when expressed in terms of $\Lunit$ and $\uLgood$ derivatives;
see equations \eqref{E:MODELFAST}-\eqref{E:SWITCHEDORDERMODELFAST}

As in the bulk of the paper, we will focus our attention here on nearly simple outgoing waves.
By a simple outgoing (that is, right-moving) plane wave, 
we mean a solution such that $\Lunit \Psi \equiv 0$ and $\slow \equiv 0$.
Due to our assumptions \eqref{E:SOMENONINEARITIESARELINEAR}
on the semilinear inhomogeneous terms on 
and RHSs~\eqref{E:FASTWAVE}-\eqref{E:SLOWWAVE},
the systems that we study in this paper admit simple plane wave solutions.

In the present subsubsection, we will consider 
plane symmetric initial data verifying 
a set of size assumptions.
Our assumptions involve the four parameters
$\Psiep$,
$\mathring{\upepsilon}$,
$\mathring{\updelta}$,
and
$\TrandatasizeWithFactor$,
which in this subsubsection only have slightly different (but analogous)
definitions than they do in the rest of the paper.
Specifically, we assume that the initial data for $\Psi$ and $\slow$ are 
given along $\Sigma_0^1$, which corresponds to the portion of $\Sigma_0$ with $0 \leq u \leq 1$,
as well as $\mathcal{P}_0^{2 \TrandatasizeWithFactor^{-1}}$, which 
is the portion of the level set $\lbrace u = 0 \rbrace$ with
$0 \leq t \leq 2 \TrandatasizeWithFactor^{-1}$,
where we define $\TrandatasizeWithFactor$ just below.
Note that in plane symmetry, $\Sigma_0^1$ can be identified with 
an orientation-reversed version of
the unit interval $[0,1]$ of $x^1$ values.
We assume the following size conditions, where all functions
on the LHSs of the inequalities are assumed to be continuous with respect
to the geometric coordinates $(t,u)$:
\begin{subequations}
\begin{align}
	\uLgood \Psi|_{\Sigma_0^1}
	& = f(u),
		\label{E:PLANESYMMETRYTHEONELARGEDATUM} \\
	\left\|
		\Psi
	\right\|_{L^{\infty}(\Sigma_0^1)}
	& \leq \Psiep,
		\label{E:PSIITSLFPLANESYMMETRYSIGMA01} \\
	\left\|
		\Lunit \Psi
	\right\|_{L^{\infty}(\Sigma_0^1)}
	& \leq \mathring{\upepsilon},
		\label{E:PLANESYMMETRYLUNITPSISMALLSIGMA01} \\
	\left\| 
		\slow
	\right\|_{L^{\infty}(\Sigma_0^1)},
		\,
	\left\| 
		\slow_0
	\right\|_{L^{\infty}(\Sigma_0^1)},
		\,
	\left\| 
		\slow_1
	\right\|_{L^{\infty}(\Sigma_0^1)}
	& \leq \mathring{\upepsilon},
		\label{E:PLANESYMMETRYWSMALLSIGMA01} \\
	\left\|
		\Psi
	\right\|_{L^{\infty}(\mathcal{P}_0^{2 \TrandatasizeWithFactor^{-1}})},
		\,
	\left\|
		\Lunit \Psi
	\right\|_{L^{\infty}(\mathcal{P}_0^{2 \TrandatasizeWithFactor^{-1}})}
	& \leq \mathring{\upepsilon},
		\label{E:PLANESYMMETRYALLPSISMALLPO} \\
	\left\| 
		\slow 
	\right\|_{L^{\infty}(\mathcal{P}_0^{2 \TrandatasizeWithFactor^{-1}})},
		\,
	\left\| 
		\slow_0
	\right\|_{L^{\infty}(\mathcal{P}_0^{2 \TrandatasizeWithFactor^{-1}})},
		\,
	\left\| 
		\slow_1
	\right\|_{L^{\infty}(\mathcal{P}_0^{2 \TrandatasizeWithFactor^{-1}})}
	& \leq \mathring{\upepsilon},
	\label{E:PLANESYMMETRYALLWSMALLPO}
\end{align}
\end{subequations}
where $f(u)$ is a continuous function and\footnote{In \eqref{E:MODELCASEBLOWUPTIMEPARAMETER}
and throughout, $[p]_- := |\min \lbrace p, 0 \rbrace|$.}
\begin{align} \label{E:MODELCASEBLOWUPTIMEPARAMETER}
	\TrandatasizeWithFactor
	& := \sup_{u \in [0,1]} [f(u)]_-.
\end{align}
Above,
$\Psiep > 0$ is a parameter that, for our subsequent bootstrap argument to close, must be small in an absolute sense,
while
$\mathring{\upepsilon} \geq 0$ is a parameter that must be 
small in an absolute sense,
small relative to
$\TrandatasizeWithFactor$, and small relative to
$\mathring{\updelta}^{-1}$,
where
\begin{align} 
\mathring{\updelta}
&:= \sup_{u \in [0,1]} |f(u)|.
\end{align}
In the remainder of this subsubsection, we will assume that
$\TrandatasizeWithFactor > 0$ and $\mathring{\updelta} > 0$.
When $\mathring{\upepsilon} = 0$, the corresponding solution
is a simple outgoing plane wave; see the end of this subsubsection for further discussion
of this point. 
It is straightforward to see that there exist data that verify the above size assumptions.
See Subsect.\ \ref{SS:EXISTENCEOFDATA} for further discussion on this point
in the context of our main theorem.

For convenience, in this subsubsection,
we study only the following specific example of 
a system of type
\eqref{E:FASTWAVE}
+
\eqref{E:SLOW0EVOLUTION}-\eqref{E:SYMMETRYOFMIXEDPARTIALS}
in one spatial dimension,
where the metric perturbation function
is given by \eqref{E:MODELMETRICPERT}
and, to simplify the discussion, 
we have chosen relatively simple semilinear terms: 
\begin{subequations}
\begin{align}
	\Lunit \uLgood \Psi
	& = 
		\Lunit \Psi \cdot \uLgood \Psi
		+
		\slow_0 \cdot \uLgood \Psi
		+ 
		\upmu \slow_0 \cdot \Psi,
			\label{E:MODELFAST} \\
		\uLgood \Lunit \Psi
		& = 
		\Lunit \Psi \cdot \uLgood \Psi
		+
		\upmu (\Lunit \Psi)^2
		+
		\slow_0 \cdot \uLgood \Psi
		+ 
		\upmu \slow_0 \cdot \Psi,
			\label{E:SWITCHEDORDERMODELFAST} \\
	\upmu \partial_t \slow_0
	& = 
		\frac{1}{4} \upmu \partial_1 \slow_1 
		+
		\Lunit \Psi \cdot \uLgood \Psi
		+
		\upmu \slow_0 \cdot \Psi,
			\label{E:MODELSLOW0} \\
	\upmu \partial_t \slow_1
	& 
	= \upmu \partial_1 \slow_0 .
	\label{E:MODELSLOW1}
\end{align}
\end{subequations}
We now make some remarks on the structure of equations \eqref{E:MODELFAST}-\eqref{E:MODELSLOW1}.
We have multiplied the equations by the inverse foliation density $\upmu$, 
which will help clarify certain aspects of the analysis.\footnote{Away from plane symmetry,
it is critically important to multiply the wave equations by $\upmu$ before commuting them
with appropriate vectorfields; the factor of $\upmu$ leads to important cancellations. In contrast,
our arguments in this subsubsection do not involve commuting the equations.}
The forms of LHSs
\eqref{E:MODELFAST}-\eqref{E:SWITCHEDORDERMODELFAST}
are a consequence of Prop.~\ref{P:GEOMETRICWAVEOPERATORFRAMEDECOMPOSED}.
In \eqref{E:MODELSLOW0}, the factor of
$
\displaystyle
\frac{1}{4}
$
accounts for our assumption that $\slow$ is the slow wave.
Note that equations \eqref{E:MODELSLOW0}-\eqref{E:MODELSLOW1}
are semilinear while for the general class of equations that we consider,
the analogous equations are typically quasilinear.
These facts play very little role in the discussion in this subsubsection.
In particular, in proving the main theorem of the paper,
we use that $\slow$ is the slow wave 
mainly when deriving energy estimates,
which we can avoid in this subsubsection
by integrating along characteristics.
That is, in this subsubsection, it is not
fundamentally important that $\slow$ is the slow wave.

To facilitate our analysis via integrating along characteristics, 
we now replace \eqref{E:MODELSLOW0}-\eqref{E:MODELSLOW1} with the following
equations,\footnote{Equations \eqref{E:REVAMPEDMODELSLOW0}-\eqref{E:REVAMPEDMODELSLOW1}
are evolution equations for the Riemann invariants of the subsystem
\eqref{E:MODELSLOW0}-\eqref{E:MODELSLOW1}.} 
which are equivalent up
to harmless constant factors on the right-hand sides:
\begin{subequations}
\begin{align}
	\upmu 
	(2 \partial_t + \partial_1)
	(\slow_0 - \frac{1}{2} \slow_1)
	& = \Lunit \Psi \cdot \uLgood \Psi
			+
			\upmu \slow_0 \cdot \Psi,
			\label{E:REVAMPEDMODELSLOW0} \\
	\upmu 
	(2 \partial_t - \partial_1)
	(\slow_0 + \frac{1}{2} \slow_1)
	& = 
		\Lunit \Psi \cdot \uLgood \Psi
		+
		\upmu \slow_0 \cdot \Psi.
	\label{E:REVAMPEDMODELSLOW1}
\end{align}
\end{subequations}

In our subsequent analysis, we will rely on the following
relations, which are simple consequences of Lemma~\ref{L:CARTESIANVECTORFIELDSINTERMSOFGEOMETRICONES},
our assumption that the metric perturbation is given by \eqref{E:MODELMETRICPERT},
the normalization condition $g(\Radunit,\Radunit) = 1$ (see \eqref{E:RADIALVECTORFIELDSLENGTHS}),
our assumption of plane symmetry,
and the fact that (under these assumptions) the vectorfield $\GeoAng$
verifies $\GeoAng = \partial_2$:
\begin{align} \label{E:PLANESYMMETRICCHOV}
	2 \partial_t
	& = 
	\Lunit
	+
	\frac{1}{\upmu} 
	\uLgood,
	\qquad
	2 \partial_1
	=
	(1 + \Psi)
	\left\lbrace
		\Lunit
		-
		\frac{1}{\upmu} 
		\uLgood
	\right\rbrace.
\end{align}
Using \eqref{E:PLANESYMMETRICCHOV},
we can replace
\eqref{E:REVAMPEDMODELSLOW0}-\eqref{E:REVAMPEDMODELSLOW1}
with the following equations, which are again equivalent
to \eqref{E:MODELSLOW0}-\eqref{E:MODELSLOW1}
up to harmless constant factors on the right-hand sides:
\begin{subequations}
\begin{align}
	\left\lbrace
		(1 - \Psi) \uLgood 
		+ 
		\upmu (3 + \Psi) \Lunit
	\right\rbrace
	(\slow_0 - \frac{1}{2} \slow_1)
	& = 
			\Lunit \Psi \cdot \uLgood \Psi
			+
			\upmu \slow_0 \cdot \Psi,
			\label{E:AGAINREVAMPEDMODELSLOW0} \\
	\left\lbrace
		(3 + \Psi) \uLgood 
		+ 
		\upmu (1 - \Psi) \Lunit
	\right\rbrace
	(\slow_0 + \frac{1}{2} \slow_1)
	& = 
		\Lunit \Psi \cdot \uLgood \Psi
		+
		\upmu \slow_0 \cdot \Psi.
	\label{E:AGAINREVAMPEDMODELSLOW1}
\end{align}
\end{subequations}
In plane symmetry, the most important aspect of
LHSs \eqref{E:AGAINREVAMPEDMODELSLOW0}-\eqref{E:AGAINREVAMPEDMODELSLOW1}
are that for $|\Psi|$ small,
the vectorfields
$(1 - \Psi) \uLgood 
		+ 
\upmu (3 + \Psi) \Lunit$
and
$
(3 + \Psi) \uLgood 
+ 
\upmu (1 - \Psi) \Lunit
$ 
are transversal to the $\mathcal{P}_u$, a simple fact that
follows from the
identities $\uLgood u = 2$ 
and
$\Lunit u = 0$
(see \eqref{E:NULLVECGEOMETRICCOMP}).

We now note that $\upmu$ (which is defined in \eqref{E:FIRSTUPMU})
verifies an evolution equation that we can schematically express as follows
(see \eqref{E:UPMUFIRSTTRANSPORT} for the precise formula):
\begin{align} \label{E:UPMUSIMPLEPLANEWAVESCHEMATICEVOLUTION}
	\Lunit \upmu
	& = 
		\uLgood \Psi
		+ 
		\upmu \Lunit \Psi.
\end{align}
In total, 
we will study the system 
\eqref{E:MODELFAST}-\eqref{E:SWITCHEDORDERMODELFAST} 
+ 
\eqref{E:AGAINREVAMPEDMODELSLOW0}-\eqref{E:AGAINREVAMPEDMODELSLOW1}
+
\eqref{E:UPMUSIMPLEPLANEWAVESCHEMATICEVOLUTION}
and sketch a proof that whenever $\mathring{\upepsilon}$
is sufficiently small
(in a manner that is allowed to depend on $\mathring{\updelta}$ and $\TranminusdatasizeWithFactor$)
and $\Psiep$ is small relative to $1$,
a shock forms in $\Psi$ in finite time.

In our analysis, we will rely on the geometric coordinates $(t,u)$.
To facilitate our analysis, we find it convenient
to make the following bootstrap assumptions for
$(t,u) \in [0, \Tboot) \times [0,1]$,
where $0 < \Tboot \leq 2 \TranminusdatasizeWithFactor^{-1}$
is a bootstrap time:
\begin{subequations}
\begin{align}
	|\Psi|
	& \leq \Psiep^{1/2},	
		\label{E:INTROPSIITSELFBOOTSTRAP} \\
	|\Lunit \Psi|,
		\,
	|\slow_0|,
		\,
	|\slow_1|
	& \leq \mathring{\upepsilon}^{1/2},
		\label{E:INTROSMALLBOOT} \\
	\left|
		\uLgood \Psi(t,u)
		-
		f(u)
	\right|
	& \leq		
		\mathring{\upepsilon}^{1/2},
		\label{E:INTROTRANSVERSALBOOT} \\
	\upmu(t,u)
	& \leq 
			1 + 2|f(u)| \TranminusdatasizeWithFactor^{-1} 
		+ \Psiep^{1/2}
		+
		\mathring{\upepsilon}^{1/2}.
		\label{E:INTROUPMUBOOT}
\end{align}
\end{subequations}
We also assume that for $(t,u) \in [0, \Tboot) \times [0,1]$, we have
\begin{align} \label{E:INTRONOSHOCKBOOT}
\upmu(t,u) > 0,
\end{align}
which is tantamount to the assumption that a shock has not
yet formed on $[0,\Tboot) \times [0,1]$,
though it allows for the possibility that a shock forms 
exactly at time $\Tboot$. 
By standard local well-posedness,
if the data verify the size assumptions \eqref{E:PLANESYMMETRYTHEONELARGEDATUM}-\eqref{E:PLANESYMMETRYALLWSMALLPO},
if $\Psiep$ and
$\mathring{\upepsilon}$ are sufficiently small in the manner described above, 
and if $\Tboot > 0$ is sufficiently small,
then there exists a classical solution for 
$(t,u) \in [0, \Tboot) \times [0,1]$ such that the bootstrap assumptions are verified in this region.
Using the identities \eqref{E:PLANESYMMETRICCHOV},
we see that if the bootstrap assumptions are not saturated and if $\upmu$ remains
uniformly positive on $[0,\Tboot) \times [0,1]$, then the solution and its
$\partial_t$ and $\partial_1$ derivatives 
remain uniformly bounded in magnitude on $[0,\Tboot) \times [0,1]$. 
It is a standard result that under these 
conditions, the solution can be classically continued past the time $\Tboot$.
Thus, in order to prove that a shock forms,
it suffices to \textbf{i)} justify the bootstrap assumptions
by deriving a strict improvement of them, 
a task that we accomplish by showing that they hold with 
$\mathring{\upepsilon}^{1/2}$
replaced by $C \mathring{\upepsilon}$ (where $\mathring{\upepsilon}$ is chosen to be sufficiently small)
and with $\Psiep^{1/2}$ replaced by $\Psiep + C \mathring{\upepsilon}$
(where $\Psiep$ is also chosen to be sufficiently small);
\textbf{ii)} to show that $\upmu$ can vanish in finite time; 
and \textbf{iii)} to show that the 
vanishing of $\upmu$ leads to the blowup of 
$\max \lbrace |\partial_t \Psi|, |\partial_1 \Psi| \rbrace$.
Note that by \eqref{E:INTROSMALLBOOT},
our proof \textbf{i)} implies that $|\slow_0|$ and $|\slow_1|$
remain bounded. 

We now explain how to improve the bootstrap assumptions,
starting with \eqref{E:INTROSMALLBOOT}.
To this end, we find it convenient to introduce
(see \eqref{E:MTUDEF} for the definition of $\mathcal{M}_{\Tboot;u}$)
\begin{align}
q(u) 
:= 
\sup_{\mathcal{M}_{\Tboot;u}}
\left\lbrace
|\Lunit \Psi|
+
|\slow_0|
+
|\slow_1|
\right\rbrace.
\end{align}
In the rest of the proof, 
we silently rely on \eqref{E:NULLVECGEOMETRICCOMP}, 
which allows us to think of
$
\displaystyle
\Lunit 
= \frac{d}{dt}
$
along the integral curves of $\Lunit$ and
$
\displaystyle
\uLgood
= 2 \frac{d}{du}
$
along the integral curves of $\uLgood$.
Similarly, we have
that
$
\displaystyle
(1 - \Psi) \uLgood 
+ 
\upmu (3 + \Psi) \Lunit
= 2 (1 - \Psi) \frac{d}{du}
$
along the integral curves of
$(1 - \Psi) \uLgood 
+ 
\upmu (3 + \Psi) \Lunit
$
and
$
\displaystyle
(3 + \Psi) \uLgood 
+ 
\upmu (1 - \Psi) \Lunit
= 2(3 + \Psi) \frac{d}{du}
$
along the integral curves of
$
(3 + \Psi) \uLgood 
+ 
\upmu (1 - \Psi) \Lunit
$.
Using these observations, we
integrate equations
\eqref{E:SWITCHEDORDERMODELFAST} 
and
\eqref{E:AGAINREVAMPEDMODELSLOW0}-\eqref{E:AGAINREVAMPEDMODELSLOW1}
and use \eqref{E:NULLVECGEOMETRICCOMP},
the bootstrap assumptions,
and the small-data assumptions 
\eqref{E:PLANESYMMETRYLUNITPSISMALLSIGMA01},
\eqref{E:PLANESYMMETRYWSMALLSIGMA01},
\eqref{E:PLANESYMMETRYALLPSISMALLPO},
and 
\eqref{E:PLANESYMMETRYALLWSMALLPO}
to obtain
\begin{align} \label{E:LITTLEQGRONWALLREADY}
	q(u)
	& \leq
	C \mathring{\upepsilon}
	+
	C
	\int_{u'=0}^u
		q(u')
	\, du',
\end{align}
where here and throughout the paper, all constants $C$ 
are allowed to depend on
$\mathring{\updelta}$ and $\TranminusdatasizeWithFactor$,
and similarly for implicit constants hidden in the notations $\lesssim$ and $\mathcal{O}$;
see Subsect.\ \ref{SS:NOTATIONANDINDEXCONVENTIONS} for a precise description 
of the way in which we allow constants to depend on the various parameters in the bulk of the paper.
From \eqref{E:LITTLEQGRONWALLREADY} and Gronwall's inequality, we conclude that
$
\sup_{u \in [0,1]} q(u) \lesssim
\mathring{\upepsilon}
$.
Next, using the already obtained bound $|\Lunit \Psi| \lesssim \mathring{\upepsilon}$,
the fundamental theorem of calculus, 
and the data-size assumptions \eqref{E:PSIITSLFPLANESYMMETRYSIGMA01} for $\Psi$,
we deduce that for $(t,u) \in [0, \Tboot) \times [0,1]$, we have
\begin{align}
	\left|
		\Psi
	\right|
	(t,u)
	& \leq \Psiep
	+
	\int_{s=0}^t
		|\Lunit \Psi|(s,u)
	\, ds
		\\
	& \leq 
		\Psiep
		+
		C \TranminusdatasizeWithFactor^{-1} \mathring{\upepsilon}
	\leq \Psiep + C \mathring{\upepsilon}.
		\notag
\end{align}
We have thus 
derived the desired improvements of the bootstrap assumptions 
\eqref{E:INTROPSIITSELFBOOTSTRAP}-\eqref{E:INTROSMALLBOOT}
(whenever $\Psiep$ and $\mathring{\upepsilon}$ are sufficiently small).

Next, using the previously obtained estimates, the bootstrap assumptions, 
the evolution equation \eqref{E:MODELFAST},
the fundamental theorem of calculus,
and the data assumption \eqref{E:PLANESYMMETRYTHEONELARGEDATUM},
we obtain
\begin{align} \label{E:ULGOODPSIBOOTSTRAPIMPROVED}
	\left|
		\uLgood \Psi(t,u)
		-
		f(u)
	\right|
	\leq
		\int_{s=0}^t
			|\Lunit \uLgood \Psi(s,u)|
		\, ds
	&
	\leq
	C
	\int_{s=0}^t
		\mathring{\upepsilon}
	\, ds
		\\
	& \leq 
	 		C \TranminusdatasizeWithFactor^{-1} \mathring{\upepsilon}
	 \leq  C \mathring{\upepsilon},
	 \notag
\end{align}
which yields an improvement of the bootstrap assumption \eqref{E:INTROTRANSVERSALBOOT}.
Similarly, from the previously obtained estimates, the bootstrap assumptions, 
equation \eqref{E:UPMUSIMPLEPLANEWAVESCHEMATICEVOLUTION},
and the fact that (by construction) $\upmu|_{t=0} = 1 + \mathcal{O}_{\mydiam}(\Psi)$,
we deduce that
\begin{align} \label{E:UPMUSIMPLEPLANEWAVESCHEMATIC}
	\upmu(t,u)
	& = 
		1
		+
		f(u) t
		+
		\mathcal{O}_{\mydiam}(\Psiep)
		+ 
		\mathcal{O}(\mathring{\upepsilon}),
\end{align}
where the implicit constants in $\mathcal{O}_{\mydiam}(\cdot)$ do not depend on
$\TranminusdatasizeWithFactor$ or $f(u)$.
We have therefore improved the bootstrap assumption \eqref{E:INTROUPMUBOOT}
(whenever $\Psiep$ and $\mathring{\upepsilon}$ are sufficiently small),
which completes our proof of the improvement of the bootstrap assumptions.

We now show that a shock forms in finite time.
We start by setting
\[
\upmu_{\star}(t)
:= \min_{u \in [0,1]} \upmu(t,u).
\]
From 
definition \eqref{E:MODELCASEBLOWUPTIMEPARAMETER}
and
\eqref{E:UPMUSIMPLEPLANEWAVESCHEMATIC},
we find that
\begin{align} \label{E:SHOCKWILLFORMUPMUSIMPLEPLANEWAVESCHEMATIC}
	\upmu_{\star}(t)
	& = 
		1
		-
		\TranminusdatasizeWithFactor t
		+
		\mathcal{O}_{\mydiam}(\Psiep)
		+ 
		\mathcal{O}(\mathring{\upepsilon}).
\end{align}
From \eqref{E:SHOCKWILLFORMUPMUSIMPLEPLANEWAVESCHEMATIC},
we easily infer that
$\upmu_{\star}(t)$ vanishes at the time
$
\displaystyle
T_{(Shock)}
=
\TranminusdatasizeWithFactor^{-1}
\left\lbrace
	1
	+
	\mathcal{O}_{\mydiam}(\Psiep)
	+ 
	\mathcal{O}(\mathring{\upepsilon})
\right\rbrace
$.
Finally, from 
\eqref{E:MODELCASEBLOWUPTIMEPARAMETER},
\eqref{E:PLANESYMMETRICCHOV},
and the bounds 
$|\Psi| \leq \Psiep + C \mathring{\upepsilon}$,
$|\Lunit \Psi| \leq C \mathring{\upepsilon}$,
and \eqref{E:ULGOODPSIBOOTSTRAPIMPROVED},
we see that
if $\Psiep$ and $\mathring{\upepsilon}$ are sufficiently small,
then as $t \uparrow T_{(Shock)}$,
$
\sup_{u \in [0,1]}
|\partial_t \Psi(t,u)|
$ 
and
$
\sup_{u \in [0,1]}
|\partial_1 \Psi(t,u)|
$ 
are equal to non-zero, bounded functions times
$
\displaystyle
\frac{1}{\upmu_{\star}(t)}
$.
Hence,
$
\sup_{u \in [0,1]}
|\partial_t \Psi(t,u)|
$ 
and
$
\sup_{u \in [0,1]}
|\partial_1 \Psi(t,u)|
$ 
blow up precisely at time $T_{(Shock)}$.

We close this subsubsection by highlighting that there exist initial data,
compactly supported in $\Sigma_0^1$,
such that $\mathring{\upepsilon} = 0$ and such that the shock formation 
argument given above goes through; see Subsect.\ \ref{SS:EXISTENCEOFDATA}
for further discussion. The corresponding solutions are simple outgoing plane waves. This clarifies
why in perturbing these simple waves, 
we can consider initial data such that
$\mathring{\upepsilon}$ is positive but small relative
to the other relevant quantities in the problem,
as the above bootstrap argument required.

\subsubsection{The full problem without symmetry assumptions: data-size assumptions, bootstrap assumptions, and $L^{\infty}$ estimates}
\label{SSS:DATAASSUMPTIONSANDBOOTSTRAPASSUMPTIONS}
In our main theorem (Theorem~\ref{T:MAINTHEOREM}),
we study (non-symmetric) perturbations of the 
plane symmetric nearly simple outgoing wave solutions studied in Subsubsect.\ \ref{SSS:NEARLYSIMPLEWAVES}. 
We now outline the size assumptions that we make on the data in proving our main theorem.
Our assumptions are similar in spirit to our data assumptions from Subsubsect.\ \ref{SSS:NEARLYSIMPLEWAVES}
but are more complicated in view of the additional spatial direction
and the necessity of deriving energy estimates away from plane symmetry;
see Subsect.\ \ref{SS:DATAASSUMPTIONS} for a precise statement of our assumptions on the data.
We study solutions such that the interesting, relatively large
portion of the data lies in $\Sigma_0^{U_0}$
when $U_0$ is near $1$
while the data on $\mathcal{P}_0^{2 \TranminusdatasizeWithFactor^{-1}}$
are very small; see Figure~\ref{F:REGION} on pg.~\pageref{F:REGION}.
Here and in the remainder of the article, $\TranminusdatasizeWithFactor > 0$ 
denotes the data-dependent parameter defined in \eqref{E:INTROCRITICALBLOWUPTIMEFACTOR}.
We consider data for $\Psi$ such that along $\Sigma_0^{U_0}$, 
$\Psi$ itself is initially of small $L^{\infty}$ size $\Psiep$,
the $\mathcal{P}_u$-tangential derivatives of $\Psi$ 
(that is, its $\Lunit$ and $\GeoAng$ derivatives)
up to top order
are of a relatively small size $\mathring{\upepsilon}$ in appropriate norms, while 
the pure $\Rad$ derivatives such as $\Rad \Psi$ and $\Rad \Rad \Psi$
are of a relatively large\footnote{$\mathring{\updelta}$ is allowed to be small in an absolute sense.} 
size $\mathring{\updelta}$. We assume that all mixed tangential-transversal derivatives such as
$\Lunit \Rad \Psi$ are also of a relatively small size $\mathring{\upepsilon}$.
Our size assumptions are such that the energies we use to control the solution
are all initially of small size $\mathcal{O}(\mathring{\upepsilon}^2)$;
see Subsubsect.\ \ref{SSS:ENERGYESTIMATES} for further discussion on this point.
These size assumptions are similar to the ones made
in \cites{jSgHjLwW2016,jLjS2016b} and correspond to data close to that of the
nearly simple outgoing plane waves studied in Subsubsect.\ \ref{SSS:NEARLYSIMPLEWAVES}.
Roughly, the relative largeness of $\mathring{\updelta}$ is tied to a Riccati-type blowup
of $\max_{\alpha=0,1,2} |\partial_{\alpha} \Psi|$.
We assume that the slow wave variable array $\bigslow$ 
and all of its derivatives 
up to top order in \emph{all directions}\footnote{Actually, in our proof, we do not
need estimates for more than two $\Rad$ derivatives of $\bigslow$.}
are initially of small size $\mathring{\upepsilon}$.
Finally, we assume that along $\mathcal{P}_0^{2 \TranminusdatasizeWithFactor^{-1}}$, 
the derivatives of $\Psi$ and $\bigslow$ up to top order 
in \emph{all directions}
are of small size $\mathring{\upepsilon}$.
The case $\mathring{\upepsilon} = 0$ corresponds to a simple outgoing (that is, right-moving) plane wave.
See Subsect.\ \ref{SS:EXISTENCEOFDATA} for a proof sketch of the existence of data
that verify our size assumptions and for discussion on why 
their existence is tied to our structural assumptions \eqref{E:SOMENONINEARITIESARELINEAR}
on the semilinear inhomogeneous terms.

\begin{remark}[\textbf{On the parameter} $\Psiep$]
	\label{R:NEWPARAMETER}
	In \cite{jSgHjLwW2016},	
	the $L^{\infty}$ smallness of $\Psi$ itself (un-differentiated) and the smallness of its $\mathcal{P}_u$-tangential derivatives
	were captured by the smallness of $\mathring{\upepsilon}$.
	That is, the parameter $\Psiep$ was not featured in the work \cite{jSgHjLwW2016}. 
	For this reason, in \cite{jSgHjLwW2016}, $\mathring{\upepsilon}$ did not vanish
	for non-trivial simple outgoing plane wave solutions, 
	which is different than in the present article. 
	In this article, we have decided that it is better to introduce $\Psiep$ so that 
	\textbf{i)} our results here apply in particular to non-trivial simple outgoing plane wave solutions
	and 
	\textbf{ii)} the existence of an open set of initial data (without symmetry assumptions)
	that verify our size assumptions 
	follows as an easy consequence of the fact that our shock formation results apply to some simple outgoing plane wave solutions
	(see Subsect.\ \ref{SS:EXISTENCEOFDATA} for further discussion).
\end{remark}

The data-size assumptions described in the previous paragraph correspond to a pair of waves
in which one wave (namely $\Psi$) is nearly simple and outgoing while the other (namely $\bigslow$)
is uniformly small. A key point of our proof is showing how to propagate 
various aspects of the $\Psiep$-$\mathring{\updelta}$-$\mathring{\upepsilon}$
hierarchy all the way up to the shock, much like 
in Subsubsect.\ \ref{SSS:NEARLYSIMPLEWAVES}.
As in Subsubsect.\ \ref{SSS:NEARLYSIMPLEWAVES}, 
to propagate the $\Psiep$-$\mathring{\updelta}$-$\mathring{\upepsilon}$ hierarchy,
we find it convenient to make $L^{\infty}$ bootstrap assumptions for 
$\Psi$, $\bigslow$, and their geometric derivatives
on a bootstrap time interval of the form 
$[0,\Tboot)$, on which $\upmu > 0$ and on which the solution exists classically.
In view of the remarks made below \eqref{E:INTROCRITICALBLOWUPTIMEFACTOR},
we can assume that $\Tboot \leq 2 \TranminusdatasizeWithFactor^{-1}$.
Our ``fundamental'' bootstrap\footnote{To close our estimates, we also find it convenient to make additional ``auxiliary'' bootstrap
assumptions; see Subsects.\ \ref{SS:AUXILIARYBOOTSTRAP} and \ref{SS:BOOTSTRAPFORHIGHERTRANSVERSAL}.} 
assumptions are (see Subsect.\ \ref{SS:PSIBOOTSTRAP})
\begin{align} \label{E:INTROPSIFUNDAMENTALC0BOUNDBOOTSTRAP}
	\left\| 
		\Tanset^{[1,10]} \Psi 
	\right\|_{L^{\infty}(\Sigma_t^u)},
		\,
	\left\| 
		\Tanset^{\leq 10} \bigslow
	\right\|_{L^{\infty}(\Sigma_t^u)}
	& \leq \varepsilon,
\end{align}
where $\Tanset^{[1,M]}$ denotes an arbitrary differential operator of order in between $1$ and $M$
corresponding to repeated differentiation with respect to the $\mathcal{P}_u$-tangential
vectorfields $\Tanset = \lbrace \Lunit, \GeoAng \rbrace$, 
$\Tanset^{\leq M}$ is defined similarly but allows for the possibility of zero differentiations,
and
$\varepsilon$ is a small bootstrap parameter that, 
at the end of the paper, by virtue of a priori energy estimates and Sobolev embedding,
will have been shown to verify
$\varepsilon \lesssim \mathring{\upepsilon}$.

\begin{remark}
	Note that the uniform boundedness of 
	$\max_{\alpha=0,1,2} |\partial_{\alpha} \slow|$ up to the shock
	is already accounted
	for in the bootstrap assumption \eqref{E:INTROPSIFUNDAMENTALC0BOUNDBOOTSTRAP}.
	Note furthermore that the same is \emph{not} true for $\max_{\alpha=0,1,2} |\partial_{\alpha} \Psi|$,
	which blows up at the shock.
\end{remark}

Using \eqref{E:INTROPSIFUNDAMENTALC0BOUNDBOOTSTRAP}
and our data-size assumptions,
we can derive $L^{\infty}$ estimates for
$\Psi$ and the low-order pure transversal
and mixed transversal-tangential
derivatives of $\Psi$ and $\bigslow$ on the bootstrap region,
and for various derivatives of $\upmu$ and the Cartesian component functions 
$\lbrace \Lunit^a \rbrace_{a=1,2}$
for times up to as large as $2 \TranminusdatasizeWithFactor^{-1}$. 
Roughly, this is the content of  
Sects.\ \ref{S:PRELIMINARYPOINTWISE} and ~\ref{S:LINFINITYESTIMATESFORHIGHERTRANSVERSAL}.
The analysis is similar in spirit to that of Subsubsect.\ \ref{SSS:NEARLYSIMPLEWAVES}
but is much more involved. We need the estimates for the derivatives of $\upmu$ and $\lbrace \Lunit^a \rbrace_{a=1,2}$
because these quantities arise as error terms when we commute the equations with the 
vectorfields $\lbrace \Lunit, \Rad, \GeoAng \rbrace$.

\subsubsection{Proof sketch of the formation of the shock and the blowup of $\partial \Psi$}
\label{SSS:FORMATIONOFSHOCK}
Given the $L^{\infty}$ estimates described in Subsubsect.\ \ref{SSS:DATAASSUMPTIONSANDBOOTSTRAPASSUMPTIONS},
the proofs that $\upmu \to 0$ in finite time and that $\max_{\alpha=0,1,2} |\partial_{\alpha} \Psi|$
blows up are not much more difficult they were in Subsubsect.\ \ref{SSS:NEARLYSIMPLEWAVES}.
We now sketch the proofs.
First, one derives (essentially as a consequence of the eikonal equation \eqref{E:INTROEIKONAL})
the following transport equation for $\upmu$
(see Lemma~\ref{L:UPMUANDLUNITIFIRSTTRANSPORT}):
\begin{align} \label{E:SCHEMATICUPMUTRANSPORT}
	\Lunit \upmu(t,u,\vartheta) 
	& = \frac{1}{2} [G_{\Lunit \Lunit} \Rad \Psi](t,u,\vartheta)
		+ \upmu \mathcal{O}(\Singletan \Psi)(t,u,\vartheta).
\end{align}
In \eqref{E:SCHEMATICUPMUTRANSPORT} and throughout, 
$\Singletan$ schematically denotes a differentiation in a
direction tangential to the characteristics $\mathcal{P}_u$.
Note that equation \eqref{E:SCHEMATICUPMUTRANSPORT} does not involve the slow wave $\bigslow$.
Using bootstrap assumptions and $L^{\infty}$ estimates of the type described in
Subsubsect.\ \ref{SSS:DATAASSUMPTIONSANDBOOTSTRAPASSUMPTIONS}, it is easy to show
that $\upmu \mathcal{O}(\Singletan \Psi)(t,u,\vartheta) = \mathcal{O}(\varepsilon)$
and that 
$[G_{\Lunit \Lunit} \Rad \Psi](t,u,\vartheta) 
= 
[G_{\Lunit \Lunit} \Rad \Psi](0,u,\vartheta)
+ 
\mathcal{O}(\varepsilon)
$.
Inserting these estimates into \eqref{E:SCHEMATICUPMUTRANSPORT}, we find that
\begin{align} \label{E:ANOTHERSCHEMATICUPMUTRANSPORT}
	\Lunit \upmu(t,u,\vartheta) 
	& = \frac{1}{2} [G_{\Lunit \Lunit} \Rad \Psi](0,u,\vartheta)
		+ \mathcal{O}(\varepsilon).
\end{align}
From \eqref{E:ANOTHERSCHEMATICUPMUTRANSPORT},
definition \eqref{E:CRITICALBLOWUPTIMEFACTOR} and the fact that
$\varepsilon$ is controlled by $\mathring{\upepsilon}$, 
we see that there exists $(u_*,\vartheta_*) \in [0,1] \times \mathbb{T}$ such that
\begin{align} \label{E:MOSTNEGATIVEANOTHERSCHEMATICUPMUTRANSPORT}
	\Lunit \upmu(t,u_*,\vartheta_*) 
	& = - \TranminusdatasizeWithFactor
		+ \mathcal{O}(\mathring{\upepsilon}).
\end{align}
Recalling that 
$
\displaystyle
\Lunit 
= 
\frac{\partial}{\partial t}
$
and that
$
\upmu|_{t=0}
= 1
+ \mathcal{O}_{\mydiam}(\Psi)
= 1 + \mathcal{O}_{\mydiam}(\Psiep)
$
(see Subsect.\ \ref{SS:NOTATIONANDINDEXCONVENTIONS} regarding the notation),
we see that 
if 
$\Psiep$ and
$\mathring{\upepsilon}$ are sufficiently small,
then 
$\upmu$
vanishes for the first time when
$t = \lbrace 1 + \mathcal{O}_{\mydiam}(\Psiep) + \mathcal{O}(\mathring{\upepsilon}) \rbrace \TranminusdatasizeWithFactor^{-1}$.
Moreover, \underline{$\upmu$ vanishes linearly} in that $\Lunit \upmu$ is \emph{strictly negative}
at the vanishing points; 
as we describe below in Subsubsect.\ \ref{SSS:ENERGYESTIMATES},
\emph{these are crucially important facts for our energy estimates}.
Finally, we note that the above argument also yields that
$|\Rad \Psi| \gtrsim 1$ at the points where $\upmu$ vanishes and thus
$
\displaystyle
\Radunit \Psi 
:= \frac{1}{\upmu} \Rad \Psi
$
blows up like 
$
\displaystyle
\frac{1}{\upmu}
$
at such points.
Since we are also able to show that the Cartesian components $\Radunit^{\alpha}$ remain close
to $- \delta_1^{\alpha}$ throughout the evolution,
it follows that $\max_{\alpha=0,1,2} |\partial_{\alpha} \Psi|$ 
blows up when $\upmu$ vanishes.

\subsubsection{Overview of the energy estimates}
\label{SSS:ENERGYESTIMATES}
Energy estimates are by far the most difficult aspect of the proof.
For reasons to be explained, to close our energy estimates, we must
commute the evolution equations up to $18$ times with
the elements the $\mathcal{P}_u$-tangential commutation set
$\Tanset = \lbrace \Lunit, \GeoAng \rbrace$
and derive energy estimates for the differentiated quantities.
The starting point for these energy estimates is 
energy identities for $\Psi$ and $\bigslow$, which we obtain
by applying the divergence theorem
on the regions depicted in Figure~\ref{F:SOLIDREGION}.
We provide the details behind these energy identities in
Sect.\ \ref{S:ENERGIES}; here we focus mainly on outlining
how to derive a priori energy estimates based on the energy identities.
\begin{center}
\begin{overpic}[scale=.2]{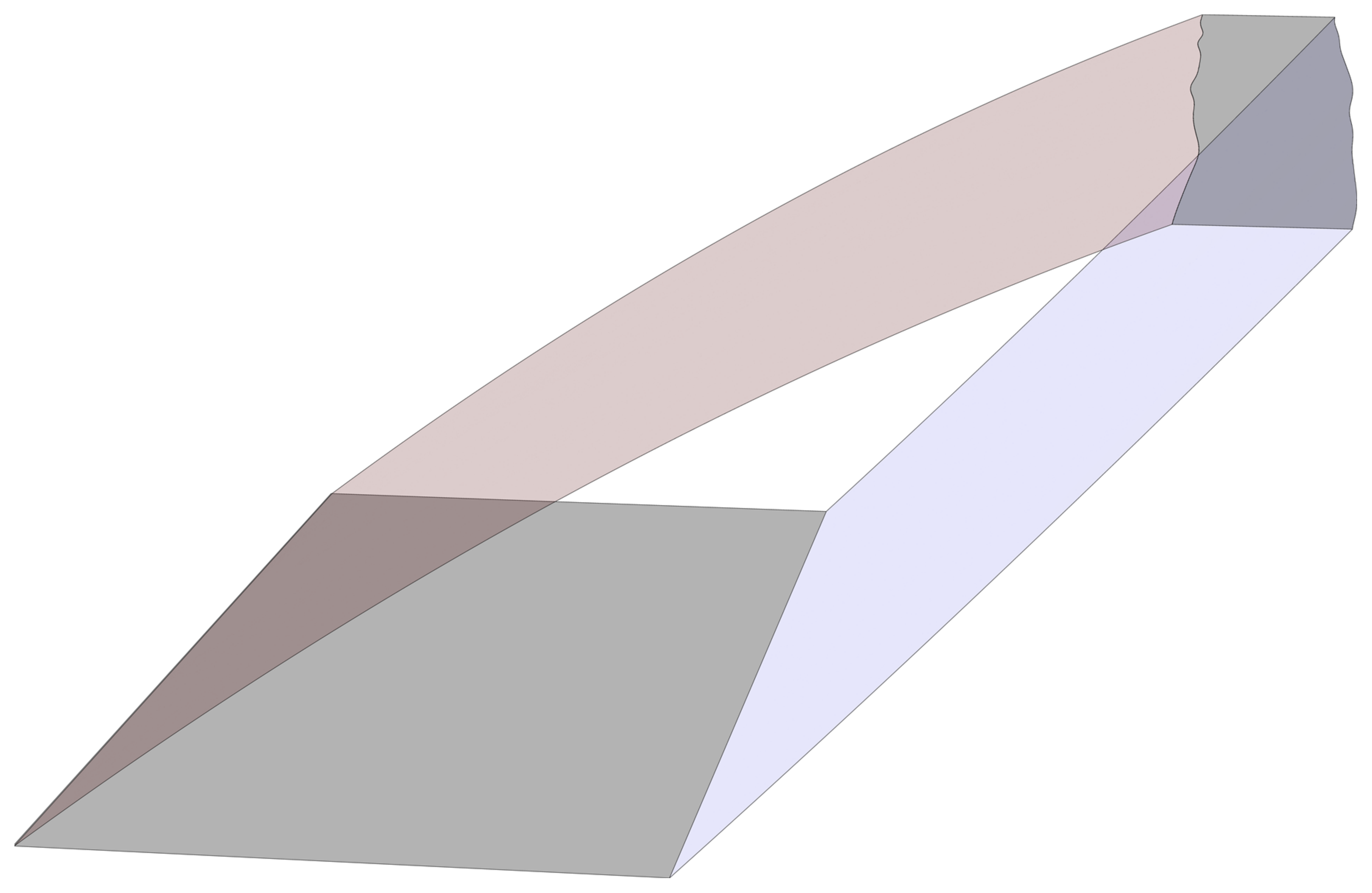} 
\put (54,32.5) {\large$\displaystyle \mathcal{M}_{t,u}$}
\put (38,33) {\large$\displaystyle \mathcal{P}_u^t$}
\put (74,34) {\large$\displaystyle \mathcal{P}_0^t$}
\put (59,15) {}
\put (32,17) {\large$\displaystyle \Sigma_0^u$}
\put (48.5,13) {\large$\displaystyle \ell_{0,0}$}
\put (12,13) {\large$\displaystyle \ell_{0,u}$}
\put (90,60.5) {\large$\displaystyle \Sigma_t^u$}
\put (93.5,56) {\large$\displaystyle \ell_{t,0}$}
\put (87.2,56) {\large$\displaystyle \ell_{t,u}$}
\put (-.6,16) {\large$\displaystyle x^2 \in \mathbb{T}$}
\put (22,-3) {\large$\displaystyle x^1 \in \mathbb{R}$}
\thicklines
\put (-.9,3){\vector(.9,1){22}}
\put (.7,1.8){\vector(100,-4.5){48}}
\end{overpic}
\captionof{figure}{The energy estimate region}
\label{F:SOLIDREGION}
\end{center}

We start by describing our energy-null flux quantities.
In Subsubsect.\ \ref{SSS:ENERGYESTIMATES} only, 
we denote the energy-null flux quantity\footnote{In our detailed analysis, 
when constructing $L^2$-controlling quantities,
we separately define energies along $\Sigma_t^u$, null fluxes along $\mathcal{P}_u^t$, and spacetime
integrals over $\mathcal{M}_{t,u}$. Here, to shorten our explanation of the main ideas, we have grouped them together.} 
for $\Psi$ by $\mathbb{H}^{(Fast)}$
and the one for $\bigslow$ by $\mathbb{H}^{(Slow)}$.
Schematically, $\mathbb{H}^{(Fast)}$ and $\mathbb{H}^{(Slow)}$
have the following strength
(see Sect.\ \ref{S:ENERGIES} for precise statements concerning the energies and their
coerciveness), where the integrals are with respect to the geometric
coordinates:\footnote{In our schematic overview of the proof,
we use the notation $A \sim B$ to imprecisely indicate that
$A$ is well-approximated by $B$.}
\begin{subequations}
\begin{align}
	\mathbb{H}^{(Fast)}(t,u)
	& \sim
			\int_{\Sigma_t^u}
				\left\lbrace
					\upmu (\Lunit \Psi)^2
					+ 
					(\Rad \Psi)^2
					+
					\upmu (\GeoAng \Psi)^2
				\right\rbrace
			\,d \vartheta  du' 
			+
			\int_{\mathcal{P}_u^t}
				\left\lbrace
					(\Lunit \Psi)^2
					+ 
					\upmu
					(\GeoAng \Psi)^2
				\right\rbrace
			\, d \vartheta dt'
				\label{E:INTROPSIENERGY}
				\\
		& \ \
			+
			\int_{\mathcal{M}_{t,u}}
				[\Lunit \upmu]_-
				(\GeoAng \Psi)^2
			\, d \vartheta d u' dt',
		 \notag \\
	\mathbb{H}^{(Slow)}(t,u)
	& \sim \int_{\Sigma_t^u}
				\upmu |\bigslow|^2
			\,d \vartheta  du' 
			+
			\int_{\mathcal{P}_u^t}
				|\bigslow|^2
			\, d \vartheta dt'.
			\label{E:INTROSLOWWAVEENERGY}
\end{align}
\end{subequations}
On RHS~\eqref{E:INTROPSIENERGY} and throughout,
$
\displaystyle
f_- 
:= \max \lbrace -f, 0 \rbrace
$.

A crucially important feature of the above energies is that some of the integrals on RHSs \eqref{E:INTROPSIENERGY}-\eqref{E:INTROSLOWWAVEENERGY}
are $\upmu$-weighted
and thus become weak near the shock, that is, in regions where $\upmu$ is near $0$. 
It turns out that the $\upmu$-weighted integrals are not able to suitably 
control all of the error terms that arise in the energy identities. The reason is that
we encounter some error terms that \emph{lack} $\upmu$ weights and are therefore relatively strong.
However, it is also true that \emph{all} elements of
$\lbrace \Lunit \Psi, \Rad \Psi, \GeoAng \Psi, \bigslow \rbrace$
are controlled by one of $\mathbb{H}^{(Fast)}$ or $\mathbb{H}^{(Slow)}$
\emph{without a $\upmu$ weight}
and thus, in total, we are able to control all error integrals.
Let us further comment on the coerciveness of the spacetime integral
$
\displaystyle
\int_{\mathcal{M}_{t,u}}
				[\Lunit \upmu]_-
				(\GeoAng \Psi)^2
			\, d \vartheta d u' dt'
$
featured on RHS~\eqref{E:INTROPSIENERGY}.
The key point is that $\Lunit \upmu$ is quantitatively negative 
(and thus $[\Lunit \upmu]_-$ is quantitatively positive)
in the difficult region where $\upmu$ is small,
which leads to the coerciveness of the integral.
The reasons behind this were outlined in Subsubsect.\ \ref{SSS:FORMATIONOFSHOCK}.
In all prior works on shock formation in more than one spatial dimension,
similar spacetime integrals were exploited to close the energy estimates.
The idea to exploit such a spacetime integral seems to have originated in
the works \cites{sA1999a,dC2007}.

We now let $\mathbb{H}^{(Fast)}_N$ denote an energy corresponding to commuting the wave equation
for $\Psi$ with a string of vectorfields $\Tanset^N$ consisting of
precisely $N$ factors of elements the $\mathcal{P}_u$-tangential commutation set
$\Tanset = \lbrace \Lunit, \GeoAng \rbrace$.
We let $\mathbb{H}^{(Slow)}_N$ be an analogous energy for $\bigslow$. 
As we alluded to above, in our detailed proof, we will have $1 \leq N \leq 18$
for $\mathbb{H}^{(Fast)}_N$ and
$N \leq 18$ for $\mathbb{H}^{(Slow)}_N$.
For such $N$ values, our initial data are such that all energies are of initially
small size $\mathcal{O}(\mathring{\upepsilon}^2)$,
where $\mathring{\upepsilon}$ is the smallness parameter from Subsubsect.\ \ref{SSS:DATAASSUMPTIONSANDBOOTSTRAPASSUMPTIONS}.
In particular, our energies \emph{completely vanish} for simple outgoing plane wave solutions
(in which $\Lunit \Psi = \GeoAng \Psi = \bigslow = 0$).
We stress that we avoid using the energy $\mathbb{H}^{(Fast)}_0$,
which involves the $L^2$ norm of the pure transversal derivative $\Rad \Psi$ and is therefore allowed to be
of a relatively large size $\mathcal{O}(\mathring{\updelta}^2)$; 
as we mentioned earlier, in order to close our proof, 
we do not need to control $\Rad \Psi$ in $L^2$, but rather only in $L^{\infty}$. 
We also need to control $\Rad \Rad \Psi$ and $\Rad \Rad \Rad \Psi$ 
in $L^{\infty}$, for reasons that
we clarify starting in Sect.\ \ref{S:LINFINITYESTIMATESFORHIGHERTRANSVERSAL}. We can obtain these $L^{\infty}$ estimates
by treating the wave equation \eqref{E:FASTWAVE} 
like a transport equation of the form $\Lunit \Rad \Psi = \cdots$ 
(i.e., a transport equation in $\Rad \Psi$),
where the source terms $\cdots$ are controlled by our energies,
and by commuting this equation up to two times with $\Rad$;
see Props.~\ref{P:IMPROVEMENTOFAUX} and \ref{P:IMPROVEMENTOFHIGHERTRANSVERSALBOOTSTRAP}
for detailed proofs.

We now provide a few more details about
how to derive the energy identities that form the starting point of our $L^2$-type analysis.
To obtain the relevant energy identities for $\Psi$,
we commute $\upmu$ times\footnote{To avoid uncontrollable error terms, it 
is essential that we first multiply the equations by $\upmu$ before commuting them.}
the wave equation \eqref{E:FASTWAVE}
with $\Tanset^N$, multiply by $\Mult \Tanset^N \Psi$, 
and then integrate by parts
over $\mathcal{M}_{t,u}$.
Here,
 \begin{align} \label{E:MULT}
\Mult := (1 + 2 \upmu) \Lunit + 2 \Rad
\end{align}
is a ``multiplier vectorfield'' 
with appropriately chosen $\upmu$ weights.
Similarly, 
to obtain the relevant energy identities for $\bigslow$,
we multiply equations 
\eqref{E:SLOW0EVOLUTION}-\eqref{E:SYMMETRYOFMIXEDPARTIALS}
with $\upmu$, commute them
with $\Tanset^N$, multiply by an appropriate quantity,
and then integrate by parts
over $\mathcal{M}_{t,u}$;
see Sect.\ \ref{S:ENERGIES} for the details
behind the integration by parts and
Sect.\ \ref{S:POINTWISEESTIMATES} for the details
behind the pointwise estimates for the semilinear inhomogeneous
terms on RHSs
\eqref{E:FASTWAVE}
and
\eqref{E:SLOW0EVOLUTION}
and for pointwise estimates for
the error terms that we generate upon commuting the equations.
In total, we can use these energy identities and pointwise estimates
to obtain a system of integral inequalities
of the following type, where here we only schematically display a few representative terms:
\begin{subequations}
\begin{align}
	\mathbb{H}^{(Fast)}_N(t,u)
	& \leq 
	  C \mathring{\upepsilon}^2
		+
		\int_{\mathcal{M}_{t,u}}
			(\Rad \Psi)
			(\Rad \GeoAng^N \Psi)
			\GeoAng^N \mytr \upchi
		\, d \vartheta d u' dt'
			\label{E:PSIMODELENERGYINEQUALITY} \\
	& \ \
		+
		\int_{\mathcal{M}_{t,u}}
			(\Lunit \GeoAng^N \Psi)
			\GeoAng^N \bigslow
		\, d \vartheta d u' dt'
		+ 
		\cdots,
			\notag \\
	\mathbb{H}^{(Slow)}_N(t,u)
	& \leq 
		C \mathring{\upepsilon}^2
		+
		\int_{\mathcal{M}_{t,u}}
			|\GeoAng^N \bigslow|^2
		\, d \vartheta d u' dt'
		+ 
		\int_{\mathcal{M}_{t,u}}
			(\GeoAng^N \bigslow)
			(\Rad \GeoAng^N \Psi)
		\, d \vartheta d u' dt'
		+
		\cdots.
		\label{E:SLOWWAVEMODELENERGYINEQUALITY}
\end{align}
\end{subequations}

The $C \mathring{\upepsilon}^2$ terms on
RHSs \eqref{E:PSIMODELENERGYINEQUALITY}-\eqref{E:SLOWWAVEMODELENERGYINEQUALITY}
are generated by the data.
The tensorfield $\upchi$ on RHS~\eqref{E:PSIMODELENERGYINEQUALITY}
is the null second fundamental form 
of the co-dimension-two tori
$\ell_{t,u}$. It is
a symmetric type $\binom{0}{2}$
tensorfield with components
$\upchi_{\CoordAng \CoordAng}
= g(\D_{\CoordAng} \Lunit, \CoordAng)$
(see \eqref{E:CHIUSEFULID} and recall that
$
\displaystyle
\CoordAng = \frac{\partial}{\partial \vartheta}
$),
where $\D$ is the Levi--Civita connection of $g$.
Moreover, $\mytr \upchi$ is the trace of $\upchi$
with respect to the Riemannian metric $\gsphere$ induced on $\ell_{t,u}$ by $g$.
Geometrically, $\mytr \upchi$ is the null mean curvature of the $g$-null hypersurfaces
$\mathcal{P}_u$. Analytically, we have $\Lunit^{\alpha} \sim \partial u$ 
(see \eqref{E:LGEOEQUATION} and \eqref{E:LUNITDEF})
and thus $\mytr \upchi \sim \partial^2 u$, where $u$
is the eikonal function. From the point of view of counting derivatives, 
one might expect to see terms such as
$\GeoAng^N \mytr \upchi \sim \partial^{N+2} u$ on the RHS of the $N$-times commuted fast wave equation
since the Cartesian components $\Singletan^{\alpha}$ of the elements  
$
\Singletan \in \Tanset
= \lbrace \Lunit, \GeoAng \rbrace
$
depend on $\partial u$; roughly, terms such as
$\GeoAng^N \mytr \upchi$ can arise
when one commutes operators of the form $\Tanset^N$ through\footnote{As we mentioned above, in practice, we commute through
$\upmu \square_g$ since the $\upmu$-weighted wave operator exhibits better commutation properties
with the elements of $\lbrace \Lunit, \Rad, \GeoAng \rbrace $.} 
the expression $\square_g \Psi$ and the maximum number of derivatives falls on the components\footnote{In practice, 
we mostly rely on
geometric decomposition formulas when 
decomposing the error terms in the commuted equations, 
rather than working with Cartesian components.}
$\Singletan^{\alpha} \sim \partial u$.
Thus, the presence of $\GeoAng^N \mytr \upchi$ on RHS~\eqref{E:PSIMODELENERGYINEQUALITY}
signifies that the energy estimates for the wave variables are coupled 
to $L^2$ estimates for the derivatives of the eikonal function. This is a fundamental
difficulty that one faces whenever one works with vectorfields constructed from 
an eikonal function adapted to the characteristics.

We already stress here that a naive treatment of the term
$\GeoAng^N \mytr \upchi$ in the energy estimates would result in the loss
of a derivative that would preclude closure of the estimates. This is because
crude estimates for $\GeoAng^N \mytr \upchi \sim \partial^{N+2} u$,
based on the eikonal equation \eqref{E:INTROEIKONAL},
lead to an $L^2$ estimate for $\partial^{N+2} u$ that depends on
$\partial^{N+2} \Psi$, which is one more derivative of $\Psi$ than is controlled by
$	\mathbb{H}^{(Fast)}_N$.
However, the term $\GeoAng^N \mytr \upchi$ has a special tensorial structure, 
and we can avoid the loss of a derivative through a procedure that we describe below.
As we explain below, we use this procedure only at the top order\footnote{More precisely, 
in treating the most difficult top-order error terms involving $\GeoAng^N \mytr \upchi$,
we use this procedure only at the top order. We also encounter less degenerate top-order
error terms with factors that behave like $\upmu \GeoAng^N \mytr \upchi$ 
(see Lemma~\ref{L:LESSDEGENERATEENERGYESTIMATEINTEGRALS} and point (5) of Subsect.\ \ref{SS:OFTENUSEDESTIMATES}),
and for these terms, thanks to the helpful factor of $\upmu$, 
it is permissible to use the procedure at all derivative levels.} 
because it leads to a rather degenerate top-order energy estimate,
and we need improved estimates below top order to close our proof.
The improved estimates are possible because below top order, 
one does not need to worry about the loss of a derivative.
In fact, as we further explain below,
to obtain the improved estimates,
it is important that one \emph{should allow the loss of a derivative} in the estimates for
$\GeoAng^N \mytr \upchi$ below top order.
Taken together, these are unusual and technically challenging
features of the study of shock formation
that were first found in the works 
\cites{sA1995,sA1999a,sA1999b,dC2007}
of Alinhac\footnote{Actually, Alinhac's approach allowed for some loss of differentiability 
stemming from his use of an eikonal function. As we mentioned 
in Subsect.\ \ref{SS:PRIORWORKSANDSUMMARY}, to overcome this difficulty, 
he used Nash--Moser estimates. In contrast, Christodoulou used an approach that avoided the derivative loss altogether,
which is the approach we take in the present article.} 
and Christodoulou.

We now provide some additional details on how we derive the a priori energy estimates.
In the usual fashion, we must control RHSs~\eqref{E:PSIMODELENERGYINEQUALITY} and
\eqref{E:SLOWWAVEMODELENERGYINEQUALITY} in terms of $\mathbb{H}^{(Fast)}_M$
and $\mathbb{H}^{(Slow)}_M$ (for suitable $M$)
so that we can use a version of Gronwall's inequality.
After a rather difficult Gronwall estimate, one obtains 
a hierarchy of energy estimates 
holding up to the shock, which we now explain.
The estimates feature the quantity
\[
	\upmu_{\star}(t,u) := \min \left\lbrace 1 , \min_{\Sigma_t^u} \upmu \right\rbrace,
\]
which essentially measures the worst case smallness for $\upmu$ along $\Sigma_t^u$.
As in all prior works on shock formation in more than one spatial dimension,
our proof allows for the possibility that the high-order energies 
might blow up like negative powers of $\upmu_{\star}$,
and, at the same time, guarantees that the energies 
become successively less singular as one reduces the number of derivatives.
Moreover, one eventually reaches a level at which the energies 
remain uniformly bounded, all the way up to the shock; 
these non-degenerate energy estimates are what allows one to improve,
via Sobolev embedding, the $L^{\infty}$ bootstrap assumptions 
(see Subsubsect.\ \ref{SSS:DATAASSUMPTIONSANDBOOTSTRAPASSUMPTIONS})
that are crucial for all aspects of the proof.
The hierarchy of energy estimates that we derive can be modeled as follows:
\begin{subequations}
\begin{align}
	\mathbb{H}_{18}^{(Fast)}(t,u),
		\,
	\mathbb{H}_{18}^{(Slow)}(t,u)	
	& \lesssim \mathring{\upepsilon}^2 \upmu_{\star}^{-11.8}(t,u),
		\label{E:INTROTOPENERGY} \\
	\mathbb{H}_{17}^{(Fast)}(t,u),
		\,
	\mathbb{H}_{17}^{(Slow)}(t,u)	
	& \lesssim \mathring{\upepsilon}^2 \upmu_{\star}^{-9.8}(t,u),
		\label{E:INTROJUSTBELOWTOPENERGY} \\
	& \cdots,
		\\
	\mathbb{H}_{13}^{(Fast)}(t,u),
		\,
	\mathbb{H}_{13}^{(Slow)}(t,u)
	& \lesssim \mathring{\upepsilon}^2 \upmu_{\star}^{-1.8}(t,u),
		\\
	\mathbb{H}_{12}^{(Fast)}(t,u),
		\,
	\mathbb{H}_{12}^{(Slow)}(t,u)
	& \lesssim \mathring{\upepsilon}^2,
		\label{E:FIRSTNONDEGENERATEENERGY} \\
	& \cdots 
		\\
	\mathbb{H}_1^{(Fast)}(t,u),
		\,
	\mathbb{H}_1^{(Slow)}(t,u)
	& \lesssim \mathring{\upepsilon}^2.
		\label{E:INTROLOWESTENERGY}
\end{align}
\end{subequations}
The estimates \eqref{E:INTROTOPENERGY}-\eqref{E:INTROLOWESTENERGY}
capture in spirit the energy estimates that we prove in this article; 
we refer the reader to Prop.~\ref{P:MAINAPRIORIENERGY} for the precise
statements. The precise numerology behind the hierarchy \eqref{E:INTROTOPENERGY}-\eqref{E:INTROLOWESTENERGY} 
is intricate, but here are the main ideas:
\textbf{i)} The top-order blowup-exponent of $11.8$ found on RHS~\eqref{E:INTROTOPENERGY}
is tied to certain universal structural constants in the equations 
(such as the constant $A$ appearing in \eqref{E:INTROHARDESTERRRORINTEGRAL} below)
and is close to optimal by our approach; for example, if one considered data belonging to the Sobolev space $H^{100}$,
then our approach would only allow us to conclude 
that the energy controlling $100$ derivatives of $\Psi$
might blow up at the rate $\upmu_{\star}^{-11.8}(t,u)$.
\textbf{ii)} The fact that the estimates become less singular by precisely two powers of 
$\upmu_{\star}$ at each step in the descent seems to be fundamental. The root of this 
phenomenon is the following: an integration in time of $\upmu_{\star}^{-B}(t,u)$ reduces the strength
of the singularity by one degree to $\upmu_{\star}^{1-B}(t,u)$; see \eqref{E:INTEGRATEMUSTARINTIMELESSSINGULAR}.
\textbf{iii)} To control error terms, it is convenient for the solutions 
to be such that slightly more than half of the energies are uniformly bounded up to the shock. 
Then, when we are bounding 
error term products in $L^2$, we can exploit the fact that
all but at-most-one factor in the product 
is uniformly bounded in $L^{\infty}$ up to the shock.

We now discuss some of the main ideas
behind deriving the energy estimate hierarchy
\eqref{E:INTROTOPENERGY}-\eqref{E:INTROLOWESTENERGY}.
By far, the most difficult integrals to estimate
are the ones on RHS~\eqref{E:PSIMODELENERGYINEQUALITY} 
involving $\GeoAng^N \mytr \upchi$
and some related ones that we have not displayed
but that create similar difficulties.
In the case of the scalar wave equations treated in \cite{jSgHjLwW2016},
these difficult integrals were handled via extensions 
of techniques developed in \cite{dC2007}.
In the present article, 
in treating these integrals,
we encounter some new terms
stemming from interactions between the fast wave, the eikonal function, and the slow wave; 
see the proof outline of Prop.~\ref{P:KEYPOINTWISEESTIMATE} for further discussion on this point.
Later in this section, we will say a few words about these difficult integrals, 
but we will not discuss them in detail here since the most challenging 
aspects of these integrals were handled in \cite{jSgHjLwW2016}. Instead, in this subsubsection, 
we focus on describing the influence of the $\GeoAng^N \bigslow$-involving integrals 
from RHSs \eqref{E:PSIMODELENERGYINEQUALITY}-\eqref{E:SLOWWAVEMODELENERGYINEQUALITY}
on the a priori estimates \eqref{E:INTROTOPENERGY}-\eqref{E:INTROLOWESTENERGY}. 
Roughly, these integrals account for the self-interactions of the slow wave
and the interaction of the slow wave with the fast wave up to the shock,
which are the main new kinds of interactions accounted for in this paper;
the remaining error integrals on RHS~\eqref{E:PSIMODELENERGYINEQUALITY}
involve self-interactions of $\Psi$
and the interaction of $\Psi$ with the eikonal function,
which were handled in \cite{jSgHjLwW2016},
as well as interactions of $\bigslow$ with the below-top-order derivatives of the eikonal function,
which are relatively easy to treat (see Lemma~\ref{L:STANDARDPSISPACETIMEINTEGRALS}).
Our main goal at present is the following:
\begin{quote}
\emph{We will sketch why the $\GeoAng^N \bigslow$-involving integrals
from RHSs \eqref{E:PSIMODELENERGYINEQUALITY}-\eqref{E:SLOWWAVEMODELENERGYINEQUALITY}
create only harmless exponential growth in
the energies}, which is allowable within our approach in view of our
sufficiently good guess about the time of first shock formation
(see the discussion following equation \eqref{E:INTROCRITICALBLOWUPTIMEFACTOR})
and our assumed smallness of $\mathring{\upepsilon}$.
\end{quote}
In particular, if the $\GeoAng^N \bigslow$-involving integrals were 
the only types of error integrals
that one encountered in the energy estimates, 
then at all derivatives levels, the energies would remain uniformly bounded 
by $\lesssim \mathring{\upepsilon}^2$ up to the shock.
This shows that the interaction between the two waves is in some sense weak near the shock,
\emph{even though the RHS of the slow wave equation \eqref{E:SLOWWAVE}
contains $\partial \Psi$ source terms that blow up at the shock.} 
The weakness of the interaction is very much a ``PDE effect'' 
(it is not easily modeled by ODE inequalities)
that is detectable only because our energies
\eqref{E:INTROPSIENERGY}-\eqref{E:INTROSLOWWAVEENERGY}
contain non-$\upmu$-degenerate $\mathcal{P}_u^t$ integrals
and spacetime integrals.

To proceed with our sketch, we let
\[
\overline{\mathbb{H}}_N^{(Fast)}(t,u)
:=
\sup_{(t',u') \in [0,t] \times [0,u]}\mathbb{H}^{(Fast)}_N(t',u'),
\qquad
\overline{\mathbb{H}}_N^{(Slow)}(t,u)
:=
\sup_{(t',u') \in [0,t] \times [0,u]} \mathbb{H}^{(Slow)}_N(t',u').
\]
We then note the following simple consequence of
\eqref{E:INTROPSIENERGY}-\eqref{E:INTROSLOWWAVEENERGY},
\eqref{E:PSIMODELENERGYINEQUALITY}-\eqref{E:SLOWWAVEMODELENERGYINEQUALITY},
and Young's inequality,
where we are ignoring the 
$\GeoAng^N \mytr \upchi$-involving integral
on RHS~\eqref{E:PSIMODELENERGYINEQUALITY}:
\begin{align}
	\overline{\mathbb{H}}_N^{(Fast)}(t,u)
	& \leq
	  C \mathring{\upepsilon}^2
		+
		C
		\int_{u'=0}^u
		  \overline{\mathbb{H}}_N^{(Fast)}(t,u')
		\, d u'
		+
		C
		\int_{u'=0}^u
		  \overline{\mathbb{H}}_N^{(Slow)}(t,u')
		\, d u'
		+ 
		\cdots,
			\label{E:GRONWALLREADYPSIMODELENERGYINEQUALITY} \\
	\overline{\mathbb{H}}_N^{(Slow)}(t,u)
	& \leq
		C \mathring{\upepsilon}^2
		+
		C
		\int_{u'=0}^u
		  \overline{\mathbb{H}}_N^{(Fast)}(t,u')
		\, d u'
		+
		C
		\int_{t'=0}^t
		  \overline{\mathbb{H}}_N^{(Slow)}(t',u)
		\, d t'
		+
		\cdots.
		\label{E:GRONWALLREADYSLOWWAVEMODELENERGYINEQUALITY}
\end{align}
Then from \eqref{E:GRONWALLREADYPSIMODELENERGYINEQUALITY}-\eqref{E:GRONWALLREADYSLOWWAVEMODELENERGYINEQUALITY}
and Gronwall's inequality in $t$ and $u$, we conclude that as long as the solution
exists classically, we have the following estimates 
for $(t,u) \in [0,2 \TranminusdatasizeWithFactor^{-1}] \times [0,1]$:
\begin{align} \label{E:INTROEASYTERMSTOPENERGYGRONWALLED}
	\overline{\mathbb{H}}_N^{(Fast)}(t,u)
	& \leq C \mathring{\upepsilon}^2,
	&&
	\overline{\mathbb{H}}_N^{(Slow)}(t,u)
	\leq C \mathring{\upepsilon}^2,
\end{align}
where, as we have mentioned, constants $C$ are allowed to depend on 
$\TranminusdatasizeWithFactor^{-1}$,
the approximate time of first shock formation
(see \eqref{E:INTROCRITICALBLOWUPTIMEFACTOR} and the discussion below it).

As we mentioned above, in reality, we are not able to prove
the non-degenerate estimate \eqref{E:INTROEASYTERMSTOPENERGYGRONWALLED} for large $N$ because the
$\GeoAng^N \mytr \upchi$-involving integral on
RHS~\eqref{E:PSIMODELENERGYINEQUALITY} leads to a much worse
a priori energy estimate in the top-order case
$N=18$. This phenomenon is explained in detail in \cite{jSgHjLwW2016} in the case of 
a homogeneous scalar covariant wave equation $\square_{g(\Psi)} \Psi = 0$; here we only describe the 
changes in the analysis of $\GeoAng^N \mytr \upchi$ compared to \cite{jSgHjLwW2016},
the new feature being the presence of the semilinear coupling terms on RHS~\eqref{E:FASTWAVE}.
Let us first describe the estimate.
Specifically, 
due to the difficult regularity theory of the eikonal function,\footnote{Recall that 
$\GeoAng^N \mytr \upchi \sim \partial^{N+2} u$.}
the following term is in fact present
on RHS~\eqref{E:GRONWALLREADYPSIMODELENERGYINEQUALITY} in the case $N=18$
(see below for more details):
\begin{align} \label{E:INTROHARDESTERRRORINTEGRAL}
		A
		\int_{t'=0}^t
			\left(
			\sup_{\Sigma_{t'}^u}
			\left|
				\frac{\Lunit \upmu}{\upmu}
			\right|
			\right)
			\mathbb{H}^{(Fast)}_{18}(t',u)
		\, dt'
			+ 
			\cdots,
\end{align}
where $A$ is a universal positive constant that is
\textbf{independent of the structure of the nonlinearities and the number of times that the equations are commuted}
and $\cdots$ denotes similar or less degenerate error terms.
By itself, the error term \eqref{E:INTROHARDESTERRRORINTEGRAL}
would change the a priori estimate 
in a way that can roughly be described as follows:
\begin{align} \label{E:INTROALLTERMSTOPENERGYGRONWALLED}
	\overline{\mathbb{H}}^{(Fast)}_{18}(t,u)
	& \leq C \mathring{\upepsilon}^2 \upmu_{\star}^{-A}(t,u),
	&&
	\overline{\mathbb{H}}^{(Slow)}_{18}(t,u)
	\leq C \mathring{\upepsilon}^2 \upmu_{\star}^{-A}(t,u).
\end{align}
The factor of $A$ on RHS~\eqref{E:INTROALLTERMSTOPENERGYGRONWALLED} is 
partially responsible for the magnitude 
of the top-order blowup-exponent $11.8$ on RHS~\eqref{E:INTROTOPENERGY},
though we stress that in a detailed proof, one encounters other types of
degenerate error integrals that further enlarge the blowup-exponents.
Throughout the paper, 
we indicate the ``important'' structural constants that substantially contribute to the blowup-exponents
by drawing boxes around them
(see, for example, the RHS of the estimates of Prop.~\ref{P:TANGENTIALENERGYINTEGRALINEQUALITIES}).
The derivation of \eqref{E:INTROALLTERMSTOPENERGYGRONWALLED} 
as a consequence of the 
presence of the error term \eqref{E:INTROHARDESTERRRORINTEGRAL}
is based on a difficult Gronwall estimate that requires
having sharp information about the way that $\upmu \to 0$
as well as the behavior of $\Lunit \upmu$.
Roughly, one can show 
(see Subsubsect.\ \ref{SSS:FORMATIONOFSHOCK} for a discussion of the main ideas)
that
\begin{align} \label{E:UPMUSTARSCHEMATIC}
	\upmu_{\star}(t,u)
	\sim 1 - \TranminusdatasizeWithFactor t,
	\qquad 
	\sup_{\Sigma_{t'}^u} |\Lunit \upmu| \sim \TranminusdatasizeWithFactor,
\end{align}
from which one can obtain \eqref{E:INTROALLTERMSTOPENERGYGRONWALLED} by 
Gronwall's inequality (where it is important that the same factor $\TranminusdatasizeWithFactor$
defined in \eqref{E:INTROCRITICALBLOWUPTIMEFACTOR}
appears in both expressions in \eqref{E:UPMUSTARSCHEMATIC}).

We now explain the origin of the difficult error integral \eqref{E:INTROHARDESTERRRORINTEGRAL}
and its connection to the following aforementioned difficulty: that of avoiding a loss of a derivative 
at the top order when bounding the error term $\GeoAng^N \mytr \upchi$ in $L^2$.
To proceed, we first note that
using geometric decompositions,\footnote{By this, we essentially mean Raychaudhuri's 
identity for the component $\mbox{\upshape Ric}_{\Lunit \Lunit}$
of the Ricci curvature of the metric $g(\Psi)$. \label{FN:RAYCH}} 
one obtains the following evolution equation
for $\GeoAng^N \mytr \upchi$, expressed here in schematic form
(see the proof outline of Prop.~\ref{P:KEYPOINTWISEESTIMATE} for further discussion):
\begin{align} \label{E:TRCHISCHEMATICEVOLUTION}
	\Lunit \GeoAng^N \mytr \upchi
	& = \Lunit \Tanset^{N+1} \Psi
		+ \angLap \Tanset^N \Psi
		+ l.o.t.,
\end{align}
where $\angLap$ is the covariant Laplacian induced on $\ell_{t,u}$
by $g(\Psi)$ and $l.o.t.$ are lower-order (in the sense of the number of derivatives involved) terms 
that do \emph{not} involve the slow wave variable $\bigslow$. 
In the top-order case $N=18$, equation \eqref{E:TRCHISCHEMATICEVOLUTION}
is not useful in its current form because the RHS involves one more
derivative of $\Psi$ than we can control by 
commuting equation \eqref{E:FASTWAVE} $18$ times and deriving 
energy estimates.
To overcome this difficulty, we follow the following strategy,
whose blueprint originates in the proof of the stability of Minkowski spacetime \cite{dCsK1993}
and that was later used in the context of low-regularity well-posedness 
for wave equations \cite{sKiR2003} and finally in the context of 
shock formation \cite{dC2007}:
one can decompose the fast wave equation \eqref{E:FASTWAVE} 
using equation \eqref{E:LONOUTSIDEGEOMETRICWAVEOPERATORFRAMEDECOMPOSED}
and then algebraically replace $\upmu \angLap \Tanset^N \Psi$ 
(note the crucial factor of $\upmu$ and see \eqref{E:LONOUTSIDEGEOMETRICWAVEOPERATORFRAMEDECOMPOSED})
with 
$\Lunit \Rad \Tanset^N \Psi + \Lunit (\upmu \Tanset^{N+1} \Psi) + \cdots$ 
(written in schematic form, where $\cdots$ denotes terms depending on $\leq N + 1$ derivatives of $\Psi$)
plus the influence of the semilinear inhomogeneous terms on RHS~\eqref{E:FASTWAVE}, that is,
plus $\GeoAng^N (\upmu \times \mbox{RHS~\eqref{E:FASTWAVE}})$.
We can then bring these perfect $\Lunit$-derivative terms over to 
LHS~\eqref{E:TRCHISCHEMATICEVOLUTION} to obtain an evolution
equation for a ``modified'' version of $\GeoAng^N \mytr \upchi$,
denoted here by {\textsf Modified},\footnote{In Prop.~\ref{P:KEYPOINTWISEESTIMATE},
we denote this modified quantity by $\upchifullmodarg{\GeoAng^N}$.}
which can be written in the following schematic form:\footnote{More precisely,
as we describe in our proof outline of Prop.~\ref{P:KEYPOINTWISEESTIMATE},
to control {\textsf Modified}, we first multiply its evolution equation by an integrating
factor denoted by $\iota$.}
\begin{align} \label{E:MODIFIEDTRCHISCHEMATICEVOLUTION}
	\Lunit 
	\left\lbrace
		\overbrace{
		\upmu \GeoAng^N \mytr \upchi
		+
		\Rad \Tanset^N \Psi
		+
		\upmu \Tanset^{N+1} \Psi
		}^{\mbox{\textsf Modified}}
	\right\rbrace
	& = 
		\GeoAng^N (\upmu \times \mbox{RHS~\eqref{E:FASTWAVE}})
		+ \cdots;
\end{align}
see \eqref{E:TOPORDERMODIFIEDTRCHI} for the precise definition of the modified quantity.
We stress that if we had allowed $g = g(\Psi,\slow)$ instead of $g=g(\Psi)$, 
then our proof of \eqref{E:MODIFIEDTRCHISCHEMATICEVOLUTION} 
would have broken down in the sense that typically, we would not  
have been able to derive an analogous evolution equation
featuring terms with allowable regularity on the RHS.

To handle the first integral on RHS~\eqref{E:PSIMODELENERGYINEQUALITY},
we now algebraically replace
$
\displaystyle
\GeoAng^N \mytr \upchi
=
\frac{1}{\upmu}\mbox{\textsf Modified}
- 
\frac{1}{\upmu} \Rad \Tanset^N \Psi
- 
\Tanset^{N+1} \Psi
$,
which leads to three error integrals.
After a bit of additional work, one finds that the integral corresponding to 
the second term
$
\displaystyle
- 
\frac{1}{\upmu} \Rad \Tanset^N \Psi
$
leads to the integral in \eqref{E:INTROHARDESTERRRORINTEGRAL}.
The error integral corresponding to the first term
$
\displaystyle
\frac{1}{\upmu}\mbox{\textsf{Modified}}
$
leads to a similar but more difficult error integral that we treat
in inequality \eqref{E:FIRSTSTEPDIFFICULTINTEGRALBOUND} and the discussion
just below it
(see also the proof outline of Prop.\ \ref{P:KEYPOINTWISEESTIMATE}). 
Note that in view of RHS~\eqref{E:MODIFIEDTRCHISCHEMATICEVOLUTION},
the $L^2$ estimates for $\textsf{Modified}$
are coupled to the semilinear inhomogeneous terms on the right-hand side of the fast wave equation \eqref{E:FASTWAVE},
which involve the slow wave variable $\bigslow$. That is, our reliance on a modified version of $\mytr \upchi$
leads to the coupling of the slow wave variable
to the top-order estimates for the null mean curvature of the $\mathcal{P}_u$.
However, due to the presence of the factor of $\upmu$ on RHS~\eqref{E:MODIFIEDTRCHISCHEMATICEVOLUTION}, 
the coupling terms are weak and are among the easier error terms to treat
(see the proof outline of Lemma~\ref{L:DIFFICULTTERML2BOUND}).
The error integral corresponding to 
the third term
$- \Tanset^{N+1} \Psi$ 
from the above algebraic decomposition
is much easier to treat 
and is handled in Lemma~\ref{L:STANDARDPSISPACETIMEINTEGRALS}.

We have now sketched why the top-order quantity
$\mathbb{H}^{(Fast)}_{18}$ from \eqref{E:INTROTOPENERGY}
can blow up like $\mathring{\upepsilon}^2 \upmu_{\star}^{-11.8}$
as $\upmu_{\star} \to 0$.
To understand why the same can occur for 
$\mathbb{H}^{(Slow)}_{18}$, we simply consider the 
integral
$
\displaystyle
\int_{u'=0}^u
		  \overline{\mathbb{H}}^{(Fast)}_{18}(t,u')
		\, d u'
$
on RHS~\eqref{E:GRONWALLREADYSLOWWAVEMODELENERGYINEQUALITY} (in the case $N=18$);
the integration with respect to $u'$ does not ameliorate the strength 
of the singularity and thus
$\mathbb{H}^{(Slow)}_{18}$ can blow up at the same rate as
$\mathbb{H}^{(Fast)}_{18}$.

It remains for us to explain why, in the energy hierarchy
\eqref{E:INTROTOPENERGY}-\eqref{E:INTROLOWESTENERGY},
the energy estimates become successively less singular
with respect to powers of $\upmu_{\star}^{-1}$ at each stage in the descent.
The main ideas are the same as in all prior works on shock formation in more than one spatial dimension.
Specifically, at each level of derivatives,
the strength of the singularity is driven by
the $\GeoAng^N \mytr \upchi$-involving integral on
RHS~\eqref{E:PSIMODELENERGYINEQUALITY} and a few other integrals similar to it.
The key point is that below top order, we can estimate these integrals 
in a more direct fashion by 
integrating the RHS of
the evolution equation \eqref{E:TRCHISCHEMATICEVOLUTION} with respect to time
(recall that 
$\displaystyle
\Lunit = \frac{\partial}{\partial t}
$)
to obtain an estimate for
$\| \GeoAng^N \mytr \upchi \|_{L^2(\Sigma_t^u)}$. Such an approach
involves the loss of one derivative (which is permissible below top order)
and thus couples the below-top-order energy estimates to the top-order one.
The gain is that the resulting error integrals are less singular
with respect to powers of $\upmu_{\star}^{-1}$ compared to the top-order integral \eqref{E:INTROHARDESTERRRORINTEGRAL}.
We now provide a few more details in the just-below-top-order case $N=17$
to illustrate the main ideas behind this ``descent scheme.'' The main idea
is that the strength of the singularity is reduced with each integration in time
due to the following estimates,\footnote{The estimates stated in \eqref{E:INTEGRATEMUSTARINTIMELESSSINGULAR}
are a quasilinear version
of the model estimates $\int_{s=t}^1 s^{-B} \, ds \lesssim t^{1 - B}$
and $\int_{s=t}^1 s^{-9/10} \, ds \lesssim 1$,
where $B>1$ and $0 < t < 1$ in the model estimates
and $t=0$ represents the time of first vanishing of $\upmu_\star$.} 
valid for constants $B > 1$
(see Prop.~\ref{P:MUINVERSEINTEGRALESTIMATES} for the precise statements):
\begin{align} \label{E:INTEGRATEMUSTARINTIMELESSSINGULAR}
	\int_{t'=0}^t
		\frac{1}{\upmu_{\star}^B(t',u)}
	\, dt'
	\lesssim 
	\upmu_{\star}^{1-B}(t',u),
	\qquad
	\int_{t'=0}^t
		\frac{1}{\upmu_{\star}^{9/10}(t',u)}
	\, dt'
	\lesssim 
	1,
\end{align}
which follow from having sharp information about the
way in which $\upmu_{\star} \to 0$ in time
(see \eqref{E:UPMUSTARSCHEMATIC}).
We now explain the role that the estimates
\eqref{E:INTEGRATEMUSTARINTIMELESSSINGULAR} play in the descent scheme.
Using the above strategy and
\eqref{E:GRONWALLREADYPSIMODELENERGYINEQUALITY},
we obtain
\begin{align} \label{E:JUSTBELOWTOPORDER}
	\overline{\mathbb{H}}^{(Fast)}_{17}(t,u)
	& \leq 
	  C \mathring{\upepsilon}^2
		+
		\int_{t'=0}^t
		  \frac{1}{\upmu_{\star}^{1/2}(t',u)}
		  \sqrt{\overline{\mathbb{H}}^{(Fast)}_{17}}(t',u)
		  \int_{s=0}^{t'}
		  	\frac{1}{\upmu_{\star}^{1/2}(s,u)}
		  	\sqrt{\overline{\mathbb{H}}^{(Fast)}_{18}}(s,u)
		  \, ds
		\, d t'
	+ \cdots,
\end{align}
where $\cdots$ denotes easier error terms,
the $\sqrt{\overline{\mathbb{H}}^{(Fast)}_{18}}(s,u)$ term corresponds to the loss of one derivative
that one encounters in estimating $\| \GeoAng^{17} \mytr \upchi \|_{L^2(\Sigma_{t'}^u)}$,
and the factors of 
$
\displaystyle
\frac{1}{\upmu_{\star}^{1/2}}
$
arise from the $\upmu$-weights found in the energies
\eqref{E:INTROPSIENERGY}
along $\Sigma_t^u$ hypersurfaces.
In reality, when deriving a priori estimates, one must treat
$\overline{\mathbb{H}}^{(Fast)}_{18}$, $\overline{\mathbb{H}}^{(Fast)}_{17}$, $\cdots$ as 
unknowns in a coupled system of integral inequalities
for which one derives a priori estimates via a complicated Gronwall argument.
Here, instead of providing the lengthy technical details behind the Gronwall argument
(which, as we describe in our proof outline of Prop.~\ref{P:MAINAPRIORIENERGY},
was essentially carried out already in the proof of \cite{jSgHjLwW2016}*{Proposition 14.1}),
to illustrate the main ideas,
we demonstrate only the \emph{consistency} of the integral inequality \eqref{E:JUSTBELOWTOPORDER}
with the less singular estimate \eqref{E:INTROJUSTBELOWTOPENERGY}
(less singular compared to \eqref{E:INTROTOPENERGY}, that is).
Specifically, inserting the estimates
\eqref{E:INTROTOPENERGY}-\eqref{E:INTROJUSTBELOWTOPENERGY}
into the double time integral on RHS~\eqref{E:JUSTBELOWTOPORDER}
and using the first of \eqref{E:INTEGRATEMUSTARINTIMELESSSINGULAR} two times,
we obtain
\begin{align} \label{E:CONSISTENCYJUSTBELOWTOPORDER}
	\overline{\mathbb{H}}^{(Fast)}_{17}(t,u)
	& \leq 
	  C \mathring{\upepsilon}^2
		+
		\mathring{\upepsilon}^2
		\int_{t'=0}^t
		  \frac{1}{\upmu_{\star}^{5.4}(t',u)}
		  \int_{s=0}^{t'}
		  	\frac{1}{\upmu_{\star}^{6.4}(s,u)}
		  \, ds
		\, d t'
	+ \cdots
		\\
	& \leq 
	  C \mathring{\upepsilon}^2
		+
		\mathring{\upepsilon}^2
		\int_{t'=0}^t
		  \frac{1}{\upmu_{\star}^{10.8}(t',u)}
		 \, d t'
	+ \cdots
		\notag
			\\
	& \leq 
	  C \mathring{\upepsilon}^2
		+
		\mathring{\upepsilon}^2
		\frac{1}{\upmu_{\star}^{9.8}(t,u)}
		+
		\cdots.
		\notag
\end{align}
Thus, the strength of the singularity on RHS~\eqref{E:CONSISTENCYJUSTBELOWTOPORDER}
is at least consistent with the estimate \eqref{E:INTROJUSTBELOWTOPENERGY}
that one aims to prove. 
Subsequent to obtaining the estimate \eqref{E:INTROJUSTBELOWTOPENERGY}, one can continue 
the descent scheme, with the energies becoming successively less singular at each step in the descent.
Eventually, one reaches a level \eqref{E:FIRSTNONDEGENERATEENERGY} at which, 
thanks to the second estimate in \eqref{E:INTEGRATEMUSTARINTIMELESSSINGULAR},
one can show that the energies remain bounded all the way up to the shock.
Finally, from the non-degenerate energy estimates
\eqref{E:FIRSTNONDEGENERATEENERGY}-\eqref{E:INTROLOWESTENERGY}
and Sobolev embedding (see Lemma~\ref{L:SOBOLEV} and Cor.\ \ref{C:IMPROVEDFUNDAMENTALLINFTYBOOTSTRAPASSUMPTIONS}),
one can recover the non-degenerate
$L^{\infty}$ estimates described in Subsubsect.\ \ref{SSS:DATAASSUMPTIONSANDBOOTSTRAPASSUMPTIONS},
which, in a detailed proof, one needs to control various error terms
and to show that $\upmu$ vanishes in finite time
(as we outlined in Subsubsect.\ \ref{SSS:FORMATIONOFSHOCK}).

\section{The remaining ingredients in the geometric setup}
\label{S:GEOMETRICSETUP}
When outlining our proof in Subsect.\ \ref{SS:OVERVIEWOFPROOF}, we defined some basic
geometric objects that we use in studying the solution.
In this section,
we construct most of the remaining such objects,
exhibit their basic properties,
and give rigorous definitions of most of the
quantities that we informally referred to in Sect.\ \ref{S:INTRO}.

\subsection{Additional constructions related to the eikonal function}
\label{SS:EIKONALFUNCTIONANDRELATED}
We recall that we constructed the eikonal function in
Def.~\ref{D:INTROEIKONAL}.
To supplement the coordinates $t$ and $u$,
we now construct a local coordinate function on the tori $\ell_{t,u}$
(which are defined in Def.~\ref{D:HYPERSURFACESANDCONICALREGIONS}).
We remark that the coordinate $\vartheta$ plays only a minor role in our analysis.

\begin{definition}[\textbf{Geometric torus coordinate}]
\label{D:GEOMETRICTORUSCOORDINATE}
We define the geometric torus coordinate $\vartheta$
to be the solution to the following 
initial value problem for a transport equation:
\begin{align} \label{E:APPENDIXGEOMETRICTORUSCOORD}
	(g^{-1})^{\alpha \beta}
	\partial_{\alpha} u \partial_{\beta} \vartheta
	& = 0,
		\\
	\vartheta|_{\Sigma_0}
	& = x^2, \label{E:APPENDIXGEOMETRICTORUSCOORDINITIALCOND}
\end{align}
where $x^2$ is the (locally defined) Cartesian coordinate function on $\mathbb{T}$.
\end{definition}

\begin{definition}[\textbf{Geometric coordinates and partial derivatives}]
	\label{D:GEOMETRICCOORDINATES}
	We refer to $(t,u,\vartheta)$ as the geometric coordinates,
	where $t$ is the Cartesian time coordinate.
	We denote the corresponding geometric coordinate partial derivative vectorfields by
	\begin{align} \label{E:GEOCOORDPARTDERIVVECTORFIELDS}
		\left\lbrace 
			\frac{\partial}{\partial t}, 
			\frac{\partial}{\partial u},
			\CoordAng := \frac{\partial}{\partial \vartheta}
		\right\rbrace.
	\end{align}
\end{definition}

\begin{remark}[\textbf{Remarks on} $\CoordAng$]
	\label{R:REMARKSONCOORDANG}
	Note that $\CoordAng$ is positively oriented and globally defined
	even though $\vartheta$ is only locally defined along $\ell_{t,u}$.
\end{remark}

\subsection{Important vectorfields, the rescaled frame, and the unit frame}
\label{SS:FRAMEANDRELATEDVECTORFIELDS}
In this subsection, we construct some vectorfields that we use in our analysis
and exhibit their basic properties.
We start by defining the gradient vectorfield of the eikonal function:
\begin{align} \label{E:LGEOEQUATION}
	\Lgeo^{\nu} 
	& := - (g^{-1})^{\nu \alpha} \partial_{\alpha} u.
\end{align}
It is straightforward to see that $\Lgeo$ is future-directed\footnote{By a future-directed
vectorfield $V$, we mean that $V^0 > 0$, where $V^0$ is the ``$0$'' Cartesian component of $V$.
Similarly, by a future-directed one-form $\xi$, we mean that its $g$-dual vectorfield,
which has the Cartesian components
$(g^{-1})^{\nu \alpha} \xi_{\alpha}$, is future-directed.
We analogously define past-directed vectorfields and one-forms
by replacing ``$V^0 > 0$'' with ``$V^0 < 0$,'' etc.} 
and $g$-null:
\begin{align}  \label{E:LGEOISNULL}
g(\Lgeo,\Lgeo) 
:= g_{\alpha \beta} \Lgeo^{\alpha} \Lgeo^{\beta}
= 0.
\end{align}
Moreover, by differentiating 
the eikonal equation \eqref{E:INTROEIKONAL}
with $\D^{\nu} := (g^{-1})^{\nu \alpha} \D_{\alpha}$, 
where $\D$ is the Levi--Civita connection of $g$,
and using that $\D_{\alpha} \D_{\beta} u = \D_{\beta} \D_{\alpha} u$,
we infer that $\Lgeo$ is geodesic:
\begin{align} \label{E:LGEOISGEODESIC}
	\D_{\Lgeo} \Lgeo & = 0.
\end{align}
In addition, it is straightforward to see that
$\Lgeo$ is $g$-orthogonal to the characteristics $\mathcal{P}_u$.
Hence, the $\mathcal{P}_u$ have $g$-null normals,
which justifies our use of the terminology
\emph{null hypersurfaces} in referring to them.

It is convenient to work with a rescaled version of
$\Lgeo$ that we denote by $\Lunit$.  
Our proof will show that the Cartesian component functions 
$\lbrace \Lunit^{\alpha} \rbrace_{\alpha = 0,1,2}$ 
remain uniformly bounded up to the shock.

\begin{definition}[\textbf{Rescaled null vectorfield}]
	\label{D:LUNITDEF}
	Let $\upmu$ be the inverse foliation density from Def.\ \ref{D:FIRSTUPMU}.
	We define 
	\begin{align} \label{E:LUNITDEF}
		\Lunit
		& := \upmu \Lgeo.
	\end{align}
\end{definition}
Note that $\Lunit$ is $g$-null since $\Lgeo$ is.
We also note that by \eqref{E:APPENDIXGEOMETRICTORUSCOORD}, we have
$\Lunit \vartheta = 0$.

We now define the vectorfields $\Radunit$, 
$\Rad$, 
and $\Timenormal$,
which are transversal to the characteristics
$\mathcal{P}_u$. For our subsequent analysis, 
it is important that $\Rad$ is rescaled by a factor of $\upmu$.

\begin{definition}[$\Radunit$, $\Rad$, \textbf{and} $\Timenormal$]
	\label{D:RADANDXIDEFS}
	We define $\Radunit$ to be the unique vectorfield
	that is $\Sigma_t$-tangent, $g$-orthogonal 
	to the $\ell_{t,u}$, and normalized by
	\begin{align} \label{E:GLUNITRADUNITISMINUSONE}
		g(\Lunit,\Radunit) = -1.
	\end{align}
	We define
	\begin{align} \label{E:RADDEF}
		\Rad := \upmu \Radunit.
	\end{align}
		We define 
	\begin{align} \label{E:TIMENORMAL}
		\Timenormal 
		& := \Lunit + \Radunit.
	\end{align}
\end{definition}

In our analysis, we find it convenient to use the following two vectorfield frames.

\begin{definition}[\textbf{Two frames}]
	\label{D:RESCALEDFRAME}
	We define, respectively, the rescaled frame and the non-rescaled frame as follows:
	\begin{subequations}
	\begin{align} \label{E:RESCALEDFRAME}
		& \lbrace \Lunit, \Rad, \CoordAng \rbrace,
		&& \mbox{Rescaled frame},	\\
		&\lbrace \Lunit, \Radunit, \CoordAng \rbrace,
		&& \mbox{Non-rescaled frame}.
			\label{E:UNITFRAME}
	\end{align}
	\end{subequations}
\end{definition}

We now exhibit some basic properties of the above vectorfields.

\begin{lemma}\cite{jSgHjLwW2016}*{Lemma~2.1; \textbf{Basic properties of}  $\Lunit$, $\Radunit$, $\Rad$, \textbf{and} $\Timenormal$}
\label{L:BASICPROPERTIESOFFRAME}
The following identities hold:
\begin{subequations}
\begin{align} \label{E:LUNITOFUANDT}
	\Lunit u & = 0, \qquad \Lunit t = \Lunit^0 = 1,
		\\
	\Rad u & = 1,
	\qquad \Rad t = \Rad^0 = 0, \label{E:RADOFUANDT}
\end{align}
\end{subequations}

\begin{subequations}
\begin{align} \label{E:RADIALVECTORFIELDSLENGTHS}
	g(\Radunit,\Radunit)
	& = 1,
	\qquad
	g(\Rad,\Rad)
	= \upmu^2,
		\\
	g(\Lunit,\Radunit)
	& = -1,
	\qquad
	g(\Lunit,\Rad) = -\upmu.
	\label{E:LRADIALVECTORFIELDSNORMALIZATIONS}
\end{align}
\end{subequations}

Moreover, relative to the geometric coordinates, we have
\begin{align} \label{E:LISDDT}
	\Lunit = \frac{\partial}{\partial t}.
\end{align}

In addition, there exists an $\ell_{t,u}$-tangent vectorfield
$\NonRadialRad = \XiCoordComp \CoordAng$ 
(where $\XiCoordComp$ is a scalar function)
such that
\begin{align} \label{E:RADSPLITINTOPARTTILAUANDXI}
	\Rad 
	& = \frac{\partial}{\partial u} - \NonRadialRad
	= \frac{\partial}{\partial u} - \XiCoordComp \CoordAng.
\end{align}

The vectorfield $\Timenormal$ defined in \eqref{E:TIMENORMAL} is future-directed, $g$-orthogonal
to $\Sigma_t$ and is normalized by
\begin{align} \label{E:TIMENORMALUNITLENGTH}
	g(\Timenormal,\Timenormal) 
	& = - 1.
\end{align}
Moreover, relative to Cartesian coordinates, we have
(for $\nu = 0,1,2$):
\begin{align} \label{E:TIMENORMALRECTANGULAR}
	\Timenormal^{\nu} = - (g^{-1})^{0 \nu}.
\end{align}

	Finally, the following identities hold relative to the Cartesian coordinates 
	(for $\nu = 0,1,2$):
	\begin{align}  \label{E:DOWNSTAIRSUPSTAIRSSRADUNITPLUSLUNITISAFUNCTIONOFPSI} 
		\Radunit_{\nu} 
		& = - \Lunit_{\nu} - \delta_{\nu}^0,
		\qquad
		\Radunit^{\nu}
		= - \Lunit^{\nu}
			- (g^{-1})^{0\nu},
	\end{align}
	where $\delta_{\nu}^0$ is the standard Kronecker delta.
\end{lemma}

\subsection{Projection tensorfields, \texorpdfstring{$G_{(Frame)}$}{frame components}, and projected Lie derivatives}
\label{SS:PROJECTIONTENSORFIELDANDPROJECTEDLIEDERIVATIVES}
Many of our constructions involve projections onto 
$\Sigma_t$ and $\ell_{t,u}$.
\begin{definition}[\textbf{Projection tensorfields}]
We define the $\Sigma_t$-projection tensorfield\footnote{In \eqref{E:SIGMATPROJECTION},
we have corrected a sign error that occurred in
\cite{jSgHjLwW2016}*{Definition 2.8}.}  
$\Sigmatproject$
and the $\ell_{t,u}$-projection tensorfield
$\Lineproject$ relative to Cartesian coordinates as follows:
\begin{subequations}
\begin{align} 
	\Sigmatproject_{\nu}^{\ \mu} 
	&:=	\delta_{\nu}^{\ \mu}
			+ 
			\Timenormal_{\nu} \Timenormal^{\mu} 
		= \delta_{\nu}^{\ \mu}
			- \delta_{\nu}^{\ 0} \Lunit^{\mu}
			- \delta_{\nu}^{\ 0} \Radunit^{\mu},
			\label{E:SIGMATPROJECTION} \\
	\Lineproject_{\nu}^{\ \mu} 
	&:=	\delta_{\nu}^{\ \mu}
			+ \Radunit_{\nu} \Lunit^{\mu} 
			+ \Lunit_{\nu} (\Lunit^{\mu} + \Radunit^{\mu}) 
		= \delta_{\nu}^{\ \mu}
			- \delta_{\nu}^{\ 0} \Lunit^{\mu} 
			+  \Lunit_{\nu} \Radunit^{\mu},
			\label{E:LINEPROJECTION}
	\end{align}
\end{subequations}
where the second equalities in \eqref{E:SIGMATPROJECTION}-\eqref{E:LINEPROJECTION} follow from
\eqref{E:TIMENORMAL} and \eqref{E:DOWNSTAIRSUPSTAIRSSRADUNITPLUSLUNITISAFUNCTIONOFPSI}.
\end{definition}

\begin{definition}[\textbf{Projections of tensorfields}]
Given any spacetime tensorfield $\xi$,
we define its $\Sigma_t$ projection $\Sigmatproject \xi$
and its $\ell_{t,u}$ projection $\Lineproject \xi$
as follows:
\begin{subequations}
\begin{align} 
(\Sigmatproject \xi)_{\nu_1 \cdots \nu_n}^{\mu_1 \cdots \mu_m}
& :=
	\Sigmatproject_{\widetilde{\mu}_1}^{\ \mu_1} \cdots \Sigmatproject_{\widetilde{\mu}_m}^{\ \mu_m}
	\Sigmatproject_{\nu_1}^{\ \widetilde{\nu}_1} \cdots \Sigmatproject_{\nu_n}^{\ \widetilde{\nu}_n} 
	\xi_{\widetilde{\nu}_1 \cdots \widetilde{\nu}_n}^{\widetilde{\mu}_1 \cdots \widetilde{\mu}_m},
		\\
(\Lineproject \xi)_{\nu_1 \cdots \nu_n}^{\mu_1 \cdots \mu_m}
& := 
	\Lineproject_{\widetilde{\mu}_1}^{\ \mu_1} \cdots \Lineproject_{\widetilde{\mu}_m}^{\ \mu_m}
	\Lineproject_{\nu_1}^{\ \widetilde{\nu}_1} \cdots \Lineproject_{\nu_n}^{\ \widetilde{\nu}_n} 
	\xi_{\widetilde{\nu}_1 \cdots \widetilde{\nu}_n}^{\widetilde{\mu}_1 \cdots \widetilde{\mu}_m}.
	\label{E:STUPROJECTIONOFATENSOR}
\end{align}
\end{subequations}
\end{definition}
We say that a spacetime tensorfield $\xi$ is $\Sigma_t$-tangent 
(respectively $\ell_{t,u}$-tangent)
if $\Sigmatproject \xi = \xi$
(respectively if $\Lineproject \xi = \xi$).
Alternatively, we say that $\xi$ is a
$\Sigma_t$ tensor (respectively $\ell_{t,u}$ tensor).

\begin{definition}[\textbf{$\ell_{t,u}$ projection notation}]
	\label{D:STUSLASHPROJECTIONNOTATION}
	If $\xi$ is a spacetime tensor, then we define
	\begin{align} \label{E:TENSORSTUPROJECTED}
		\angxi := \Lineproject \xi.
	\end{align}

	If $\xi$ is a symmetric type $\binom{0}{2}$ spacetime tensor and $V$ is a spacetime
	vector, then we define
	\begin{align} \label{E:TENSORVECTORANDSTUPROJECTED}
		\angxiarg{V} 
		& := \Lineproject (\xi_V),
	\end{align}
	where $\xi_V$ is the spacetime co-vector with 
	Cartesian components $\xi_{\alpha \nu} V^{\alpha}$, $(\nu = 0,1,2)$.
\end{definition}

We often refer to the following arrays of $\ell_{t,u}$-tangent tensorfields in our analysis.

\begin{definition}[\textbf{Components of $G$ and $G'$ relative to the non-rescaled frame}]
	\label{D:GFRAMEARRAYS}
	We define
	\[
		G_{(Frame)} := \left(G_{\Lunit \Lunit}, G_{\Lunit \Radunit}, G_{\Radunit \Radunit}, \angGarg{\Lunit}, \angGarg{\Radunit}, \angG \right)
	\]
	to be the array of components of the tensorfield $G$ defined in \eqref{E:BIGGDEF} relative 
	to the non-rescaled frame \eqref{E:UNITFRAME}.
	Similarly, we define $G_{(Frame)}'$ to be the analogous array for the tensorfield 
	$G_{(Frame)}'$ defined in \eqref{E:BIGGDEF}.
\end{definition}

Throughout, 
$\Lie_V \xi$
denotes the Lie derivative of the tensorfield $\xi$ with respect to $V$.
If $V$ and $W$ are both vectorfields,
then we often use the standard Lie bracket notation
$[V,W] := \Lie_V W$.
In our analysis, we often differentiate various quantities with the projected Lie derivatives
featured in the following definition.

\begin{definition}[$\ell_{t,u}$ \textbf{and} $\Sigma_t$-\textbf{projected Lie derivatives}]
\label{D:PROJECTEDLIE}
Given a tensorfield $\xi$
and a vectorfield $V$,
we define the $\Sigma_t$-projected Lie derivative
$\SigmatLie_V \xi$ of $\xi$
and the $\ell_{t,u}$-projected Lie derivative
$\angLie_V \xi$ of $\xi$ as follows:
\begin{align} 
	\SigmatLie_V \xi
	& := \Sigmatproject \Lie_V \xi,
		\qquad
	\angLie_V \xi
	:= \Lineproject \Lie_V \xi.
	\label{E:PROJECTIONS}
\end{align}
\end{definition}

\subsection{First and second fundamental forms, 
the trace of a tensorfield,
covariant differential operators, and the geometric torus differential}

\begin{definition}[\textbf{First fundamental forms}] \label{D:FIRSTFUND}
	We define the first fundamental form $\gt$ of $\Sigma_t$ and the 
	first fundamental form $\gsphere$ of $\ell_{t,u}$ as follows:
	\begin{align}
		\gt
		& := \Sigmatproject g,
		&&
		\gsphere
		:= \Lineproject g.
		\label{E:GTANDGSPHERESPHEREDEF}
	\end{align}

	We define the corresponding inverse first fundamental forms by raising the 
	indices with $g^{-1}$:
	\begin{align}
		(\gt^{-1})^{\mu \nu}
		& := (g^{-1})^{\mu \alpha} (g^{-1})^{\nu \beta} \gt_{\alpha \beta},
		&& 
		(\gsphere^{-1})^{\mu \nu}
		:= (g^{-1})^{\mu \alpha} (g^{-1})^{\nu \beta} \gsphere_{\alpha \beta}.
		\label{E:GGTINVERSEANDGSPHEREINVERSEDEF}
	\end{align}
\end{definition}

\begin{definition}[\textbf{$\gsphere$-trace of a tensorfield}]
If $\xi$ is a type $\binom{0}{2}$ $\ell_{t,u}$-tangent tensor,
then $\mytr \xi := (\ginversesphere)^{\alpha \beta} \xi_{\alpha \beta}$ 
denotes its $\gsphere$-trace.
\end{definition}

\begin{definition}[\textbf{Differential operators associated to the metrics}] 
\label{D:CONNECTIONS}
	We use the following notation for various differential operators associated 
	to the spacetime metric $g$ and the Riemannian metric
	$\gsphere$ induced on the $\ell_{t,u}$.
	\begin{itemize}
		\item $\D$ denotes the Levi--Civita connection of the spacetime metric $g$.
		\item $\angD$ denotes the Levi--Civita connection of $\gsphere$.
		\item If $\xi$ is an $\ell_{t,u}$-tangent one-form,
			then $\angdiv \xi$ is the scalar-valued function
			$\angdiv \xi := \ginversesphere \cdot \angD \xi$.
		\item Similarly, if $V$ is an $\ell_{t,u}$-tangent vectorfield,
			then $\angdiv V := \ginversesphere \cdot \angD V_{\flat}$,
			where $V_{\flat}$ is the one-form $\gsphere$-dual to $V$.
		\item If $\xi$ is a symmetric type $\binom{0}{2}$ 
		 $\ell_{t,u}$-tangent tensorfield, then
		 $\angdiv \xi$ is the $\ell_{t,u}$-tangent 
		 one-form $\angdiv \xi := \ginversesphere \cdot \angD \xi$,
		 where the two contraction indices in $\angD \xi$
		 correspond to the operator $\angD$ and the first index of $\xi$.
		\end{itemize}
\end{definition}

\begin{definition}[\textbf{Covariant wave operator and Laplacian}] 
	\label{D:WAVEOPERATORSANDLAPLACIANS}
	We use the following standard notation.
	\begin{itemize}
		\item $\square_g := (g^{-1})^{\alpha \beta} \D_{\alpha \beta}^2$ denotes the covariant wave operator 
			corresponding to the spacetime metric $g$.
		\item $\angLap := \ginversesphere \cdot \angD^2$ denotes the covariant Laplacian 
			corresponding to $\gsphere$. 
	\end{itemize}
\end{definition}

\begin{definition}[\textbf{Geometric torus differential}]
	\label{D:ANGULARDIFFERENTIAL}
	If $f$ is a scalar function on $\ell_{t,u}$, 
	then $\angdiff f := \angD f = \Lineproject \D f$,
	where $\D f$ is the gradient one-form associated to $f$.
\end{definition}

Def.~\ref{D:ANGULARDIFFERENTIAL} allows us to avoid potentially confusing notation
such as $\angD \Lunit^i$ by instead using
$\angdiff \Lunit^i$; the latter notation clarifies that
$\Lunit^i$ is to be viewed as a scalar Cartesian component function
under differentiation.

\begin{definition}[\textbf{Second fundamental forms}]
We define the second fundamental form $k$ of $\Sigma_t$, 
which is a symmetric type $\binom{0}{2}$ $\Sigma_t$-tangent tensorfield,
by
\begin{align} \label{E:SECONDFUNDSIGMATDEF}
	k 
	&:= \frac{1}{2} \SigmatLie_{\Timenormal} \gt.
\end{align}

	We define the null second fundamental form $\upchi$, 
	which is a symmetric type $\binom{0}{2}$ $\ell_{t,u}$-tangent tensorfield,
	by
\begin{align} \label{E:CHIDEF}
	\upchi
	& := \frac{1}{2} \angLie_{\Lunit} \gsphere.
\end{align}

\end{definition}

\subsection{Identities for various tensorfields}
\label{SS:EXPRESSIONSFORCONNECTIONCOEFFICIETNS}
We now provide some identities for various $\ell_{t,u}$ tensorfields.

\begin{lemma}\cite{jSgHjLwW2016}*{Lemma 2.3; \textbf{Alternate expressions for the second fundamental forms}}
We have the following identities:
\begin{subequations}
\begin{align} 
	\upchi_{\CoordAng \CoordAng}
	& = g(\D_{\CoordAng} \Lunit, \CoordAng),
		\label{E:CHIUSEFULID} 
	&&
	\angkdoublearg{\Radunit}{\CoordAng}
	= g(\D_{\CoordAng} \Lunit, \Radunit).
\end{align}
\end{subequations}
\end{lemma}

In the next lemma, we decompose some $\ell_{t,u}$-tangent tensorfields that arise in our analysis.

\begin{lemma}\cite{jSgHjLwW2016}*{Lemma 2.13; \textbf{Expressions for} $\upzeta$ and $\angk$}
	Let $\upzeta$ be the $\ell_{t,u}$ one-form defined by
	\begin{align} \label{E:ZETADEF}
		\upzeta_{\CoordAng}
		& := \angkdoublearg{\Radunit}{\CoordAng}.
	\end{align}
	We have the following identities for
	the $\ell_{t,u}$ tensorfields 
	$\angk$ and $\upzeta$:
	\begin{subequations}
	\begin{align}
	\upzeta
	& = 
		\upmu^{-1} \upzeta^{(Trans-\Psi)}
		+ 
		\upzeta^{(Tan-\Psi)},
		 \label{E:ZETADECOMPOSED} \\
	\angk 
	& 
	= \upmu^{-1} \angk^{(Trans-\Psi)}
	+ 
	\angk^{(Tan-\Psi)},
		\label{E:ANGKDECOMPOSED}
	\end{align}
\end{subequations}
where
\begin{subequations}
	\begin{align}
		\upzeta^{(Trans-\Psi)} 
		& :=
			- \frac{1}{2} \angGarg{\Lunit} \Rad \Psi,
			\label{E:ZETATRANSVERSAL} \\
		\angk^{(Trans-\Psi)} 
		& := \frac{1}{2} \angG \Rad \Psi,
			\label{E:KABTRANSVERSAL}
	\end{align}
\end{subequations}
and
\begin{subequations}
\begin{align}
	\upzeta^{(Tan-\Psi)}
	& := \frac{1}{2} \angGarg{\Radunit} \Lunit \Psi
			- \frac{1}{2} G_{\Lunit \Radunit} \angdiff \Psi
			- \frac{1}{2} G_{\Radunit \Radunit} \angdiff \Psi,
		\label{E:ZETAGOOD} \\
	\angk^{(Tan-\Psi)} 
	& := \frac{1}{2} \angG \Lunit \Psi
			- \frac{1}{2} \angGarg{\Lunit} \otimes \angdiff \Psi
			- \frac{1}{2} \angdiff \Psi \otimes \angGarg{\Lunit} 
			- \frac{1}{2} \angGarg{\Radunit} \otimes \angdiff \Psi
			- \frac{1}{2} \angdiff \Psi \otimes \angGarg{\Radunit}.
			\label{E:KABGOOD}
\end{align}
\end{subequations}
	
\end{lemma}

\subsection{Metric decompositions}
\label{SS:METRICEXPRESSIONS}

\begin{lemma}\cite{jSgHjLwW2016}*{Lemma 2.4; \textbf{Expressions for $g$ and $g^{-1}$ in terms of the non-rescaled frame}}
\label{L:METRICDECOMPOSEDRELATIVETOTHEUNITFRAME}
We have the following identities:
\begin{subequations}
\begin{align}
	g_{\mu \nu} 
	& = - \Lunit_{\mu} \Lunit_{\nu}
			- (
					\Lunit_{\mu} \Radunit_{\nu} 
					+ 
					\Radunit_{\mu} \Lunit_{\nu}
				)
			+ \gsphere_{\mu \nu} 
			\label{E:METRICFRAMEDECOMPLUNITRADUNITFRAME},
			\\
	(g^{-1})^{\mu \nu} 
	& = 
			- \Lunit^{\mu} \Lunit^{\nu}
			- (
					\Lunit^{\mu} \Radunit^{\nu} 
					+ \Radunit^{\mu} \Lunit^{\nu}
				)
			+ (\ginversesphere)^{\mu \nu}.
			\label{E:GINVERSEFRAMEWITHRECTCOORDINATESFORGSPHEREINVERSE}
\end{align}
\end{subequations}

\end{lemma}

The following scalar function captures the ``$\ell_{t,u}$ part of $g$''.

\begin{definition}[\textbf{The metric component} $\gtancomp$]
\label{D:METRICANGULARCOMPONENT}
We define the scalar function $\gtancomp > 0$ by
\begin{align} \label{E:METRICANGULARCOMPONENT}
	\gtancomp^2
	& :=  g(\CoordAng,\CoordAng) = \gsphere(\CoordAng,\CoordAng).
\end{align}
\end{definition}

%

\subsection{The change of variables map}
\label{SS:CHOV}
In this subsection, we define the change of variables map between the geometric and Cartesian coordinates
and illustrate some of its basic properties.

\begin{definition}\label{D:CHOVMAP}
	We define $\Upsilon: [0,T) \times [0,U_0] \times \mathbb{T} \rightarrow \mathcal{M}_{T,U_0}$,
	$\Upsilon(t,u,\vartheta) := (t,x^1,x^2)$,
	to be the change of variables map from geometric to Cartesian coordinates.
\end{definition}

\begin{lemma}\cite{jSgHjLwW2016}*{Lemma 2.7; \textbf{Basic properties of the change of variables map}}
\label{L:CHOV}
	We have the following expression for the Jacobian of $\Upsilon$:
	\begin{align} \label{E:CHOV}
		\frac{\partial \Upsilon}{\partial (t,u,\vartheta)}
		& :=
		\frac{\partial (x^0,x^1,x^2)}{\partial (t,u,\vartheta)}
		= 
		\left(
		\begin{array}{ccc}
			1 & 0 & 0  \\
			\Lunit^1 & \Rad^1 + \NonRadialRad^1 & \CoordAng^1 \\
			\Lunit^2 & \Rad^2 + \NonRadialRad^2 & \CoordAng^2 \\
		\end{array}
		\right).
	\end{align}
Moreover, we have\footnote{There was a typo in \cite{jSgHjLwW2016}*{Lemma 2.7} in that the
	minus sign was missing from the RHS of the analog of \eqref{E:JACOBIAN}.}
\begin{align} \label{E:JACOBIAN}
	\mbox{\upshape{det}}
		\frac{\partial (x^0,x^1,x^2)}{\partial (t,u,\vartheta)}
	= 
		\mbox{\upshape{det}} \frac{\partial (x^1,x^2)}{\partial (u,\vartheta)}
	= - \upmu (\mbox{\upshape{det}} \gt_{ij})^{-1/2} \gtancomp,
\end{align}
where $\gtancomp$ is the metric component from Def.~\ref{D:METRICANGULARCOMPONENT}
and $(\mbox{\upshape{det}} \gt_{ij})^{-1/2}$ is a smooth function of $\Psi$ in a neighborhood of $0$ 
with $(\mbox{\upshape{det}} \gt_{ij})^{-1/2}(\Psi=0) = 1$. In \eqref{E:JACOBIAN},
$\gt$ is viewed as the Riemannian metric on $\Sigma_t^{U_0}$ defined by \eqref{E:GTANDGSPHERESPHEREDEF}
and $\mbox{\upshape{det}} \gt_{ij}$ is the determinant of the corresponding $2 \times 2$ matrix
of components of $\gt$ relative to the Cartesian spatial coordinates.
\end{lemma}

\subsection{Commutation vectorfields and a basic vectorfield commutation identity}
\label{SS:COMMUTATIONVECTORFIELDS}
In this subsection, we define the commutation vectorfields
that we use when commuting equations to obtain
estimates for the solution's derivatives.

\begin{definition}[\textbf{The vectorfields} $\GeoAng_{(Flat)}$ \textbf{and} $\GeoAng$]
\label{D:ANGULARVECTORFIELDS}
We define the Cartesian components of
the $\Sigma_t$-tangent vectorfields 
$\GeoAng_{(Flat)}$ and $\GeoAng$ 
as follows ($i=1,2$):
\begin{align}
	\GeoAng_{(Flat)}^i
	&: = \delta_2^i,
		\label{E:GEOANGEUCLIDEAN} \\
	\GeoAng^i
	& :=
	\Lineproject_a^{\ i} \GeoAng_{(Flat)}^a
		= \Lineproject_2^{\ i},
		\label{E:GEOANGDEF}
\end{align}
where $\Lineproject$ is the
$\ell_{t,u}$ projection tensorfield 
defined in \eqref{E:LINEPROJECTION}.
\end{definition}

To derive estimates for the solution's derivatives, we commute the equations
with the elements of the following two sets of vectorfields.
\begin{definition}[\textbf{Commutation vectorfields}]
	\label{D:COMMUTATIONVECTORFIELDS}
	We define the commutation set $\Fullset$ as follows:
	\begin{subequations}
	\begin{align} \label{E:COMMUTATIONVECTORFIELDS}
		\Fullset
		:= \lbrace \Lunit, \Rad, \GeoAng \rbrace,
	\end{align}
	where $\Lunit$, $\Rad$, and $\GeoAng$ are respectively defined by
	\eqref{E:LUNITDEF}, \eqref{E:RADDEF}, and \eqref{E:GEOANGDEF}.
	
	We define the $\mathcal{P}_u$-tangent commutation set $\Tanset$ as follows:
	\begin{align} \label{E:TANGENTIALCOMMUTATIONVECTORFIELDS}
		\Tanset
		:= \lbrace \Lunit, \GeoAng \rbrace.
	\end{align}
	\end{subequations}
\end{definition}

We use the following commutation identity throughout our analysis.

\begin{lemma}\cite{jSgHjLwW2016}*{Lemma 2.10; $\Lunit$, $\Rad$, $\GeoAng$ \textbf{commute with} $\angdiff$}
\label{L:LANDRADCOMMUTEWITHANGDIFF}
For scalar functions $f$ and $V \in \lbrace \Lunit, \Rad, \GeoAng \rbrace$,
we have
\begin{align} \label{E:ANGLIECOMMUTESWITHANGDIFF}
	\angLie_V \angdiff f
	& = \angdiff V f.
\end{align}
\end{lemma}

The following quantities are convenient to study because
for the solutions that we study in our main theorem, 
they are small.
\begin{definition}[\textbf{Perturbed part of various vectorfields}]
\label{D:PERTURBEDPART}
For $i=1,2$, we define the following scalar functions:
\begin{align} \label{E:PERTURBEDPART}
	\Lunit_{(Small)}^i
	& := \Lunit^i
		- \delta_1^i,
	\qquad
	\Radunit_{(Small)}^i
	:= \Radunit^i
		+ \delta_1^i,
	\qquad
	\GeoAng_{(Small)}^i
	:= \GeoAng^i - \delta_2^i.
\end{align}
The vectorfields 
$\Lunit$,
$\Radunit$,
and
$\GeoAng$
in \eqref{E:PERTURBEDPART}
are defined in
Defs.~\ref{D:LUNITDEF},
\ref{D:RADANDXIDEFS},
and \ref{D:ANGULARVECTORFIELDS}.
\end{definition}

In the next lemma, we characterize the
discrepancy between $\GeoAng_{(Flat)}$ and $\GeoAng$.

\begin{lemma}\cite{jSgHjLwW2016}*{Lemma 2.8; \textbf{Decomposition of} $\GeoAng_{(Flat)}$}
\label{L:GEOANGDECOMPOSITION}
We can decompose $\GeoAng_{(Flat)}$ into an $\ell_{t,u}$-tangent vectorfield
and a vectorfield parallel to $\Radunit$ as follows:
there exists a scalar function $\GeoAngFlatRadComponent$ such that
\begin{subequations}
\begin{align} \label{E:GEOANGINTERMSOFEUCLIDEANANGANDRADUNIT}
	\GeoAng_{(Flat)}^i
	& = \GeoAng^i 
		+ \GeoAngFlatRadComponent \Radunit^i,
			\\
	\GeoAng_{(Small)}^i
	& = - \GeoAngFlatRadComponent \Radunit^i.
	\label{E:GEOANGSMALLINTERMSOFRADUNIT}
\end{align}
\end{subequations}
Moreover, we have\footnote{In the last term in equation \eqref{E:FLATYDERIVATIVERADIALCOMPONENT},
we have corrected a sign error that occurred in \cite{jSgHjLwW2016}*{Equation (2.55)}.}
\begin{align} \label{E:FLATYDERIVATIVERADIALCOMPONENT}
	\GeoAngFlatRadComponent 
	= g(\GeoAng_{(Flat)},\Radunit)
	= g_{ab} \GeoAng_{(Flat)}^a \Radunit^b
	= g_{2a} \Radunit^a
	= g_{21}^{(Small)} \Radunit^1 + g_{22} \Radunit_{(Small)}^2.
\end{align}
\end{lemma}

\subsection{Deformation tensors}
\label{SS:BASICVECTORFIELDCOMMUTATOR}
In this subsection, we provide the standard definition of the deformation tensor
of a vectorfield.

\begin{definition}[\textbf{Deformation tensor of a vectorfield} $V$]
\label{D:DEFORMATIONTENSOR}
If $V$ is a spacetime vectorfield,
then its deformation tensor $\deform{V}$
(relative to the spacetime metric $g$)
is the symmetric type $\binom{0}{2}$ spacetime tensorfield
\begin{align} \label{E:DEFORMATIONTENSOR}
	\deformarg{V}{\alpha}{\beta}
	:= \Lie_V g_{\alpha \beta}
	= \D_{\alpha} V_{\beta} 
		+
		\D_{\beta} V_{\alpha},
\end{align}
where the last equality in \eqref{E:DEFORMATIONTENSOR} is a well-known consequence of 
the torsion-free property of the connection $\D$.
\end{definition}

\subsection{Transport equations for the eikonal function quantities}
\label{SS:TRANSPORTEQUATIONSFOREIKONALFUNCTION}
In this subsection, we provide the main evolution equations that we use to control
$\upmu$, the Cartesian component functions $\lbrace \Lunit_{(Small)}^a \rbrace_{a=1,2}$,
the $\ell_{t,u}$-tensor $\upchi$,
and their derivatives, except at the top derivative level.
We sometimes refer to these tensorfields as the \emph{eikonal function quantities}
since they are constructed out of the eikonal function.
To control their top-order derivatives, one needs to rely on
the modified quantities described in Subsubsect.\ \ref{SSS:ENERGYESTIMATES}.

\begin{lemma}\cite{jSgHjLwW2016}*{Lemma~2.12; \textbf{The transport equations verified by} $\upmu$ \textbf{and} $\Lunit_{(Small)}^i$}
\label{L:UPMUANDLUNITIFIRSTTRANSPORT}
The inverse foliation density $\upmu$ defined in \eqref{E:FIRSTUPMU} verifies the following transport equation:
\begin{align} \label{E:UPMUFIRSTTRANSPORT}
	\Lunit \upmu 
	& 
		=
		\frac{1}{2} G_{\Lunit \Lunit} \Rad \Psi
		- \frac{1}{2} \upmu G_{\Lunit \Lunit} \Lunit \Psi
		- \upmu G_{\Lunit \Radunit} \Lunit \Psi.
\end{align}

Moreover, the scalar-valued Cartesian component functions $\Lunit_{(Small)}^i$,
($i=1,2$), defined in \eqref{E:PERTURBEDPART},
verify the following transport equation:
\begin{align}
	\Lunit \Lunit_{(Small)}^i
	&  = - \frac{1}{2} G_{\Lunit \Lunit} (\Lunit \Psi) \Lunit^i
			- \frac{1}{2} G_{\Lunit \Lunit} (\Lunit \Psi) (g^{-1})^{0i}
			- \angGmixedarg{\Lunit}{\#} \cdot (\angdiff x^i) (\Lunit \Psi) 
			+ \frac{1}{2} G_{\Lunit \Lunit} (\angdiffuparg{\#} \Psi) \cdot \angdiff x^i.
				\label{E:LLUNITI} 
\end{align}

\end{lemma}

The next lemma provides a useful identity for $\Rad \Lunit_{(Small)}^i$.

\begin{lemma}\cite{jSgHjLwW2016}*{Lemma~2.14; \textbf{Formula for} $\Rad \Lunit_{(Small)}^i$}
\label{L:RADLUNITI}
We have the following identity for the scalar-valued functions $\Lunit_{(Small)}^i$, ($i=1,2$):
\begin{align}
	\Rad \Lunit_{(Small)}^i
	& = \left\lbrace
				- \frac{1}{2} G_{\Lunit \Lunit} \Rad \Psi
				+ \frac{1}{2} \upmu G_{\Lunit \Lunit} \Lunit \Psi
				+ \upmu G_{\Lunit \Radunit} \Lunit \Psi
				+ \frac{1}{2} \upmu G_{\Radunit \Radunit} \Lunit \Psi
			\right\rbrace
			\Lunit^i
			\label{E:RADLUNITI} 	\\
	& \ \ 
			+
			\left\lbrace
				- \frac{1}{2} G_{\Lunit \Lunit} \Rad \Psi
				+ \frac{1}{2} \upmu G_{\Lunit \Lunit} \Lunit \Psi
				+ \upmu G_{\Lunit \Radunit} \Lunit \Psi
				+ \frac{1}{2} \upmu G_{\Radunit \Radunit} \Lunit \Psi
			\right\rbrace
			(g^{-1})^{0i}
			\notag \\
	& \ \ 
		 - \left\lbrace
					\angGmixedarg{\Lunit}{\#} \Rad \Psi
					+ \frac{1}{2} \upmu G_{\Radunit \Radunit} \angdiffuparg{\#} \Psi
				\right\rbrace
				\cdot
				\angdiff x^i
			+ (\angdiffuparg{\#} \upmu) \cdot \angdiff x^i.
			\notag
\end{align}
\end{lemma}

\subsection{Useful expressions for the null second fundamental form}
The identities provided by the following lemma are convenient for deriving estimates for
$\upchi$ and related quantities.

\begin{lemma}\cite{jSgHjLwW2016}*{Lemma 2.15; \textbf{Identities involving} $\upchi$}
	\label{L:IDFORCHI}
	We have the following identities:
	\begin{subequations}
	\begin{align} \label{E:CHIINTERMSOFOTHERVARIABLES}
	\upchi 
	& = g_{ab} (\angdiff \Lunit^a) \otimes \angdiff x^b
		+ \frac{1}{2} \angG \Lunit \Psi,
			\\
	\mytr \upchi 
	& = g_{ab} \ginversesphere \cdot \left\lbrace (\angdiff \Lunit^a) \otimes \angdiff x^b \right\rbrace
		+ \frac{1}{2} \ginversesphere \cdot \angG \Lunit \Psi,
			 \label{E:TRCHIINTERMSOFOTHERVARIABLES}
			\\
	\Lunit \ln \gtancomp
	& = \mytr \upchi,
	\label{E:LDERIVATIVEOFVOLUMEFORMFACTOR}
\end{align}
\end{subequations}
where 
$\upchi$ is the $\ell_{t,u}$-tangent tensorfield defined by \eqref{E:CHIDEF}
and $\gtancomp$ is the metric component from Def.~\ref{D:METRICANGULARCOMPONENT}.
\end{lemma}

\subsection{Arrays of unknowns and schematic notation}
\label{SS:ARRAYS}
We often use the following arrays for convenient shorthand notation
when analyzing quantities that are tied to the fast wave variable
and the eikonal function quantities.

\begin{definition}[\textbf{Shorthand notation for the fast wave variable and the eikonal function quantities}]
\label{D:ABBREIVATEDVARIABLES}
We define the following arrays $\GdVar$ and $\BadVar$ of scalar functions:
\begin{subequations}
\begin{align}
	\GdVar 
	& := \left(\Psi, \Lunit_{(Small)}^1, \Lunit_{(Small)}^2 \right),
			\label{E:GOODABBREIVATEDVARIABLES} \\
	\BadVar 
	& := \left(\Psi, \upmu - 1, \Lunit_{(Small)}^1, \Lunit_{(Small)}^2 \right).
	\label{E:BADABBREIVATEDVARIABLES}
\end{align}
\end{subequations}
\end{definition}

\begin{notation}[\textbf{Schematic functional dependence}]
\label{N:SCHEMATICTENSORFIELDPRODUCTS}
In the remainder of the article, we use the notation
$\smoothfunction(\xi_{(1)},\xi_{(2)},\cdots,\xi_{(m)})$ to schematically depict
an expression (often tensorial and involving contractions)
that depends smoothly on the $\ell_{t,u}$-tangent tensorfields $\xi_{(1)}, \xi_{(2)}, \cdots, \xi_{(m)}$.
Note that in general, $\smoothfunction(0) \neq 0$.
\end{notation}

\begin{notation}[\textbf{The meaning of the symbol} $\Singletan$]
	Throughout, $\Singletan$ schematically denotes a 
	differential operator that is tangent to the characteristics 
	$\mathcal{P}_u$, such as $\Lunit$, $\GeoAng$, or $\angdiff$.
	For example, $\Singletan f$ might denote
	$\angdiff f$ or $\Lunit f$. We use such notation when
	the precise details of $\Singletan$ are not important.
\end{notation}

Many of the geometric tensorfields that we have defined can be expressed as functions
of $\GdVar$, $\BadVar$, $\ginversesphere$, and $\lbrace \angdiff x^a \rbrace_{a=1,2}$.
When deriving estimates for these tensorfields, it often will suffice for us to have only
crude information about the functional dependence. The next lemma is the main result 
in this direction.

\begin{lemma}[\textbf{Schematic structure of various scalar functions and tensorfields}]
	\label{L:SCHEMATICDEPENDENCEOFMANYTENSORFIELDS}
	We have the following schematic relations for scalar functions:
	\begin{subequations}
	\begin{align} \label{E:SCALARSDEPENDINGONGOODVARIABLES}
		g_{\alpha \beta},
		(g^{-1})^{\alpha \beta},
		(\gt^{-1})^{\alpha \beta},
		\gsphere_{\alpha \beta},
		(\ginversesphere)^{\alpha \beta},
		G_{\alpha \beta},
		G_{\alpha \beta}',
		\Lineproject_{\beta}^{\ \alpha},
		\Lunit^{\alpha}, 
		\Radunit^{\alpha},
		\GeoAng^{\alpha}
		& = \smoothfunction(\GdVar),
			\\
		G_{\Lunit \Lunit},
		G_{\Lunit \Radunit},
		G_{\Radunit \Radunit},
		G_{\Lunit \Lunit}',
		G_{\Lunit \Radunit}',
		G_{\Radunit \Radunit}'
		& = \smoothfunction(\GdVar),
			\label{E:GFRAMESCALARSDEPENDINGONGOODVARIABLES} \\
		g_{\alpha \beta}^{(Small)},
		\GeoAng_{(Small)}^{\alpha},
		\Radunit_{(Small)}^{\alpha},
		\GeoAngFlatRadComponent
		& = \smoothfunction(\GdVar) \GdVar,
			\label{E:LINEARLYSMALLSCALARSDEPENDINGONGOODVARIABLES}
			\\
		\Rad^{\alpha} 
		& = \smoothfunction(\BadVar).
			\label{E:SCALARSDEPENDINGONBADVARIABLES}
	\end{align}
	\end{subequations}

	Moreover, we have the following schematic relations for $\ell_{t,u}$-tangent tensorfields:
	\begin{subequations}
	\begin{align}
		\gsphere,
		\angGarg{\Lunit},
		\angGarg{\Radunit},
		\angG,
		\angGprimearg{\Lunit},
		\angGprimearg{\Radunit},
		\angGprime
		& = \smoothfunction(\GdVar,\angdiff x^1,\angdiff x^2),
			\label{E:TENSORSDEPENDINGONGOODVARIABLES} \\
	\GeoAng
	& = \smoothfunction(\GdVar,\ginversesphere,\angdiff x^1,\angdiff x^2),
			\label{E:TENSORSDEPENDINGONGOODVARIABLESANDGINVERSESPHERE}
			\\
	\upzeta^{(Tan-\Psi)},
	\angk^{(Tan-\Psi)} 
	& = \smoothfunction(\GdVar,\angdiff x^1,\angdiff x^2) \Singletan \Psi,
		\label{E:TENSORSDEPENDINGONGOODVARIABLESGOODPSIDERIVATIVES}
			\\
	\upzeta^{(Trans-\Psi)},
	\angk^{(Trans-\Psi)}
	& = \smoothfunction(\GdVar,\angdiff x^1,\angdiff x^2) \Rad \Psi,
		\label{E:TENSORSDEPENDINGONGOODVARIABLESBADDERIVATIVES} \\
	\upchi 
	& = \smoothfunction(\GdVar,\angdiff x^1,\angdiff x^2) \Singletan \GdVar,
		\label{E:TENSORSDEPENDINGONGOODVARIABLESGOODDERIVATIVES}
			\\
	\mytr \upchi 
	& = \smoothfunction(\GdVar,\ginversesphere,\angdiff x^1,\angdiff x^2) \Singletan \GdVar.
	\label{E:TENSORSDEPENDINGONGOODVARIABLESGOODDERIVATIVESANDGINVERSESPHERE}
  \end{align}
\end{subequations}
	
Finally, the null form $\mathcal{Q}^g(\partial \Psi, \partial \Psi)$ defined in \eqref{E:STANDARDNULLFORM},
upon being multiplied by $\upmu$,
has the following schematic structure:
\begin{align} \label{E:UPMUTIMESNULLFORMSSCHEMATIC}
	\upmu \mathcal{Q}^g(\partial \Psi, \partial \Psi)
	& =
	\smoothfunction(\BadVar,\Rad \Psi,\Singletan \Psi) 
	\Singletan \Psi.
\end{align}
\end{lemma}

\begin{proof}
	All relations were proved in 
	\cite{jSgHjLwW2016}*{Lemma 2.19} except for
	\eqref{E:UPMUTIMESNULLFORMSSCHEMATIC}.
	\eqref{E:UPMUTIMESNULLFORMSSCHEMATIC} follows from the identity
	$
	g^{-1} 
	= - \Lunit \otimes \Lunit
		- \Lunit \otimes \Radunit
		- \Radunit \otimes \Lunit
		+ \frac{1}{g_{ab} \GeoAng^a \GeoAng^b} \GeoAng \otimes \GeoAng
	$
	(which follows easily from \eqref{E:GINVERSEFRAMEWITHRECTCOORDINATESFORGSPHEREINVERSE})
	and the other schematic relations provided by the lemma.
	
\end{proof}

\subsection{Frame decomposition of the wave operator}
\label{SS:FRAMEDCOMPOFBOX}

In the following proposition, 
we decompose $\upmu \square_{g(\Psi)} f$ relative to the rescaled frame \eqref{E:RESCALEDFRAME}.

\begin{proposition}\cite{jSgHjLwW2016}*{Proposition~2.16; \textbf{Frame decomposition of $\upmu \square_{g(\Psi)} f$}}
	\label{P:GEOMETRICWAVEOPERATORFRAMEDECOMPOSED}
	Let $f$ be a scalar function.
	Then relative to the rescaled frame $\lbrace \Lunit, \Rad, \CoordAng \rbrace$,
	$\upmu \square_{g(\Psi)} f$ can be expressed in either of the following two forms:
	\begin{subequations}
	\begin{align} \label{E:LONOUTSIDEGEOMETRICWAVEOPERATORFRAMEDECOMPOSED}
		\upmu \square_{g(\Psi)} f 
			& = - \Lunit(\upmu \Lunit f + 2 \Rad f)
				+ \upmu \angLap f
				- \mytr \upchi \Rad f
				- \upmu \mytr \angk \Lunit f
				- 2 \upmu \upzeta^{\#} \cdot \angdiff f,
					\\
			& = - (\upmu \Lunit + 2 \Rad) (\Lunit f) 
				+ \upmu \angLap f
				- \mytr \upchi \Rad f
				- \upomega \Lunit f
				+ 2 \upmu \upzeta^{\#} \cdot \angdiff f
				+ 2 (\angdiffuparg{\#} \upmu) \cdot \angdiff f,
				\label{E:LONINSIDEGEOMETRICWAVEOPERATORFRAMEDECOMPOSED}
	\end{align}
	\end{subequations}
	where the $\ell_{t,u}$-tangent tensorfields
	$\upchi$,
	$\upzeta$,
	and
	$\angk$
	can be expressed via
	\eqref{E:CHIINTERMSOFOTHERVARIABLES},
	\eqref{E:ZETADECOMPOSED},
	and
	\eqref{E:ANGKDECOMPOSED}.
\end{proposition}

\subsection{Relationship between Cartesian and geometric partial derivative vectorfields}
\label{SS:CARTESIANVSGEOMETRICPARTIALDERIVATIVES}
In the next lemma, we provide explicit expressions for the Cartesian coordinate partial derivative
vectorfields $\partial_{\nu}$ as 
(solution-dependent) linear combinations of the commutation vectorfields of Def.~\ref{D:COMMUTATIONVECTORFIELDS}.

\begin{lemma}[{\textbf{Expression for} $\partial_{\nu}$ \textbf{in terms of geometric vectorfields}}]
	\label{L:CARTESIANVECTORFIELDSINTERMSOFGEOMETRICONES}
	We can express the Cartesian coordinate partial derivative 
	vectorfields in terms of 
	$\Lunit$, $\Radunit$, and $\GeoAng$ as follows,
	$(i=1,2)$:
	\begin{subequations}
	\begin{align} 
		\partial_t
		& = 
			\Lunit
			-  
			(g_{\alpha 0} \Lunit^{\alpha}) \Radunit
			+
			\left(
				\frac{g_{a0} \GeoAng^a}{g_{cd} \GeoAng^c \GeoAng^d}
			\right)
			\GeoAng,
			\label{E:PARTIALTINTERMSOFLUNITRADUNITANDGEOANG} \\
		\partial_i
		& = (g_{ai} \Radunit^a) \Radunit 
			+ 
			\left(
				\frac{g_{ai} \GeoAng^a}{g_{cd} \GeoAng^c \GeoAng^d}
			\right)
			\GeoAng.
			\label{E:PARTIALIINTERMSOFRADUNITANDGEOANG}
	\end{align}
	\end{subequations}
\end{lemma}

\begin{proof}
	We expand 
	$\partial_i = \upalpha_i \Radunit + \upbeta_i \GeoAng$
	for scalars $\upalpha_i$ and $\upbeta_i$.
	Taking the $g$-inner product of each side with respect to
	$\Radunit$, we obtain
	$\upalpha_i = g(\Radunit,\partial_i) = g_{ab} \Radunit^a \delta_i^b = g_{ai} \Radunit^a$.
	Similarly, 
	$\upbeta_i g_{cd} \GeoAng^c \GeoAng^d = g_{ai} \GeoAng^a$.
	Using these identities to substitute for
	$\upalpha_i$ and $\upbeta_i$,
	we conclude \eqref{E:PARTIALIINTERMSOFRADUNITANDGEOANG}.
	The identity \eqref{E:PARTIALTINTERMSOFLUNITRADUNITANDGEOANG} follows similarly
	with the help of \eqref{E:DOWNSTAIRSUPSTAIRSSRADUNITPLUSLUNITISAFUNCTIONOFPSI}.
	\end{proof}

\subsection{An algebraic expression for the transversal derivative of the slow wave}
\label{SS:EXPRESSIONFORRADSLOW}
We will use the following algebraic lemma in order to
control $\Rad \bigslow$ in terms of 
$\mathcal{P}_u$-tangential derivatives of $\bigslow$
and other simple error terms.

\begin{lemma}[\textbf{Algebraic expression for} $\Rad \bigslow$]
	\label{L:RADOFSLOWWAVEALGEBRAICALLYEXPRESSED}
	Equations \eqref{E:SLOW0EVOLUTION}-\eqref{E:SLOWEVOLUTION}
	imply the following schematic algebraic relation,
	where the $\smoothfunction$ depend smoothly on their arguments whenever
	$|\GdVar| + |\bigslow|$ is sufficiently small:
	\begin{align} \label{E:RADOFSLOWWAVEALGEBRAICALLYEXPRESSED}
		\Rad \bigslow
		& =	\smoothfunction(\BadVar,\bigslow,\Rad \Psi, \Singletan \Psi) \Singletan \Psi
				+
				\smoothfunction(\BadVar,\bigslow,\Rad \Psi, \Singletan \Psi) \Singletan \bigslow
				+ 
				\smoothfunction(\BadVar,\bigslow,\Rad \Psi, \Singletan \Psi) \bigslow.
	\end{align}
\end{lemma}
\begin{proof}
	We first write the sub-system \eqref{E:SLOW0EVOLUTION}-\eqref{E:SLOWEVOLUTION} in the matrix-vector form
	$\upmu A^{\alpha} \partial_{\alpha} \bigslow = \upmu F$, where 
	$A^{\alpha} = A^{\alpha}(\Psi,\bigslow)$,
	$A^0$ is the $4 \times 4$ identity matrix,
	and $F$ corresponds to the semilinear inhomogeneous terms on the second line of RHS~\eqref{E:SLOW0EVOLUTION}.
	It is straightforward to check (say, at an arbitrary given point $p$, relative to a frame in which 
	$h|_p = \mbox{\upshape diag}(-1,1,1)$)
	that if $\xi$ is $h$-timelike,\footnote{By 
	$h$-timelike, we mean that $(h^{-1})^{\alpha \beta} \xi_{\alpha} \xi_{\beta} < 0$. \label{FN:HTIMELIKE}}
	then the matrix $A^{\alpha} \xi_{\alpha}$ is invertible
	(we note that in fact,
	$\mbox{\upshape det} (A^{\alpha} \xi_{\alpha}) = - (\xi_0)^2 (h^{-1})^{\alpha \beta} \xi_{\alpha} \xi_{\beta}$,
	though we do not use this precise formula in this proof).
	Decomposing $\upmu A^{\alpha} \partial_{\alpha} \bigslow
	= 
	\upmu \upalpha_1 \Lunit \bigslow
	+
	\upalpha_2 \Rad \bigslow
	+
	\upmu \upalpha_3 \GeoAng \bigslow
	$,
	where the $\upalpha_i$ are $4 \times 4$ matrices,
	we compute, with the help of
	Lemma~\ref{L:CARTESIANVECTORFIELDSINTERMSOFGEOMETRICONES}
	and \eqref{E:DOWNSTAIRSUPSTAIRSSRADUNITPLUSLUNITISAFUNCTIONOFPSI},
	that 
	$
	\upalpha_2
	=
	- A^0 \Lunit_0 + A^a \Radunit_a
	= - A^{\alpha} \Lunit_{\alpha}
	$.
	Since $\Lunit$ is $g$-null, 
	we deduce from \eqref{E:VECTORSGCAUSALIMPLIESHTIMELIKE}
	that the one-form $\Lunit_{\alpha}$ is $h$-timelike.
	Thus, from the above observations, we see that $\upalpha_2$ is invertible.
	Moreover, with the help of Lemmas~\ref{L:SCHEMATICDEPENDENCEOFMANYTENSORFIELDS} and \ref{L:CARTESIANVECTORFIELDSINTERMSOFGEOMETRICONES}, 
	we deduce that $\upalpha_i = \smoothfunction(\GdVar,\bigslow)$, $(i=1,2,3)$.
	The desired relation \eqref{E:RADOFSLOWWAVEALGEBRAICALLYEXPRESSED}
	is a simple consequence of these facts,
	the assumptions on the semilinear inhomogeneous terms stated in \eqref{E:SOMENONINEARITIESARELINEAR},
	and Lemma~\ref{L:SCHEMATICDEPENDENCEOFMANYTENSORFIELDS}.
\end{proof}

\subsection{Geometric integration}
\label{SS:GEOMETRICINTEGRATION}
We define our geometric integrals in terms of length, area, and volume forms\footnote{Throughout the paper,
we blur the distinction between these forms and the corresponding (non-negative) integration measures that they induce.
The precise meaning will be clear from context. \label{FN:FORMSBLURRED}} 
that remain non-degenerate throughout the evolution,
all the way up to the shock.

\begin{definition}[\textbf{Geometric forms and related integrals}]
	\label{D:NONDEGENERATEVOLUMEFORMS}
	We define the length form
	$d \spherevol$ on $\ell_{t,u}$,
	the area form $d \tvol$ on $\Sigma_t^u$,
	the area form $d \conevol$ on $\mathcal{P}_u^t$,
	and the volume form $d \vol$ on $\mathcal{M}_{t,u}$
	as follows (relative to the geometric coordinates):
	\begin{align} \label{E:RESCALEDVOLUMEFORMS}
			d \spherevol
			& = d \spherevol(t,u,\vartheta)
			:= \gtancomp(t,u,\vartheta) d \vartheta,
				&&
			d \tvol
			=
			d \tvol(t,u',\vartheta)
			:= d \spherevol(t,u',\vartheta) du',
				\\
			d \conevol 
			& = d \conevol(t',u,\vartheta)
			:= d \spherevol(t',u,\vartheta) dt',
				&&
			d \vol 
			= d \vol(t',u',\vartheta)
			:= d \spherevol(t',u',\vartheta) du' dt',
				\notag
	\end{align}
	where $\gtancomp$ is the scalar function from Def.~\ref{D:METRICANGULARCOMPONENT}.

	If $f$ is a scalar function, then we define
	\begin{subequations}
	\begin{align}
	\int_{\ell_{t,u}}
			f
		\, d \spherevol
		& 
		:=
		\int_{\vartheta \in \mathbb{T}}
			f(t,u,\vartheta)
		\, \gtancomp(t,u,\vartheta) d \vartheta,
			\label{E:LINEINTEGRALDEF} \\
		\int_{\Sigma_t^u}
			f
		\, d \tvol
		& 
		:=
		\int_{u'=0}^u
		\int_{\vartheta \in \mathbb{T}}
			f(t,u',\vartheta)
		\, \gtancomp(t,u',\vartheta) d \vartheta du',
			\label{E:SIGMATUINTEGRALDEF} \\
		\int_{\mathcal{P}_u^t}
			f
		\, d \conevol
		& 
		:=
		\int_{t'=0}^t
		\int_{\vartheta \in \mathbb{T}}
			f(t',u,\vartheta)
		\, \gtancomp(t',u,\vartheta) d \vartheta dt',
			\label{E:PUTINTEGRALDEF} \\
		\int_{\mathcal{M}_{t,u}}
			f
		\, d \vol
		& 
		:=
		\int_{t'=0}^t
		\int_{u'=0}^u
		\int_{\vartheta \in \mathbb{T}}
			f(t',u',\vartheta)
		\, \gtancomp(t',u',\vartheta) d \vartheta du' dt'.
		\label{E:MTUTUINTEGRALDEF}
	\end{align}
	\end{subequations}
\end{definition}

\begin{remark}
	One can check that the canonical forms associated to
	$\gt$ and $g$ are, respectively,
	$\upmu d \tvol$ and $\upmu d \vol$.
\end{remark}

\subsection{Integration with respect to Cartesian forms}
\label{SS:INTEGRATIONWITHRESPECTTOCARTESIAN}
In deriving energy \emph{identities} for the slow wave variables $\bigslow$, 
it is convenient to carry out calculations relative to the Cartesian coordinates.
In this subsection, we define some basic objects that play a role in these identities.

\begin{definition}[\textbf{The one-form} $\covL$]
	\label{D:EUCLIDEANNORMALTONULLHYPERSURFACE}
 	We define $\covL$ to be the one-form with the following Cartesian components:
	\begin{align} \label{E:EUCLIDEANNORMALTONULLHYPERSURFACE}
		\covL_{\alpha}
		& := - \frac{1}{(\delta^{\kappa \lambda} \Lunit_{\kappa} \Lunit_{\lambda})^{1/2}} \Lunit_{\alpha},
	\end{align}	
	where $\delta^{\kappa \lambda}$ is the standard inverse Euclidean metric on $\mathbb{R} \times \Sigma$
	(that is, $\delta^{\kappa \lambda} = \mbox{\upshape diag (1,1,1)}$ relative to the Cartesian coordinates). 
	Note that $\covL$ is the Euclidean-unit-length co-normal to $\mathcal{P}_u$.
\end{definition}

\begin{definition}[\textbf{Cartesian volume and area forms and related integrals}]
	\label{D:CARTESIANFORMS}
	We define 
	\[d \mathcal{M} := dx^1 dx^2 dt,
	\qquad
	d \Sigma := dx^1 dx^2, 
	\qquad
	d \mathcal{P}
	\]
	to be, respectively, the standard volume form on
	$\mathcal{M}_{t,u}$ induced by the Euclidean metric\footnote{By definition, the Euclidean metric
	has the components $\mbox{\upshape diag}(1,1,1)$ relative to the standard
	Cartesian coordinates $(t,x^1,x^2)$ on $\mathbb{R} \times \Sigma$.} 
	on $\mathbb{R} \times \Sigma$,
	the standard area form induced on 
	$\Sigma_t^u$ by the Euclidean metric on $\mathbb{R} \times \Sigma$,
	and the standard area form induced on
	$\mathcal{P}_u$ by the Euclidean metric on $\mathbb{R} \times \Sigma$.

\end{definition}

\begin{remark}[\textbf{We do not use the Cartesian forms when deriving estimates}]
	\label{R:NOFLATFORMSINESTIMATES}
	We could of course provide explicit expressions
	for the integrals of functions over various domains
	with respect to the Cartesian forms of Def.~\ref{D:CARTESIANFORMS}.
	For example, we have
	$
	\int_{\Sigma_t^U}
			f
		\, d \Sigma
		=
		\int_{\lbrace 0 \leq u(t,x^1,x^2) \leq U \rbrace}
			f(t,x^1,x^2)
		\, dx^1 dx^2
	$.
	We avoid providing further detailed expressions because we do not need them. 
	The reason is that we never \emph{estimate} integrals involving the Cartesian volume forms;
	before deriving estimates,
	we will always use Lemma~\ref{L:VOLFORMRELATION} below
	order to replace the Cartesian forms with the geometric ones of
	Def.~\ref{D:NONDEGENERATEVOLUMEFORMS}.
	We use the Cartesian forms only when deriving energy \emph{identities}
	relative to the Cartesian coordinates, in which the Cartesian forms naturally appear.
\end{remark}

\subsection{Relationship between the Cartesian integration measures and the geometric integration measures}
\label{E:CARTESIANEUCLIDEANFORMCOMPARISON}
After we derive energy identities for the slow wave relative to the Cartesian coordinates, 
it will be convenient for us to express the corresponding integrals
in terms of the geometric integration measures. The following lemma provides
some identities that are useful in this regard.

\begin{lemma}[\textbf{Relationship between Cartesian and geometric integration measures}]
	\label{L:VOLFORMRELATION}
	There exist scalar functions, schematically denoted by
	$\smoothfunction(\GdVar)$, 
	that are smooth for $|\GdVar|$ sufficiently small
	and such that
	the following relationship holds
	between the geometric integration measures corresponding to Def.~\ref{D:NONDEGENERATEVOLUMEFORMS}
	and the Cartesian integration measures corresponding to Def.~\ref{D:CARTESIANFORMS}
	(see Footnote~\ref{FN:FORMSBLURRED}):
	\begin{align} \label{E:VOLFORMRELATION}
		d \mathcal{M}
		& = \upmu 
				\left\lbrace 
					1 + \upgamma \smoothfunction(\upgamma)
				\right\rbrace 
				d \vol,
		&
		d \Sigma 
		& = \upmu
				\left\lbrace 
					1 + \upgamma \smoothfunction(\upgamma)
				\right\rbrace 
				d \tvol,
		&
		d \mathcal{P}
		& =  
			\left\lbrace 
				\sqrt{2} + \GdVar \smoothfunction(\GdVar)
			\right\rbrace 
			d \conevol. 	
	\end{align}
\end{lemma}

\begin{proof}
	We prove only the identity
	$
	d \mathcal{P}
	=  
			\left\lbrace 
				\sqrt{2} + \GdVar \smoothfunction(\GdVar)
			\right\rbrace 
			d \conevol
	$
	since the other two 
	identities in \eqref{E:VOLFORMRELATION}
	are a straightforward consequence of
	Lemma~\ref{L:CHOV} 
	(in particular, the Jacobian determinant\footnote{Note that
	the minus sign in equation \eqref{E:JACOBIAN}
	does not appear in equation \eqref{E:VOLFORMRELATION}
	since we are viewing \eqref{E:VOLFORMRELATION} as a relationship between integration measures.} 
	expressions in \eqref{E:JACOBIAN}).
	In the proof, 
	we view $d \conevol$ 
	(see \eqref{E:RESCALEDVOLUMEFORMS})
	to be the two-form $\gtancomp dt \wedge d \vartheta$ on $\mathcal{P}_u$,
	where $dt \wedge d \vartheta = dt \otimes d \vartheta - d \vartheta \otimes dt$.
	Similarly, we view $d \mathcal{P}$ to be the two-form induced on 
	$\mathcal{P}_u$ by the standard Euclidean metric 
	$\delta_{\alpha \beta} = \mbox{\upshape diag (1,1,1)}$
	on $\mathbb{R} \times \Sigma$.
	Then relative to Cartesian coordinates, 
	we have $d \mathcal{P} = (dx^0 \wedge dx^1 \wedge dx^2) \cdot V$,
	where $V$ is the future-directed Euclidean normal vectorfield to 
	$\mathcal{P}_u$ and $(dx^0 \wedge dx^1 \wedge dx^2) \cdot V$ denotes contraction 
	of $V$ against the first slot of $dx^0 \wedge dx^1 \wedge dx^2$.
	Note that $V^{\alpha} = \delta^{\alpha \beta} \covL_{\beta}$,
	where $\covL_{\alpha}$ is defined in \eqref{E:EUCLIDEANNORMALTONULLHYPERSURFACE}
	and $\delta^{\alpha \beta} = \mbox{\upshape diag (1,1,1)}$ is the standard inverse Euclidean metric
	on $\mathbb{R} \times \Sigma$.
	Since $d \conevol$ and $d \mathcal{P}$ are proportional
	and since $dt \wedge d \vartheta \cdot (\Lunit \otimes \CoordAng) = 1$,
	it suffices to show that
	$
	\left\lbrace 
		\sqrt{2} + \GdVar \smoothfunction(\GdVar)
	\right\rbrace 
	\gtancomp
	=  
	(dx^0 \wedge dx^1 \wedge dx^2) \cdot (V \otimes \Lunit \otimes \CoordAng)
	$.
	To proceed, we note that
	$(dx^0 \wedge dx^1 \wedge dx^2) \cdot (V \otimes \Lunit \otimes \CoordAng)$
	is equal to the determinant of the $3 \times 3$ matrix
	$
	N:
	=
	\left(
		\begin{array}{ccc}
			V^0 & \Lunit^0 & 0  \\
			V^1 & \Lunit^1 & \CoordAng^1 \\
			V^2 & \Lunit^2 & \CoordAng^2 \\
		\end{array}
		\right)
	$.
	Next, we consider the $3 \times 3$ matrix
	$M := N^{\top} \cdot g \cdot N$,
	where we view $g$
	as a $3 \times 3$ matrix expressed relative to the Cartesian coordinates.
	Since
	\eqref{E:LITTLEGDECOMPOSED}-\eqref{E:METRICPERTURBATIONFUNCTION}
	imply that
	$|\mbox{\upshape det} g| = 1 + \GdVar \smoothfunction(\GdVar)$
	relative to the Cartesian coordinates
	and since
	$|\mbox{\upshape det} M| = |\mbox{\upshape det} g| (\mbox{\upshape det} N)^2$,
	the desired conclusion will follow once we prove
	$|\mbox{\upshape det} M| = \gtancomp^2 \left\lbrace 2 + \GdVar \smoothfunction(\GdVar) \right\rbrace$.
	To obtain this relation, we first compute that
	$
	M 
	=
	\left(
		\begin{array}{ccc}
			g(V,V) & g(V,\Lunit) & g(V,\CoordAng) \\
			g(\Lunit,V) & 0 & 0 \\
			g(\CoordAng,V) & 0 & \gtancomp^2 \\
		\end{array}
		\right)
	$
	and thus
	$
	\mbox{\upshape det} M
	=
	- 
	\left(g(\Lunit,V) \right)^2 \gtancomp^2
	$.
	Finally, from 
	\eqref{E:LITTLEGDECOMPOSED},
	\eqref{E:METRICPERTURBATIONFUNCTION},
	\eqref{E:PERTURBEDPART},
	and
	\eqref{E:EUCLIDEANNORMALTONULLHYPERSURFACE},
	we compute (relative to the Cartesian coordinates)
	that $g(\Lunit,V) = - \sqrt{2} + \GdVar \smoothfunction(\GdVar)$,
	which, in conjunction with the above computations, 
	yields $|\mbox{\upshape det} M| = \gtancomp^2 \left\lbrace 2 + \GdVar \smoothfunction(\GdVar) \right\rbrace$ as desired.

\end{proof}

\section{Norms, initial data, bootstrap assumptions, and smallness assumptions}
\label{S:NORMSANDBOOTSTRAP}
In this section, we state our size assumptions on the data,
formulate appropriate bootstrap assumptions for the solution,
and state our smallness assumptions. 
As we mentioned in the introduction,
the solutions that we study in this article 
are perturbations of simple outgoing plane waves.
In Subsect.\ \ref{SS:EXISTENCEOFDATA}, we show
the existence of data for the system
\eqref{E:FASTWAVE}
+
\eqref{E:SLOW0EVOLUTION}-\eqref{E:SYMMETRYOFMIXEDPARTIALS}
that verify the size assumptions of the present article.
In particular, we show that under the assumptions \eqref{E:SOMENONINEARITIESARELINEAR} on the semilinear inhomogeneous terms,
the system admits simple outgoing plane wave solutions,
thereby justifying our study of their perturbations.

\subsection{Norms}
\label{SS:NORMS}
In our analysis, we will derive estimates
for scalar functions and $\ell_{t,u}$-tangent tensorfields. 
We use the metric $\gsphere$
when taking the pointwise norm of 
$\ell_{t,u}$-tangent tensorfields,
a concept that we make precise 
in the next definition.

\begin{definition}[\textbf{Pointwise norms}]
	\label{D:POINTWISENORM}
	Let $\gsphere$ be the Riemannian metric on $\ell_{t,u}$ from Def.\ \ref{D:FIRSTFUND}.
	If $\xi_{\nu_1 \cdots \nu_n}^{\mu_1 \cdots \mu_m}$ 
	is a type $\binom{m}{n}$ $\ell_{t,u}$ tensor,
	then we define the norm $|\xi| \geq 0$ by
	\begin{align} \label{E:POINTWISENORM}
		|\xi|^2
		:= 
		\gsphere_{\mu_1 \widetilde{\mu}_1} \cdots \gsphere_{\mu_m \widetilde{\mu}_m}
		(\ginversesphere)^{\nu_1 \widetilde{\nu}_1} \cdots (\ginversesphere)^{\nu_n \widetilde{\nu}_n}
		\xi_{\nu_1 \cdots \nu_n}^{\mu_1 \cdots \mu_m}
		\xi_{\widetilde{\nu}_1 \cdots \widetilde{\nu}_n}^{\widetilde{\mu}_1 \cdots \widetilde{\mu}_m}.
	\end{align}
\end{definition}

We use $L^2$ and $L^{\infty}$ norms in our analysis.

\begin{definition}[$L^2$ \textbf{and} $L^{\infty}$ \textbf{norms}]
In terms of the geometric forms of Def.~\ref{D:NONDEGENERATEVOLUMEFORMS},
we define the following norms for 
$\ell_{t,u}$-tangent tensorfields:
\label{D:SOBOLEVNORMS}
	\begin{subequations}
	\begin{align}  \label{E:L2NORMS}
			\left\|
				\xi
			\right\|_{L^2(\ell_{t,u})}^2
			& :=
			\int_{\ell_{t,u}}
				|\xi|^2
			\, d \spherevol,
				\qquad
			\left\|
				\xi
			\right\|_{L^2(\Sigma_t^u)}^2
			:=
			\int_{\Sigma_t^u}
				|\xi|^2
			\, d \tvol,
				\\
			\left\|
				\xi
			\right\|_{L^2(\mathcal{P}_u^t)}^2
			& :=
			\int_{\mathcal{P}_u^t}
				|\xi|^2
			\, d \conevol,
			\notag
	\end{align}

	\begin{align} 
			\left\|
				\xi
			\right\|_{L^{\infty}(\ell_{t,u})}
			& :=
				\mbox{ess sup}_{\vartheta \in \mathbb{T}}
				|\xi|(t,u,\vartheta),
			\qquad
			\left\|
				\xi
			\right\|_{L^{\infty}(\Sigma_t^u)}
			:=
			\mbox{ess sup}_{(u',\vartheta) \in [0,u] \times \mathbb{T}}
				|\xi|(t,u',\vartheta),
			\label{E:LINFTYNORMS}
				\\
			\left\|
				\xi
			\right\|_{L^{\infty}(\mathcal{P}_u^t)}
			& :=
			\mbox{ess sup}_{(t',\vartheta) \in [0,t] \times \mathbb{T}}
				|\xi|(t',u,\vartheta).
			\notag
	\end{align}
	\end{subequations}
\end{definition}

\begin{remark}[\textbf{Subset norms}]
	\label{R:SUBSETNORMS}
	We sometimes use norms 
	$\| \cdot \|_{L^2(\Omega)}$
	and
	$\| \cdot \|_{L^{\infty}(\Omega)}$,
	where $\Omega$ is a subset of $\Sigma_t^u$.
	These norms are defined by replacing 
	$\Sigma_t^u$ with $\Omega$ in 
	\eqref{E:L2NORMS} and \eqref{E:LINFTYNORMS}.
\end{remark}

\subsection{Strings of commutation vectorfields and vectorfield seminorms}
\label{SS:STRINGSOFCOMMUTATIONVECTORFIELDS}
We use the following shorthand notation to capture the relevant structure
of our vectorfield operators and to schematically depict estimates.

\begin{definition}[\textbf{Strings of commutation vectorfields and vectorfield seminorms}] \label{D:VECTORFIELDOPERATORS}
	\ \\
	\begin{itemize}
		\item $\Fullset^{N;M} f$ 
			denotes an arbitrary string of $N$ commutation
			vectorfields in $\Fullset$ (see \eqref{E:COMMUTATIONVECTORFIELDS})
			applied to $f$, where the string contains \emph{at most} $M$ factors of the $\mathcal{P}_u^t$-transversal
			vectorfield $\Rad$. 
		\item $\Tanset^N f$
			denotes an arbitrary string of $N$ commutation
			vectorfields in $\Tanset$ (see \eqref{E:TANGENTIALCOMMUTATIONVECTORFIELDS})
			applied to $f$.
		\item 
			For $N \geq 1$,
			$\Fullset_*^{N;M} f$
			denotes an arbitrary string of $N$ commutation
			vectorfields in $\Fullset$ 
			applied to $f$, where the string contains \emph{at least} one $\mathcal{P}_u$-tangent factor 
			and \emph{at most} $M$ factors of $\Rad$.
			We also set  $\Fullset_*^{0;0} f := f$.
		\item For $N \geq 1$,
					$\Tanset_*^N f$ 
					denotes an arbitrary string of $N$ commutation
					vectorfields in $\Tanset$ 
					applied to $f$, where the string contains
					\emph{at least one factor} of $\GeoAng$ or \emph{at least two factors} of $\Lunit$.
		\item For $\ell_{t,u}$-tangent tensorfields $\xi$, 
					we similarly define strings of $\ell_{t,u}$-projected Lie derivatives 
					such as $\angLie_{\Fullset}^{N;M} \xi$.
	\end{itemize}

	We also define pointwise seminorms constructed out of sums of the above strings of vectorfields:
	\begin{itemize}
		\item $|\Fullset^{N;M} f|$ 
		simply denotes the magnitude of one of the $\Fullset^{N;M} f$ as defined above
		(there is no summation).
	\item $|\Fullset^{\leq N;M} f|$ is the \emph{sum} over all terms of the form $|\Fullset^{N';M} f|$
			with $N' \leq N$ and $\Fullset^{N';M} f$ as defined above.
			When $N=M=1$, we sometimes write $|\Fullset^{\leq 1} f|$ instead of $|\Fullset^{\leq 1;1} f|$.
		\item $|\Fullset^{[1,N];M} f|$ is the sum over all terms of the form $|\Fullset^{N';M} f|$
			with $1 \leq N' \leq N$ and $\Fullset^{N';M} f$ as defined above.
		\item Sums such as 
			$|\Tanset^{\leq N} f|$,
			$|\Tanset_*^{[1,N]} f|$,
			$|\angLie_{\Fullset}^{\leq N;M} \xi|$,
			$|\GeoAng^{\leq 1} f|$,
			$|\Rad^{[1,N]} f|$,
			etc.,
			are defined analogously.
			For example, 
			$|\Rad^{[1,N]} f| 
			= |\Rad f| 
			+ |\Rad \Rad f| 
			+ \cdots 
			+ |\overbrace{\Rad \Rad \cdots \Rad}^{N \mbox{ \upshape copies}} f|
			$.
		\end{itemize}

\end{definition}

\begin{remark}
	Some operators in Def.~\ref{D:VECTORFIELDOPERATORS} are decorated with a $*$.
	These operators involve $\mathcal{P}_u$-tangential differentiations that often
	lead to a gain in smallness in the estimates.
	More precisely, the operators $\Tanset_*^N$ always lead to a gain in smallness while
	the operators $\Fullset_*^{N;M}$ lead to a gain in smallness except 
	perhaps when they are applied to $\upmu$
	(because 
	$\Lunit \upmu$ and its $\Rad$ derivatives are not generally small for the solutions under study).
	We clarify that for the simple plane wave solutions (whose perturbations we study in our main theorem),
	the $\Tanset_*^N$ derivatives of all variables in the array
	$\BadVar$ (see \eqref{E:BADABBREIVATEDVARIABLES})
	completely vanish, and that the same is true for
	$\Fullset_*^{N;M} \bigslow$ (by our definition of a simple wave).
\end{remark}

\subsection{Assumptions on the data}
\label{SS:DATAASSUMPTIONS}
In this subsection, we state our size assumptions on the data.
Our assumptions involve three parameters: $\Psiep > 0$, $\mathring{\upepsilon} \geq 0$,
and $\mathring{\updelta} > 0$, whose sizes we describe in Subsect.\ \ref{SS:SMALLNESSASSUMPTIONS}.
As we mentioned in Subsubsect.\ \ref{SSS:SPACETIMEREGION},
our analysis also involves the following data-dependent parameter.

\begin{definition}[\textbf{The quantity that controls the time of first shock formation}]
	\label{D:CRITICALBLOWUPTIMEFACTOR}
	We define
	\begin{align} \label{E:CRITICALBLOWUPTIMEFACTOR}
		\TranminusdatasizeWithFactor
		& := \frac{1}{2} 
		\sup_{\Sigma_0^1} 
		\left[G_{\Lunit \Lunit} \Rad \Psi \right]_-.
	\end{align}
\end{definition}
Our main theorem shows that
the reciprocal of $\TranminusdatasizeWithFactor$
is approximately equal to the time of first shock formation.

We assume that the data verify the following size assumptions.

\medskip

\noindent \underline{\textbf{Assumptions along} $\Sigma_0^1$}.
\begin{subequations}
\begin{align} \label{E:PSIL2SMALLDATAASSUMPTIONSALONGSIGMA0}
		\left\|
			\Fullset_{\ast}^{[1,19];1} \Psi
		\right\|_{L^2(\Sigma_0^1)}
		& \leq \mathring{\upepsilon},
			\\
		\left\|
			\Tanset^{\leq 18} \bigslow
		\right\|_{L^2(\Sigma_0^1)}
		& \leq \mathring{\upepsilon},
			\label{E:SLOWL2SMALLDATAASSUMPTIONSALONGSIGMA0} 
\end{align}
\end{subequations}

\begin{subequations}
\begin{align}
		\left\|
			\Psi
		\right\|_{L^{\infty}(\Sigma_0^1)}
			& \leq \Psiep,
			 \label{E:PSIITSELFLINFTYSMALLDATAASSUMPTIONSALONGSIGMA0} 
			\\
		\left\|
			\Fullset_{\ast}^{[1,11];1} \Psi
		\right\|_{L^{\infty}(\Sigma_0^1)},
			\,
		\left\|
			\Fullset_{\ast}^{[1,10];2} \Psi
		\right\|_{L^{\infty}(\Sigma_0^1)},
			\,
		\left\|
			\Lunit \Rad \Rad \Rad \Psi
		\right\|_{L^{\infty}(\Sigma_0^1)}
		& \leq \mathring{\upepsilon},
			 \label{E:PSILINFTYSMALLDATAASSUMPTIONSALONGSIGMA0} 
			\\
		\left\|
			\Tanset^{\leq 10} \bigslow
		\right\|_{L^{\infty}(\Sigma_0^1)},
			\,
		\left\|
			\Fullset^{\leq 9;1} \bigslow
		\right\|_{L^{\infty}(\Sigma_0^1)},
			\,
		\left\|
			\Rad \Rad \bigslow
		\right\|_{L^{\infty}(\Sigma_0^1)}
		& \leq \mathring{\upepsilon},
			 \label{E:SLOWWAVELINFTYSMALLDATAASSUMPTIONSALONGSIGMA0}  
\end{align}
\end{subequations}

\begin{align} \label{E:FASTWAVELINFTYLARGEDATAASSUMPTIONSALONGSIGMA0}
	\left\|
		\Rad^{[1,3]} \Psi
	\right\|_{L^{\infty}(\Sigma_0^1)}
	& \leq \mathring{\updelta}.
\end{align}

\noindent \underline{\textbf{Assumptions along} $\mathcal{P}_0^{2 \TranminusdatasizeWithFactor^{-1}}$}.
\begin{subequations}
\begin{align} \label{E:PSIL2SMALLDATAASSUMPTIONSALONGP0}
		\left\|
			\Tanset^{[1,19]} \Psi
		\right\|_{L^2\left(\mathcal{P}_0^{2 \TranminusdatasizeWithFactor^{-1}}\right)}
		& \leq \mathring{\upepsilon},
			\\
		\left\|
			\Tanset^{\leq 18} \bigslow
		\right\|_{L^2\left(\mathcal{P}_0^{2 \TranminusdatasizeWithFactor^{-1}}\right)}
		& \leq \mathring{\upepsilon},
			 \label{E:SLOWL2SMALLDATAASSUMPTIONSALONGP0} 
\end{align}
\end{subequations}

\begin{subequations}
\begin{align}
		\left\|
			\Tanset^{\leq 17} \Psi
		\right\|_{L^{\infty}\left(\mathcal{P}_0^{2 \TranminusdatasizeWithFactor^{-1}}\right)}
		& \leq \mathring{\upepsilon},
			\label{E:PSILINFTYSMALLDATAASSUMPTIONSALONGP0}
				\\
		\left\|
			\Tanset^{\leq 16} \bigslow
		\right\|_{L^{\infty} \left(\mathcal{P}_0^{2 \TranminusdatasizeWithFactor^{-1}}\right)}
		& \leq \mathring{\upepsilon}.
			\label{E:SLOWLINFTYSMALLDATAASSUMPTIONSALONGP0}
\end{align}
\end{subequations}

\noindent \underline{\textbf{Assumptions along} $\ell_{t,0}$}.
We assume that for $t \in [0,2 \TranminusdatasizeWithFactor^{-1}]$, we have
\begin{align} \label{E:LINFTYPSISMALLDATAASSUMPTIONSALONGLT0}
		\left\|
			\Fullset^{\leq 1} \Psi
		\right\|_{L^{\infty}(\ell_{t,0})}
		& \leq \mathring{\upepsilon}.
\end{align}

\noindent \underline{\textbf{Assumptions along} $\ell_{0,u}$}.
We assume that for $u \in [0,1]$, we have
\begin{subequations}
\begin{align}
		\left\|
			\Tanset^{[1,18]} \Psi
		\right\|_{L^2(\ell_{0,u})}
		& \leq \mathring{\upepsilon},
			\label{E:PSIL2SMALLDATAASSUMPTIONSALONGL0U} \\
		\left\|
			\Tanset^{\leq 17} \bigslow
		\right\|_{L^2(\ell_{0,u})}
		& \leq \mathring{\upepsilon}.
		\label{E:SLOWL2SMALLDATAASSUMPTIONSALONGL0U}
\end{align}
\end{subequations}

\begin{remark}[\textbf{A brief description of the data-size parameters}]
	\label{R:DATAPARAMETERSBRIEFDESCRIPTION}
	To prove our main theorem, we assume that $\Psiep$ and $\mathring{\upepsilon}$
	are small in a sense that we make precise in Subsect.\ \ref{SS:SMALLNESSASSUMPTIONS}.
	The parameters $\TranminusdatasizeWithFactor > 0$ and $\mathring{\updelta} > 0$
	are allowed to be small or large in an absolute sense, but the smallness of
	$\mathring{\upepsilon}$ must be adapted to $\TranminusdatasizeWithFactor^{-1}$ and $\mathring{\updelta}$.
	The parameter $\mathring{\upepsilon}$ vanishes for simple outgoing plane wave solutions
	(see Subsect.\ \ref{SS:EXISTENCEOFDATA} for further discussion).
\end{remark}

\subsection{The data of the eikonal function quantities}
\label{SS:INITIALBEHAVIOROFEIKONAL}
The data-size assumptions of Subsect.\ \ref{SS:DATAASSUMPTIONS}
determine the initial size of various quantities 
constructed out of the eikonal function. In the this subsection,
under appropriate smallness assumptions,
we estimate these data-dependent quantities in various norms.
The main result is Lemma~\ref{L:BEHAVIOROFEIKONALFUNCTIONQUANTITIESALONGSIGMA0}.
We start with a simple lemma that provides algebraic identities 
that hold along $\Sigma_0$.

\begin{lemma}\cite{jSgHjLwW2016}*{Lemma~7.2; \textbf{Algebraic identities along} $\Sigma_0$}
\label{L:ALGEBRAICIDALONGSIGMA0}
The following identities hold along $\Sigma_0$ (for $i=1,2$):
\begin{align}
	\upmu
	& = \frac{1}{\sqrt{(\gtinverse)^{11}}},
	\qquad
	\Lunit_{(Small)}^i
	 = 	\frac{(\gtinverse)^{i1}}{\sqrt{(\gtinverse)^{11}}} 
	 		- \delta^{i1}
			- (g^{-1})^{0i},
	\qquad
	\NonRadialRad^i
	 =  \frac{(\gtinverse)^{i1}}{(\gtinverse)^{11}}
	 	- \delta^{i1},
		\label{E:INITIALRELATIONS}
\end{align}
where $\gt$ is viewed as the $2 \times 2$ matrix of Cartesian spatial components 
of the Riemannian metric on $\Sigma_0$ defined by \eqref{E:GTANDGSPHERESPHEREDEF},
$\gt^{-1}$ is the corresponding inverse matrix,
and $\NonRadialRad$ is the $\ell_{t,u}$-tangent vectorfield from \eqref{E:RADSPLITINTOPARTTILAUANDXI}.
\end{lemma}

We now provide the main result of this subsection.

\begin{lemma}[\textbf{Behavior of the eikonal function quantities along} $\Sigma_0^1$ \textbf{and} 
$\mathcal{P}_0^{2 \TranminusdatasizeWithFactor^{-1}}$]
\label{L:BEHAVIOROFEIKONALFUNCTIONQUANTITIESALONGSIGMA0}
For data verifying 
the assumptions of Subsect.\ \ref{SS:DATAASSUMPTIONS},
the following $L^2$ and $L^{\infty}$ estimates hold along $\Sigma_0^1$
and $\mathcal{P}_0^{2 \TranminusdatasizeWithFactor^{-1}}$
whenever 
$\Psiep$ and
$\mathring{\upepsilon}$ are sufficiently small,
where the constants $C$ are allowed to depend 
on $\mathring{\updelta}$
and $\TranminusdatasizeWithFactor^{-1}$, 
the constants $C_{\mydiam}$ can be chosen to be independent\footnote{Some of the constants denoted by ``$C$'' can also
be chosen to be independent of $\mathring{\updelta}$ and $\TranminusdatasizeWithFactor^{-1}$. We use the symbol ``$C_{\mydiam}$'' 
only when it is important that the constant can be chosen to be independent of $\mathring{\updelta}$ and 
$\TranminusdatasizeWithFactor^{-1}$.} 
of $\mathring{\updelta}$ and $\TranminusdatasizeWithFactor^{-1}$,
and $i=1,2$ 
(see Subsect.\ \ref{SS:STRINGSOFCOMMUTATIONVECTORFIELDS} regarding the vectorfield operator notation):
\begin{align} 
	\left\|
		\Fullset_*^{[1,19];3} \Lunit_{(Small)}^i
	\right\|_{L^2(\Sigma_0^1)}
	& \leq C \mathring{\upepsilon},
	\label{E:LUNITIDATAL2CONSEQUENCES}
\end{align}

\begin{align}  \label{E:UPMUDATATANGENTIALL2CONSEQUENCES}
	\left\|
		\Tanset_*^{[1,19]} \upmu
	\right\|_{L^2(\Sigma_0^1)}
	& \leq C \mathring{\upepsilon},
\end{align}

\begin{subequations}
\begin{align} 
	\left\|
		\Lunit_{(Small)}^i
	\right\|_{L^{\infty}(\Sigma_0^1)}
	& \leq C_{\mydiam} \Psiep, 
		\label{E:LUNITIITSEFLSMALLDATALINFINITYCONSEQUENCES} \\
	\left\|
		\Fullset_*^{[1,17];2} \Lunit_{(Small)}^i
	\right\|_{L^{\infty}(\Sigma_0^1)}
	& \leq C \mathring{\upepsilon},
		\label{E:LUNITIDATASMALLLINFTYCONSEQUENCES} \\
	\left\|
		\Rad^{[1,2]} \Lunit_{(Small)}^i
	\right\|_{L^{\infty}(\Sigma_0^1)}
	& \leq C,
\end{align}
\end{subequations}

\begin{subequations}
\begin{align}  
	\left\|
		\upmu - 1
	\right\|_{L^{\infty}(\Sigma_0^1)}
	& \leq C_{\mydiam} \Psiep,
		\label{E:UPITSELFLINFINITYSIGMA0CONSEQUENCES} \\
	\left\|
		\Tanset_*^{[1,17]} \upmu
	\right\|_{L^{\infty}(\Sigma_0^1)}
	& \leq C \mathring{\upepsilon},
		\label{E:UPMUDATATANGENTIALLINFINITYCONSEQUENCES} \\
	\left\|
		\Lunit \Rad^{[0,2]} \upmu 
	\right\|_{L^{\infty}(\Sigma_0^1)},
		\,
	\left\|
		\Rad^{[0,2]} \Lunit \upmu 
	\right\|_{L^{\infty}(\Sigma_0^1)},
		\,
	\left\|
		\Rad \Lunit \Rad \upmu 
	\right\|_{L^{\infty}(\Sigma_0^1)},
		\,
	\left\|
		\Rad^{[1,2]} \upmu 
	\right\|_{L^{\infty}(\Sigma_0^1)}
	& \leq C,
	\label{E:UPMUDATARADIALLINFINITYCONSEQUENCES}
\end{align}
\end{subequations}

\begin{subequations}
\begin{align}
	\left\|
		\Lunit_{(Small)}^i
	\right\|_{L^{\infty}(\mathcal{P}_0^{2 \TranminusdatasizeWithFactor^{-1}})}
	& \leq C \mathring{\upepsilon},
		\label{E:LUNITIDATASMALLP0LINFTYCONSEQUENCES} \\
	\left\|
		\upmu - 1
	\right\|_{L^{\infty}(\mathcal{P}_0^{2 \TranminusdatasizeWithFactor^{-1}})}
	& \leq C \mathring{\upepsilon}.
		\label{E:UPITSELFLINFINITYP0CONSEQUENCES} 
\end{align}
\end{subequations}

\end{lemma}

\begin{proof}
Readers can consult \cite{jSgHjLwW2016}*{Lemma~7.3} for the main ideas on how to prove 
the estimates along $\Sigma_0^1$ based on Lemma~\ref{L:ALGEBRAICIDALONGSIGMA0}
and the assumptions of Subsect.\ \ref{SS:DATAASSUMPTIONS}
on the data for $\Psi$ and $ \bigslow$.
The data in \cite{jSgHjLwW2016}
were compactly supported in $\Sigma_0^1$, which is different than the present context,
but that minor detail does not necessitate any substantial changes in the proof.

To obtain the $L^{\infty}(\mathcal{P}_0^{2 \TranminusdatasizeWithFactor^{-1}})$ bounds for
$\Lunit_{(Small)}^i$ and $\upmu - 1$
stated in
\eqref{E:LUNITIDATASMALLP0LINFTYCONSEQUENCES}-\eqref{E:UPITSELFLINFINITYP0CONSEQUENCES},
we first use the evolution equations
\eqref{E:UPMUFIRSTTRANSPORT}-\eqref{E:LLUNITI}
(recall that
$
\displaystyle
\Lunit 
= \frac{\partial}{\partial t}
$),
Lemma \ref{L:SCHEMATICDEPENDENCEOFMANYTENSORFIELDS},
and the fundamental theorem of calculus to deduce
that for $(t,\vartheta) \in [0,2 \TranminusdatasizeWithFactor^{-1}] \times \mathbb{T}$,
we have the following estimate, where $\smoothfunction$ is smooth in its arguments:
\begin{align} \label{E:MUANDLUNITISMALLALONGP0GRONWALLREADY}
	\left|
		\myarray[\upmu - 1]
			{\sum_{a=1}^2 |\Lunit_{(Small)}^a|}
	\right|
	(t,0,\vartheta)
	& \leq
		\left|
		\myarray[\upmu - 1]
			{\sum_{a=1}^2 |\Lunit_{(Small)}^a|}
	\right|
	(0,0,\vartheta)
	\\
& \ \
	+
	C
	\int_{s=0}^t
		\left\lbrace
		\left|
			\smoothfunction(\upmu - 1, \Lunit_{(Small)}^1, \Lunit_{(Small)}^2,\Psi) 
		\right|
		\left|
			\Fullset^{\leq 1} \Psi
		\right|
		\right\rbrace
		(s,0,\vartheta)	
	\, ds.
	\notag
\end{align}
Moreover,
from \eqref{E:LITTLEGDECOMPOSED}-\eqref{E:METRICPERTURBATIONFUNCTION}
and \eqref{E:INITIALRELATIONS},
we deduce
the schematic relations
$\upmu|_{\Sigma_0^1} = 1 + \Psi \smoothfunction(\Psi)$
and
$\Lunit_{(Small)}^i = \Psi \smoothfunction(\Psi)$, where the functions $\smoothfunction$ are smooth.
Hence,
from the smallness assumption \eqref{E:LINFTYPSISMALLDATAASSUMPTIONSALONGLT0},
we deduce that the first term on RHS~\eqref{E:MUANDLUNITISMALLALONGP0GRONWALLREADY} 
is $\leq C \mathring{\upepsilon}$.
Moreover, from 
\eqref{E:LINFTYPSISMALLDATAASSUMPTIONSALONGLT0},
we deduce that the time integral on RHS~\eqref{E:MUANDLUNITISMALLALONGP0GRONWALLREADY} 
is
$
\leq
C \mathring{\upepsilon}
	\int_{s=0}^t
		\left\lbrace
		\left|
			\smoothfunction(\upmu - 1, \Lunit_{(Small)}^1, \Lunit_{(Small)}^2) 
		\right|
		\right\rbrace
		(s,0,\vartheta)	
	\, ds
$.
Hence, from Gronwall's inequality,
we conclude that
$
\left|
		\myarray[\upmu - 1]
			{\sum_{a=1}^2 |\Lunit_{(Small)}^a|}
	\right|
(t,0,\vartheta)
\leq C \mathring{\upepsilon}
$
for $t \in [0,2 \TranminusdatasizeWithFactor^{-1}]$,
which yields the desired bounds
\eqref{E:LUNITIDATASMALLP0LINFTYCONSEQUENCES}-\eqref{E:UPITSELFLINFINITYP0CONSEQUENCES}.

\end{proof}

\subsection{\texorpdfstring{$\Tboot$}{The bootstrap time}, the positivity of 
\texorpdfstring{$\upmu$}{the inverse foliation density}, and the diffeomorphism property of 
\texorpdfstring{$\Upsilon$}{the change of variables map}}
\label{SS:SIZEOFTBOOT}
To control the solution up to the shock and to derive estimates, 
we find it convenient to rely on a set of bootstrap assumptions.
In this subsection, we state some basic bootstrap assumptions.

We start by fixing a real number $\Tboot$ with
\begin{align} \label{E:TBOOTBOUNDS}
	0 < \Tboot \leq 2 \TranminusdatasizeWithFactor^{-1},
\end{align}
where $\TranminusdatasizeWithFactor > 0$ is defined in \eqref{E:CRITICALBLOWUPTIMEFACTOR}.

We assume that on the spacetime domain $\mathcal{M}_{\Tboot,U_0}$
(see \eqref{E:MTUDEF}), we have
\begin{align} \label{E:BOOTSTRAPMUPOSITIVITY} \tag{$\mathbf{BA} \upmu > 0$}
	\upmu > 0.
\end{align}
Inequality \eqref{E:BOOTSTRAPMUPOSITIVITY} implies that no shocks are present in
$\mathcal{M}_{\Tboot,U_0}$.

We also assume that
\begin{align} \label{E:BOOTSTRAPCHOVISDIFFEO}
	& \mbox{The change of variables map $\Upsilon$ from Def.~\ref{D:CHOVMAP}
	is a $C^1$ diffeomorphism from} \\
	& [0,\Tboot) \times [0,U_0] \times \mathbb{T}
	\mbox{ onto its image}.
	\notag
\end{align}

\subsection{Fundamental \texorpdfstring{$L^{\infty}$}{essential sup-norm} bootstrap assumptions}
\label{SS:PSIBOOTSTRAP}
 Our fundamental bootstrap assumptions for $\Psi$ and $\bigslow$
 are that the following inequalities hold on $\mathcal{M}_{\Tboot,U_0}$
 (see Subsect.\ \ref{SS:STRINGSOFCOMMUTATIONVECTORFIELDS} regarding the vectorfield operator notation):
\begin{align} \label{E:PSIFUNDAMENTALC0BOUNDBOOTSTRAP} \tag{$\mathbf{BA}\Psi-\bigslow$}
	\left\| 
		\Tanset^{[1,11]} \Psi 
	\right\|_{L^{\infty}(\Sigma_t^u)},
		\,
	\left\| 
		\Tanset^{\leq 10} \bigslow
	\right\|_{L^{\infty}(\Sigma_t^u)}
	& \leq \varepsilon,
\end{align}
where $\varepsilon$ is a small positive bootstrap parameter whose smallness 
we describe in Sect.\ \ref{SS:SMALLNESSASSUMPTIONS}.

\subsection{Auxiliary \texorpdfstring{$L^{\infty}$}{essential sup-norm} bootstrap assumptions}
\label{SS:AUXILIARYBOOTSTRAP}
In deriving pointwise estimates, we find it convenient to make the
following auxiliary bootstrap assumptions.
In Prop.~\ref{P:IMPROVEMENTOFAUX}, we will derive strict improvements
of these assumptions.

\medskip

\noindent \underline{\textbf{Auxiliary bootstrap assumptions for small quantities}.}
We assume that the following inequalities hold on $\mathcal{M}_{\Tboot,U_0}$:
\begin{align} 
	\left\| 
		\Psi 
	\right\|_{L^{\infty}(\Sigma_t^u)}
	& \leq \Psiep + \varepsilon^{1/2},
	\label{E:PSIITSELFAUXLINFINITYBOOTSTRAP} \tag{$\mathbf{AUX1}\Psi$}
		\\
	\left\| 
		\Fullset_*^{[1,10];1} \Psi 
	\right\|_{L^{\infty}(\Sigma_t^u)}
	& \leq \varepsilon^{1/2},
	\label{E:PSIAUXLINFINITYBOOTSTRAP} \tag{$\mathbf{AUX2}\Psi$}
		\\
	\left\| 
		\Fullset^{\leq 10;1} \bigslow 
	\right\|_{L^{\infty}(\Sigma_t^u)}
	& \leq \varepsilon^{1/2},
	\label{E:SLOWAUXLINFINITYBOOTSTRAP} \tag{$\mathbf{AUX}\bigslow$}
\end{align}

\begin{align}
	\left\| 
		\Lunit \Tanset^{[1,9]} \upmu
	\right\|_{L^{\infty}(\Sigma_t^u)},
		\,
	\left\| 
		\Tanset_*^{[1,9]} \upmu
	\right\|_{L^{\infty}(\Sigma_t^u)}
	& \leq \varepsilon^{1/2},
		\label{E:UPMUBOOT}  \tag{$\mathbf{AUX1}\upmu$} 
		\\
	\left\| 
		\Lunit_{(Small)}^i 
	\right\|_{L^{\infty}(\Sigma_t^u)}
	& \leq \Psiep^{1/2},
		\label{E:FRAMECOMPONENTS1BOOT} \tag{$\mathbf{AUX1}\Lunit_{(Small)}$} 
		\\
	\left\| 
		\Fullset_*^{[1,9];1} \Lunit_{(Small)}^i 
	\right\|_{L^{\infty}(\Sigma_t^u)}
	& \leq \varepsilon^{1/2},
		\label{E:FRAMECOMPONENTSIBOOT} \tag{$\mathbf{AUX2}\Lunit_{(Small)}$} 
		\\
	\left\| 
		\angLie_{\Fullset}^{\leq 8;1} \upchi
	\right\|_{L^{\infty}(\Sigma_t^u)}
	& \leq \varepsilon^{1/2}.
	\label{E:CHIBOOT} \tag{$\mathbf{AUX1}\upchi$}
\end{align}

\medskip

\noindent \underline{\textbf{Auxiliary bootstrap assumptions for quantities that are allowed to be large}.}
\begin{align}
	\left\| 
		\Rad \Psi 
	\right\|_{L^{\infty}(\Sigma_t^u)}
	& \leq 
	\left\| 
		\Rad \Psi 
	\right\|_{L^{\infty}(\Sigma_0^u)}
	+ 
	\varepsilon^{1/2},
	&& 
		\label{E:PSITRANSVERSALLINFINITYBOUNDBOOTSTRAP} \tag{$\mathbf{AUX2}\Psi$}
\end{align}

\begin{align}
	\left\| 
		\Lunit \upmu
	\right\|_{L^{\infty}(\Sigma_t^u)}
	& \leq 
		\frac{1}{2}
		\left\| 
			G_{\Lunit \Lunit} \Rad \Psi
		\right\|_{L^{\infty}(\Sigma_0^u)}
		+ \varepsilon^{1/2},
		 \label{E:LUNITUPMUBOOT}  \tag{$\mathbf{AUX2}\upmu$}  \\
		\left\| 
			\upmu
		\right\|_{L^{\infty}(\Sigma_t^u)}
		& \leq
		1
		+ 
		\TranminusdatasizeWithFactor^{-1} 
		\left\| 
				G_{\Lunit \Lunit} \Rad \Psi
		\right\|_{L^{\infty}(\Sigma_0^u)}
		+
		\Psiep^{1/2}
		+ 
		\varepsilon^{1/2},
			\label{E:UPMUTRANSVERSALBOOT}  \tag{$\mathbf{AUX3}\upmu$} \\
		\left\| 
			\Rad \Lunit_{(Small)}^i
		\right\|_{L^{\infty}(\Sigma_t^u)}
		& \leq
	 	\left\| 
			\Rad \Lunit_{(Small)}^i
		\right\|_{L^{\infty}(\Sigma_0^u)}
		+ \varepsilon^{1/2}.
		\label{E:LUNITITRANSVERSALBOOT}  \tag{$\mathbf{AUX2}\Lunit_{(Small)}$} 
\end{align}

\subsection{Smallness assumptions}
\label{SS:SMALLNESSASSUMPTIONS}
For the remainder of the article, 
when we say that ``$A$ is small relative to $B$,''
we mean that there exists a continuous increasing function 
$f :(0,\infty) \rightarrow (0,\infty)$ 
such that 
$
\displaystyle
A < f(B)
$.
In principle, the functions $f$ could always be chosen to be 
polynomials with positive coefficients or exponential functions.
However, to avoid lengthening the paper, we typically do not 
specify the form of $f$.

Throughout the rest of the paper, we make the following
smallness assumptions. We
continually adjust the required smallness
in order to close our estimates.
\begin{itemize}
\item The bootstrap parameter $\varepsilon$ 
		and the data smallness parameter $\mathring{\upepsilon}$ from Subsect.\ \ref{SS:DATAASSUMPTIONS}
		are small relative to $1$ (i.e., small in an absolute sense, without regard for the other parameters).
	\item $\varepsilon$ and $\mathring{\upepsilon}$
		are small relative to $\mathring{\updelta}^{-1}$,
		where $\mathring{\updelta}$ is the data-size parameter 
		from Subsect.\ \ref{SS:DATAASSUMPTIONS}.
 \item $\varepsilon$ and $\mathring{\upepsilon}$ are small relative to
		the data-size parameter $\TranminusdatasizeWithFactor$ 
		from Def.\ \ref{D:CRITICALBLOWUPTIMEFACTOR}.
	\item The data-size parameter $\Psiep$ 
		from Subsect.\ \ref{SS:DATAASSUMPTIONS} is small relative to $1$.
	\item We assume that\footnote{In the proof of the main theorem (Theorem~\ref{T:MAINTHEOREM}),
	one sets $\varepsilon = C' \mathring{\upepsilon}$,
	where $C' > 1$ is chosen to be sufficiently large
	and $\mathring{\upepsilon}$ is assumed to be sufficiently small.
	This is compatible with \eqref{E:DATAEPSILONVSBOOTSTRAPEPSILON}.}
	\begin{align} \label{E:DATAEPSILONVSBOOTSTRAPEPSILON}
		\mathring{\upepsilon} 
		& \leq \varepsilon.
	\end{align}
	\end{itemize}
The first two assumptions will allow us, in particular, to treat error terms of size
$\varepsilon \mathring{\updelta}^k$ as small quantities, where
$k \geq 0$ is an integer. The third assumption is relevant
because we only need to control the solution for times
$t < 2 \TranminusdatasizeWithFactor^{-1}$,
which is plenty of time for us to show that a shock forms;
hence, in many estimates, we can therefore consider factors of
$t$ as being bounded by the ``constant'' $C = 2 \TranminusdatasizeWithFactor^{-1}$.
The assumption \eqref{E:DATAEPSILONVSBOOTSTRAPEPSILON} is convenient
for closing our bootstrap argument. The smallness assumption on $\Psiep$ 
allows us to control the size of some key structural coefficients in our estimates
(see, for example, RHS~\eqref{E:TOPORDERTANGENTIALENERGYINTEGRALINEQUALITIES})
and ensures that the solution remains in the regime of hyperbolicity of the equations.

\subsection{Existence of simple plane waves and of admissible initial data}
\label{SS:EXISTENCEOFDATA}
In this subsection, we show that
under the assumptions \eqref{E:SOMENONINEARITIESARELINEAR} on the semilinear inhomogeneous terms,
there exist plane symmetric initial data (i.e., initial data on $\Sigma_0$ that depend only on the Cartesian coordinate $x^1$) 
for the system \eqref{E:FASTWAVE}-\eqref{E:SLOWWAVE}
that are compactly supported in $\Sigma_0^1$
and such that the solution has the following four properties.
\begin{enumerate}
	\item The solution is plane symmetric, i.e., it depends only on the Cartesian coordinates
		$t$ and $x^1$, and, since $\GeoAng = \partial_2$ in plane symmetry, 
		the $\GeoAng$ derivatives of the solution are $0$.
	\item The solution is simple and outgoing, i.e., $\slow = 0$ and $\Lunit \Psi = 0$.
		\item The assumptions of Subsect.\ \ref{SS:DATAASSUMPTIONS} 
		hold with $\mathring{\upepsilon} = 0$,
		consistent with the smallness assumptions of Subsect.\ \ref{SS:SMALLNESSASSUMPTIONS}.
	\item The $L^{\infty}$ assumption \eqref{E:PSIITSELFLINFTYSMALLDATAASSUMPTIONSALONGSIGMA0} holds, 
		where $\Psiep$ can be made arbitrarily small,
		consistent with the smallness assumptions of Subsect.\ \ref{SS:SMALLNESSASSUMPTIONS}.
\end{enumerate}

The key point is that (non-symmetric) perturbations of the above solutions
will obey the smallness assumptions of Subsect.\ \ref{SS:SMALLNESSASSUMPTIONS}
and hence fall under the scope of our main results.

\begin{remark}[\textbf{The vanishing of} $\mathring{\upepsilon}$ \textbf{for simple plane symmetric waves}]
	\label{R:VANISHINGOFUPEPSILON}
	It is straightforward, though somewhat tedious, 
	to show that if assumptions (1) and (2) above hold, 
	then (3) follows automatically. That is, for simple outgoing plane wave solutions
	(which by definition verify $\Lunit \Psi \equiv 0$ and $\bigslow \equiv 0$)
	whose initial data are supported in $\Sigma_0^1$,
	the assumptions of Subsect.\ \ref{SS:DATAASSUMPTIONS} hold with $\mathring{\upepsilon} = 0$.
	The model problem of Subsubsect.\ \ref{SSS:NEARLYSIMPLEWAVES} is a good starting point
	for readers interested in working out the details.
\end{remark}

\begin{remark}[\textbf{Ignoring the $x^2$ direction in this subsection}]
	\label{R:IGNORINGX2}
	In this subsection only, we ignore the $x^2$ direction (since we are considering plane symmetric solutions).
	For example, here we identity $\Sigma_s$ with the constant-time hypersurface
	$\lbrace (t,x^1) \in \mathbb{R} \times \mathbb{R} \ | \ t = s \rbrace$.
\end{remark}

To show that solutions with the properties (1)-(4) exist, we first show that
for plane symmetric 
initial data that are compactly supported in $\Sigma_0^1$
and such that
$\slow|_{\Sigma_0^1} = \partial_t \slow|_{\Sigma_0^1} = \Lunit \Psi|_{\Sigma_0^1} = 0$,
the solution to \eqref{E:FASTWAVE}-\eqref{E:SLOWWAVE}
is also plane symmetric (i.e., it depends only on $t$ and $x^1$)
and moreover, that $\slow = \partial_t \slow = \Lunit \Psi = 0$, as long as the solution remains smooth.
The fact that plane symmetric initial data lead to plane symmetric solutions 
with $\GeoAng = \partial_2$
is a standard consequence 
of the fact that equations \eqref{E:FASTWAVE}-\eqref{E:SLOWWAVE} and the eikonal equation
\eqref{E:APPENDIXGEOMETRICTORUSCOORD} are invariant under the Cartesian coordinate translations $x^2 \rightarrow x^2 + a$ 
(with $a \in [0,1)$ and $x^2 + a$ interpreted mod $\mathbb{T}$), 
and we will not discuss these facts further here.
Next, we note that in plane symmetry,
equations \eqref{E:FASTWAVE}-\eqref{E:SLOWWAVE}
can be written in the following (schematic) form:
\begin{subequations}
\begin{align} 
	(\upmu \Lunit + 2 \Rad) (\Lunit \Psi) 
	& = \smoothfunction \cdot \Lunit \Psi + \smoothfunction \cdot (\slow,\partial \slow),
		\label{E:FASTWAVEINPLANENSYMMETRY}
			\\
	(h^{-1})^{\alpha \beta}(\Psi,\bigslow)  \partial_{\alpha} \partial_{\beta} \slow 
	& = \smoothfunction \cdot \Lunit \Psi + \smoothfunction \cdot (\slow,\partial \slow),
	\label{E:SLOWWAVEINPLANENSYMMETRY}
\end{align}
\end{subequations}
where the $\smoothfunction$ are smooth functions of 
$\Psi$, 
$\slow$,
$\partial \slow$, etc.\ (the precise details of the $\smoothfunction$ are not important here).
Equations \eqref{E:FASTWAVEINPLANENSYMMETRY}-\eqref{E:SLOWWAVEINPLANENSYMMETRY}
follow from equations \eqref{E:FASTWAVE}-\eqref{E:SLOWWAVE}, 
\eqref{E:CHIINTERMSOFOTHERVARIABLES},
and \eqref{E:LONINSIDEGEOMETRICWAVEOPERATORFRAMEDECOMPOSED},
Lemma~\ref{L:SCHEMATICDEPENDENCEOFMANYTENSORFIELDS},
our assumptions \eqref{E:SOMENONINEARITIESARELINEAR} on the semilinear inhomogeneous terms,
and the fact that all $\angdiff$ derivatives of all scalar unknowns
vanish in plane symmetry.

We will now sketch a proof that if the (plane symmetric) initial data for
\eqref{E:FASTWAVEINPLANENSYMMETRY}-\eqref{E:SLOWWAVEINPLANENSYMMETRY}
verify $\slow|_{\Sigma_0^1} = \partial_t \slow|_{\Sigma_0^1} = \Lunit \Psi|_{\Sigma_0^1} = 0$,
then $\slow = \partial_t \slow = \Lunit \Psi = 0$, as long as the solution remains smooth.
To start, we assume that we have a smooth plane symmetric solution to 
\eqref{E:FASTWAVEINPLANENSYMMETRY}-\eqref{E:SLOWWAVEINPLANENSYMMETRY}
with $\upmu > 0$. Multiplying equation \eqref{E:FASTWAVEINPLANENSYMMETRY} by $\Lunit \Psi$,
integrating by parts over $\Sigma_t$, 
and using the assumption $\Lunit \Psi|_{\Sigma_0^1} = 0$,
we obtain the a priori energy estimate
$
	\int_{\Sigma_t} 
		(\Lunit \Psi)^2 
	\, dx^1
	\lesssim
	\int_{s=0}^t
		\int_{\Sigma_s} 
			(\Lunit \Psi)^2 
		\, dx^1
	\, ds
	+
	\int_{s=0}^t
		\int_{\Sigma_s} 	
		|\Lunit \Psi|(|\slow| + |\partial \slow|)
		\, dx^1
	\, ds
$,
where the implicit constants depend on the solution
(in particular they can depend on an upper bound for $1/\upmu$).
Similarly, using standard energy estimates for the wave equation \eqref{E:SLOWWAVEINPLANENSYMMETRY}
(in the spirit of the estimates of Prop.\ \ref{P:SLOWWAVEDIVTHM})
and the assumption $\slow|_{\Sigma_0} = \partial_t \slow|_{\Sigma_0} =  0$,
we obtain the a priori energy estimate
$
	\int_{\Sigma_t} 
		\left\lbrace
			|\partial \slow|^2
			+ 
			\slow^2 
		\right\rbrace
	\, dx^1
	\lesssim
	\int_{s=0}^t
		\int_{\Sigma_s} 
			\left\lbrace
				|\partial \slow|^2
				+ 	
				\slow^2 
			\right\rbrace
		\, dx^1
	\, ds
	+
	\int_{s=0}^t
		\int_{\Sigma_s} 	
		|\Lunit \Psi|(|\slow| + |\partial \slow|)
		\, dx^1
	\, ds
	$.
	Hence, setting
	$
	Q(t) 
	:=
	\int_{\Sigma_t} 
		\left\lbrace
			(\Lunit \Psi)^2 
			+
			|\partial \slow|^2
			+ 
			\slow^2 
		\right\rbrace
	\, dx^1
	$
	and adding the two a priori energy estimates,
	we find that
	$Q(t) \lesssim  \int_{s=0}^t Q(s) \, ds$.
	From Gronwall's inequality, we conclude that $Q(t) = 0$
	(for times $t$ such that the solution is smooth), 
	which is the desired result. We have therefore shown that equations
	\eqref{E:FASTWAVE}-\eqref{E:SLOWWAVE}
	admit simple plane wave solutions with $\slow = 0$
	and $\Lunit \Psi = 0$,
	starting from any smooth initial data 
	that are compactly supported in $\Sigma_0^1$
	and such that
	$\slow|_{\Sigma_0^1} = \partial_t \slow|_{\Sigma_0^1} = \Lunit \Psi|_{\Sigma_0^1} = 0$.

To complete our discussion, it remains for us to
show that there exist plane symmetric initial data for $\Psi$ 
that satisfy our smallness assumptions and such that
$\Lunit \Psi|_{\Sigma_0^1} = 0$. 
In our construction, we will think of $\Psi|_{\Sigma_0}$ and $\partial_t \Psi|_{\Sigma_0}$
as the initial data that we are free to prescribe.
To proceed, we let $\phi = \phi(x^1)$
be any smooth, non-trivial function of the Cartesian coordinate
$x^1$ that is compactly supported in $[0,1]$,
and we set $\Psi|_{\Sigma_0} := \phi$.
We now construct data for $\partial_t \Psi$
such $\Lunit \Psi|_{\Sigma_0} = 0$.
Using \eqref{E:PERTURBEDPART} and
the second identities in \eqref{E:LUNITOFUANDT}
and \eqref{E:INITIALRELATIONS},
we see that
$\Lunit \Psi|_{\Sigma_0} = 0$ 
is equivalent to
$\partial_t \Psi|_{\Sigma_0}
= - \Lunit^1 \partial_1 \Psi|_{\Sigma_0}
= - \Lunit^1 \partial_1 \phi
= - \left\lbrace
			\sqrt{(\gtinverse)^{11}\circ \phi} 
	 		- 
			(g^{-1})^{01}\circ \phi	
		\right\rbrace
		\partial_1 \phi
$.
Put differently, we can prescribe
$\partial_t \Psi|_{\Sigma_0}$
in terms of $\phi$ to achieve the desired result $\Lunit \Psi|_{\Sigma_0} = 0$.

We have therefore constructed plane symmetric initial data such that $\Lunit \Psi|_{\Sigma_0} = 0$.
To complete the discussion in this subsection, it remains only for us to point out
that by rescaling $\phi \rightarrow \uplambda \phi$ 
(where $\phi$ is as in the previous paragraph)
and choosing the real number $\uplambda$ to be 
sufficiently small but non-zero, we can ensure that
$\| \Psi \|_{L^{\infty}(\Sigma_0^1)}$ is as small as we want,
consistent with our assumption
\eqref{E:PSIITSELFLINFTYSMALLDATAASSUMPTIONSALONGSIGMA0}
and with the smallness assumption for
$\Psiep$ that we imposed in Subsect.\ \ref{SS:SMALLNESSASSUMPTIONS}.

\section{Energies, null fluxes, and energy-null flux identities}
\label{S:ENERGIES}
In this section, we define the energies and null fluxes that we
use in our $L^2$ analysis. We also provide the basic energy-null flux
identities for solutions to the fast wave equation \eqref{E:FASTWAVE}
and to the slow wave equation, 
in the first-order form \eqref{E:SLOW0EVOLUTION}-\eqref{E:SYMMETRYOFMIXEDPARTIALS}.

\subsection{Definitions of the energies and null fluxes}
\label{SS:ENERGYDEFINITIONS}

\subsubsection{Energies and null fluxes for the fast wave}

\begin{definition}[\textbf{Energy and null flux for the fast wave}]
\label{D:ENERGYFLUX}
In terms of the geometric forms of Def.~\ref{D:NONDEGENERATEVOLUMEFORMS},
we define the fast wave energy functional $\enzero[\cdot]$ and the fast wave null flux functional
$\flzero[\cdot]$ as follows:
\begin{subequations}
\begin{align} \label{E:ENERGYORDERZEROCOERCIVENESS}
	\enzero[f](t,u)
		& = 
		\int_{\Sigma_t^u} 
			\left\lbrace
				\frac{1}{2} (1 + 2 \upmu) \upmu (\Lunit f)^2
				+ 2 \upmu (\Lunit f) \Rad f
				+ 2 (\Rad \Psi)^2
				+ \frac{1}{2} (1 + 2 \upmu)\upmu |\angdiff f|^2
			\right\rbrace
		\, d \tvol,
		\\
	\flzero[f](t,u)
		& = 
		\int_{\mathcal{P}_u^t} 
			\left\lbrace
				(1 + \upmu)(\Lunit f)^2
				+ \upmu |\angdiff f|^2
			\right\rbrace
		\, d \conevol.
		\label{E:NULLFLUXENERGYORDERZEROCOERCIVENESS}
\end{align}
\end{subequations}

\end{definition}

\subsubsection{Energies and null fluxes for the slow wave}
Let
\begin{align} \label{E:ALTSLOWARRAY}
	\altbigslow
	& := (\altslow,\altslow_0,\altslow_1,\altslow_2)
\end{align}
be an array comprising four scalar functions.
In later applications, the entries of $\altbigslow$ will be
derivatives of the entries of the slow wave array $\bigslow$ defined in \eqref{E:SLOWWAVEVARIABLES}.
To construct energies and null fluxes for the slow wave,
we will rely on the compatible current vectorfield
$J = J[\altbigslow]$, which we define relative to the Cartesian coordinates as follows:
\begin{align} \label{E:SLOWCURRENT}
	J^{\alpha}[\altbigslow]
	& := 
		2 \enmomtensor^{\alpha \beta}[\altbigslow] \delta_{\beta}^0
		- 
		\altslow^2 (h^{-1})^{\alpha \beta} \delta_{\beta}^0.
\end{align}
In \eqref{E:SLOWCURRENT},
\begin{align} \label{E:ENMOMENTUMSLOWWAVE}
\enmomtensor^{\alpha \beta}[\altbigslow]
:= (h^{-1})^{\alpha \kappa} (h^{-1})^{\beta \lambda} \altslow_{\kappa} \altslow_{\lambda} 
- 
\frac{1}{2} (h^{-1})^{\alpha \beta} (h^{-1})^{\kappa \lambda} \altslow_{\kappa} \altslow_{\lambda}
\end{align}
is the type $\binom{2}{0}$ energy-momentum tensorfield of the slow wave equation \eqref{E:SLOWWAVE}.

\begin{remark}[\textbf{Suppression of some arguments of} $\enmomtensor$]
	\label{R:ENMOMEMSUPPRESSIONOFARGUMENTS}
	Although $\enmomtensor^{\alpha \beta}[\altbigslow]$ depends on
	$(\Psi,\bigslow)$ through the Cartesian component functions 
	$(h^{-1})^{\alpha \beta} = (h^{-1})^{\alpha \beta}(\Psi,\bigslow)$,
	we typically suppress this dependence.
\end{remark}
Note that \eqref{E:VECTORSGCAUSALIMPLIESHTIMELIKE}
and our assumption 
$(g^{-1})^{\alpha \beta}(\Psi = 0) 
= (m^{-1})^{\alpha \beta} 
= \mbox{\upshape diag}(-1,1,1)$
together imply that
when $|\GdVar| + |\bigslow|$ is sufficiently small,
the one-form with Cartesian components $\delta_{\alpha}^0$ is 
past-directed 
(by \eqref{E:ZEROZEROISMINUSONE})
and $h$-timelike.
Thus, the product $2 \enmomtensor^{\alpha \beta}[\altbigslow] \delta_{\beta}^0$ 
on RHS~\eqref{E:SLOWCURRENT} is the contraction of the energy-momentum
tensorfield of the slow wave with a past-directed $h$-timelike (see Footnote~\ref{FN:HTIMELIKE})
one-form. We also note the following explicit formulas, 
the first of which relies on the second relation in \eqref{E:ZEROZEROISMINUSONE},
where $(i=1,2)$:
\begin{subequations}
\begin{align} \label{E:SLOWCURRENTTIMECARTESIANCOMPONENT}
	J^0[\altbigslow]
	& =
		2 (h^{-1})^{\alpha 0} (h^{-1})^{\beta 0} \altslow_{\alpha} \altslow_{\beta}
		+
		(h^{-1})^{\alpha \beta} \altslow_{\alpha} \altslow_{\beta}
		+
		\altslow^2,
			\\
	J^i[\altbigslow] 
	& = 
		2 (h^{-1})^{\alpha i} (h^{-1})^{\beta 0} \altslow_{\alpha} \altslow_{\beta}
		-
		(h^{-1})^{i0} (h^{-1})^{\alpha \beta} \altslow_{\alpha} \altslow_{\beta}
		-
		(h^{-1})^{i0}
		\altslow^2.
		\label{E:SLOWCURRENTSPATIALCARTESIANCOMPONENTS}
\end{align}
\end{subequations}

We now again consider the past-directed one-form with Cartesian components
$\delta_{\alpha}^0$, which we have already shown to be $h$-timelike when 
$|\GdVar| + |\bigslow|$ is sufficiently small.
Since, on RHS~\eqref{E:SLOWCURRENT},
this one-form is contracted against the energy-momentum tensorfield $\enmomtensor[\altbigslow]$,
the well-known dominant energy condition for $\enmomtensor[\altbigslow]$ implies that
if $\omega$ is any non-trivial past-directed, $h$-timelike one-form
and if $\altbigslow \neq 0$,
then $J^{\alpha} \omega_{\alpha} > 0$. 
It follows that $h(J,J) \leq 0$ and thus, by definition, $J$ is $h$-causal.
Moreover, taking $\omega_{\alpha} := \delta_{\alpha}^0$,
we find that $J^0 > 0$ whenever $\altbigslow \neq 0$.
Thus, by \eqref{E:VECTORSHCAUSALIMPLIESGTIMELIKE} (with $J[\altbigslow]$ in the role of $V$ in \eqref{E:VECTORSHCAUSALIMPLIESGTIMELIKE}), 
the following holds when $|\GdVar| + |\bigslow|$ is sufficiently small:
\begin{align} \label{E:JISGTIMELIKE}
	\altbigslow \neq 0 \implies \mbox{\upshape the vectorfield $J[\altbigslow]$ is future-directed and $g$-timelike.}
\end{align}
We will use these facts in the proof of Lemma~\ref{L:COERCIVENESSOFSLOWWAVEENERGIESANDFLUXES}.

We find it convenient to \emph{define} the slow wave energy and the slow wave null flux
relative to the Cartesian forms. However, when describing
their coerciveness properties and deriving energy estimates, 
we use the geometric forms; see Lemma~\ref{L:COERCIVENESSOFSLOWWAVEENERGIESANDFLUXES}.

\begin{definition}[\textbf{Energy and null flux for the slow wave}]
	\label{D:SLOWWAVEENERGYFLUX}
	Let $J^{\alpha}[\altbigslow]$ be the compatible current vectorfield with Cartesian components given
	by \eqref{E:SLOWCURRENTTIMECARTESIANCOMPONENT}-\eqref{E:SLOWCURRENTSPATIALCARTESIANCOMPONENTS}.
	In terms of the Cartesian forms of Def.~\ref{D:CARTESIANFORMS}
	and the one-form 
	$\covL_{\alpha}$ defined in \eqref{E:EUCLIDEANNORMALTONULLHYPERSURFACE},
	we define the slow wave energy functional $\slowen[\cdot]$ and the slow wave null flux functional
	$\slowfl[\cdot]$ as follows:
	\begin{align} \label{E:SLOWWAVEENERGYFLUX}
		\slowen[\altbigslow](t,u)
		& 
		:= \int_{\Sigma_t^u}
				J^0[\altbigslow]
			\, d \Sigma,
		&&
		\slowfl[\altbigslow](t,u)
		:= \int_{\mathcal{P}_u^t}
				J^{\alpha}[\altbigslow] \covL_{\alpha}
			\, d \mathcal{P}.
	\end{align}
\end{definition}

\subsection{Coerciveness of the energy and null flux for the slow wave}
\label{SS:ENERGYCOERCIVENESSSLOWWAVE}
The coerciveness properties of the energy and null flux for the fast wave
are fairly apparent from Def.~\ref{D:ENERGYFLUX}.
In contrast, it takes some effort to reveal the coerciveness of the energy
and null flux for the slow wave. The next lemma yields the desired coerciveness.

\begin{lemma}[\textbf{Coerciveness of the energy and null flux for the slow wave}]
	\label{L:COERCIVENESSOFSLOWWAVEENERGIESANDFLUXES}
	In terms of the geometric forms of Def.~\ref{D:NONDEGENERATEVOLUMEFORMS},
	the slow wave energy $\slowen[\cdot]$
	and the slow wave null flux 
	$\slowfl[\cdot]$
	from Def.~\ref{D:SLOWWAVEENERGYFLUX} enjoy the following 
	coerciveness properties, valid when $|\GdVar| + |\bigslow|$ is sufficiently small:
		\begin{subequations}
		\begin{align}
		\slowen[\altbigslow](t,u)
		& 
		\approx 
			\int_{\Sigma_t^u}
				\upmu
				\left\lbrace
					\altslow^2 + \sum_{\alpha = 0}^2 \altslow_{\alpha}^2
				\right\rbrace
			\, d \tvol,
			\label{E:SLOWSIGMATENERGYCOERCIVENESS}	
			\\
	\slowfl[\altbigslow](t,u)
	 & 
		\approx 
			\int_{\mathcal{P}_u^t}
				\left\lbrace
					\altslow^2 + \sum_{\alpha = 0}^2 \altslow_{\alpha}^2
				\right\rbrace
		\, d \conevol.
		\label{E:SLOWNULLFLUXCOERCIVENESS}
	\end{align}
	\end{subequations}
\end{lemma}

\begin{remark}[\textbf{On the necessity of the null fluxes and the necessity of the slow speed of the slow wave}]
	\label{R:NEEDNULLFLUX}
	It is critically important for our analysis that
	the null flux $\slowfl[\altbigslow]$ 
	controls $\altbigslow$ \emph{without any degenerate factor of} $\upmu$,
	as RHS~\eqref{E:SLOWNULLFLUXCOERCIVENESS} shows.
	Similar remarks apply to the fast wave null flux 
	$\flzero[f](t,u)$
	defined in \eqref{E:NULLFLUXENERGYORDERZEROCOERCIVENESS},
	which controls $(\Lunit f)^2$ with a weight that does not degenerate as $\upmu \to 0$.
	We also recall that, as we described in Subsubsect.\ \ref{SSS:ENERGYESTIMATES},
	we use the spacetime integrals from Def.\ \ref{D:COERCIVEINTEGRAL}
	to obtain non-degenerate control of the $\angdiff$ derivatives of the fast wave.
	We also note that to derive the estimate \eqref{E:SLOWNULLFLUXCOERCIVENESS},
	we fundamentally rely on the assumption that the wave speed of $\bigslow$ is slower than
	the wave speed of $\Psi$, and that this is the main spot in the article where    
	we use this assumption.
\end{remark}

\begin{proof}[Proof of Lemma~\ref{L:COERCIVENESSOFSLOWWAVEENERGIESANDFLUXES}]
Just above equation \eqref{E:JISGTIMELIKE}, 
we showed that $J^0 = J^0[\altbigslow]> 0$ as long as
$\altbigslow \neq 0$ and $|\GdVar| + |\bigslow|$ is sufficiently small.
Since $J^0[\altbigslow]$ is precisely quadratic in its arguments $\altbigslow$,
the desired estimate \eqref{E:SLOWSIGMATENERGYCOERCIVENESS} follows from this fact,
the second identity in \eqref{E:VOLFORMRELATION},
and the first definition in \eqref{E:SLOWWAVEENERGYFLUX}.

To prove \eqref{E:SLOWNULLFLUXCOERCIVENESS},
we first note that the one-form
$\covL$ defined in \eqref{E:EUCLIDEANNORMALTONULLHYPERSURFACE}
is past-directed (in view of the last identity in \eqref{E:LUNITOFUANDT})
and $g$-null. Thus, by \eqref{E:JISGTIMELIKE},
$J^{\alpha} \covL_{\alpha} > 0$ whenever $\altbigslow \neq 0$
and $|\GdVar| + |\bigslow|$ is sufficiently small.
Since $J^{\alpha} \covL_{\alpha}$ is precisely quadratic in $\altbigslow$,
the desired estimate \eqref{E:SLOWNULLFLUXCOERCIVENESS}
follows from the last identity in \eqref{E:VOLFORMRELATION}
and the second definition in \eqref{E:SLOWWAVEENERGYFLUX}.

\end{proof}

\subsection{Energy-null flux identities}
\label{SS:ENERGYIDENTITIES}
In this subsection, we derive energy-null flux identities 
that we later will use to control solutions to the system
\eqref{E:FASTWAVE}
+
\eqref{E:SLOW0EVOLUTION}-\eqref{E:SYMMETRYOFMIXEDPARTIALS}.

\subsubsection{Energy-null flux identities for the fast wave}
\label{SSS:FASTWAVEENERGYID}
\begin{proposition}\cite{jSgHjLwW2016}*{Proposition~3.5; \textbf{Fundamental energy-null flux identity for the fast wave}}
	\label{P:DIVTHMWITHCANCELLATIONS}
		For solutions $f$ to the inhomogeneous wave equation
		\begin{align*}
			\upmu \square_{g(\Psi)} f & = \waveinhom,
		\end{align*}
	we have the following identity for $t \geq 0$ and $u \in [0,U_0]$:
	\begin{align} \label{E:E0DIVID}
				&
				\enzero[f](t,u)
				+
				\flzero[f](t,u)
					\\
				& 
				= 
				\enzero[f](0,u)
				+
				\flzero[f](t,0)
				- 
				\int_{\mathcal{M}_{t,u}}
					\left\lbrace
						(1 + 2 \upmu) (\Lunit f)
						+ 
						2 \Rad f 
					\right\rbrace
				\waveinhom 
				\, d \vol
					\notag \\
			& \ \
				- 
				\frac{1}{2} 
				\int_{\mathcal{M}_{t,u}}
				[\Lunit \upmu]_- |\angdiff f|^2 
				\, d \vol
				+ 
				\sum_{i=1}^5
				\int_{\mathcal{M}_{t,u}}
				 	\basicenergyerrorarg{\Mult}{i}[f]
				\, d \vol,
				\notag
			\end{align}
where $f_+: = \max \lbrace f,0 \rbrace$,
$f_- := \max \lbrace -f, 0 \rbrace$,
and
\begin{subequations}
		\begin{align}
			\basicenergyerrorarg{\Mult}{1}[f] 
				& := (\Lunit f)^2 
						\left\lbrace
							- \frac{1}{2} \Lunit \upmu
							+ \Rad \upmu
							- \frac{1}{2} \upmu \mytr \upchi
							- \mytr \angk^{(Trans-\Psi)}
							- \upmu \mytr \angk^{(Tan-\Psi)}
						\right\rbrace,
						\label{E:MULTERRORINTEG1} \\
			\basicenergyerrorarg{\Mult}{2}[f] 
			& := - (\Lunit f) (\Rad f)
					\left\lbrace
						 \mytr \upchi
						+ 2 \mytr \angk^{(Trans-\Psi)}
						+ 2 \upmu \mytr \angk^{(Tan-\Psi)}
					\right\rbrace,
						\label{E:MULTERRORINTEG2} \\
		  \basicenergyerrorarg{\Mult}{3}[f] 
			& := 
				\upmu |\angdiff f|^2
				\left\lbrace
					\frac{1}{2} \frac{[\Lunit \upmu]_+}{\upmu}
					+ \frac{\Rad \upmu}{\upmu}
					+ 2 \Lunit \upmu
					- \frac{1}{2} \mytr \upchi
					- \mytr \angk^{(Trans-\Psi)}
					- \upmu \mytr \angk^{(Tan-\Psi)}
				\right\rbrace,
				\label{E:MULTERRORINTEG3} \\
			\basicenergyerrorarg{\Mult}{4}[f] 
			& := 	(\Lunit f)(\angdiffuparg{\#} f) 
					\cdot
					\left\lbrace
						(1 - 2 \upmu) \angdiff \upmu 
						+ 
						2 \upzeta^{(Trans-\Psi)}
						+
						2 \upmu \upzeta^{(Tan-\Psi)}
					\right\rbrace,
					\label{E:MULTERRORINTEG4} \\
			\basicenergyerrorarg{\Mult}{5}[f] 
			& := - 2 (\Rad f)(\angdiffuparg{\#} f)
					\cdot
					 \left\lbrace
							\angdiff \upmu 
						 	+ 
						 	2 \upzeta^{(Trans-\Psi)}
							+
							2 \upmu \upzeta^{(Tan-\Psi)}
					 \right\rbrace.
					\label{E:MULTERRORINTEG5} 
		\end{align}
\end{subequations}
	The tensorfields
	$\upchi$, 
	$\upzeta^{(Trans-\Psi)}$,
	$\angk^{(Trans-\Psi)}$,
	$\upzeta^{(Tan-\Psi)}$,
	and
	$\angk^{(Tan-\Psi)}$
	from above 
	are as in
	\eqref{E:CHIINTERMSOFOTHERVARIABLES},
	\eqref{E:ZETATRANSVERSAL},
	\eqref{E:KABTRANSVERSAL},
	\eqref{E:ZETAGOOD},
	and \eqref{E:KABGOOD},
	while the symbol ``$\Mult$''
	in \eqref{E:MULTERRORINTEG1}-\eqref{E:MULTERRORINTEG5}
	merely signifies that the energies and error terms are tied
	to the multiplier vectorfield from \eqref{E:MULT},
	as is shown by the proof of \cite{jSgHjLwW2016}*{Proposition~3.5}.
\end{proposition}

\subsubsection{Energy-null flux identities for the slow wave}
\label{SSS:SLOWWAVEENERGYID}
We now derive an analog of Prop.~\ref{P:DIVTHMWITHCANCELLATIONS} for the slow wave variable.

\begin{proposition}[\textbf{Fundamental energy-null flux identity for the slow wave}]
\label{P:SLOWWAVEDIVTHM}
Solutions $\altbigslow := (\altslow,\altslow_0,\altslow_1,\altslow_2)$ to the inhomogeneous system
\begin{subequations}
\begin{align}
	\upmu \partial_t \altslow_0
	& = \upmu (h^{-1})^{ab} \partial_a \altslow_b
		+ 2 \upmu (h^{-1})^{0a} \partial_a \altslow_0
		+ F_0,
		 \label{E:SLOWTIMEINHOM} \\
	\upmu \partial_t \altslow_i
	& = \upmu \partial_i \altslow_0
		+ F_i,
		\label{E:SLOWSPACEINHOM} \\
	\upmu \partial_t \altslow
	& = \upmu \altslow_0
	+ F,
	\label{E:SLOWINHOM}
		\\
	\upmu \partial_i \altslow_j
	& = \upmu \partial_j \altslow_i
	+ F_{ij}
	\label{E:MIXEDPARTIALSINHOM}
\end{align}
\end{subequations}
verify the following energy identity for $t \geq 0$ and $u \in [0,U_0]$:
\begin{align} \label{E:SLOWENERGYID}
	&
	\slowen[\altbigslow](t,u)
	+
	\slowfl[\altbigslow](t,u)
		\\
	& = 
		\slowen[\altbigslow](0,u)
		+
		\slowfl[\altbigslow](t,0)
			\notag \\
	& \ \
		+ \int_{\mathcal{M}_{t,u}}
				\left\lbrace 
					1 + \upgamma \smoothfunction(\upgamma)
				\right\rbrace 
				\slowbasicenergyerror[\altbigslow]
			 \, d \vol
			 \notag \\
	& \ \ 
		+ \int_{\mathcal{M}_{t,u}}
				\left\lbrace 
					1 + \upgamma \smoothfunction(\upgamma)
				\right\rbrace 
				\left\lbrace
					4 (h^{-1})^{\alpha 0} (h^{-1})^{\beta 0} \altslow_{\alpha} F_{\beta}
					+
					2 (h^{-1})^{\alpha \beta} \altslow_{\alpha} F_{\beta}
					+
					2 (h^{-1})^{\alpha a} (h^{-1})^{b 0}  \altslow_{\alpha} F_{ab}
					+ 
					2 \altslow F
				\right\rbrace
			\, d \vol,
			\notag
\end{align}
where the schematically denoted functions $\upgamma \smoothfunction(\upgamma)$ are smooth and vanish
when $\GdVar = 0$, and
\begin{align} \label{E:SLOWWAVEBASICENERGYINTEGRAND}
	\slowbasicenergyerror[\altbigslow]
	& := 4 \upmu (\partial_t (h^{-1})^{\alpha 0}) (h^{-1})^{\beta 0} \altslow_{\alpha} \altslow_{\beta}
			+
			\upmu (\partial_t(h^{-1})^{\alpha \beta}) \altslow_{\alpha} \altslow_{\beta}
				\\
		& \ \
			+
			2 \upmu (\partial_a (h^{-1})^{\alpha a}) (h^{-1})^{\beta 0} \altslow_{\alpha} \altslow_{\beta}
			+
			2 \upmu (h^{-1})^{\alpha a} (\partial_a (h^{-1})^{\beta 0}) \altslow_{\alpha} \altslow_{\beta}
				\notag \\
		& \ \
			-
			\upmu (\partial_a (h^{-1})^{a0}) (h^{-1})^{\alpha \beta} \altslow_{\alpha} \altslow_{\beta}
			-
			\upmu (h^{-1})^{a0} (\partial_a (h^{-1})^{\alpha \beta}) \altslow_{\alpha} \altslow_{\beta}
			-
			\upmu (\partial_a (h^{-1})^{a0}) \altslow^2
			\notag \\
		& \ \
			- 
			2 \upmu (h^{-1})^{\alpha 0} \altslow \altslow_{\alpha}.
			\notag	
\end{align}

\end{proposition}

\begin{remark}
	Note that the above expressions depend on
	$(\Psi,\bigslow)$ through $(h^{-1})^{\alpha \beta} = (h^{-1})^{\alpha \beta}(\Psi,\bigslow)$.
\end{remark}

\begin{proof}[Proof of Prop.\ \ref{P:SLOWWAVEDIVTHM}]
	We consider the vectorfield $J = J[\altbigslow]$ defined relative to the Cartesian coordinates by \eqref{E:SLOWCURRENT}.
	Using the second relation in \eqref{E:ZEROZEROISMINUSONE},
	the identities \eqref{E:SLOWCURRENTTIMECARTESIANCOMPONENT}-\eqref{E:SLOWCURRENTSPATIALCARTESIANCOMPONENTS},
	and equations \eqref{E:SLOWTIMEINHOM}-\eqref{E:MIXEDPARTIALSINHOM},
	we compute that
	\begin{align} \label{E:UPMUDIVOFSLOWJ}
		\partial_{\alpha} J^{\alpha}
		& = 
			4 (\partial_t (h^{-1})^{\alpha 0}) (h^{-1})^{\beta 0} \altslow_{\alpha} \altslow_{\beta}
			+
			(\partial_t(h^{-1})^{\alpha \beta}) \altslow_{\alpha} \altslow_{\beta}
				\\
		& \ \
			+
			2 (\partial_a (h^{-1})^{\alpha a}) (h^{-1})^{\beta 0} \altslow_{\alpha} \altslow_{\beta}
			+
			2 (h^{-1})^{\alpha a} (\partial_a (h^{-1})^{\beta 0}) \altslow_{\alpha} \altslow_{\beta}
				\notag \\
		& \ \
			-
			(\partial_a (h^{-1})^{a0}) (h^{-1})^{\alpha \beta} \altslow_{\alpha} \altslow_{\beta}
			-
			(h^{-1})^{a0} (\partial_a (h^{-1})^{\alpha \beta}) \altslow_{\alpha} \altslow_{\beta}
			-
			(\partial_a (h^{-1})^{a0}) \altslow^2
			\notag \\
		& \ \
			- 
			2 (h^{-1})^{\alpha 0} \altslow \altslow_{\alpha}
			\notag	\\
		& \ \
			+
			\frac{4}{\upmu} (h^{-1})^{\alpha 0} (h^{-1})^{\beta 0} F_{\alpha} \altslow_{\beta}
			+
			\frac{2}{\upmu} (h^{-1})^{\alpha \beta} F_{\alpha} \altslow_{\beta}
			+
			\frac{2}{\upmu} (h^{-1})^{\alpha a} (h^{-1})^{b 0}  \altslow_{\alpha} F_{ab}
			+ 
			\frac{2}{\upmu} \altslow F.
			\notag
	\end{align}	
	We now apply the divergence theorem
	to the vectorfield $J$ on the region $\mathcal{M}_{t,u}$,
	where we use the Cartesian coordinates,
	the Euclidean metric $\delta^{\alpha \beta} := \mbox{\upshape diag}(1,1,1)$
	on $\mathbb{R} \times \Sigma$, 
	and the Cartesian forms of Def.~\ref{D:CARTESIANFORMS}
	in all computations.
	As the final step, we use the identities in \eqref{E:VOLFORMRELATION}
	to express all integrals 
	as integrals with respect to the geometric integration measures 
	corresponding to Def.~\ref{D:NONDEGENERATEVOLUMEFORMS}.
	Also taking into account Def.~\ref{D:SLOWWAVEENERGYFLUX},
	we arrive at the desired identity \eqref{E:SLOWENERGYID}.
	Note that the one-form $\covL_{\alpha}$ on RHS~\eqref{E:SLOWWAVEENERGYFLUX}
	is the Euclidean unit-length co-normal to the hypersurfaces $\mathcal{P}_u$,
	which is the reason that $J^{\alpha}[\altbigslow] \covL_{\alpha}$ arises when
	we apply the standard divergence theorem relative to the Cartesian coordinates
	on $\mathcal{M}_{t,u}$.
\end{proof}

\section{Preliminary pointwise estimates}
\label{S:PRELIMINARYPOINTWISE}
In this section, we use the data-size assumptions, 
smallness assumptions, 
and bootstrap assumptions of
Sect.\ \ref{S:NORMSANDBOOTSTRAP} to derive preliminary $L^{\infty}$ and pointwise
estimates for the solution.
These estimates serve as the starting point for related estimates
that we derive in 
Sects.\ \ref{S:LINFINITYESTIMATESFORHIGHERTRANSVERSAL}-\ref{S:POINTWISEESTIMATES}

\begin{remark}[\textbf{Many estimates were proved in} \cite{jSgHjLwW2016}]
\label{R:ESTIMATESPROVEDINOTHERPAPER}
Many of the estimates involving the fast wave variable $\Psi$ and the eikonal function
are independent of the slow wave variable $\bigslow$ and were proved in
\cite{jSgHjLwW2016}; thus, we cite \cite{jSgHjLwW2016} for many of the estimates.
\end{remark}

\subsection{Notation for repeated differentiation}
\label{SS:NOTATIONFORREPEATED}

\begin{definition}[\textbf{Notation for repeated differentiation}]
\label{D:REPEATEDDIFFERENTIATIONSHORTHAND}
Recall that the commutation vectorfield sets $\Fullset$ and $\Tanset$ are defined in Def.~\ref{D:COMMUTATIONVECTORFIELDS}.
We label the three vectorfields in 
$\Fullset$ as follows: $Z_{(1)} = \Lunit, Z_{(2)} = \GeoAng, Z_{(3)} = \Rad$.
Note that $\Tanset = \lbrace Z_{(1)}, Z_{(2)} \rbrace$.
We define the following vectorfield operators:
\begin{itemize}
	\item If $\vec{I} = (\iota_1, \iota_2, \cdots, \iota_N)$ is a multi-index
		of order $|\vec{I}| := N$
		with $\iota_1, \iota_2, \cdots, \iota_N \in \lbrace 1,2,3 \rbrace$,
		then $\Fullset^{\vec{I}} := Z_{(\iota_1)} Z_{(\iota_2)} \cdots Z_{(\iota_N)}$ 
		denotes the corresponding $N^{th}$ order differential operator.
		We write $\Fullset^N$ rather than $\Fullset^{\vec{I}}$
		when we are not concerned with the structure of $\vec{I}$,
		and we sometimes omit the superscript when $N=1$.
	\item Similarly, $\angLie_{\Fullset}^{\vec{I}} 
	:= \angLie_{Z_{(\iota_1)}} \angLie_{Z_{(\iota_2})} \cdots \angLie_{Z_{(\iota_N})}$
		denotes an $N^{th}$ order $\ell_{t,u}$-projected Lie derivative operator
		(see Def.~\ref{D:PROJECTEDLIE}),
		and we write $\angLie_{\Fullset}^N$
		when we are not concerned with the structure of $\vec{I}$.
	\item If $\vec{I} = (\iota_1, \iota_2, \cdots, \iota_N)$,
		then 
		$\vec{I}_1 + \vec{I}_2 = \vec{I}$ 
		means that
		$\vec{I}_1 = (\iota_{k_1}, \iota_{k_2}, \cdots, \iota_{k_m})$
		and
		$\vec{I}_2 = (\iota_{k_{m+1}}, \iota_{k_{m+2}}, \cdots, \iota_{k_N})$,
		where $1 \leq m \leq N$ and
		$k_1, k_2, \cdots, k_N$ is a permutation of 
		$1,2,\cdots,N$. 
	\item Sums such as $\vec{I}_1 + \vec{I}_2 + \cdots + \vec{I}_M = \vec{I}$
		have an analogous meaning.
	\item $\mathcal{P}_u$-tangential vectorfield operators such as 
		$\Tanset^{\vec{I}}$ are defined analogously,  
		except in this case we have
		$\iota_1, \iota_2, \cdots, \iota_N \in \lbrace 1,2 \rbrace$.
		We write $\Tanset^N$ rather than $\Tanset^{\vec{I}}$
		when we are not concerned with the structure of $\vec{I}$,
		and we sometimes omit the superscript when $N=1$.
\end{itemize}
\end{definition}

\subsection{Basic assumptions, facts, and estimates that we use silently}
\label{SS:OFTENUSEDESTIMATES}
For the reader's convenience, we present here some basic assumptions, facts, and estimates
(similar to those from \cite{jSgHjLwW2016}*{Section 8.2})
that we silently use throughout the rest of the paper when deriving
estimates.

\begin{enumerate}
	\item All of the estimates that we derive
		hold on the bootstrap region $\mathcal{M}_{\Tboot,U_0}$.
		Moreover, in deriving estimates,
		we rely on the data-size and bootstrap assumptions 
		of Subsects.\ \ref{SS:DATAASSUMPTIONS}-\ref{SS:AUXILIARYBOOTSTRAP},
		the smallness assumptions of Subsect.\ \ref{SS:SMALLNESSASSUMPTIONS},
		and in particular the estimates for the data of the eikonal function quantities provided by
		Lemma~\ref{L:BEHAVIOROFEIKONALFUNCTIONQUANTITIESALONGSIGMA0}.
		Moreover, when we refer to ``the bootstrap assumptions,''
		we mean the fundamental bootstrap assumptions of
		Subsect.\ \ref{SS:PSIBOOTSTRAP}
		and the auxiliary bootstrap assumptions of Subsect.\ \ref{SS:AUXILIARYBOOTSTRAP}.
	\item We often use the assumption \eqref{E:DATAEPSILONVSBOOTSTRAPEPSILON} without explicitly mentioning it.
	\item All quantities that we estimate can be controlled in terms of the quantities
		$\BadVar = \lbrace \Psi, \upmu - 1, \Lunit_{(Small)}^1, \Lunit_{(Small)}^2 \rbrace$,
		$\bigslow$,
		and their derivatives. Note that many of these quantities are small, but
		for the solutions under consideration,
		the $\Rad$ derivatives of $\BadVar$ do not have to be small,
		nor do $\upmu - 1$ or $\Lunit \upmu$.
	\item We typically use the Leibniz rule for 
		the operators $\angLie_Z$ and $\angD$ when deriving
		pointwise estimates for the $\angLie_Z$ and $\angD$ 
		derivatives of tensor products of the schematic form 
		$\prod_{i=1}^m v_i$, where the $v_i$ are scalar functions or
		$\ell_{t,u}$-tangent tensors. Our derivative counts are such that
		all $v_i$ except at most one are uniformly bounded in $L^{\infty}$
		on $\mathcal{M}_{\Tboot,U_0}$.
		Thus, our pointwise estimates often explicitly feature 
		(on the right-hand sides)
		only the factor with the most derivatives on it,
		multiplied by a constant (often implicit)
		that uniformly bounds the other factors.
		In some estimates, the right-hand sides also gain a smallness factor,
		such as $\Psiep$ or $\varepsilon^{1/2}$,
		generated by the remaining $v_i's$.
	\item The operators $\angLie_{\Fullset}^N$ commute through
		$\angdiff$, as is shown by Lemma~\ref{L:LANDRADCOMMUTEWITHANGDIFF}.
	\item For scalar functions $f$, we have
		$	\left| \GeoAng f \right|
		= 
			\left\lbrace
				1 + \mathcal{O}(\GdVar) 
			\right\rbrace
			\left|
				\angdiff f
			\right|
		= \left\lbrace
				1 + \mathcal{O}_{\mydiam}(\Psiep^{1/2}) + \mathcal{O}(\varepsilon^{1/2})
			\right\rbrace
			\left|
				\angdiff f
			\right|
		$,
		a fact which follows from 
		the proofs of 
		Lemmas~\ref{L:TENSORSIZECONTROLLEDBYYCONTRACTIONS} and
		Lemma~\ref{L:ANGDERIVATIVESINTERMSOFTANGENTIALCOMMUTATOR}
		and the bootstrap assumptions.
		Hence, for scalar functions $f$, we sometimes schematically depict $\angdiff f$ 
		as 
		$\left(1 + \mathcal{O}(\GdVar)\right) \Singletan f$
		or
		$\Singletan f$
		when the factor 
		$1 + \mathcal{O}(\GdVar)$ is not important.
		Similarly, 
		Lemma~\ref{L:SCHEMATICDEPENDENCEOFMANYTENSORFIELDS},
		the schematic identity
		$
		\angD_{\GeoAng} \xi 
		= \angLie_{\GeoAng} \xi  
		+ 
		\sum \xi \cdot \angD \GeoAng
		$,
		and the proof of Lemma~\ref{L:ANGDERIVATIVESINTERMSOFTANGENTIALCOMMUTATOR}
		imply that we can depict $\angLap f$ by\footnote{In \cite{jSgHjLwW2016},
		we schematically denoted $\angLap f$ by
		$\smoothfunction(\Tanset^{\leq 1} \GdVar,\ginversesphere)
		\Tanset_*^{[1,2]} f$. Here we note that in fact, the dependence on 
		$\smoothfunction$ on $\ginversesphere$ is not needed.} 
		$
		\smoothfunction(\Tanset^{\leq 1} \GdVar)
		\Tanset_*^{[1,2]} f
		$
		(see Subsubsect.\ \ref{SS:STRINGSOFCOMMUTATIONVECTORFIELDS} regarding the notation $\Tanset_*^{[1,2]} f$)
		or $\Tanset_*^{[1,2]} f$
		when the factor $\smoothfunction(\Tanset^{\leq 1} \GdVar)$
		is not important,
		and, 
		for type $\binom{0}{n}$ $\ell_{t,u}$-tangent tensorfields $\xi$,
		$\angD \xi$ by 
		$
		\smoothfunction(\Tanset^{\leq 1} \GdVar,\ginversesphere,\angdiff x^1, \angdiff x^2)
		\angLie_{\Tanset}^{\leq 1} \xi
		$
		(or $\angLie_{\Tanset}^{\leq 1} \xi$
		when the factor $\smoothfunction(\Tanset^{\leq 1} \GdVar,\ginversesphere,\angdiff x^1, \angdiff x^2)$
		is not important\footnote{In the analogous discussion in \cite{jSgHjLwW2016},
		the dependence of $\smoothfunction$ on $\angdiff x^1,\angdiff x^2$ was mistakenly omitted.}).
	\item The constants $C$ 
		and the implicit constants
		in our estimates are allowed to depend on 
		the data-size parameters
		$\mathring{\updelta}$
		and 
		$\TranminusdatasizeWithFactor^{-1}$.
		In contrast, the constants $C_{\mydiam}$ can be chosen to be independent of
		$\mathring{\updelta}$
		and 
		$\TranminusdatasizeWithFactor^{-1}$.
		See Subsect.\ \ref{SS:NOTATIONANDINDEXCONVENTIONS}
		for a precise description 
		of the way in which we allow constants to depend on the various parameters.
\end{enumerate}

\subsection{Omission of the independent variables in some expressions}
\label{SS:OMISSION}
We use the following notational conventions in the rest of the article.
\begin{itemize}
	\item
	Many of our pointwise estimates are stated in the form
	\[
		|f_1| \lesssim F(t)|f_2|
	\]
	for some function $F$.
	Unless we otherwise indicate, it is understood that both $f_1$ and $f_2$
	are evaluated at the point with geometric coordinates
	$(t,u,\vartheta)$.
	\item Unless we otherwise indicate,
		in integrals $\int_{\ell_{t,u}} f \, d \spherevol$,
		the integrand $f$ and the length form $d \spherevol$ are viewed
		as functions of $(t,u,\vartheta)$ and $\vartheta$ is the integration variable.
	\item Unless we otherwise indicate,
		in integrals $\int_{\Sigma_t^u} f \, d \tvol$,
		the integrand $f$ and the area form $d \tvol$ are viewed
		as functions of $(t,u',\vartheta)$ and $(u',\vartheta)$ are the integration variables.
	\item Unless we otherwise indicate,
		in integrals $\int_{\mathcal{P}_u^t} f \, d \conevol$,
		the integrand $f$ and the area form $d \conevol$ are viewed
		as functions of $(t',u,\vartheta)$ and $(t',\vartheta)$ are the integration variables.
	\item Unless we otherwise indicate,
		in integrals $\int_{\mathcal{M}_{t,u}} f \, d \vol$,
		the integrand  $f$ and the volume form $d \vol$ are viewed
		as functions of $(t',u',\vartheta)$ and
		$(t',u',\vartheta)$ are the integration variables.
\end{itemize}

\subsection{Differential operator comparison estimates}
\label{SS:DIFFOPCOMPARISON}
Our main goal in this subsection is to derive differential operator comparison estimates.
We start with a simple lemma in which we show that the pointwise norm $|\cdot|$ of $\ell_{t,u}$-tangent
tensors can be controlled by contractions against the vectorfield $\GeoAng$.

\begin{lemma}[\textbf{The norm of $\ell_{t,u}$-tangent tensors can be measured via $\GeoAng$ contractions}]
	\label{L:TENSORSIZECONTROLLEDBYYCONTRACTIONS}
	Let $\xi_{\alpha_1 \cdots \alpha_n}$ be a type $\binom{0}{n}$ $\ell_{t,u}$-tangent tensor with $n \geq 1$.
	Under the data-size and bootstrap assumptions 
	of Subsects.\ \ref{SS:DATAASSUMPTIONS}-\ref{SS:AUXILIARYBOOTSTRAP}
	and the smallness assumptions of Subsect.\ \ref{SS:SMALLNESSASSUMPTIONS}, 
	we have
	\begin{align} \label{E:TENSORSIZECONTROLLEDBYYCONTRACTIONS}
		|\xi| = 
			\left\lbrace 
				1 + \mathcal{O}_{\mydiam}(\Psiep^{1/2}) + \mathcal{O}(\varepsilon^{1/2})
			\right\rbrace
			|\xi_{\GeoAng \GeoAng \cdots \GeoAng}|.
	\end{align}
	The same result holds if 
	$|\xi_{\GeoAng \GeoAng \cdots \GeoAng}|$
	is replaced with 
	$|\xi_{\GeoAng \cdot}|$, 
	$|\xi_{\GeoAng \GeoAng \cdot}|$,
	etc., where $\xi_{\GeoAng \cdot}$
	is the type $\binom{0}{n-1}$ tensor with components
  $\GeoAng^{\alpha_1} \xi_{\alpha_1 \alpha_2 \cdots \alpha_n}$,
  and similarly for $\xi_{\GeoAng \GeoAng \cdot}$, etc.
\end{lemma}

\begin{proof}
	See Subsect.\ \ref{SS:OFTENUSEDESTIMATES} for some comments on the analysis.
	\eqref{E:TENSORSIZECONTROLLEDBYYCONTRACTIONS} is easy to derive 
	relative to the Cartesian coordinates by using 
	the decomposition $(\ginversesphere)^{ij} = \frac{1}{|\GeoAng|^2} \GeoAng^i \GeoAng^j$
	and the estimate $|\GeoAng| = 1 + \mathcal{O}_{\mydiam}(\Psiep^{1/2}) + \mathcal{O}(\varepsilon^{1/2})$.
	This latter estimate follows from
	the identity 
	$|\GeoAng|^2 
	= g_{ab} \GeoAng^a \GeoAng^b
	= (\delta_{ab} + g_{ab}^{(Small)}) (\delta_2^a + \GeoAng_{(Small)}^a)(\delta_2^b + \GeoAng_{(Small)}^b)$,
	the fact that $g_{ab}^{(Small)} = \smoothfunction(\GdVar) \GdVar$ with $\smoothfunction$ smooth
	and similarly for $\GeoAng_{(Small)}^a$
	(see Lemma~\ref{L:SCHEMATICDEPENDENCEOFMANYTENSORFIELDS}),
	and the bootstrap assumptions.
\end{proof}

We now provide the main result of this subsection.

\begin{lemma}[\textbf{Controlling $\angD$ derivatives in terms of $\GeoAng$ derivatives}]
\label{L:ANGDERIVATIVESINTERMSOFTANGENTIALCOMMUTATOR}
	Let $f$ be a scalar function on $\ell_{t,u}$.
	Under the data-size and bootstrap assumptions 
	of Subsects.\ \ref{SS:DATAASSUMPTIONS}-\ref{SS:AUXILIARYBOOTSTRAP}
	and the smallness assumptions of Subsect.\ \ref{SS:SMALLNESSASSUMPTIONS}, 
	the following comparison estimates hold
	on $\mathcal{M}_{\Tboot,U_0}$:
\begin{align} \label{E:ANGDERIVATIVESINTERMSOFTANGENTIALCOMMUTATOR}
		|\angdiff f|
		& \leq (1 + C_{\mydiam} \Psiep^{1/2} + C \varepsilon^{1/2})\left| \GeoAng f \right|,
		\qquad
		|\angD^2 f|
		\leq (1 + C_{\mydiam} \Psiep^{1/2} + C \varepsilon^{1/2})\left| \angdiff(\GeoAng f) \right|
			+
			C \varepsilon^{1/2} |\angdiff f|.
\end{align}
\end{lemma}

\begin{proof}
	See Subsect.\ \ref{SS:OFTENUSEDESTIMATES} for some comments on the analysis.
	The first inequality in \eqref{E:ANGDERIVATIVESINTERMSOFTANGENTIALCOMMUTATOR}
	follows directly from Lemma~\ref{L:TENSORSIZECONTROLLEDBYYCONTRACTIONS}.
	To prove the second, we first use
	Lemma~\ref{L:TENSORSIZECONTROLLEDBYYCONTRACTIONS},
	the identity 
	$\angD_{\GeoAng \GeoAng}^2 f = \GeoAng \cdot \angdiff (\GeoAng f) - \angD_{\GeoAng} \GeoAng \cdot \angdiff f$,
	and the estimate $|\GeoAng| = 1 + \mathcal{O}_{\mydiam}(\Psiep^{1/2}) + \mathcal{O}(\varepsilon^{1/2})$ noted in the proof of
	Lemma~\ref{L:TENSORSIZECONTROLLEDBYYCONTRACTIONS}
	to deduce that 
	\begin{align} \label{E:ANGDSQUAREFUNCTIONFIRSTBOUNDINTERMSOFGEOANG}
		|\angD^2 f| 
		& \leq (1 + C_{\mydiam} \Psiep^{1/2} + C \varepsilon^{1/2})|\angD_{\GeoAng \GeoAng}^2 f|
			\leq (1 + C_{\mydiam} \Psiep^{1/2} + C \varepsilon^{1/2})
				\left\lbrace
					|\angdiff(\GeoAng f)|
					+ 
					|\angD_{\GeoAng} \GeoAng||\angdiff f|
				\right\rbrace.
	\end{align}
	Next, we use Lemma~\ref{L:TENSORSIZECONTROLLEDBYYCONTRACTIONS} 
	and the identity 
	$	\angdeformarg{\GeoAng}{\GeoAng}{\GeoAng} 
		= \angD_{\GeoAng} (\gsphere(\GeoAng, \GeoAng))
		= \GeoAng (g_{ab} \GeoAng^a \GeoAng^b)
	$
	to deduce that
	\begin{align} \label{E:ANGDGEOANGOFGEOANGINTERMSOFGEOANGDEFORMSPHERE}
		\left|
			\angD_{\GeoAng} \GeoAng
		\right|
		& \lesssim
			\left|
				g(\angD_{\GeoAng} \GeoAng,\GeoAng)
			\right|
			\lesssim
			\left|
				\angdeformarg{\GeoAng}{\GeoAng}{\GeoAng}
			\right|
			\lesssim
			\left|
				\GeoAng(g_{ab} \GeoAng^a \GeoAng^b)
			\right|.
	\end{align}
	Since Lemma~\ref{L:SCHEMATICDEPENDENCEOFMANYTENSORFIELDS} implies that
	$g_{ab} \GeoAng^a \GeoAng^b = \smoothfunction(\GdVar)$ with $\smoothfunction$ smooth,
	the bootstrap assumptions yield that RHS~\eqref{E:ANGDGEOANGOFGEOANGINTERMSOFGEOANGDEFORMSPHERE}
	is $\lesssim |\GeoAng \GdVar| \lesssim \varepsilon^{1/2}$.
	The desired second inequality in \eqref{E:ANGDERIVATIVESINTERMSOFTANGENTIALCOMMUTATOR} 
	now follows from this estimate,
	\eqref{E:ANGDSQUAREFUNCTIONFIRSTBOUNDINTERMSOFGEOANG},
	and
	\eqref{E:ANGDGEOANGOFGEOANGINTERMSOFGEOANGDEFORMSPHERE}.
\end{proof}

\subsection{Pointwise estimates for the derivatives of the \texorpdfstring{$x^i$}{Cartesian spatial coordinate functions} 
and for the Lie derivatives of the Riemannian metric 
induced on \texorpdfstring{$\ell_{t,u}$}{the tori}}

\begin{lemma}[{\textbf{Pointwise estimates for} $x^i$}]
\label{L:POINTWISEFORRECTANGULARCOMPONENTSOFVECTORFIELDS}
	Assume that $1 \leq N \leq 18$.
	Let $x^i = x^i(t,u,\vartheta)$ denote the Cartesian coordinate function
	and let $\mathring{x}^i = \mathring{x}^i(u,\vartheta) := x^i(0,u,\vartheta)$.
	Then the following estimates hold
	for $i = 1,2$
	(see Subsect.\ \ref{SS:STRINGSOFCOMMUTATIONVECTORFIELDS} regarding the vectorfield operator notation):
	\begin{subequations}
		\begin{align} \label{E:XIPOINTWISE}
			\left|
				x^i - \mathring{x}^i
			\right|
			& \lesssim 1,
				\\
			\left|
				\angdiff x^i
			\right|
			& \lesssim 
				1,
				 \label{E:ANGDIFFXI} 
				\\
			\left|
				\angdiff \Tanset^{[1,N]} x^i
			\right|
			& \lesssim 
				\left| 
					\Tanset^{[1,N]} \GdVar 
				\right|,
					\label{E:ANGDIFFXIPURETANGENTIALDIFFERENTIATED} \\
			\left|
				\angdiff \Fullset^{[1,N];1} x^i
			\right|
			& \lesssim 
				\left| 
					\Fullset_*^{[1,N];1} \GdVar 
				\right|
				+
				\left| 
					\Tanset_*^{[1,N]} \BadVar
				\right|.
					\label{E:ANGDIFFXIONERADIALDIFFERENTIATED} 
		\end{align} 
	\end{subequations}
	In the case $i=2$ at fixed $u,\vartheta$,
	LHS~\eqref{E:XIPOINTWISE} is to be interpreted as
	the Euclidean distance traveled
	by the point $x^2$
	in the flat universal covering space 
	$\mathbb{R}$ of $\mathbb{T}$
	along the corresponding integral curve of $\Lunit$
	over the time interval $[0,t]$.
\end{lemma}

\begin{proof}
See Subsect.\ \ref{SS:OFTENUSEDESTIMATES} for some comments on the analysis.
Lemma~\ref{L:SCHEMATICDEPENDENCEOFMANYTENSORFIELDS} implies
that for $V \in \lbrace \Lunit, \Radunit, \GeoAng \rbrace$,
the component $V^i = V x^i$ verifies $V^i = \smoothfunction(\GdVar)$
with $\smoothfunction$ smooth.
The estimates of the lemma therefore follow easily
from the bootstrap assumptions,
except for \eqref{E:XIPOINTWISE}.
To obtain \eqref{E:XIPOINTWISE}, 
we first argue as above to deduce $|\Lunit x^i| = |\Lunit^i| = |\smoothfunction(\GdVar)| \lesssim 1$.
Since 
$
\displaystyle
\Lunit = \frac{\partial}{\partial t}
$, 
we can integrate with respect to time starting from $t = 0$
and use the previous estimate to conclude \eqref{E:XIPOINTWISE}.

\end{proof}

\begin{lemma}[{\textbf{Crude pointwise estimates for the Lie derivatives of $\gsphere$ and $\ginversesphere$}}]
\label{L:POINTWISEESTIMATESFORGSPHEREANDITSDERIVATIVES}
	Assume that $N \leq 18$.
	Then the following estimates hold
	(see Subsect.\ \ref{SS:STRINGSOFCOMMUTATIONVECTORFIELDS} regarding the vectorfield operator notation):
	\begin{subequations}
	\begin{align} \label{E:POINTWISEESTIMATESFORGSPHEREANDITSTANGENTIALDERIVATIVES}
		\left|
			\angLie_{\Tanset}^{N+1} \gsphere
		\right|,
			\,
		\left|
			\angLie_{\Tanset}^{N+1} \ginversesphere
		\right|
			\,
		\left|
			\angLie_{\Tanset}^N \upchi
		\right|,
			\,
		\left|
			\Tanset^N \mytr \upchi
		\right|
		& \lesssim 
			\left| 
				\Tanset^{[1,N+1]} \GdVar
			\right|,
				\\
		\left|
			\angLie_{\Fullset_*}^{N+1;1} \gsphere
		\right|,
			\,
		\left|
			\angLie_{\Fullset_*}^{N+1;1} \ginversesphere
		\right|,
			\,
		\left|
			\angLie_{\Fullset}^{N;1} \upchi
		\right|,
			\,
		\left|
			\Fullset^{N;1} \mytr \upchi
		\right|
		& \lesssim 
			\left|
				\Fullset_*^{[1,N+1];1} \GdVar
			\right|
			+
			\left|
				\Tanset_*^{[1,N+1]} \BadVar
			\right|,
			\label{E:ONERADIALNOTPURERADIALFORGSPHERE}
				\\
		\left|
			\angLie_{\Fullset}^{N+1;1} \gsphere
		\right|,
			\,
		\left|
			\angLie_{\Fullset}^{N+1;1} \ginversesphere
		\right|
		& \lesssim 
			\left|
				\Fullset^{[1,N+1];1} \GdVar
			\right|
			+
			\left| 
				\Tanset_*^{[1,N+1]} \BadVar
			\right|.
			\label{E:ONERADIALFORGSPHERE}
\end{align}
\end{subequations}
\end{lemma}

\begin{proof}
		See Subsect.\ \ref{SS:OFTENUSEDESTIMATES} for some comments on the analysis.
		By Lemma~\ref{L:SCHEMATICDEPENDENCEOFMANYTENSORFIELDS}, we have
		$\gsphere = \smoothfunction(\GdVar,\angdiff x^1,\angdiff x^2)$.
		The desired estimates
		for $\angLie_{\Tanset}^{N+1} \gsphere$
		thus follow from
		Lemma~\ref{L:POINTWISEFORRECTANGULARCOMPONENTSOFVECTORFIELDS}
		and the bootstrap assumptions.
		The desired estimates for 
		$\angLie_{\Tanset}^{N+1} \ginversesphere$
		then follow from repeated use of the schematic identity
		$\angLie_{\Singletan} \ginversesphere = - (\ginversesphere)^{-2} \angLie_{\Singletan} \gsphere$
		(which is standard, see \cite{jSgHjLwW2016}*{Lemma~2.9})
		and the estimates for $\angLie_{\Tanset}^{N+1} \gsphere$.
		The estimates for $\angLie_{\Tanset}^N \upchi$ 
		and
		$\Tanset^N \mytr \upchi$
		follow from the estimates for $\angLie_{\Tanset}^{N+1} \gsphere$
		and $\angLie_{\Tanset}^{N+1} \ginversesphere$
		since $\upchi \sim \angLie_{\Singletan} \gsphere$
		(see \eqref{E:CHIDEF})
		and $\mytr \upchi \sim \ginversesphere \cdot \angLie_{\Singletan} \gsphere$.
\end{proof}

\subsection{Commutator estimates}
\label{SS:COMMUTATORESTIMATES}
In this subsection, we establish some commutator estimates.

\begin{lemma}[{\textbf{Pure $\mathcal{P}_u$-tangential commutator estimates}}]
\label{L:COMMUTATORESTIMATES}
	Assume that $1 \leq N \leq 18$.
	Let $\vec{I}$ be an order $|\vec{I}| = N + 1$
	multi-index for the set $\Tanset$ of 
	$\mathcal{P}_u$-tangential commutation vectorfields
	(see Def.~\ref{D:REPEATEDDIFFERENTIATIONSHORTHAND}),
	and let $\vec{I}'$ be any permutation of $\vec{I}$.
	Let $f$ be a scalar function, and let
	$\xi$ be an $\ell_{t,u}$-tangent one-form or a type $\binom{0}{2}$ $\ell_{t,u}$-tangent tensorfield.
	Then the following commutator estimates hold, where products involving
	the operators
	$\Tanset_*^{[1,\lfloor N/2 \rfloor]}$
	or
	$\angLie_{\Tanset}^{[1,N-1]}$
	are absent when $N=1$:
	\begin{align}
		\left|
			\Tanset^{\vec{I}} f
			-
			\Tanset^{\vec{I}'} f
		\right|
		& \lesssim
			\varepsilon^{1/2}
			\left|
				\Tanset_*^{[1,N]} f
			\right|
			+
			\left|
				\Tanset_*^{[1,\lfloor N/2 \rfloor]} f
			\right|
			\left|
				\Tanset^{[1,N]} \GdVar
			\right|.
			\label{E:PURETANGENTIALFUNCTIONCOMMUTATORESTIMATE}  
		\end{align}
		
		Moreover, if $1 \leq N \leq 17$
		and $\vec{I}$ is as above,
		then the following commutator estimates hold:
		\begin{subequations}
		\begin{align}
		\left|
			[\angD^2, \Tanset^N] f
		\right|
		& \lesssim
			\varepsilon^{1/2}
			\left|
				\Tanset_*^{[1,N]} f
			\right|
			+ 
			\left|
				\Tanset_*^{[1,\lceil N/2 \rceil]} f
			\right|
			\left|
				\Tanset^{[1,N+1]} \GdVar
			\right|,
				\label{E:ANGDSQUAREDPURETANGENTIALFUNCTIONCOMMUTATOR} \\
		\left|
			[\angLap, \Tanset^N] f
		\right|
		& \lesssim
			\varepsilon^{1/2}
			\left|
				\Tanset_*^{[1,N+1]} f
			\right|
			+ 
			\left|
				\Tanset_*^{[1,\lceil N/2 \rceil]} f
			\right|
			\left|
				\Tanset^{[1,N+1]} \GdVar
			\right|,
			\label{E:ANGLAPPURETANGENTIALFUNCTIONCOMMUTATOR}
		\end{align}
		\end{subequations}

		\begin{subequations}
		\begin{align}
		\left|
			\angLie_{\Tanset}^{\vec{I}} \xi
			-
			\angLie_{\Tanset}^{\vec{I}'} \xi
		\right|
		& \lesssim
			\varepsilon^{1/2}
			\left|
				\angLie_{\Tanset}^{[1,N]} \xi
			\right|
			+
			\left|
				\angLie_{\Tanset}^{\leq \lfloor N/2 \rfloor} \xi
			\right|
			\left|
				\Tanset^{[1,N+1]} \GdVar
			\right|,
			\label{E:PURETANGENTIALTENSORFIELDCOMMUTATORESTIMATE}
		\\
		\left|
			[\angD, \angLie_{\Tanset}^N] \xi
		\right|
		& \lesssim
			\varepsilon^{1/2}
			\left|
				\angLie_{\Tanset}^{[1,N-1]} \xi
			\right|
			+
			\left|
				\angLie_{\Tanset}^{\leq \lfloor N/2 \rfloor} \xi
			\right|
			\left|
				\Tanset^{[1,N+1]} \GdVar
			\right|,
				\label{E:ANGDANGLIETANGENTIALTENSORFIELDCOMMUTATORESTIMATE} \\
		\left|
			[\angdiv, \angLie_{\Tanset}^N] \xi
		\right|
		& \lesssim
			\varepsilon^{1/2}
			\left|
				\angLie_{\Tanset}^{[1,N]} \xi
			\right|
			+
			\left|
				\angLie_{\Tanset}^{\leq \lfloor N/2 \rfloor} \xi
			\right|
			\left|
				\Tanset^{[1,N+1]} \GdVar
			\right|.
			\label{E:ANGDIVANGLIETANGENTIALTENSORFIELDCOMMUTATORESTIMATE}
		\end{align}
		\end{subequations}

		Finally, 
		if $1 \leq N \leq 17$,
		then we have the following alternate version of
		\eqref{E:ANGDSQUAREDPURETANGENTIALFUNCTIONCOMMUTATOR}:
		\begin{align}	 \label{E:ALTERNATEANGDSQUAREDPURETANGENTIALFUNCTIONCOMMUTATOR}
		\left|
			[\angD^2, \Tanset^N] f
		\right|
		& \lesssim
			\left|
				\Tanset^{[1,\lceil N/2 \rceil +1]} \GdVar
			\right|
			\left|
				\Tanset_*^{[1,N]} f
			\right|
			+ 
			\left|
				\Tanset_*^{[1,\lceil N/2 \rceil]} f
			\right|
			\left|
				\Tanset^{[1,N+1]} \GdVar
			\right|.
	\end{align}

\end{lemma}

\begin{proof}[Discussion of proof]
The estimates of Lemma~\ref{L:COMMUTATORESTIMATES} were essentially proved in
\cite{jSgHjLwW2016}*{Lemma~8.7}, based in part on bootstrap 
assumptions that are analogs of the bootstrap assumptions of the present article and
estimates that are analogs of the estimates of
Lemmas~\ref{L:POINTWISEFORRECTANGULARCOMPONENTSOFVECTORFIELDS} and \ref{L:POINTWISEESTIMATESFORGSPHEREANDITSDERIVATIVES}.
We note that the right-hand sides of the estimates of Lemma~\ref{L:COMMUTATORESTIMATES}
are slightly different than the right-hand sides
of the estimates of \cite{jSgHjLwW2016}*{Lemma~8.7}. The difference is that
on the right-hand sides of the estimates of Lemma~\ref{L:COMMUTATORESTIMATES},
all products contain a factor involving at least one differentiation with respect to a vectorfield
belonging to the $\mathcal{P}_u$-tangential subset $\Tanset$. In particular, 
the estimates hold true without the presence of
pure order-zero products such as
$
|f||\GdVar|
$
on the right-hand sides.
This structure was not stated in
the estimates of \cite{jSgHjLwW2016}*{Lemma~8.7}, 
although the availability of this structure follows
from the proof of \cite{jSgHjLwW2016}*{Lemma~8.7}.
\end{proof}

\begin{lemma}[{\textbf{Mixed $\mathcal{P}_u$-transversal-tangent commutator estimates}}]
		\label{L:TRANSVERALTANGENTIALCOMMUTATOR}
		Assume that $1 \leq N \leq 18$.
		Let $\Fullset^{\vec{I}}$ be a 
		$\Fullset$-multi-indexed operator containing
		\textbf{exactly one} $\Rad$ factor, and assume that
		$|\vec{I}| = N+1$. 
		Let $\vec{I}'$ be any permutation of $\vec{I}$.
		Let $f$ be a scalar function.
		Then the following commutator estimates hold
		(see Subsect.\ \ref{SS:STRINGSOFCOMMUTATIONVECTORFIELDS} regarding the vectorfield operator notation):
		\begin{align}
		\left|
			\Fullset^{\vec{I}} f
			-
			\Fullset^{\vec{I}'} f
		\right|
		& \lesssim
			\left|
				\Tanset_*^{[1,N]} f
			\right|
			+
			\underbrace{
			\varepsilon^{1/2}
			\left|
				\GeoAng \Fullset^{\leq N-1;1} f
			\right|}_{\mbox{\upshape Absent if $N=1$}}
				\label{E:ONERADIALTANGENTIALFUNCTIONCOMMUTATORESTIMATE} \\
	& \ \
			+
			\left|
				\Tanset_*^{[1,\lfloor N/2 \rfloor]} f
			\right|
			\left|
				\myarray
					[\Tanset_*^{[1,N]} \BadVar]
					{\Fullset_*^{[1,N];1} \GdVar}
			\right|
			+
			\underbrace{
			\left|
				\GeoAng \Fullset^{[1,\lfloor N/2 \rfloor - 1];1} f
			\right|
			\left|
				\Tanset^{[1,N]} \GdVar
			\right|}_{\mbox{\upshape Absent if $N \leq 3$}}.
			\notag 
		\end{align}

		Moreover, if $1 \leq N \leq 17$, then the following estimates hold:
		\begin{subequations}
		\begin{align}
		\left|
			[\angD^2, \Fullset^{N;1}] f
		\right|
		& \lesssim
			\left|
				\Fullset_*^{[1,N];1} f
			\right|
				\label{E:ANGDSQUAREDONERADIALTANGENTIALFUNCTIONCOMMUTATOR} 
				\\
		& \ \
			+
			\left|
				\Tanset_*^{[1,\lceil N/2 \rceil]} f
			\right|
			\left|
				\myarray
					[\Tanset_*^{[1,N+1]} \BadVar]
					{\Fullset_*^{[1,N+1];1} \GdVar}
			\right|
			+
			\left|
				\Fullset_*^{[1,\lceil N/2 \rceil]} f 
			\right|
			\left|
				\Tanset^{[1,N+1]} \GdVar
			\right|,
			\notag \\
		\left|
			[\angLap, \Fullset^{N;1}] f
		\right|
		& \lesssim
			\left|
				\Fullset_*^{[1,N+1];1} f
			\right|
				\label{E:ANGLAPONERADIALTANGENTIALFUNCTIONCOMMUTATOR} 
				\\
		& \ \
			+
			\left|
				\Tanset_*^{[1,\lceil N/2 \rceil]} f
			\right|
			\left|
				\myarray
					[\Tanset_*^{[1,N+1]} \BadVar]
					{\Fullset_*^{[1,N+1];1} \GdVar}
			\right|
			+
			\left|
				\Fullset_*^{[1,\lceil N/2 \rceil]} f
			\right|
			\left|
				\Tanset^{[1,N+1]} \GdVar
			\right|.
			\notag 
		\end{align}
		\end{subequations}
\end{lemma}

\begin{proof}
	The estimates were essentially proved as
	\cite{jSgHjLwW2016}*{Lemma~8.8},
	based in part on bootstrap 
	assumptions that are analogs of the bootstrap assumptions of the present article and
	estimates that are analogs of the estimates of
	Lemmas~\ref{L:POINTWISEFORRECTANGULARCOMPONENTSOFVECTORFIELDS} and \ref{L:POINTWISEESTIMATESFORGSPHEREANDITSDERIVATIVES}.
	We note that the right-hand sides of the estimates of Lemma~\ref{L:TRANSVERALTANGENTIALCOMMUTATOR}
	are slightly different than the right-hand sides
	of the estimates of \cite{jSgHjLwW2016}*{Lemma~8.8}. 
	The difference is that
	on the right-hand sides of the estimates of Lemma~\ref{L:TRANSVERALTANGENTIALCOMMUTATOR},
	no pure order-zero terms such as
	$|f|$ or $|\GdVar|$ appear.
	This structure was not stated in
	the estimates of \cite{jSgHjLwW2016}*{Lemma~8.8}, although the availability of this structure follows
	from the proof of \cite{jSgHjLwW2016}*{Lemma~8.8}.
\end{proof}

\subsection{Transport inequalities and strict improvements of the auxiliary bootstrap assumptions}
\label{SS:IMPROVEMENTOFAUX}
In this subsection, 
we use the previous estimates to derive transport inequalities for the eikonal function quantities and 
strict improvements of the auxiliary bootstrap assumptions stated in Subsect.\ \ref{SS:AUXILIARYBOOTSTRAP}. 
The transport inequalities form the starting point for our
derivation of $L^2$ estimates for the below-top-order derivatives of the eikonal function quantities
as well as their top-order derivatives involving at least one $\Lunit$ differentiation
(see Lemma~\ref{L:EASYL2BOUNDSFOREIKONALFUNCTIONQUANTITIES}).
The main challenge in proving the proposition 
is to propagate the smallness of the
$\Psiep$-sized and the
$\mathring{\upepsilon}$-sized quantities, 
even though some terms in the evolution equations involve $\mathring{\updelta}$-sized
quantities, which are allowed to be large.
To this end, we must find and exploit effective partial decoupling 
between various quantities, which is present because of the special
structure of the evolution equations relative to the geometric coordinates,
because of our assumptions on the structure of the semilinear inhomogeneous
terms in the wave equations (especially \eqref{E:SOMENONINEARITIESARELINEAR}),
and because of the good properties of the commutation vectorfield
sets $\Fullset$ and $\Tanset$.

\begin{proposition}[\textbf{Transport inequalities and strict improvements of the auxiliary bootstrap assumptions}] 
\label{P:IMPROVEMENTOFAUX}
The following estimates hold
(see Subsect.\ \ref{SS:STRINGSOFCOMMUTATIONVECTORFIELDS} regarding the vectorfield operator notation).

\medskip
\noindent \underline{\textbf{Transport inequalities for the eikonal function quantities}.}

\medskip

\noindent \textbf{$\bullet$Transport inequalities for} $\upmu$.
	The following pointwise estimate holds:
	\begin{subequations}
	\begin{align} 
		\left|
			\Lunit \upmu
		\right|
		& \lesssim 
			\left|
				\Fullset \Psi
			\right|.
			\label{E:LUNITUPMUPOINTWISE} 
		\end{align}
	
	Moreover, for $1 \leq N \leq 18$, the following estimates hold:
	\begin{align} \label{E:PURETANGENTIALLUNITUPMUCOMMUTEDESTIMATE}
		\left|
			\Lunit \Tanset^N \upmu
		\right|,
			\,
		\left|
			\Tanset^N \Lunit \upmu
		\right|
		& \lesssim 
			\left|
				\Fullset_*^{[1,N+1];1} \Psi
			\right|
			+	
			\left|
				\Tanset^{[1,N]} \GdVar
			\right|
			+
			\varepsilon
			\left|
				\Tanset_*^{[1,N]} \BadVar
			\right|.
	\end{align}
	\end{subequations}
	
	\noindent \textbf{$\bullet$Transport inequalities for} $\Lunit_{(Small)}^i$ and $\mytr \upchi$.
	For $N \leq 18$, the following estimates hold:
	\begin{subequations}
	\begin{align}
		\left|
			\myarray
				[\Lunit \Tanset^N \Lunit_{(Small)}^i]
				{\Lunit \Tanset^{N-1} \mytr \upchi}
		\right|,
			\,
		\left|
			\myarray
				[\Tanset^N \Lunit \Lunit_{(Small)}^i]
				{\Tanset^{N-1} \Lunit \mytr \upchi}
		\right|
		& \lesssim 
			\left|
				\Tanset^{[1,N+1]} \Psi
			\right|
		+ \varepsilon
			\left|
				\Tanset^{[1,N]} \GdVar
			\right|,
				\label{E:LUNITTANGENTDIFFERENTIATEDLUNITSMALLIMPROVEDPOINTWISE} \\
		\left|
			\myarray
				[\Lunit \Fullset^{N;1} \Lunit_{(Small)}^i]
				{\Lunit \Fullset^{N-1;1} \mytr \upchi}
		\right|,
			\,
		\left|
			\myarray
				[\Fullset^{N;1} \Lunit \Lunit_{(Small)}^i]
				{\Fullset^{N-1;1} \Lunit \mytr \upchi}
		\right|
		& \lesssim 
		\left|
			\Fullset_*^{[1,N+1];1} \Psi
		\right|
		+ 
		\left|
			\myarray[\varepsilon \Tanset_*^{[1,N]} \BadVar]
				{\Fullset_*^{[1,N];1} \GdVar}
		\right|.
			\label{E:LUNITONERADIALTANGENTDIFFERENTIATEDLUNITSMALLIMPROVEDPOINTWISE} 
\end{align}
\end{subequations}

\medskip		
\noindent \underline{$L^{\infty}$ \textbf{estimates for} 
$\Psi$, 
$\bigslow$,
\textbf{and the eikonal function quantities}.}

\medskip

\noindent \textbf{$\bullet$$L^{\infty}$ estimates involving at most one transversal derivative of $\Psi$}. 	
The following estimates hold:
\begin{subequations}	
\begin{align} 
	\left\| 
		\Psi
	\right\|_{L^{\infty}(\Sigma_t^u)}
		& \leq \Psiep + C \varepsilon,
		\label{E:PSIITSELFBOOTSTRAPIMPROVED}
			\\
		\left\| 
			\Fullset_*^{[1,10];1} \Psi
		\right\|_{L^{\infty}(\Sigma_t^u)}
		& \leq C \varepsilon,
		\label{E:PSIMIXEDTRANSVERSALTANGENTBOOTSTRAPIMPROVED}
			\\
		\left\| 
		\Rad \Psi 
	\right\|_{L^{\infty}(\Sigma_t^u)}
	& \leq 
	\left\| 
		\Rad \Psi 
	\right\|_{L^{\infty}(\Sigma_0^u)}
	+ C \varepsilon.
		\label{E:PSITRANSVERSALLINFINITYBOUNDBOOTSTRAPIMPROVED}
\end{align}
\end{subequations}

\noindent \textbf{$\bullet$$L^{\infty}$ estimates involving at most one transversal derivative of $\bigslow$}.	
The following estimates hold:
\begin{align}
	\left\| 
		\Fullset^{\leq 10;1} \bigslow
	\right\|_{L^{\infty}(\Sigma_t^u)}
	& \leq C \varepsilon.
		\label{E:SLOWWAVETRANSVERSALTANGENT}
\end{align}	

\noindent \textbf{$\bullet$$L^{\infty}$ estimates for $\upmu$}. 		
	The following estimates hold:
	\begin{subequations}
	\begin{align} \label{E:LUNITUPMULINFINITY}
		\left\| 
			\Lunit \upmu
		\right\|_{L^{\infty}(\Sigma_t^u)}
		& 
		= 
		\frac{1}{2}
		\left\| 
			G_{\Lunit \Lunit} \Rad \Psi
		\right\|_{L^{\infty}(\Sigma_0^u)}
		+ \mathcal{O}(\varepsilon),
			\\
		\left\|
			\Lunit \Tanset^{[1,9]} \upmu
		\right\|_{L^{\infty}(\Sigma_t^u)},
			\,
		\left\| 
			\Tanset_*^{[1,9]} \upmu
		\right\|_{L^{\infty}(\Sigma_t^u)}
		& \leq
			C \varepsilon,
			\label{E:LUNITAPPLIEDTOTANGENTIALUPMUANDTANSETSTARLINFTY}
	\end{align}
	\end{subequations}
	
	\begin{subequations}
	\begin{align} \label{E:UPMULINFTY}
	\left\| 
			\upmu 
		\right\|_{L^{\infty}(\Sigma_t^u)}
		& \leq
		1
		+
	 	\TranminusdatasizeWithFactor^{-1} 
		\left\| 
			G_{\Lunit \Lunit} \Rad \Psi
		\right\|_{L^{\infty}(\Sigma_0^u)}
		+ 
		C_{\mydiam} \Psiep
		+ 
		C \varepsilon.
	\end{align}
	\end{subequations}
	
\noindent \textbf{$\bullet$$L^{\infty}$ estimates for $\Lunit_{(Small)}^i$ and $\upchi$}.
The following estimates hold:
\begin{subequations}
\begin{align}  
	\left\|
		\Lunit_{(Small)}^i
	\right\|_{L^{\infty}(\Sigma_t^u)}
		& \leq C_{\mydiam} \Psiep + C \varepsilon,
		\label{E:LUNITISMALLITSELFLSMALLINFTYESTIMATE} \\
	\left\|
		\Lunit \Tanset^{\leq 10} \Lunit_{(Small)}^i
	\right\|_{L^{\infty}(\Sigma_t^u)},
		\,
	\left\|
		\Tanset^{[1,10]} \Lunit_{(Small)}^i
	\right\|_{L^{\infty}(\Sigma_t^u)}
	& \leq C \varepsilon,
		\label{E:PURETANGENTIALLUNITAPPLIEDTOLISMALLANDLISMALLINFTYESTIMATE} \\
	\left\|
		\Lunit \Fullset^{\leq 9;1} \Lunit_{(Small)}^i
	\right\|_{L^{\infty}(\Sigma_t^u)},
		\,
	\left\|
		\Fullset_*^{[1,9];1} \Lunit_{(Small)}^i
	\right\|_{L^{\infty}(\Sigma_t^u)}
	& \leq C \varepsilon,
		\label{E:LUNITAPPLIEDTOLISMALLANDLISMALLINFTYESTIMATE} \\
	\left\|
		\Rad \Lunit_{(Small)}^i
	\right\|_{L^{\infty}(\Sigma_t^u)}
	& \leq
	\left\| 
		\Rad \Lunit_{(Small)}^i
	\right\|_{L^{\infty}(\Sigma_0^u)}
	+  C \varepsilon,	
	\label{E:LISMALLLONERADIALINFINITYESTIMATE}
\end{align}	
\end{subequations}	

\begin{subequations}
\begin{align}  \label{E:PURETANGENTIALCHICOMMUTEDLINFINITY}
		\left\|
			\angLie_{\Tanset}^{\leq 9} \upchi
		\right\|_{L^{\infty}(\Sigma_t^u)},
			\,
		\left\|
			\angLie_{\Tanset}^{\leq 9} \upchi^{\#}
		\right\|_{L^{\infty}(\Sigma_t^u)},
			\,
		\left\|
			\Tanset^{\leq 9} \mytr \upchi
		\right\|_{L^{\infty}(\Sigma_t^u)}
		& \leq C \varepsilon,
			\\
		\left\|
			\angLie_{\Fullset}^{\leq 8;1} \upchi
		\right\|_{L^{\infty}(\Sigma_t^u)},
			\,
		\left\|
			\angLie_{\Fullset}^{\leq 8;1} \upchi^{\#}
		\right\|_{L^{\infty}(\Sigma_t^u)},
			\,
		\left\|
			\Fullset^{\leq 8;1} \mytr \upchi
		\right\|_{L^{\infty}(\Sigma_t^u)}
		& \leq C \varepsilon.
		\label{E:ONERADIALCHICOMMUTEDLINFINITY}
\end{align}
\end{subequations}
\end{proposition}

\begin{proof}[Proof outline]
	See Subsect.\ \ref{SS:OFTENUSEDESTIMATES} for some comments on the analysis.
	Throughout, we refer to the data-size assumptions of
	Subsect.\ \ref{SS:DATAASSUMPTIONS} and the bounds of Lemma~\ref{L:BEHAVIOROFEIKONALFUNCTIONQUANTITIESALONGSIGMA0}
	as the ``conditions on the data.'' 
	
	To derive \eqref{E:SLOWWAVETRANSVERSALTANGENT} in the case
	$\Fullset^{\leq 10;1} = \Tanset^{\leq 10}$, 
	we simply note that the desired bound is one of
	the bootstrap assumptions from \eqref{E:PSIFUNDAMENTALC0BOUNDBOOTSTRAP}.
	To prove \eqref{E:SLOWWAVETRANSVERSALTANGENT} in the remaining case in which
	$\Fullset^{\leq 10;1}$ contains a factor of $\Rad$,
	we first apply $\Tanset^{\leq 9}$
	to the identity \eqref{E:RADOFSLOWWAVEALGEBRAICALLYEXPRESSED}
	and use the bootstrap assumptions
	to deduce that
	$
	\left\| 
		\Tanset^{\leq 9} \Rad \bigslow
	\right\|_{L^{\infty}(\Sigma_t^u)}
	\lesssim \varepsilon
	$.
	We then use the commutator estimate
	\eqref{E:ONERADIALTANGENTIALFUNCTIONCOMMUTATORESTIMATE}
	with $f = \bigslow$, 
	the estimate just proved for
	$\Tanset^{\leq 9} \Rad \bigslow$,
	and the bootstrap assumptions, 
	which allow us to 
	arbitrarily commute the vectorfields in the expression
	$\Tanset^{\leq 9} \Rad \bigslow$
	up to errors bounded in $\| \cdot \|_{L^{\infty}(\Sigma_t^u)}$  by $\lesssim \varepsilon$.
	In total, we have derived
	the desired bound
	$
	\left\| 
		\Fullset^{\leq 10;1} \bigslow
	\right\|_{L^{\infty}(\Sigma_t^u)}
	\lesssim \varepsilon
	$.
	
	The remaining estimates
	in Prop.~\ref{P:IMPROVEMENTOFAUX}
	can be established using 
	arguments nearly identical to the ones used in proving
	\cite{jSgHjLwW2016}*{Proposition~8.10}, 
	as we now outline.
	Specifically, one uses the transport equations of Lemma~\ref{L:UPMUANDLUNITIFIRSTTRANSPORT},
	the estimates of 
	Lemmas~\ref{L:POINTWISEFORRECTANGULARCOMPONENTSOFVECTORFIELDS}-\ref{L:POINTWISEESTIMATESFORGSPHEREANDITSDERIVATIVES},
	the commutator estimates of Subsect.\ \ref{SS:COMMUTATORESTIMATES},
	and the conditions on the data
	to derive the desired bounds for $\upmu$ and $\Lunit_{(Small)}^i$;
	these bounds are not explicitly tied to $\bigslow$ and hence the proofs
	from \cite{jSgHjLwW2016}*{Proposition~8.10} go through nearly verbatim.
	There are two minor differences that we now highlight.
	\textbf{i)} Note that the estimates 
	\eqref{E:LUNITUPMUPOINTWISE},
	\eqref{E:PURETANGENTIALLUNITUPMUCOMMUTEDESTIMATE},
	\eqref{E:LUNITTANGENTDIFFERENTIATEDLUNITSMALLIMPROVEDPOINTWISE},
	and \eqref{E:LUNITONERADIALTANGENTDIFFERENTIATEDLUNITSMALLIMPROVEDPOINTWISE} 
	do not feature any order $0$ terms such as $|\GdVar|$ on the RHS.
	This is different compared to the analogous
	estimates stated in \cite{jSgHjLwW2016}*{Proposition~8.10},
	but follows from the proof given there
	and from the commutator estimates of 
	Lemmas~\ref{L:COMMUTATORESTIMATES} and
	\ref{L:TRANSVERALTANGENTIALCOMMUTATOR}
	(see also the remarks made in the discussion of the proofs of
		Lemmas~\ref{L:COMMUTATORESTIMATES} and
	\ref{L:TRANSVERALTANGENTIALCOMMUTATOR}).
	\textbf{ii)} The estimates
	\eqref{E:PSIITSELFBOOTSTRAPIMPROVED},
	\eqref{E:UPMULINFTY},
	\eqref{E:LUNITISMALLITSELFLSMALLINFTYESTIMATE}
	feature the parameter $\Psiep$ on the RHS,
	which is different compared to the analogous
	estimates stated in \cite{jSgHjLwW2016}*{Proposition~8.10}.
	The difference stems 
	from the fact that some of the order $0$ quantities in this paper are controlled by
	the data-size parameter $\Psiep$, which is not featured in \cite{jSgHjLwW2016};
	see also Remark~\ref{R:NEWPARAMETER}.
	
	The estimates for $\upchi$ then follow from the estimates for
	$\upmu$ and $\Lunit_{(Small)}^i$ described above
	and Lemmas~\ref{L:IDFORCHI} and ~\ref{L:SCHEMATICDEPENDENCEOFMANYTENSORFIELDS}.
	
	To derive the desired estimate
	\eqref{E:PSITRANSVERSALLINFINITYBOUNDBOOTSTRAPIMPROVED}
	for $\Psi$
	and the estimate \eqref{E:PSIMIXEDTRANSVERSALTANGENTBOOTSTRAPIMPROVED}
	for $\Psi$ when $\Fullset_*^{[1,10];1}$ contains exactly one factor of $\Rad$
	(the desired bounds in the case case $\Fullset_*^{[1,10];1} = \Tanset^{[1,10]}$ 
	are restatements of one of the bootstrap assumptions \eqref{E:PSIFUNDAMENTALC0BOUNDBOOTSTRAP}),
	one can use equation \eqref{E:LONOUTSIDEGEOMETRICWAVEOPERATORFRAMEDECOMPOSED},
	equation \eqref{E:UPMUFIRSTTRANSPORT},
	Lemma~\ref{L:SCHEMATICDEPENDENCEOFMANYTENSORFIELDS},
	Lemma~\ref{L:CARTESIANVECTORFIELDSINTERMSOFGEOMETRICONES},
	and the assumptions \eqref{E:SOMENONINEARITIESARELINEAR} on the semilinear terms
	to rewrite the wave equation \eqref{E:FASTWAVE} for $\Psi$ 
	in the following schematic ``transport equation'' form:
	\begin{align} \label{E:WAVEEQUATIONTRANSPORTINTERPRETATION}
		\Lunit \Rad \Psi 
		& 
		= \smoothfunction(\BadVar) \angLap \Psi 
			+ \smoothfunction(\BadVar,\ginversesphere,\angdiff x^1,\angdiff x^2,\Singletan \Psi, \Rad \Psi) 
				\Singletan \Singletan \Psi
			+ \smoothfunction(\BadVar,\bigslow,\ginversesphere,\angdiff x^1,\angdiff x^2,\Singletan \Psi, \Rad \Psi) 
				\Singletan \GdVar
					\\
		& \ \
			+ \smoothfunction(\BadVar, \bigslow, \Singletan \Psi, \Rad \Psi) \bigslow.
			\notag
	\end{align}
	Then by applying $\Tanset^{\leq 9}$ to \eqref{E:WAVEEQUATIONTRANSPORTINTERPRETATION}
	and using the commutator estimates of Subsect.\ \ref{SS:COMMUTATORESTIMATES}
	and the bootstrap assumptions, one can show that
	$\left|\Lunit \Tanset^{\leq 9} \Rad \Psi \right| 
	\lesssim \varepsilon$,
	from which the bounds
	$\| \Rad \Psi \|_{L^{\infty}(\Sigma_t^u)} \leq \| \Rad \Psi \|_{L^{\infty}(\Sigma_0^u)} + C \varepsilon$
	and
	$\| \Tanset^{[1,9]} \Rad \Psi \|_{L^{\infty}(\Sigma_t^u)} \lesssim \varepsilon$
	easily follow by integrating in time
	(recall that 
	$\displaystyle
	\Lunit = \frac{\partial}{\partial t}
	$) and using the conditions on the data.
	Then by further applications of the commutator estimates of Subsect.\ \ref{SS:COMMUTATORESTIMATES},
	we obtain
	$\| \Fullset_*^{[1,10];1} \Psi \|_{L^{\infty}(\Sigma_t^u)} \lesssim \varepsilon$.
	More precisely, all terms that arise from differentiating RHS~\eqref{E:WAVEEQUATIONTRANSPORTINTERPRETATION}
	with $\Tanset^{\leq 9}$
	were handled in the proof of \cite{jSgHjLwW2016}*{Proposition~8.10}
	except for the ones involving $\bigslow$.
	Note in particular that the commutator estimates 
	needed to commute $\Tanset^{\leq 9}$ through the operator $\Lunit$ on LHS~\eqref{E:WAVEEQUATIONTRANSPORTINTERPRETATION}
	and through the operator $\angLap$ on RHS~\eqref{E:WAVEEQUATIONTRANSPORTINTERPRETATION}
	do not involve $\bigslow$ in any way.
	That is, the only influence of $\bigslow$ on the estimates 
	under consideration
	is through the terms $\Tanset^{\leq 9} \bigslow$
	that arise from RHS \eqref{E:WAVEEQUATIONTRANSPORTINTERPRETATION}.
	Due to the bootstrap assumption
	$\| \Tanset^{\leq 10} \bigslow \|_{L^{\infty}(\Sigma_t^u)} \leq \varepsilon$
	stated in \eqref{E:PSIFUNDAMENTALC0BOUNDBOOTSTRAP},
	the products containing a factor of $\Tanset^{\leq 9} \bigslow$
	make only a negligible $\mathcal{O}(\varepsilon)$ contribution
	to the estimates. For this reason, the analysis for $\Psi$
	in the present context is a negligible $\mathcal{O}(\varepsilon)$ perturbation
	of the analogous analysis
	carried out in the proof of \cite{jSgHjLwW2016}*{Proposition~8.10}.

	\end{proof}

The following corollary is an immediate consequence of the fact that we have improved the auxiliary
bootstrap assumptions of Subsect.\ \ref{SS:AUXILIARYBOOTSTRAP}
by showing that they hold with $\varepsilon^{1/2}$ replaced by $C \varepsilon$
and with $\Psiep^{1/2}$ replaced by $C_{\mydiam} \Psiep$.

\begin{corollary}[$\varepsilon^{1/2} \rightarrow C \varepsilon$ 
	\textbf{and}
	$\Psiep^{1/2} \rightarrow C_{\mydiam} \Psiep$]
	\label{C:SQRTEPSILONTOCEPSILON}
	All prior inequalities whose right-hand sides feature an explicit factor of $\varepsilon^{1/2}$
	remain true with $\varepsilon^{1/2}$ replaced by $C \varepsilon$.
	Moreover, all prior inequalities whose right-hand sides feature an explicit factor of $\Psiep^{1/2}$
	remain true with $\Psiep^{1/2}$ replaced by $C_{\mydiam} \Psiep$.
	This is true in particular for the auxiliary bootstrap assumptions 
	of Subsect.\ \ref{SS:AUXILIARYBOOTSTRAP}.
\end{corollary}

\begin{remark}[\textbf{The auxiliary bootstrap assumptions are now redundant}]
	\label{R:AUXAREREDUNDANT}
	Since we have derived strict improvements of the auxiliary
	bootstrap assumptions of Subsect.\ \ref{SS:AUXILIARYBOOTSTRAP}, 
	when proving estimates later in the paper,
	we no longer need to state them as assumptions.
\end{remark}

\section{\texorpdfstring{$L^{\infty}$}{Essential sup-norm} estimates involving higher-order transversal derivatives}
\label{S:LINFINITYESTIMATESFORHIGHERTRANSVERSAL}
Our energy estimates rely on the delicate estimate 
$
\displaystyle
\left\|
		\frac{[\Rad \upmu]_+}{\upmu}
	\right\|_{L^{\infty}(\Sigma_t^u)}
	\leq 
		\frac{C}{\sqrt{\Tboot - t}}
$
(see \eqref{E:UNIFORMBOUNDFORMRADMUOVERMU}),
whose proof relies on
the bound
$
\left\|
\Rad \Rad \upmu
\right\|_{L^{\infty}(\Sigma_t^u)}
\lesssim 1
$.
In this section, we derive this bound and related ones
that are needed to prove it.
In particular, it turns out that to obtain the desired
estimates for $\upmu$, we must show that
$
\left\|
\Rad \Rad \Rad \Psi
\right\|_{L^{\infty}(\Sigma_t^u)}
\lesssim 1
$.

\subsection{Auxiliary \texorpdfstring{$L^{\infty}$}{essential sup-norm} bootstrap assumptions}
\label{SS:BOOTSTRAPFORHIGHERTRANSVERSAL}
To facilitate the analysis, we introduce the following 
auxiliary bootstrap assumptions.
In Prop.~\ref{P:IMPROVEMENTOFHIGHERTRANSVERSALBOOTSTRAP}, we derive strict improvements of the assumptions based on 
the estimates of Sect.\ \ref{S:PRELIMINARYPOINTWISE}
and our assumptions on the data.

\medskip

\noindent \underline{\textbf{Auxiliary bootstrap assumptions for small quantities}.}
We assume that the following inequalities hold on $\mathcal{M}_{\Tboot,U_0}$
(see Subsect.\ \ref{SS:STRINGSOFCOMMUTATIONVECTORFIELDS} regarding the vectorfield operator notation):
\begin{align}
	\left\| 
		\Fullset_*^{[1,4];2} \Psi 
	\right\|_{L^{\infty}(\Sigma_t^u)}
	& \leq \varepsilon^{1/2},
		\label{E:HIGHERTRANSVERSALPSIMIXEDFUNDAMENTALLINFINITYBOUNDBOOTSTRAP} 
		\tag{$\mathbf{BA'}1\Psi$}
			\\
	\left\| 
		\Fullset^{\leq 3;2} \bigslow
	\right\|_{L^{\infty}(\Sigma_t^u)}
	& \leq \varepsilon^{1/2},
		\label{E:HIGHERTRANSVERSALSLOWWAVEMIXEDFUNDAMENTALLINFINITYBOUNDBOOTSTRAP} 
		\tag{$\mathbf{BA'} \bigslow$}
\end{align}

\begin{subequations}
\begin{align} \label{E:UPMUONERADIALNOTPURERADIALBOOTSTRAP} \tag{$\mathbf{BA'}1\upmu$}
	\left\| 
		\Rad \GeoAng \upmu
	\right\|_{L^{\infty}(\Sigma_t^u)},
		\,
	\left\| 
		\Rad \Lunit \Lunit \upmu
	\right\|_{L^{\infty}(\Sigma_t^u)},
		\,
	\left\| 
		\Rad \GeoAng \GeoAng \upmu
	\right\|_{L^{\infty}(\Sigma_t^u)},
		\,
	\left\| 
		\Rad \Lunit \GeoAng \upmu
	\right\|_{L^{\infty}(\Sigma_t^u)}
	& \leq \varepsilon^{1/2},
\end{align}
\end{subequations}
and
\begin{align} \label{E:PERMUTEDUPMUONERADIALNOTPURERADIALBOOTSTRAP} \tag{$\mathbf{BA''}1\upmu$}
	\eqref{E:UPMUONERADIALNOTPURERADIALBOOTSTRAP} \mbox{ also holds for all permutations of the vectorfield operators on LHS } \eqref{E:UPMUONERADIALNOTPURERADIALBOOTSTRAP},
\end{align}

\begin{align} \label{E:UPMUTWORADIALNOTPURERADIALBOOTSTRAP} \tag{$\mathbf{BA'}1\Lunit_{(Small)}$}
	\left\|
		\Fullset_*^{[1,3];2}
		\Lunit_{(Small)}^i
	\right\|_{L^{\infty}(\Sigma_t^u)}
	& \leq \varepsilon^{1/2}.
\end{align}

\noindent \underline{\textbf{Auxiliary bootstrap assumptions for quantities that are allowed to be large}.}
We assume that the following inequalities hold on $\mathcal{M}_{\Tboot,U_0}$:
\begin{align}
	\left\| 
		\Rad^M \Psi 
	\right\|_{L^{\infty}(\Sigma_t^u)}
	& \leq 
	\left\| 
		\Rad^M \Psi 
	\right\|_{L^{\infty}(\Sigma_0^u)}
	+ \varepsilon^{1/2},
	&& (2 \leq M \leq 3),
		\label{E:HIGHERPSITRANSVERSALFUNDAMENTALC0BOUNDBOOTSTRAP} 
		\tag{$\mathbf{BA'}2\Psi$}
\end{align}

\begin{align}
	\left\| 
		\Lunit \Rad^M \upmu
	\right\|_{L^{\infty}(\Sigma_t^u)}
	& \leq 
		\frac{1}{2}
		\left\| 
			\Rad^M
			\left\lbrace
				G_{\Lunit \Lunit} \Rad \Psi
			\right\rbrace
		\right\|_{L^{\infty}(\Sigma_0^u)}
		+ \varepsilon^{1/2},
		&& (1 \leq M \leq 2),
			\label{E:HIGHERLUNITUPMUBOOT}  \tag{$\mathbf{BA'}2\upmu$}  \\
		\left\| 
			\Rad^M \upmu
		\right\|_{L^{\infty}(\Sigma_t^u)}
		& \leq
	 	\left\| 
				\Rad^M \upmu
		\right\|_{L^{\infty}(\Sigma_0^u)}
		+ 
		\TranminusdatasizeWithFactor^{-1} 
		\left\| 
			\Rad^M
			\left\lbrace
				G_{\Lunit \Lunit} \Rad \Psi
			\right\rbrace
		\right\|_{L^{\infty}(\Sigma_0^u)}
		+ \varepsilon^{1/2},
		&& (1 \leq M \leq 2),
			\label{E:HIGHERUPMUTRANSVERSALBOOT} 
			\tag{$\mathbf{BA'}3\upmu$} \\
		\left\| 
			\Rad \Rad \Lunit_{(Small)}^i
		\right\|_{L^{\infty}(\Sigma_t^u)}
		& \leq
	 	\left\| 
				\Rad \Rad \Lunit_{(Small)}^i
		\right\|_{L^{\infty}(\Sigma_0^u)}
		+ \varepsilon^{1/2}.
		 &&
			\label{E:HIGHERLUNITITRANSVERSALBOOT}  
			\tag{$\mathbf{BA'}2\Lunit_{(Small)}$} 
\end{align}

\subsection{Commutator estimates involving two transversal derivatives}
\label{SS:TWORADDERIVATIVESCOMMUTATORESTIMATES}
We now provide some basic commutator estimates involving two factors of the 
$\mathcal{P}_u$-transversal vectorfield $\Rad$.

\begin{lemma}\cite{jSgHjLwW2016}*{Lemma~9.1; \textbf{Mixed $\mathcal{P}_u$-transversal-tangent 
commutator estimates involving two $\Rad$ derivatives}}
	\label{L:HIGHERTRANSVERALTANGENTIALCOMMUTATOR}
		Let $\Fullset^{\vec{I}}$ be a 
		$\Fullset$-multi-indexed operator containing
		exactly two $\Rad$ factors, and assume that
		$3 \leq |\vec{I}| := N+1 \leq 4$. 
		Let $\vec{I}'$ be any permutation of $\vec{I}$.
		Under the data-size and bootstrap assumptions
		of Subsects.\ \ref{SS:DATAASSUMPTIONS}-\ref{SS:PSIBOOTSTRAP}
		and Subsect.\ \ref{SS:BOOTSTRAPFORHIGHERTRANSVERSAL}
		and the smallness assumptions of Subsect.\ \ref{SS:SMALLNESSASSUMPTIONS}, 
		the following commutator estimates hold for functions $f$ on $\mathcal{M}_{\Tboot,U_0}$
		(see Subsect.\ \ref{SS:STRINGSOFCOMMUTATIONVECTORFIELDS} regarding the vectorfield operator notation):
		\begin{align}
		\left|
			\Fullset^{\vec{I}} f
			-
			\Fullset^{\vec{I}'} f
		\right|
		& \lesssim
			\left|
				\GeoAng \Fullset^{\leq N-1;1} f
			\right|
			+
			\underbrace{
			\left|
				\GeoAng \Fullset^{\leq N-1;2} f
			\right|}_{\mbox{\upshape Absent if $N=2$}}.
				\label{E:TWORADIALTANGENTIALFUNCTIONCOMMUTATORESTIMATE} 
		\end{align}

		Moreover, we have
		\begin{align}
		\left|
			[\angLap, \Rad \Rad] f
		\right|
		& \lesssim
			\left|
				\GeoAng \Fullset^{\leq 2;1} f
			\right|.
				\label{E:ANGLAPTWORADIALTANGENTIALFUNCTIONCOMMUTATOR} 
		\end{align}
\end{lemma}

\begin{proof}[Discussion of the proof]
	Thanks in part to the $L^{\infty}$ estimates of Prop.\ \ref{P:IMPROVEMENTOFAUX}
	and the bootstrap assumption of Subsect.\ \ref{SS:BOOTSTRAPFORHIGHERTRANSVERSAL},
	the lemma follows from the same arguments given in
	\cite{jSgHjLwW2016}*{Lemma~9.1}.
	In particular, we stress that these estimates do 
	not depend on the slow wave variables $\bigslow$.
	We note that we have ignored a smallness factor from \cite{jSgHjLwW2016}*{Lemma~9.1}
	that could have been placed in front of the second product 
	on RHS~\eqref{E:TWORADIALTANGENTIALFUNCTIONCOMMUTATORESTIMATE};
	the smallness factor is not important for our estimates.
\end{proof}

\subsection{The main estimates involving higher-order transversal derivatives}
\label{SS:HIGHERORDERTRANSVERALMAINESTIMATES}
We now prove the main result of Sect.\ \ref{S:LINFINITYESTIMATESFORHIGHERTRANSVERSAL}.

\begin{proposition}[\textbf{$L^{\infty}$ estimates involving higher-order transversal derivatives}]
	\label{P:IMPROVEMENTOFHIGHERTRANSVERSALBOOTSTRAP}
 		Under the data-size and bootstrap assumptions
		of Subsects.\ \ref{SS:DATAASSUMPTIONS}-\ref{SS:PSIBOOTSTRAP}
		and Subsect.\ \ref{SS:BOOTSTRAPFORHIGHERTRANSVERSAL}
		and the smallness assumptions of Subsect.\ \ref{SS:SMALLNESSASSUMPTIONS}, 
		the following statements hold true
		on $\mathcal{M}_{\Tboot,U_0}$
		(see Subsect.\ \ref{SS:STRINGSOFCOMMUTATIONVECTORFIELDS} regarding the vectorfield operator notation).
	
	\medskip
	
	\noindent \underline{\textbf{$L^{\infty}$ estimates involving two or three transversal derivatives of $\Psi$}.}
	The following estimates hold:
	\begin{subequations}
	\begin{align}
	\left\| 
		\Fullset_*^{[1,4];2} \Psi
	\right\|_{L^{\infty}(\Sigma_t^u)}
	& \leq C \varepsilon,
		\label{E:IMPROVEDTRANSVERALESTIMATESFORTWORADBUTNOTPURERAD} 
		\\
	\left\| 
		\Rad \Rad \Psi
	\right\|_{L^{\infty}(\Sigma_t^u)}
	& \leq 
		\left\| 
			\Rad \Rad \Psi
		\right\|_{L^{\infty}(\Sigma_0^u)}
		+
		C \varepsilon,
	\label{E:IMPROVEDTRANSVERALESTIMATESFORTWORAD}
		\\
	\left\| 
		\Lunit \Rad \Rad \Rad \Psi
	\right\|_{L^{\infty}(\Sigma_t^u)}
	& \leq 
		C \varepsilon,
	\label{E:IMPROVEDTRANSVERALESTIMATESFORLUNITTHREERAD}
		\\
	\left\| 
		\Rad \Rad \Rad \Psi
	\right\|_{L^{\infty}(\Sigma_t^u)}
	& \leq 
		\left\| 
			\Rad \Rad \Rad \Psi
		\right\|_{L^{\infty}(\Sigma_0^u)}
		+
		C \varepsilon.
	\label{E:IMPROVEDTRANSVERALESTIMATESFORTHREERAD}
	\end{align}
	\end{subequations}

	\noindent \underline{\textbf{$L^{\infty}$ estimates involving one or two transversal derivatives of $\upmu$}.}
	The following estimates hold:
	\begin{subequations}
	\begin{align}
		\left\|
			\Lunit \Rad \upmu
		\right\|_{L^{\infty}(\Sigma_t^u)}
		& \leq
			\frac{1}{2}
			\left\|
				\Rad
				\left(
					G_{\Lunit \Lunit} \Rad \Psi
				\right)
			\right\|_{L^{\infty}(\Sigma_0^u)}
			+
			C \varepsilon,
				\label{E:LUNITRADUPMULINFTY}
					\\
		\left\|
			\Rad \upmu
		\right\|_{L^{\infty}(\Sigma_t^u)}
		& \leq
			\left\|
				\Rad \upmu
			\right\|_{L^{\infty}(\Sigma_0^u)}
			+
			\TranminusdatasizeWithFactor^{-1}
			\left\|
				\Rad
				\left(
					G_{\Lunit \Lunit} \Rad \Psi
				\right)
			\right\|_{L^{\infty}(\Sigma_0^u)}
			+
			C \varepsilon,
				\label{E:RADUPMULINFTY}
\end{align}

\begin{align}
	\left\| 
		\Lunit \Rad \GeoAng \upmu
	\right\|_{L^{\infty}(\Sigma_t^u)},
		\,
	\left\| 
		\Lunit \Rad \Lunit \Lunit \upmu
	\right\|_{L^{\infty}(\Sigma_t^u)},
		\,
	\left\| 
		\Lunit \Rad \GeoAng \GeoAng \upmu
	\right\|_{L^{\infty}(\Sigma_t^u)},
		\,
	\left\| 
		\Lunit \Rad \Lunit \GeoAng \upmu
	\right\|_{L^{\infty}(\Sigma_t^u)}
	& \leq C \varepsilon,
				\label{E:LUNITRADTANGENTIALUPMULINFTY}
				\\
	\left\| 
		\Rad \GeoAng \upmu
	\right\|_{L^{\infty}(\Sigma_t^u)},
		\,
	\left\| 
		\Rad \Lunit \Lunit \upmu
	\right\|_{L^{\infty}(\Sigma_t^u)},
		\,
	\left\| 
		\Rad \GeoAng \GeoAng \upmu
	\right\|_{L^{\infty}(\Sigma_t^u)},
		\,
	\left\| 
		\Rad \Lunit \GeoAng \upmu
	\right\|_{L^{\infty}(\Sigma_t^u)}
	& \leq C \varepsilon,
				\label{E:RADTANGENTIALUPMULINFTY}
\end{align}

\begin{align} \label{E:PERMUTEDRADTANGENTIALUPMULINFTY}
	\eqref{E:LUNITRADTANGENTIALUPMULINFTY}-\eqref{E:RADTANGENTIALUPMULINFTY} \mbox{ also hold for all permutations of the vectorfield 
	operators on the LHS },
\end{align}

\begin{align}
		\left\|
			\Lunit \Rad \Rad \upmu
		\right\|_{L^{\infty}(\Sigma_t^u)}
		& \leq
			\frac{1}{2}
			\left\|
				\Rad
				\Rad
				\left(
					G_{\Lunit \Lunit} \Rad \Psi
				\right)
			\right\|_{L^{\infty}(\Sigma_0^u)}
			+
			C \varepsilon
		\label{E:LUNITRADRADUPMULINFTY},
			\\
		\left\|
			\Rad \Rad \upmu
		\right\|_{L^{\infty}(\Sigma_t^u)}
		& \leq
			\left\|
				\Rad \Rad \upmu
			\right\|_{L^{\infty}(\Sigma_0^u)}
			+
			\TranminusdatasizeWithFactor^{-1}
			\left\|
				\Rad
				\Rad
				\left(
					G_{\Lunit \Lunit} \Rad \Psi
				\right)
			\right\|_{L^{\infty}(\Sigma_0^u)}
			+
			C \varepsilon.
		\label{E:RADRADUPMULINFTY}
	\end{align}
	\end{subequations}

	\noindent \underline{\textbf{$L^{\infty}$ estimates involving one or two transversal derivatives of $\Lunit_{(Small)}^i$}.}
	The following estimates hold:
	\begin{subequations}
	\begin{align} 
		\left\| 
		\Fullset_*^{[1,3];2} \Lunit_{(Small)}^i
	\right\|_{L^{\infty}(\Sigma_t^u)}
	& \leq C \varepsilon,
		\label{E:TANGENTIALANDTWORADAPPLIEDTOLUNITILINFINITY}
		\\
	\left\| 
		\Rad \Rad \Lunit_{(Small)}^i
	\right\|_{L^{\infty}(\Sigma_t^u)}
	& \leq 
	\left\| 
		\Rad \Rad \Lunit_{(Small)}^i
	\right\|_{L^{\infty}(\Sigma_0^u)}
	+
	C \varepsilon.
	\label{E:TWORADAPPLIEDTOLUNITILINFINITY}
	\end{align}
	\end{subequations}
	
	\noindent \underline{\textbf{$L^{\infty}$ estimates involving two transversal derivatives of $\bigslow$}}.
The following estimates hold: 
\begin{align}
 		\left\| 
			 \Fullset^{\leq 3;2} \bigslow
		\right\|_{L^{\infty}(\Sigma_t^u)}
			& \leq C \varepsilon.
		\label{E:UPTOTWOTRANSVERSALDERIVATIVESOFSLOWINLINFINITY} 
 \end{align}
	
	\noindent \underline{\textbf{Sharp pointwise estimates involving the critical factor $G_{\Lunit \Lunit}$}.}
	Moreover, if $0 \leq M \leq 2$
	and $0 \leq s \leq t < \Tboot$, 
	then we have the following estimates:
	\begin{align}
		\left|
			\Rad^M G_{\Lunit \Lunit}(t,u,\vartheta)
			-
			\Rad^M G_{\Lunit \Lunit}(s,u,\vartheta)
		\right|
		& \leq C \varepsilon (t - s),
			\label{E:RADDERIVATIVESOFGLLDIFFERENCEBOUND} \\
		\left|
			\Rad^M 
			\left\lbrace
				G_{\Lunit \Lunit}
				\Rad \Psi
			\right\rbrace
			(t,u,\vartheta)
			-
			\Rad^M 
			\left\lbrace
				G_{\Lunit \Lunit}
				\Rad \Psi
			\right\rbrace
			(s,u,\vartheta)
		\right|
		& \leq C \varepsilon (t - s).
		\label{E:RADDERIVATIVESOFGLLRADPSIDIFFERENCEBOUND}
	\end{align}

	Furthermore, with $\Lunit_{(Flat)} := \partial_t + \partial_1$,
	we have
	\begin{align} \label{E:LUNITUPMUDOESNOTDEVIATEMUCHFROMTHEDATA}
		\Lunit \upmu(t,u,\vartheta)
		& = \frac{1}{2} 
				\left\lbrace
					1 + \mathcal{O}_{\mydiam}(\Psiep)
				\right\rbrace	
				G_{\Lunit_{(Flat)} \Lunit_{(Flat)}}(\Psi = 0) \Rad \Psi(t,u,\vartheta)
			+ \mathcal{O}(\varepsilon),
	\end{align}
	where 
	$G_{\Lunit_{(Flat)} \Lunit_{(Flat)}}(\Psi = 0)$ is a non-zero constant
	(see \eqref{E:NONVANISHINGNONLINEARCOEFFICIENT}).
\end{proposition}

\begin{remark}[\textbf{Strict improvement of the auxiliary bootstrap assumptions}]
	\label{R:HIGHERTRANSAUXBOOTSTRAPIMPROVED}
	Note in particular that the estimates of
	Prop.\ \ref{P:IMPROVEMENTOFHIGHERTRANSVERSALBOOTSTRAP}
	yield strict improvements of the auxiliary bootstrap assumptions of
	Subsect.\ \ref{SS:BOOTSTRAPFORHIGHERTRANSVERSAL} whenever $\varepsilon$ 
	is sufficiently small. Hence, when proving estimates later in the paper,
	we no longer need to state them as assumptions.
\end{remark}

\begin{proof}[Proof of Prop.\ \ref{P:IMPROVEMENTOFHIGHERTRANSVERSALBOOTSTRAP}]
	See Subsect.\ \ref{SS:OFTENUSEDESTIMATES} for some comments on the analysis.
	Throughout this proof, we refer to the data-size assumptions of
	Subsect.\ \ref{SS:DATAASSUMPTIONS} and the bounds of Lemma~\ref{L:BEHAVIOROFEIKONALFUNCTIONQUANTITIESALONGSIGMA0}
	as the ``conditions on the data.''
	Moreover, we refer to the auxiliary bootstrap assumptions of
	Subsect.\ \ref{SS:BOOTSTRAPFORHIGHERTRANSVERSAL}
	simply as ``the bootstrap assumptions.''
		
\medskip 
\noindent \textbf{Proof of \eqref{E:UPTOTWOTRANSVERSALDERIVATIVESOFSLOWINLINFINITY}:}
We first note that by \eqref{E:SLOWWAVETRANSVERSALTANGENT}, it suffices
	to show that
	$
	\left\| 
		\Fullset^{\leq 3;2} \bigslow
	\right\|_{L^{\infty}(\Sigma_t^u)}
	\lesssim \varepsilon
	$
	whenever
	$\Fullset^{\leq 3;2}$ contains precisely two factors of $\Rad$.
	To proceed, we apply 
	$\Tanset^{\leq 1} \Rad$
	to the identity \eqref{E:RADOFSLOWWAVEALGEBRAICALLYEXPRESSED}.
	Using the $L^{\infty}$ estimates of Prop.~\ref{P:IMPROVEMENTOFAUX}
	and the bootstrap assumptions,
	we deduce that
	$
	\left\| 
		\Tanset^{\leq 1} \Rad \Rad \bigslow
	\right\|_{L^{\infty}(\Sigma_t^u)}
	\lesssim \varepsilon
	$.
	Then using the commutator estimate \eqref{E:TWORADIALTANGENTIALFUNCTIONCOMMUTATORESTIMATE} 
	with $f = \bigslow$ and the estimate \eqref{E:SLOWWAVETRANSVERSALTANGENT},
	we can arbitrarily permute the vectorfield factors
	in the expression $\Tanset^{\leq 1} \Rad \Rad \bigslow$
	up to error terms that are bounded in $\| \cdot \|_{L^{\infty}(\Sigma_t^u)}$ 
	by $\lesssim \varepsilon$,
	which yields the desired bound
	\eqref{E:UPTOTWOTRANSVERSALDERIVATIVESOFSLOWINLINFINITY}.
	
	\medskip
	\noindent \textbf{Proof of \eqref{E:IMPROVEDTRANSVERALESTIMATESFORTWORADBUTNOTPURERAD}-\eqref{E:IMPROVEDTRANSVERALESTIMATESFORTWORAD}:}
	By \eqref{E:PSIMIXEDTRANSVERSALTANGENTBOOTSTRAPIMPROVED}, it suffices to prove
	\eqref{E:IMPROVEDTRANSVERALESTIMATESFORTWORADBUTNOTPURERAD} when the operator
	$\Fullset_*^{[1,4];2}$ contains precisely two factors of $\Rad$. To proceed,
	we commute the wave equation \eqref{E:WAVEEQUATIONTRANSPORTINTERPRETATION}
	with $\Tanset^{\leq 2} \Rad$
	and use 
	Lemmas~\ref{L:POINTWISEFORRECTANGULARCOMPONENTSOFVECTORFIELDS}
	and
	\ref{L:POINTWISEESTIMATESFORGSPHEREANDITSDERIVATIVES},
	the commutator estimate \eqref{E:ANGLAPONERADIALTANGENTIALFUNCTIONCOMMUTATOR} with $f = \Psi$
	(to commute $\Tanset^{\leq 2} \Rad$ through $\angLap$),
	the commutator estimate \eqref{E:ONERADIALTANGENTIALFUNCTIONCOMMUTATORESTIMATE}
	with $f = \Rad \Psi$
	(to commute $\Tanset^{\leq 2} \Rad$ through the operator $\Lunit$ on LHS~\eqref{E:WAVEEQUATIONTRANSPORTINTERPRETATION}),
	the $L^{\infty}$ estimates of Prop.~\ref{P:IMPROVEMENTOFAUX},
	Cor.\ \ref{C:SQRTEPSILONTOCEPSILON},
	and the bootstrap assumptions
	to deduce
	\begin{align} \label{E:WAVEEQNONCETRANSVERSALLYCOMMUTEDTRANSPORTPOINTOFVIEW} 
	\left|
		\Lunit \Tanset^{\leq 2} \Rad \Rad \Psi
	\right|
		& \lesssim \varepsilon.
	\end{align}
	Since 
	$
	\displaystyle
	\Lunit = \frac{\partial}{\partial t}
	$, 
	from \eqref{E:WAVEEQNONCETRANSVERSALLYCOMMUTEDTRANSPORTPOINTOFVIEW},
	the fundamental theorem of calculus,
	and the conditions on the data,
	we deduce
	\eqref{E:IMPROVEDTRANSVERALESTIMATESFORTWORAD}
	as well as the bound 
	$|\Tanset^{[1,2]} \Rad \Rad \Psi| \lesssim \varepsilon$.
	Next, using the commutator estimate
	\eqref{E:TWORADIALTANGENTIALFUNCTIONCOMMUTATORESTIMATE},
	the bootstrap assumptions,
	and the $L^{\infty}$ estimates of Prop.~\ref{P:IMPROVEMENTOFAUX},
	we can reorder the vectorfield factors in the terms
	$\Singletan \Rad \Rad \Psi$
	up to error terms that are bounded
	in $\| \cdot \|_{L^{\infty}(\Sigma_t^u)}$ by $\lesssim \varepsilon$
	to deduce that
	$
	\left\| 
		\Fullset_*^{3;2} \Psi
	\right\|_{L^{\infty}(\Sigma_t^u)}
	\leq C \varepsilon
	$.
	Finally, using this bound,
	we can similarly reorder the vectorfield factors in the terms
	$\Tanset^2 \Rad \Rad \Psi$
	up to error terms that are bounded
	in $\| \cdot \|_{L^{\infty}(\Sigma_t^u)}$ by $\lesssim \varepsilon$,
	which in total yields \eqref{E:IMPROVEDTRANSVERALESTIMATESFORTWORADBUTNOTPURERAD}.
	
	\medskip
	\noindent \textbf{Proof of \eqref{E:LUNITUPMUDOESNOTDEVIATEMUCHFROMTHEDATA}:}
	First, we note that
	\eqref{E:PSIMIXEDTRANSVERSALTANGENTBOOTSTRAPIMPROVED}
	implies that
	$\| \Lunit \Rad \Psi \|_{L^{\infty}(\Sigma_t^u)} \lesssim \varepsilon$.
	Since 
	$
	\displaystyle
	\Lunit = \frac{\partial}{\partial t}
	$, 
	from this bound and the fundamental theorem of calculus,
	we deduce that
	$
	\Rad \Psi(t,u,\vartheta)
		= 
		\Rad \Psi(0,u,\vartheta)
	 + \mathcal{O}(\varepsilon)
	$.
	Similarly, from \eqref{E:LUNITAPPLIEDTOTANGENTIALUPMUANDTANSETSTARLINFTY}, 
	we deduce that
	$\Lunit \upmu(t,u,\vartheta)
		= 
		\Lunit \upmu(0,u,\vartheta)
	 + \mathcal{O}(\varepsilon)
	$.
	Next, we use \eqref{E:UPMUFIRSTTRANSPORT},
	the fact that 
	$G_{\Lunit \Lunit}, G_{\Lunit \Radunit} = \smoothfunction(\GdVar)$
	(see Lemma~\ref{L:SCHEMATICDEPENDENCEOFMANYTENSORFIELDS}),
	and the $L^{\infty}$ estimates of Prop.~\ref{P:IMPROVEMENTOFAUX}
	to deduce that 
	$\Lunit \upmu(0,u,\vartheta)
		= 
		\frac{1}{2}
	 [G_{\Lunit \Lunit} \Rad \Psi](0,u,\vartheta)
	 + \mathcal{O}(\varepsilon)
	$.
	Since 
	$\Lunit^0 = \Lunit_{(Flat)}^0 = 1$,
	$\Lunit^i = \Lunit_{(Flat)}^i + \Lunit_{(Small)}^i$,
	and $G_{\alpha \beta} = G_{\alpha \beta}(\Psi = 0) + \mathcal{O}_{\mydiam}(\Psi)$,
	we can use the conditions on the data to deduce,
	with the help of \eqref{E:NONVANISHINGNONLINEARCOEFFICIENT},
	the estimate
	$G_{\Lunit \Lunit}(0,u,\vartheta) 
		= 
		\left\lbrace
			1 + \mathcal{O}_{\mydiam}(\Psiep)
		\right\rbrace	
		G_{\Lunit_{(Flat)} \Lunit_{(Flat)}}(\Psi = 0)
		$.
	Combining this estimate with the previous ones
	and using \eqref{E:PSITRANSVERSALLINFINITYBOUNDBOOTSTRAPIMPROVED},
	we conclude \eqref{E:LUNITUPMUDOESNOTDEVIATEMUCHFROMTHEDATA}.
	
	\medskip
	\noindent \textbf{Proof of 
	\eqref{E:RADDERIVATIVESOFGLLDIFFERENCEBOUND}-\eqref{E:RADDERIVATIVESOFGLLRADPSIDIFFERENCEBOUND}
	in the cases $0 \leq M \leq 1$:}
	It suffices to prove that for $0 \leq M \leq 1$, we have
	\begin{align} \label{E:LDERIVATIVEOFRADIALDERIVATIVESOFCRITICALFACTOR}
		\left|
			\Lunit \Rad^M G_{\Lunit \Lunit}
		\right|
		& \lesssim \varepsilon,
		&&
		\left|
			\Lunit \Rad^M 
			\left\lbrace
				G_{\Lunit \Lunit}
				\Rad \Psi
			\right\rbrace
		\right|
		\lesssim \varepsilon.
	\end{align}
	Once we have shown \eqref{E:LDERIVATIVEOFRADIALDERIVATIVESOFCRITICALFACTOR},
	we can use the fact that 
	$
	\displaystyle
	\Lunit = \frac{\partial}{\partial t}
	$
	to obtain the desired estimates by integrating 
	from time $s$ to $t$
	and using the estimates \eqref{E:LDERIVATIVEOFRADIALDERIVATIVESOFCRITICALFACTOR}.
	To proceed, we first use Lemma~\ref{L:SCHEMATICDEPENDENCEOFMANYTENSORFIELDS} to deduce that
	$G_{\Lunit \Lunit} = \smoothfunction(\GdVar)$ 
	and $G_{\Lunit \Lunit} \Rad \Psi = \smoothfunction(\GdVar) \Rad \Psi$.
	Hence, to obtain \eqref{E:LDERIVATIVEOFRADIALDERIVATIVESOFCRITICALFACTOR}
	when $M=0$, we differentiate these two identities with $\Lunit$ and 
	use the $L^{\infty}$ estimates of Prop.~\ref{P:IMPROVEMENTOFAUX}
	and the bootstrap assumptions.
	The proof is similar in the case $M=1$,
	but we must also use the estimate
	$\left\|\Lunit \Rad \Rad \Psi \right\|_{L^{\infty}(\Sigma_t^u)} \lesssim \varepsilon$,
	which is a consequence of the previously established estimate 
	\eqref{E:IMPROVEDTRANSVERALESTIMATESFORTWORADBUTNOTPURERAD}.
	
	\medskip
	\noindent \textbf{Proof of \eqref{E:LUNITRADUPMULINFTY}-\eqref{E:PERMUTEDRADTANGENTIALUPMULINFTY}:}
	Let $1 \leq K \leq 3$ be an integer and let $\Fullset^{K;1}$ be an operator containing exactly one factor of $\Rad$.
	We commute equation \eqref{E:UPMUFIRSTTRANSPORT}
	with $\Fullset^{K;1}$ and
	use the aforementioned relations $G_{\Lunit \Lunit}, G_{\Lunit \Radunit} = \smoothfunction(\GdVar)$,
	the $L^{\infty}$ estimates of Prop.~\ref{P:IMPROVEMENTOFAUX},
	and the bootstrap assumptions to deduce
	\begin{align} \label{E:UPMUEVOLUTIONEQUATIONRADANDTANGENTIALCOMMUTEDFIRSTBOUND}
		\left|
			\Lunit \Fullset^{K;1} \upmu
		\right|
		& \leq 
			\frac{1}{2}
			\left|
				\Fullset^{K;1}
				\left\lbrace
					G_{\Lunit \Lunit}
					\Rad \Psi
				\right\rbrace
			\right|
			+
			\left|
				\Fullset_*^{[1,K+1];1} \Psi
			\right|
			+
			\left|
				[\Lunit, \Fullset^{K;1}] \upmu
			\right|.
		\end{align}
		We now show that the last two terms on RHS~\eqref{E:UPMUEVOLUTIONEQUATIONRADANDTANGENTIALCOMMUTEDFIRSTBOUND}
		are $\lesssim \varepsilon$.
		We already proved 
		$
		\left|
			\Fullset_*^{[1,K+1];1} \Psi
		\right|
		\lesssim \varepsilon
		$
		in Prop.~\ref{P:IMPROVEMENTOFAUX}.
		To bound $[\Lunit, \Fullset^{K;1}] \upmu$,
		we use the commutator estimate \eqref{E:ONERADIALTANGENTIALFUNCTIONCOMMUTATORESTIMATE}
		with $f = \upmu$,
		the $L^{\infty}$ estimates of Prop.~\ref{P:IMPROVEMENTOFAUX},
		and Cor.\ \ref{C:SQRTEPSILONTOCEPSILON}
		to deduce that
		$
		\left|
			[\Lunit, \Fullset^{K;1}] \upmu
		\right|
		\lesssim 
		\left|
			\Tanset_*^{[1, K]} \upmu
		\right|
		+
		\varepsilon
		\left|
			\GeoAng \Fullset^{\leq K-1;1} \upmu
		\right|
		$.
		The $L^{\infty}$ estimates of Prop.~\ref{P:IMPROVEMENTOFAUX} imply that 
		$
		\left|
			\Tanset_*^{[1,K]} \upmu
		\right|
		\lesssim \varepsilon
		$,
		while the bootstrap assumptions
		imply that
		$
		\varepsilon
		\left|
			\GeoAng \Fullset^{\leq K-1;1} \upmu
		\right|
		\lesssim \varepsilon
		$
		as well. We have thus shown that
		\begin{align} \label{E:UPMUEVOLUTIONEQUATIONRADANDTANGENTIALCOMMUTEDSECONDBOUND}
		\left\|
			\Lunit \Fullset^{K;1} \upmu
		\right\|_{L^{\infty}(\Sigma_t^u)}
		& \leq 
			\frac{1}{2}
			\left\|
				\Fullset^{K;1}
				\left\lbrace
					G_{\Lunit \Lunit}
					\Rad \Psi
				\right\rbrace
			\right\|_{L^{\infty}(\Sigma_t^u)}
			+ 
			C \varepsilon.
		\end{align}
		We split the remainder of the proof into two cases, starting with the case $\Fullset^{K;1} = \Rad$.
		Using the bound \eqref{E:RADDERIVATIVESOFGLLRADPSIDIFFERENCEBOUND} 
		with $s=0$ and $M=1$ (established above),
		we can replace the norm 
		$\| \cdot \|_{L^{\infty}(\Sigma_t^u)}$
		on RHS~\eqref{E:UPMUEVOLUTIONEQUATIONRADANDTANGENTIALCOMMUTEDSECONDBOUND}
		with the norm $\| \cdot \|_{L^{\infty}(\Sigma_0^u)}$
		plus an error term that is 
		bounded in the norm $\| \cdot \|_{L^{\infty}(\Sigma_t^u)}$
		by $\leq C \varepsilon$,
		which yields \eqref{E:LUNITRADUPMULINFTY}.
		Using \eqref{E:LUNITRADUPMULINFTY},
		recalling that 
		$
		\displaystyle
		\Lunit = \frac{\partial}{\partial t}
		$,
		and using the fundamental theorem of calculus
		as well as the assumption
		$\Tboot \leq 2 \TranminusdatasizeWithFactor^{-1}$,
		we conclude \eqref{E:RADUPMULINFTY}.
		In the remaining case, $\Fullset^{K;1}$ is not the operator $\Rad$.
		That is, $2 \leq K \leq 3$ and $\Fullset^{K;1}$ must contain a $\mathcal{P}_u$-tangent factor,
		which is equivalent to $\Fullset^{K;1} = \Fullset_*^{K;1}$.
		Recalling that
		$G_{\Lunit \Lunit} \Rad \Psi = \smoothfunction(\GdVar) \Rad \Psi$
		and using the estimates of Prop.~\ref{P:IMPROVEMENTOFAUX},
		the bootstrap assumptions,
		and
		\eqref{E:IMPROVEDTRANSVERALESTIMATESFORTWORADBUTNOTPURERAD},
		we find that
		$
		\left\|
			\Fullset_*^{K;1}
		\left\lbrace
					G_{\Lunit \Lunit}
					\Rad \Psi
				\right\rbrace
			\right\|_{L^{\infty}(\Sigma_t^u)}
			\lesssim \varepsilon
		$.
		Thus, in this case, we have shown that 
		$
		\mbox{RHS~\eqref{E:UPMUEVOLUTIONEQUATIONRADANDTANGENTIALCOMMUTEDSECONDBOUND}}
		\lesssim
		\varepsilon
		$.
		From the estimate
		$
		\left\|
			\Lunit \Fullset_*^{K;1} \upmu
		\right\|_{L^{\infty}(\Sigma_t^u)}
		\lesssim 
		\varepsilon
		$,
		the fact that 
		$
		\displaystyle
		\Lunit = \frac{\partial}{\partial t}
		$, 
		and the fundamental theorem of calculus,
		we conclude that
		$
		\left\|
			\Fullset_*^{K;1} \upmu
		\right\|_{L^{\infty}(\Sigma_t^u)}
		\leq
		\left\|
			\Fullset_*^{K;1} \upmu
		\right\|_{L^{\infty}(\Sigma_0^u)}
		+ C \varepsilon
		$.
		The bounds \eqref{E:LUNITRADTANGENTIALUPMULINFTY}-\eqref{E:RADTANGENTIALUPMULINFTY}
		now follow from the previous estimates and the conditions on the data.
		It remains for us to prove the estimate \eqref{E:PERMUTEDRADTANGENTIALUPMULINFTY} concerning
		the permutations of the vectorfields in \eqref{E:LUNITRADTANGENTIALUPMULINFTY}-\eqref{E:RADTANGENTIALUPMULINFTY}.
		To obtain the desired bound, we use the commutator estimate
		\eqref{E:ONERADIALTANGENTIALFUNCTIONCOMMUTATORESTIMATE}
		with $f = \upmu$,
		the $L^{\infty}$ estimates of Prop.~\ref{P:IMPROVEMENTOFAUX},
		the estimates \eqref{E:LUNITRADTANGENTIALUPMULINFTY}-\eqref{E:RADTANGENTIALUPMULINFTY},
		and the bootstrap assumptions.
	
	\medskip
\noindent \textbf{Proof of \eqref{E:TANGENTIALANDTWORADAPPLIEDTOLUNITILINFINITY} and \eqref{E:TWORADAPPLIEDTOLUNITILINFINITY}:}
We may assume that the operator $\Fullset_*^{[1,3];2}$ in \eqref{E:TANGENTIALANDTWORADAPPLIEDTOLUNITILINFINITY}
contains two factors of $\Rad$ since otherwise the desired estimate is implied by \eqref{E:LUNITAPPLIEDTOLISMALLANDLISMALLINFTYESTIMATE}.
To proceed, we first use Lemma \ref{L:SCHEMATICDEPENDENCEOFMANYTENSORFIELDS} 
to express \eqref{E:RADLUNITI} in the schematic form
$
\Rad \Lunit_{(Small)}^i 
= 
\smoothfunction(\GdVar,\ginversesphere,\angdiff x^1,\angdiff x^2) \Rad \Psi 
+ 
\smoothfunction(\BadVar,\ginversesphere,\angdiff x^1,\angdiff x^2) \Singletan \Psi
+ \smoothfunction(\ginversesphere,\angdiff x^1,\angdiff x^2) \angdiff \upmu
$.
We now apply $\Singletan \Rad$ to this identity, where $\Singletan \in \Tanset$.
Using Lemmas~\ref{L:POINTWISEFORRECTANGULARCOMPONENTSOFVECTORFIELDS}
and
\ref{L:POINTWISEESTIMATESFORGSPHEREANDITSDERIVATIVES},
the $L^{\infty}$ estimates of Prop.~\ref{P:IMPROVEMENTOFAUX},
the already proven estimates 
\eqref{E:IMPROVEDTRANSVERALESTIMATESFORTWORADBUTNOTPURERAD}
and
\eqref{E:LUNITRADTANGENTIALUPMULINFTY}-\eqref{E:PERMUTEDRADTANGENTIALUPMULINFTY},
and the bootstrap assumptions, 
we deduce that
\begin{align} \label{E:TANGENTIALRADRADLUNITIFIRSTBOUND}
	\left|
		\Singletan \Rad \Rad \Lunit_{(Small)}^i
	\right|
	& \lesssim
		\left|
			\Fullset_*^{[1,3];2} \Psi
		\right|
		+
		\left|
			\Fullset_*^{[1,3];1} \GdVar
		\right|
		+
		\left|
			\GeoAng \Fullset^{\leq 2;1} \upmu
		\right|
		+
		\left|
			\Tanset_*^{[1,2]} \upmu
		\right|
		\lesssim \varepsilon.
\end{align}
Also using the commutator estimate \eqref{E:TWORADIALTANGENTIALFUNCTIONCOMMUTATORESTIMATE} 
with $f = \Lunit_{(Small)}^i$ 
and the $L^{\infty}$ estimates of Prop.~\ref{P:IMPROVEMENTOFAUX},
we can arbitrarily reorder the vectorfield factors in the expression $\Singletan \Rad \Rad \Lunit_{(Small)}^i$
up to error terms bounded in the norm 
$\| \cdot \|_{L^{\infty}(\Sigma_t^u)}$ by $\lesssim \varepsilon$,
which yields \eqref{E:TANGENTIALANDTWORADAPPLIEDTOLUNITILINFINITY}.
Moreover, a special case of \eqref{E:TANGENTIALANDTWORADAPPLIEDTOLUNITILINFINITY}
is the bound $\left|\Lunit \Rad \Rad \Lunit_{(Small)}^i \right| \lesssim \varepsilon$.
From this estimate,
the fact that 
$
\displaystyle
\Lunit = \frac{\partial}{\partial t}
$,
and the fundamental theorem of calculus,
we conclude \eqref{E:TWORADAPPLIEDTOLUNITILINFINITY}.

\medskip
\noindent \textbf{Proof of \eqref{E:IMPROVEDTRANSVERALESTIMATESFORLUNITTHREERAD} and \eqref{E:IMPROVEDTRANSVERALESTIMATESFORTHREERAD}:}
	As a preliminary step, we establish the bounds
	$|\angLie_{\Rad \Rad} \ginversesphere| \lesssim 1$
	and
	$|\angLie_{\Rad \Rad} \angdiff x^i| \lesssim 1$.	
	To handle the terms $\angLie_{\Rad \Rad} \angdiff x^i = \angdiff \Rad \Rad x^i$,
	we first note that Lemma~\ref{L:SCHEMATICDEPENDENCEOFMANYTENSORFIELDS} yields
	$\Rad x^i = \Rad^i = \smoothfunction(\BadVar)$.
	Thus, from the $L^{\infty}$ estimates of Prop.~\ref{P:IMPROVEMENTOFAUX}
	and the bootstrap assumptions,
	we obtain $|\angdiff \Rad \Rad x^i| \lesssim |\Fullset^{\leq 2;1} \BadVar| \lesssim 1$ as desired.
	To handle the terms 
	$\angLie_{\Rad \Rad} \ginversesphere$,
	we rely on the basic identity
	$\angLie_{\Rad} \ginversesphere = - (\angLie_{\Rad} \gsphere)^{\# \#}$,
	which was proved in \cite{jSgHjLwW2016}*{Lemma~2.9}.
	From this identity,
	Lemma~\ref{L:POINTWISEESTIMATESFORGSPHEREANDITSDERIVATIVES},
	and the $L^{\infty}$ estimates of Prop.~\ref{P:IMPROVEMENTOFAUX},
	we deduce that
	$|\angLie_{\Rad \Rad} \ginversesphere| \lesssim |\angLie_{\Rad \Rad} \gsphere| + 1$.
	Moreover, Lemma~\ref{L:SCHEMATICDEPENDENCEOFMANYTENSORFIELDS}
	yields that $\gsphere = \smoothfunction(\GdVar,\angdiff x^1, \angdiff x^2)$.
	Thus, from 
	Lemma~\ref{L:POINTWISEFORRECTANGULARCOMPONENTSOFVECTORFIELDS},
	the $L^{\infty}$ estimates of Prop.~\ref{P:IMPROVEMENTOFAUX},
	the bootstrap assumptions,
	and the bound $|\angdiff \Rad \Rad x^i| \lesssim 1$ proved above,
	we conclude the desired bound 
	$|\angLie_{\Rad \Rad} \ginversesphere| \lesssim 1$.
	
	We now commute equation \eqref{E:WAVEEQUATIONTRANSPORTINTERPRETATION} with $\Rad \Rad$
	and use 
	Lemmas~\ref{L:POINTWISEFORRECTANGULARCOMPONENTSOFVECTORFIELDS}
	and
	\ref{L:POINTWISEESTIMATESFORGSPHEREANDITSDERIVATIVES},
	the $L^{\infty}$ estimates of Prop.~\ref{P:IMPROVEMENTOFAUX},
	the bootstrap assumptions,
	and the bounds
	$\left|
		\angLie_{\Rad} \angLie_{\Rad} \ginversesphere
	\right|
	\lesssim 1
	$ 
	and
	$
	\left|
		\angLie_{\Rad} \angLie_{\Rad} \angdiff x
	\right|
	\lesssim 1
	$
	proved above to deduce that
	\begin{align} \label{E:WAVEEQNTWICETRANSVERSALLYCOMMUTEDTRANSPORTPOINTOFVIEW}
		\left|
			\Lunit \Rad \Rad \Rad \Psi
		\right|
		& \lesssim
				\left|
					\Fullset_*^{[1,4];2} \Psi
				\right|
				+
				\left|
					\Fullset_*^{[1,3];2} \GdVar
				\right|
				+
				\left|
					\Rad^{\leq 2} \bigslow
				\right|
					\\
		& \ \
			+ \left|
					[\angLap, \Rad \Rad] \Psi
				\right|
			+ 
		\left|
			\Lunit \Rad \Rad \Rad \Psi
			- \Rad \Rad \Lunit \Rad \Psi
		\right|.
		\notag
	\end{align}
	Next, we note that the already proven estimates
	\eqref{E:IMPROVEDTRANSVERALESTIMATESFORTWORADBUTNOTPURERAD},
	\eqref{E:TANGENTIALANDTWORADAPPLIEDTOLUNITILINFINITY},
	and \eqref{E:UPTOTWOTRANSVERSALDERIVATIVESOFSLOWINLINFINITY} 
	imply that 
	$
	\left|
		\Fullset_*^{[1,4];2} \Psi
	\right|
	\lesssim \varepsilon
	$,
	$
	\left|
		\Fullset_*^{[1,3];2} \GdVar
	\right|
	\lesssim \varepsilon
	$,
	and
	$
	\left|
		\Rad^{\leq 2} \bigslow
	\right|
	\lesssim \varepsilon
	$.
	Next, we use \eqref{E:ANGLAPTWORADIALTANGENTIALFUNCTIONCOMMUTATOR}
	with $f = \Psi$
	to bound the commutator term
	$\left|
		[\angLap, \Rad \Rad] \Psi
	\right|$
	by $\lesssim$ the first term on RHS~\eqref{E:WAVEEQNTWICETRANSVERSALLYCOMMUTEDTRANSPORTPOINTOFVIEW}
	(and hence it is $\lesssim \varepsilon$ too).
	Next, we use
 	\eqref{E:TWORADIALTANGENTIALFUNCTIONCOMMUTATORESTIMATE} 
 	with $f=\Rad \Psi$ and $N=2$ 
	and the bound
	$
	\left|
		\Fullset_*^{[1,4];2} \Psi
	\right|
	\lesssim \varepsilon
	$
	mentioned above
	to deduce that
 	$
 	\left|
		\Lunit \Rad \Rad \Rad \Psi
		- \Rad \Rad \Lunit \Rad \Psi
	\right|
	\lesssim
	\left|
		\Fullset_*^{[1,3];2} \Psi
	\right|
	\lesssim \varepsilon
	$.
 Combining these estimates,
 we deduce that
	$
	\left|
		\Lunit \Rad \Rad \Rad \Psi
	\right|
	\lesssim \varepsilon
	$,
	which implies \eqref{E:IMPROVEDTRANSVERALESTIMATESFORLUNITTHREERAD}.
	From \eqref{E:IMPROVEDTRANSVERALESTIMATESFORLUNITTHREERAD},
	the fact that 
	$
	\displaystyle
	\Lunit = \frac{\partial}{\partial t}
	$,
	and the fundamental theorem of calculus,
	we conclude the desired estimate \eqref{E:IMPROVEDTRANSVERALESTIMATESFORTHREERAD}.

	\medskip
\noindent \textbf{Proof of \eqref{E:RADDERIVATIVESOFGLLDIFFERENCEBOUND}-\eqref{E:RADDERIVATIVESOFGLLRADPSIDIFFERENCEBOUND}
	in the case $M=2$:}
	The proof is very similar to the proof given above in the cases
	$M=0,1$, so we only highlight the main new ingredients needed in the case $M=2$:
	we must also use the estimates
	$
	\left|
		\Lunit \Rad \Rad \Rad \Psi
	\right|
	\lesssim \varepsilon
	$
	and
	$\left|
		\Lunit \Rad \Rad \Lunit_{(Small)}^i 
	\right| \lesssim \varepsilon
	$
	established in 
	\eqref{E:IMPROVEDTRANSVERALESTIMATESFORLUNITTHREERAD}
	and 
	\eqref{E:TANGENTIALANDTWORADAPPLIEDTOLUNITILINFINITY}
	in order to deduce
	\eqref{E:LDERIVATIVEOFRADIALDERIVATIVESOFCRITICALFACTOR}
	in the case $M=2$.

\medskip
	\noindent \textbf{Proof of \eqref{E:LUNITRADRADUPMULINFTY}-\eqref{E:RADRADUPMULINFTY}:}
	We commute equation \eqref{E:UPMUFIRSTTRANSPORT} with $\Rad \Rad$ 
	and argue as in the proof of \eqref{E:UPMUEVOLUTIONEQUATIONRADANDTANGENTIALCOMMUTEDFIRSTBOUND}
	to obtain
		\begin{align} \label{E:UPMUEVOLUTIONEQUATIONRADDOUBLECOMMUTEDFIRSTBOUND}
		\left|
			\Lunit \Rad \Rad \upmu
		\right|
		& \leq 
			\frac{1}{2}
			\left|
				\Rad \Rad
				\left\lbrace
					G_{\Lunit \Lunit}
					\Rad \Psi
				\right\rbrace
			\right|
			+
			\left|
				\Fullset_*^{[1,3];2} \Psi
			\right|
			+
			\left|
				\Lunit \Rad \Rad \upmu
				- 
				\Rad \Rad \Lunit \upmu
			\right|.
		\end{align}
		Using the commutator estimate
		\eqref{E:TWORADIALTANGENTIALFUNCTIONCOMMUTATORESTIMATE}
		with $f = \upmu$,
		the $L^{\infty}$ estimates of Prop.~\ref{P:IMPROVEMENTOFAUX}, 
		and the already proven bounds \eqref{E:RADTANGENTIALUPMULINFTY}-\eqref{E:PERMUTEDRADTANGENTIALUPMULINFTY},
		we deduce that 
		$\left|
				\Lunit \Rad \Rad \upmu
				- 
				\Rad \Rad \Lunit \upmu
		\right|
		\lesssim
			\left|
				\GeoAng \Fullset^{\leq 1} \upmu
			\right|
		\lesssim 
		\varepsilon
		$.
		Next, we use
		\eqref{E:IMPROVEDTRANSVERALESTIMATESFORTWORADBUTNOTPURERAD}
		to deduce that 
		$
		\left|
			\Fullset_*^{[1,3];2} \Psi
		\right|
		\lesssim \varepsilon
		$.
		Thus, we have shown that the last two terms 
		on RHS~\eqref{E:UPMUEVOLUTIONEQUATIONRADDOUBLECOMMUTEDFIRSTBOUND} are 
		$\lesssim \varepsilon$. 
		The remainder of the proof of \eqref{E:LUNITRADRADUPMULINFTY}-\eqref{E:RADRADUPMULINFTY}
		now proceeds as in the proof
		of \eqref{E:LUNITRADUPMULINFTY}-\eqref{E:RADUPMULINFTY},
		thanks to the availability of the already proven estimates
		\eqref{E:RADDERIVATIVESOFGLLDIFFERENCEBOUND}-\eqref{E:RADDERIVATIVESOFGLLRADPSIDIFFERENCEBOUND}
		in the case $M=2$.

\end{proof}

\section{Sharp estimates for \texorpdfstring{$\upmu$}{the inverse foliation density}}
\label{S:SHARPESTIMATESFORINVERSEFOLIATIONDENSITY}
In this section, we derive sharp estimates for $\upmu$, its derivatives,
and various time integrals, 
many of which involve the singular factor 
$
\displaystyle
\frac{1}{\upmu}
$.
These estimates play a fundamental role in our energy estimates because
our energies contain $\upmu$ weights and because in our energy identities,
we will encounter error integrals that involve
the derivatives of $\upmu$ and/or
factors of 
$
\displaystyle
\frac{1}{\upmu}
$.
The main results of this section are 
Props.~\ref{P:SHARPMU}
and Prop.~\ref{P:MUINVERSEINTEGRALESTIMATES}.

\subsection{Sharp \texorpdfstring{$L^{\infty}$}{essential sup-norm} estimates and pointwise estimates for 
\texorpdfstring{$\upmu$}{the inverse foliation density}}
\label{SS:MUSHARPSUPNORM}
We define the following quantities in order to
facilitate our analysis of $\upmu$.

\begin{definition}[\textbf{Auxiliary quantities used to analyze $\upmu$}]
	\label{D:AUXQUANTITIES}
	We define the following quantities, 
	where $0 \leq s \leq t$:
	\begin{subequations}
	\begin{align}
	M(s,u,\vartheta;t) 
	& := \int_{s'=s}^{s'=t} 
					\left\lbrace
						\Lunit \upmu(t,u,\vartheta) - \Lunit \upmu(s',u,\vartheta) 
					\right\rbrace
				\, ds',
		\label{E:BIGMDEF} \\
	\mathring{\upmu}(u,\vartheta)
	& := \upmu(s=0,u,\vartheta),
		\\
	\widetilde{M}(s,u,\vartheta;t)
	& := \frac{M(s,u,\vartheta;t)}{\mathring{\upmu}(u,\vartheta) 
		- M(0,u,\vartheta;t)},
		\label{E:WIDETILDEBIGMDEF} \\
	\upmu_{(Approx)}(s,u,\vartheta;t)
		& := 1
			+  \frac{\Lunit \upmu(t,u,\vartheta)}{
				\mathring{\upmu}(u,\vartheta) 
				- M(0,u,\vartheta;t)}s
			+ \widetilde{M}(s,u,\vartheta;t).
		\label{E:MUAPPROXDEF}
	\end{align}
	\end{subequations}
\end{definition}

As we outlined in Subsubsect.\ \ref{SSS:ENERGYESTIMATES},
our high-order energies are allowed
to blow up as the shock forms.
Specifically, 
the best estimates 
that we are able to derive
allow for the possibility that
the high-order energies blow up like
negative powers of the quantity $\upmu_{\star}$, which we now define;
see Prop.~\ref{P:MAINAPRIORIENERGY} for the detailed statement.

\begin{definition}[\textbf{Definition of} $\upmu_{\star}$]
	\label{D:MUSTARDEF}
	\begin{align} \label{E:MUSTARDEF}
		\upmu_{\star}(t,u)
		& := \min \lbrace 1, \min_{\Sigma_t^u} \upmu \rbrace.
	\end{align}
\end{definition}

The following simple estimates play a role in our ensuing analysis.

\begin{lemma}[\textbf{First estimates for the auxiliary quantities}]
\label{L:FIRSTESTIMATESFORAUXILIARYUPMUQUANTITIES}
The following  
estimates hold for $(t,u,\vartheta) \in [0,\Tboot) \times [0,U_0] \times \mathbb{T}$
and $0 \leq s \leq t$
(see Subsect.\ \ref{SS:NOTATIONANDINDEXCONVENTIONS} regarding our use of the notation $\mathcal{O}_{\mydiam}(\cdot)$):
\begin{align}
	\mathring{\upmu}(u,\vartheta)
	& = 1 + \mathcal{O}_{\mydiam} (\Psiep),
		\label{E:MUINITIALDATAESTIMATE}
		\\
	\mathring{\upmu}(u,\vartheta)
	& = 1 + M(0,u,\vartheta;t) + \mathcal{O}_{\mydiam} (\Psiep) + \mathcal{O}(\varepsilon).
		\label{E:MUAMPLITUDENEARONE}
\end{align}
In addition, the following pointwise estimates hold:
\begin{align} 
	\left|
		\Lunit \upmu(t,u,\vartheta) 
		- 
		\Lunit \upmu(s,u,\vartheta)
	\right|
	& \lesssim \varepsilon(t-s),
		\label{E:LUNITUPMUATTIMETMINUSLUNITUPMUATTIMESPOINTWISEESTIMATE} \\
	|M(s,u,\vartheta;t)|, |\widetilde{M}(s,u,\vartheta;t)|
	& \lesssim 
		\varepsilon (t - s)^2,
		\label{E:BIGMEST} 
		\\
	\upmu(s,u,\vartheta)
	& = \left\lbrace
				1 + \mathcal{O}_{\mydiam} (\Psiep) + \mathcal{O}(\varepsilon)
			\right\rbrace
	\upmu_{(Approx)}(s,u,\vartheta;t).
	\label{E:MUAPPROXMISLIKEMU}
\end{align}
\end{lemma}

\begin{proof}
	\eqref{E:MUINITIALDATAESTIMATE} is a restatement of \eqref{E:UPITSELFLINFINITYSIGMA0CONSEQUENCES}.
	The estimate \eqref{E:LUNITUPMUATTIMETMINUSLUNITUPMUATTIMESPOINTWISEESTIMATE}
	follows from the mean value theorem and
	the estimate $\left| \Lunit \Lunit \upmu \right| \lesssim \varepsilon$,
	which is a special case of \eqref{E:LUNITAPPLIEDTOTANGENTIALUPMUANDTANSETSTARLINFTY}.
	The estimate \eqref{E:MUAMPLITUDENEARONE}
	and the estimate \eqref{E:BIGMEST} for $M$
	then follow from definition \eqref{E:BIGMDEF} 
	and the estimates 
	\eqref{E:MUINITIALDATAESTIMATE}
	and
	\eqref{E:LUNITUPMUATTIMETMINUSLUNITUPMUATTIMESPOINTWISEESTIMATE}.
	The estimate \eqref{E:BIGMEST} for $\widetilde{M}$
	follows from definition \eqref{E:WIDETILDEBIGMDEF}, 
	the estimate \eqref{E:BIGMEST} for $M$,
	and \eqref{E:MUAMPLITUDENEARONE}.
To prove \eqref{E:MUAPPROXMISLIKEMU}, we first note the following identity, which is
a straightforward consequence of Def.\ \ref{D:AUXQUANTITIES}:
\begin{align} \label{E:MUSPLIT}
\upmu(s,u,\vartheta) 
	= \left\lbrace
				\mathring{\upmu}(u,\vartheta) - M(0,u,\vartheta;t)
			\right\rbrace
			\upmu_{(Approx)}(s,u,\vartheta;t).
\end{align}
From \eqref{E:MUSPLIT} and \eqref{E:MUAMPLITUDENEARONE},
we conclude \eqref{E:MUAPPROXMISLIKEMU}.
\end{proof}

To derive some of the most important estimates, 
we will distinguish between regions in which $\upmu$ is appreciably shrinking 
and regions in which it is not. We define the relevant regions in the next definition.

\begin{definition}[\textbf{Regions of distinct $\upmu$ behavior}]
\label{D:REGIONSOFDISTINCTUPMUBEHAVIOR}
For each 
$t \in [0,\Tboot)$,
$s \in [0,t]$, 
and $u \in [0,U_0]$, 
we partition 
\begin{subequations}
\begin{align}
	[0,u] \times \mathbb{T} 
	& = \Vplus{t}{u} \cup \Vminus{t}{u},
		\label{E:OUINTERVALCROSSS2SPLIT} \\
	\Sigma_s^u
	& = \Sigmaplus{s}{t}{u} \cup \Sigmaminus{s}{t}{u},
	\label{E:SIGMASSPLIT}
\end{align}
\end{subequations}
where
\begin{subequations}
\begin{align}
	\Vplus{t}{u}
	& := 
	\left\lbrace
		(u',\vartheta) \in [0,u] \times \mathbb{T} \ | \  
			\frac{\Lunit \upmu(t,u',\vartheta)}{\mathring{\upmu}(u',\vartheta) - M(0,u',\vartheta;t)}
		\geq 0
	\right\rbrace,
		\label{E:ANGLESANDUWITHNONDECAYINUPMUGBEHAVIOR} \\
	\Vminus{t}{u}
	& := 
	\left\lbrace
		(u',\vartheta) \in [0,u] \times \mathbb{T} \ | \ 
			\frac{\Lunit \upmu(t,u',\vartheta)}{\mathring{\upmu}(u',\vartheta) - M(0,u',\vartheta;t)} < 0
	\right\rbrace,
		\label{E:ANGLESANDUWITHDECAYINUPMUGBEHAVIOR} \\
	\Sigmaplus{s}{t}{u}
	& := 
	\left\lbrace
		(s,u',\vartheta) \in \Sigma_s^u \ | \ (u',\vartheta) \in \Vplus{t}{u}
	\right\rbrace,
		\label{E:SIGMAPLUS} \\
	\Sigmaminus{s}{t}{u}
	& := 
	\left\lbrace
		(s,u',\vartheta) \in \Sigma_s^u \ | \ (u',\vartheta) \in \Vminus{t}{u}
	\right\rbrace.
	\label{E:SIGMAMINUS}
\end{align}
\end{subequations}
\end{definition}

\begin{remark}
	Note that by \eqref{E:MUAMPLITUDENEARONE}, the denominators in 
	\eqref{E:ANGLESANDUWITHNONDECAYINUPMUGBEHAVIOR}-\eqref{E:ANGLESANDUWITHDECAYINUPMUGBEHAVIOR}
	are positive.
\end{remark}

The following proposition provides our main sharp estimates for
$\upmu$ and its derivatives. The estimates play a fundamental role
in controlling the error integrals in our energy estimates.

\begin{proposition}[\textbf{Sharp pointwise estimates for $\upmu$, $\Lunit \upmu$, and $\Rad \upmu$}]
\label{P:SHARPMU} 
The following  
estimates hold for $(t,u,\vartheta) \in [0,\Tboot) \times [0,U_0] \times \mathbb{T}$
and $0 \leq s \leq t$.
\medskip

\noindent \underline{\textbf{Upper bound for $\displaystyle \frac{[\Lunit \upmu]_+}{\upmu}$}}.
\begin{align} \label{E:POSITIVEPARTOFLMUOVERMUISBOUNDED}
	\left\|
		\frac{[\Lunit \upmu]_+}{\upmu}
	\right\|_{L^{\infty}(\Sigma_s^u)}
	& \leq C.
\end{align}

\medskip

\noindent \underline{\textbf{Small $\upmu$ implies $\Lunit \upmu$ is quantitatively negative}}.
\begin{align} \label{E:SMALLMUIMPLIESLMUISNEGATIVE}
	\upmu(s,u,\vartheta) \leq \frac{1}{4}
	\implies
	\Lunit \upmu(s,u,\vartheta) \leq - \frac{1}{4} \TranminusdatasizeWithFactor,
\end{align}
where $\TranminusdatasizeWithFactor > 0$ is defined in \eqref{E:CRITICALBLOWUPTIMEFACTOR}.

\medskip

\noindent \underline{\textbf{Upper bound for 
$\displaystyle \frac{[\Rad \upmu]_+}{\upmu}$}}.
\begin{align} \label{E:UNIFORMBOUNDFORMRADMUOVERMU}
	\left\|
		\frac{[\Rad \upmu]_+}{\upmu}
	\right\|_{L^{\infty}(\Sigma_s^u)}
	& \leq 
		\frac{C}{\sqrt{\Tboot - s}}.
\end{align}

\medskip

\noindent \underline{\textbf{Sharp spatially uniform estimates}}.
Consider a time interval $s \in [0,t]$ and consider 
the ($t,u$-dependent) constant $\LateTimeLUnitMu$ defined by
\begin{align} \label{E:CRUCIALLATETIMEDERIVATIVEDEF}
	\LateTimeLUnitMu 
	& := 
		\sup_{(u',\vartheta) \in [0,u] \times \mathbb{T}} 
		\frac{[\Lunit \upmu]_-(t,u',\vartheta)}{\mathring{\upmu}(u',\vartheta) - M(0,u',\vartheta;t)},
\end{align}
and note that $\LateTimeLUnitMu \geq 0$ in view of the estimate \eqref{E:MUAMPLITUDENEARONE}.
Then the following estimates hold
(see Subsect.\ \ref{SS:NOTATIONANDINDEXCONVENTIONS} regarding our use of the notation $\mathcal{O}_{\mydiam}(\cdot)$):
\begin{subequations}
\begin{align}
	\upmu_{\star}(s,u)
	& = \left\lbrace
				1 + \mathcal{O}_{\mydiam} (\Psiep) + \mathcal{O}(\varepsilon)
			\right\rbrace
			\left\lbrace 
				1 - \LateTimeLUnitMu s 
			\right\rbrace,
	\label{E:MUSTARBOUNDS}  
		\\
	\left\| 
		[\Lunit \upmu]_- 
	\right\|_{L^{\infty}(\Sigma_s^u)}
	& =
	\begin{cases}
		\left\lbrace
				1 + \mathcal{O}_{\mydiam}(\Psiep) + \mathcal{O}(\varepsilon^{1/2})
		\right\rbrace
		\LateTimeLUnitMu,
		& \mbox{if } \LateTimeLUnitMu \geq \varepsilon^{1/2},
		\\
		\mathcal{O}(\varepsilon^{1/2}),
		& \mbox{if } \LateTimeLUnitMu \leq \varepsilon^{1/2}.
	\end{cases}
	\label{E:LUNITUPMUMINUSBOUND}
\end{align}
\end{subequations}

Furthermore, we have
\begin{subequations}
\begin{align} \label{E:UNOTNECESSARILYEQUALTOONECRUCIALLATETIMEDERIVATIVECOMPAREDTODATAPARAMETER}
	\LateTimeLUnitMu 
	& \leq 
		\left\lbrace
			1 + \mathcal{O}_{\mydiam}(\Psiep) + \mathcal{O}(\varepsilon)
		\right\rbrace
		\TranminusdatasizeWithFactor.
\end{align}

Moreover, when $u = 1$, we have
\begin{align} \label{E:CRUCIALLATETIMEDERIVATIVECOMPAREDTODATAPARAMETER}
	\LateTimeLUnitMu 
	& = \left\lbrace
				1 + \mathcal{O}_{\mydiam}(\Psiep) + \mathcal{O}(\varepsilon)
			\right\rbrace
			\TranminusdatasizeWithFactor,
\end{align}
\end{subequations}
and
\begin{align}
	\upmu_{\star}(s,1)
	& = \left\lbrace
				1 + \mathcal{O}_{\mydiam} (\Psiep) + \mathcal{O}(\varepsilon)
			\right\rbrace
			\left\lbrace 
				1 - 
				\left[
						1 + \mathcal{O}_{\mydiam}(\Psiep) + \mathcal{O}(\varepsilon)
				\right]
				\TranminusdatasizeWithFactor s
			\right\rbrace.
	\label{E:MUSTARBOUNDSUISONE}
\end{align}

\medskip

\noindent \underline{\textbf{Sharp estimates when $(u',\vartheta) \in \Vplus{t}{u}$}}.
We recall that the set $\Vplus{t}{u}$ is defined in \eqref{E:ANGLESANDUWITHNONDECAYINUPMUGBEHAVIOR}.
If $0 \leq s_1 \leq s_2 \leq t$, then the following estimate holds:
\begin{align} \label{E:LOCALIZEDMUCANTGROWTOOFAST}
	\sup_{(u',\vartheta) \in \Vplus{t}{u}}
	\frac{\upmu(s_2,u',\vartheta)}{\upmu(s_1,u',\vartheta)}
	& \leq C.
\end{align}

In addition, if $s \in [0,t]$ and $\Sigmaplus{s}{t}{u}$ is as defined in \eqref{E:SIGMAPLUS}, then 
\begin{align}  \label{E:KEYMUNOTDECAYBOUND}
		\inf_{\Sigmaplus{s}{t}{u}} \upmu 
	& \geq 1 - C_{\mydiam} \Psiep - C \varepsilon.
\end{align}

Moreover, if $s \in [0,t]$, then 
\begin{align} 
	\left\| \frac{[\Lunit \upmu]_-}{\upmu} \right\|_{L^{\infty}(\Sigmaplus{s}{t}{u})}
	& \leq C \varepsilon.
		\label{E:KEYMUNOTDECAYINGMINUSPARTLMUOVERMUBOUND}
\end{align}

\medskip

\noindent \underline{\textbf{Sharp estimates when $(u',\vartheta) \in \Vminus{t}{u}$}}.
We recall that $\Vminus{t}{u}$ is the set defined in \eqref{E:ANGLESANDUWITHDECAYINUPMUGBEHAVIOR}.
Let $\LateTimeLUnitMu > 0$
be as in \eqref{E:CRUCIALLATETIMEDERIVATIVEDEF}
and consider a time interval $s \in [0,t]$.
Then there exists a constant $C > 0$ such that
\begin{align} \label{E:LOCALIZEDMUMUSTSHRINK}
	\mathop{\sup_{0 \leq s_1 \leq s_2 \leq t}}_{(u',\vartheta) \in \Vminus{t}{u}}
	\frac{\upmu(s_2,u',\vartheta)}{\upmu(s_1,u',\vartheta)}
	& \leq 1 + C \varepsilon.
\end{align}

Furthermore, if $s \in [0,t]$ and $\Sigmaminus{s}{t}{u}$
is as defined in \eqref{E:SIGMAMINUS}, then 
\begin{align} \label{E:LMUPLUSNEGLIGIBLEINSIGMAMINUS}
	\left\| [\Lunit \upmu]_+ \right\|_{L^{\infty}(\Sigmaminus{s}{t}{u})}
	& \leq C \varepsilon.
\end{align}

Finally, there exist constants $C_{\mydiam} > 0$ and $C > 0$ such that if $0 \leq s \leq t$, then
\begin{align}		\label{E:HYPERSURFACELARGETIMEHARDCASEOMEGAMINUSBOUND}
	\left\| 
		[\Lunit \upmu]_- 
	\right\|_{L^{\infty}(\Sigmaminus{s}{t}{u})}
	& \leq 
		\begin{cases}
		\left\lbrace
			1 + C_{\mydiam} \Psiep + C \varepsilon^{1/2}
		\right\rbrace
		\LateTimeLUnitMu,
		& \mbox{if } \LateTimeLUnitMu \geq \varepsilon^{1/2},
		\\
		C \varepsilon^{1/2},
		& \mbox{if } \LateTimeLUnitMu \leq \varepsilon^{1/2}.
	\end{cases}
\end{align}

\noindent \underline{\textbf{Approximate time-monotonicity of $\upmu_{\star}^{-1}(s,u)$}}.
There exist constants $C_{\mydiam} > 0$ and $C > 0$ such that if 
$0 \leq s_1 \leq s_2 \leq t$, then
\begin{align} \label{E:MUSTARINVERSEMUSTGROWUPTOACONSTANT}
	\upmu_{\star}^{-1}(s_1,u) & \leq (1 + C_{\mydiam} \Psiep + C \varepsilon) \upmu_{\star}^{-1}(s_2,u).
\end{align}

\end{proposition}

\begin{proof}
See Subsect.\ \ref{SS:OFTENUSEDESTIMATES} for some comments on the analysis.

\medskip

\noindent \textbf{Proof of \eqref{E:POSITIVEPARTOFLMUOVERMUISBOUNDED}}:
We may assume that $\Lunit \upmu(s,u,\vartheta) > 0$ since otherwise
\eqref{E:POSITIVEPARTOFLMUOVERMUISBOUNDED} is trivial. 
Then by \eqref{E:LUNITUPMUATTIMETMINUSLUNITUPMUATTIMESPOINTWISEESTIMATE},
for $0 \leq s' \leq s \leq t < \Tboot \leq 2 \TranminusdatasizeWithFactor^{-1}$,
we have that 
$\Lunit \upmu(s',u,\vartheta) 
\geq 
\Lunit \upmu(s,u,\vartheta)
- C \varepsilon(s-s') 
\geq 
- C \varepsilon
$.
Integrating this estimate with respect to $s'$ starting from $s'=0$
and using \eqref{E:MUINITIALDATAESTIMATE},
we find that 
$\upmu(s,u,\vartheta) \geq 1 - C_{\mydiam} \Psiep - C \varepsilon$
and thus $1/\upmu(s,u,\vartheta) \leq 1 + C_{\mydiam} \Psiep + C \varepsilon$.
Also using the bound 
$\left|
	\Lunit \upmu(s,u,\vartheta)
\right|
\leq C
$
proved in \eqref{E:LUNITUPMULINFINITY},
we conclude the desired estimate.

\medskip

\noindent \textbf{Proof of \eqref{E:SMALLMUIMPLIESLMUISNEGATIVE}}:
By \eqref{E:LUNITUPMUATTIMETMINUSLUNITUPMUATTIMESPOINTWISEESTIMATE},
for $0 \leq s \leq t < \Tboot \leq 2 \TranminusdatasizeWithFactor^{-1}$,
we have that 
$
\Lunit \upmu(s,u,\vartheta) = \Lunit \upmu(0,u,\vartheta) + \mathcal{O}(\varepsilon)
$.
Integrating this estimate with respect to $s$ starting from $s=0$
and using \eqref{E:MUINITIALDATAESTIMATE},
we find that 
$\upmu(s,u,\vartheta) 
	= 1 
	+
	\mathcal{O}_{\mydiam}(\Psiep)
	+ 
	\mathcal{O}(\varepsilon) 
	+
	s \Lunit \upmu(0,u,\vartheta) 
$.
Again using \eqref{E:LUNITUPMUATTIMETMINUSLUNITUPMUATTIMESPOINTWISEESTIMATE}
to deduce that $\Lunit \upmu(0,u,\vartheta) = \Lunit \upmu(s,u,\vartheta) + \mathcal{O}(\varepsilon)$,
we find that
$\upmu(s,u,\vartheta) 
	= 1 
	+
	\mathcal{O}_{\mydiam}(\Psiep)
	+ 
	\mathcal{O}(\varepsilon) 
	+
	s \Lunit \upmu(s,u,\vartheta) 
$.
It follows that 
whenever $\upmu(s,u,\vartheta) < 1/4$, we have
\[
\Lunit \upmu(s,u,\vartheta) 
< 
- \frac{1}{s}
	\left\lbrace
		3/4 
		+
		\mathcal{O}_{\mydiam}(\Psiep)
		+ 
		\mathcal{O}(\varepsilon) 
	\right\rbrace
<
- \frac{1}{2} 
	\left\lbrace
		3/4 
		+
		\mathcal{O}_{\mydiam}(\Psiep)
		+ 
		\mathcal{O}(\varepsilon) 
	\right\rbrace
	\TranminusdatasizeWithFactor 
< - \frac{1}{4} \TranminusdatasizeWithFactor 
\]
as desired.

\medskip

\noindent \textbf{Proof of \eqref{E:UNOTNECESSARILYEQUALTOONECRUCIALLATETIMEDERIVATIVECOMPAREDTODATAPARAMETER} and \eqref{E:CRUCIALLATETIMEDERIVATIVECOMPAREDTODATAPARAMETER}}:
We prove only 
\eqref{E:UNOTNECESSARILYEQUALTOONECRUCIALLATETIMEDERIVATIVECOMPAREDTODATAPARAMETER}
since
\eqref{E:CRUCIALLATETIMEDERIVATIVECOMPAREDTODATAPARAMETER}
 follows from nearly identical arguments. 
From 
\eqref{E:UPMUFIRSTTRANSPORT},
Lemma~\ref{L:SCHEMATICDEPENDENCEOFMANYTENSORFIELDS},
\eqref{E:MUAMPLITUDENEARONE},
\eqref{E:LUNITUPMUATTIMETMINUSLUNITUPMUATTIMESPOINTWISEESTIMATE},
and the $L^{\infty}$ estimates of Prop.\ \ref{P:IMPROVEMENTOFAUX},
we have 
\begin{align} \label{E:WEIGHTEDLMUCOMPAREDTOCRUCIALPOINTWISE}
\frac{\Lunit \upmu(t,u,\vartheta)}{\mathring{\upmu}(u,\vartheta) - M(0,u,\vartheta;t)}
& = 
\left\lbrace
	1 + \mathcal{O}_{\mydiam}(\Psiep) 
\right\rbrace
\Lunit \upmu(0,u,\vartheta)
+ \mathcal{O}(\varepsilon)
	\\
& = \frac{1}{2} 
	\left\lbrace
		1 + \mathcal{O}_{\mydiam}(\Psiep) 
  \right\rbrace
	[G_{\Lunit \Lunit} \Rad \Psi](0,u,\vartheta)
 + 
\mathcal{O}(\varepsilon).
	\notag
\end{align}
From \eqref{E:WEIGHTEDLMUCOMPAREDTOCRUCIALPOINTWISE}
and definitions
\eqref{E:CRITICALBLOWUPTIMEFACTOR}
and
\eqref{E:CRUCIALLATETIMEDERIVATIVEDEF},
we conclude that
$\LateTimeLUnitMu 
\leq 
\left\lbrace
	1 + \mathcal{O}_{\mydiam}(\Psiep) 
\right\rbrace
\TranminusdatasizeWithFactor 
+ \mathcal{O}(\varepsilon) 
=
\left\lbrace
	1 + \mathcal{O}_{\mydiam}(\Psiep) + \mathcal{O}(\varepsilon)
\right\rbrace
\TranminusdatasizeWithFactor 
$
as desired.

\medskip

\noindent \textbf{Proof of \eqref{E:MUSTARBOUNDS}, \eqref{E:MUSTARBOUNDSUISONE}, \textbf{and} \eqref{E:MUSTARINVERSEMUSTGROWUPTOACONSTANT}}:
We first prove \eqref{E:MUSTARBOUNDS}.
We start by establishing the following preliminary estimate
for the crucial quantity $\LateTimeLUnitMu = \LateTimeLUnitMu(t,u)$ 
(see \eqref{E:CRUCIALLATETIMEDERIVATIVEDEF}):
\begin{align} \label{E:LATETIMELMUTIMESTISLESSTHANONE}
	t \LateTimeLUnitMu
	< 1.
\end{align} 
We may assume that
$\LateTimeLUnitMu > 0$
since otherwise \eqref{E:LATETIMELMUTIMESTISLESSTHANONE} is trivial.
To proceed, we use
\eqref{E:MUAPPROXDEF},
\eqref{E:MUAMPLITUDENEARONE},
\eqref{E:BIGMEST},
and \eqref{E:MUSPLIT}
to deduce that the following estimate holds
for $(s,u',\vartheta) \in [0,t] \times [0,u] \times \mathbb{T}$:
\begin{align} \label{E:MUFIRSTLOWERBOUND}
	\upmu(s,u',\vartheta)
	& =
		\left\lbrace
				1 + \mathcal{O}_{\mydiam} (\Psiep) + \mathcal{O}(\varepsilon)
		\right\rbrace
		\left\lbrace
			1 
			+
			\frac{\Lunit \upmu(t,u',\vartheta)}
			{\mathring{\upmu}(u',\vartheta) - M(0,u',\vartheta;t)} s
			+ 
			\mathcal{O}(\varepsilon) (t-s)^2
		\right\rbrace.
\end{align}
Setting $s=t$ in equation \eqref{E:MUFIRSTLOWERBOUND},
taking the min of both sides  
over $(u',\vartheta) \in [0,u] \times \mathbb{T}$,
and appealing to definitions
\eqref{E:MUSTARDEF} and 
\eqref{E:CRUCIALLATETIMEDERIVATIVEDEF},
we deduce that
$\upmu_{\star}(t,u)
= 
\left\lbrace
				1 + \mathcal{O}_{\mydiam} (\Psiep) + \mathcal{O}(\varepsilon)
		\right\rbrace
(1-\LateTimeLUnitMu t)
$.
Since $\upmu_{\star}(t,u) > 0$ by \eqref{E:BOOTSTRAPMUPOSITIVITY},
we conclude \eqref{E:LATETIMELMUTIMESTISLESSTHANONE}.

Having established the preliminary estimate, 
we now take the min of both sides of \eqref{E:MUFIRSTLOWERBOUND}
over $(u',\vartheta) \in [0,u] \times \mathbb{T}$,
appeal to definition \eqref{E:CRUCIALLATETIMEDERIVATIVEDEF},
and use the estimate \eqref{E:UPITSELFLINFINITYP0CONSEQUENCES}
to obtain:
\begin{align} \label{E:HARDERCASEMUFIRSTLOWERBOUND}
	\min_{(u',\vartheta) \in [0,u] \times \mathbb{T}} \upmu(s,u',\vartheta)
	& = 
		\left\lbrace
				1 + \mathcal{O}_{\mydiam} (\Psiep) + \mathcal{O}(\varepsilon)
			\right\rbrace
		\left\lbrace
			1 
			- \LateTimeLUnitMu s
			+ \mathcal{O}(\varepsilon) (t-s)^2
		\right\rbrace.
\end{align}
We will show that the terms in the second pair of braces on RHS~\eqref{E:HARDERCASEMUFIRSTLOWERBOUND} verify
\begin{align} \label{E:MUSECONDLOWERBOUND}
	1 
	- \LateTimeLUnitMu s
	+ \mathcal{O}(\varepsilon) (t-s)^2
	& 
	=
	\left\lbrace
		1 + \smoothfunction(s,u;t)
	\right\rbrace
	\left\lbrace
		1 - \LateTimeLUnitMu s
	\right\rbrace,
\end{align}
where
\begin{align} \label{E:AMPLITUDEDEVIATIONFUNCTIONMUSECONDLOWERBOUND}
	\smoothfunction(s,u;t)
	& = \mathcal{O}(\varepsilon).
\end{align}
The desired estimate \eqref{E:MUSTARBOUNDS}
then follows easily from 
\eqref{E:HARDERCASEMUFIRSTLOWERBOUND}-\eqref{E:AMPLITUDEDEVIATIONFUNCTIONMUSECONDLOWERBOUND}
and 
definition \eqref{E:MUSTARDEF}.
To prove \eqref{E:AMPLITUDEDEVIATIONFUNCTIONMUSECONDLOWERBOUND}, 
we first use \eqref{E:MUSECONDLOWERBOUND} to solve for $\smoothfunction(s,u;t)$: 
\begin{align} \label{E:AMPLITUDEDEVIATIONFUNCTIONEXPRESSION}
	\smoothfunction(s,u;t)
	=
	\frac{\mathcal{O}(\varepsilon) (t-s)^2
				}
				{
				1 - \LateTimeLUnitMu s
				}
	=
	\frac{\mathcal{O}(\varepsilon) (t-s)^2
				}
				{
				1 - \LateTimeLUnitMu t
				+ 
				\LateTimeLUnitMu (t-s)
				}.
\end{align}
We start by considering the case $\LateTimeLUnitMu \leq (1/4) \TranminusdatasizeWithFactor$.
Since $0 \leq s \leq t < \Tboot \leq 2 \TranminusdatasizeWithFactor^{-1}$,
the denominator in the middle expression in \eqref{E:AMPLITUDEDEVIATIONFUNCTIONEXPRESSION}
is $\geq 1/2$, and the desired estimate
\eqref{E:AMPLITUDEDEVIATIONFUNCTIONMUSECONDLOWERBOUND}
follows easily.
In remaining case, we have
$\LateTimeLUnitMu > (1/4) \TranminusdatasizeWithFactor$.
Using \eqref{E:LATETIMELMUTIMESTISLESSTHANONE},
we deduce that 
RHS~\eqref{E:AMPLITUDEDEVIATIONFUNCTIONEXPRESSION}
$
\displaystyle
\leq \frac{1}{\LateTimeLUnitMu} \mathcal{O}(\varepsilon) (t-s)
\leq C \varepsilon \TranminusdatasizeWithFactor^{-2} 
\lesssim \varepsilon
$
as desired.

Inequality \eqref{E:MUSTARINVERSEMUSTGROWUPTOACONSTANT} then follows as a simple consequence of 
\eqref{E:MUSTARBOUNDS}.

Finally, we observe that the estimate \eqref{E:MUSTARBOUNDSUISONE}
follows from
\eqref{E:MUSTARBOUNDS}
and \eqref{E:CRUCIALLATETIMEDERIVATIVECOMPAREDTODATAPARAMETER}.
\medskip

\noindent \textbf{Proof of \eqref{E:LUNITUPMUMINUSBOUND} and \eqref{E:HYPERSURFACELARGETIMEHARDCASEOMEGAMINUSBOUND}}:
To prove \eqref{E:LUNITUPMUMINUSBOUND},
we first use \eqref{E:LUNITUPMUATTIMETMINUSLUNITUPMUATTIMESPOINTWISEESTIMATE}
to deduce that for $0 \leq s \leq t < \Tboot \leq 2 \TranminusdatasizeWithFactor^{-1}$ 
and $(u',\vartheta) \in [0,u] \times \mathbb{T}$,
we have $\Lunit \upmu(s,u',\vartheta) = \Lunit \upmu(t,u',\vartheta) + \mathcal{O}(\varepsilon)$.
Appealing to definition \eqref{E:CRUCIALLATETIMEDERIVATIVEDEF} and
using the estimates \eqref{E:LUNITUPMULINFINITY} and \eqref{E:MUAMPLITUDENEARONE},
we find that
$
\displaystyle
\left\| 
	[\Lunit \upmu]_- 
\right\|_{L^{\infty}(\Sigma_s^u)}
=
	\left\lbrace
		1 + \mathcal{O}_{\mydiam} (\Psiep)
	\right\rbrace
	\LateTimeLUnitMu
	+ 
	\mathcal{O}(\varepsilon)
$.
If
$
\displaystyle
\varepsilon^{1/2}
\leq \LateTimeLUnitMu
$,
we see that if $\varepsilon$ is sufficiently small, 
then we have the desired bound
$
\displaystyle
	\left\lbrace
		1 + \mathcal{O}_{\mydiam} (\Psiep)
	\right\rbrace
	\LateTimeLUnitMu
	+ 
	\mathcal{O}(\varepsilon)
	=
	\left\lbrace
		1 
		+ 
		\mathcal{O}_{\mydiam} (\Psiep)
		+ 
		\mathcal{O}(\varepsilon^{1/2})
	\right\rbrace
	\LateTimeLUnitMu
$.
On the other hand, if
$
\displaystyle
\LateTimeLUnitMu
\leq 
\varepsilon^{1/2}
$,
then similar reasoning yields that
$
\displaystyle
\left\| 
	[\Lunit \upmu]_- 
\right\|_{L^{\infty}(\Sigma_s^u)}
=
	\left\lbrace
		1 + \mathcal{O}_{\mydiam} (\Psiep)
	\right\rbrace
	\LateTimeLUnitMu
	+ 
	\mathcal{O}(\varepsilon)
= \mathcal{O}(\varepsilon^{1/2})
$
as desired. We have thus proved \eqref{E:LUNITUPMUMINUSBOUND}.

The estimate \eqref{E:HYPERSURFACELARGETIMEHARDCASEOMEGAMINUSBOUND} 
can be proved via a similar argument
and we omit the details.

\medskip

\noindent \textbf{Proof of \eqref{E:UNIFORMBOUNDFORMRADMUOVERMU}}:
We fix times $s$ and $t$ with $0 \leq s \leq t < \Tboot \leq 2 \TranminusdatasizeWithFactor^{-1}$
and a point $p \in \Sigma_s^u$ with geometric coordinates 
$(s,\widetilde{u},\widetilde{\vartheta})$.
Let $\iota : [0,u] \rightarrow \Sigma_s^u$ be the integral curve
of $\Rad$ that passes through $p$ and that is parametrized by the values $u'$ of the eikonal function.
We set 
\[
\displaystyle
F(u') := \upmu \circ \iota(u'),
\qquad
\dot{F}(u') := \frac{d}{d u'} F(u') = (\Rad \upmu)\circ \iota(u').
\]
We must bound 
$
\displaystyle
	\frac{[\Rad \upmu]_+}{\upmu}|_p
	= \frac{[\dot{F}(\widetilde{u})]_+}{F(\widetilde{u})}
$.
We may assume that $\dot{F}(\widetilde{u}) > 0$ since otherwise the desired estimate is trivial.
We now set 
\[
H := \sup_{\mathcal{M}_{\Tboot,U_0}} \Rad \Rad \upmu.
\]
If 
$
\displaystyle
F(\widetilde{u}) > \frac{1}{2}
$,
then the desired estimate is a simple consequence of
\eqref{E:RADUPMULINFTY}.
We may therefore also assume that
$
\displaystyle
F(\widetilde{u}) \leq \frac{1}{2}
$.
Then in view of the estimate
$
\left\|
	\upmu - 1
\right\|_{L^{\infty}\left(\mathcal{P}_0^{\Tboot}\right)}
\lesssim \varepsilon
$
along $\mathcal{P}_0^{\Tboot}$
(see \eqref{E:UPITSELFLINFINITYP0CONSEQUENCES} and \eqref{E:DATAEPSILONVSBOOTSTRAPEPSILON}),
we deduce that there exists a $u'' \in [0,\widetilde{u}]$ such that
$\dot{F}(u'') < 0$. Considering also the assumption $\dot{F}(\widetilde{u}) > 0$,
we see that $H > 0$. Moreover, by \eqref{E:RADRADUPMULINFTY}, we have $H \leq C$.
Furthermore, by continuity, there exists a smallest $u_{\ast} \in [0,\widetilde{u}]$
such that $\dot{F}(u') \geq 0$ for $u' \in [u_{\ast},\widetilde{u}]$.
We also set
\begin{align} \label{E:MUMIN}
\upmu_{(Min)}(s,u') 
:= \min_{(u'',\vartheta) \in [0,u'] \times \mathbb{T}} \upmu(s,u'',\vartheta).
\end{align}
The two main steps in the proof are showing that 
\begin{align} \label{E:RADMUOVERMUALGEBRAICBOUND}
	\frac{[\Rad \upmu(s,\widetilde{u},\widetilde{\vartheta})]_+}
		{\upmu(s,\widetilde{u},\widetilde{\vartheta})}
	& \leq 
		H^{1/2}\frac{1}{\sqrt{\upmu_{(Min)}}(s,\widetilde{u})}
\end{align}
and showing that for
$0 \leq s \leq t < \Tboot$, 
we have
\begin{align} \label{E:UPMUMINLOWERBOUND}
	\upmu_{(Min)}(s,u)
	& \geq 
		\max 
		\left\lbrace
			\left\lbrace
				1 - C_{\mydiam} \Psiep - C \varepsilon
			\right\rbrace
			\LateTimeLUnitMu 
			(t-s),
			\left\lbrace
				1 - C_{\mydiam} \Psiep - C \varepsilon
			\right\rbrace 
			(1 - \LateTimeLUnitMu s)
		\right\rbrace,
\end{align}
where $\LateTimeLUnitMu = \LateTimeLUnitMu(t,u)$
is defined in \eqref{E:CRUCIALLATETIMEDERIVATIVEDEF}.
Once we have obtained \eqref{E:RADMUOVERMUALGEBRAICBOUND}-\eqref{E:UPMUMINLOWERBOUND}
(we establish these estimates below), 
we split the remainder of the proof 
(which is relatively easy)
into the two cases 
$
\displaystyle
\LateTimeLUnitMu \leq \frac{1}{4} \TranminusdatasizeWithFactor
$
and 
$
\displaystyle
\LateTimeLUnitMu > \frac{1}{4} \TranminusdatasizeWithFactor
$.
In the first case 
$
\displaystyle
\LateTimeLUnitMu \leq \frac{1}{4} \TranminusdatasizeWithFactor
$,
we have
$
\displaystyle
1 - \LateTimeLUnitMu s \geq 1 - \frac{1}{4} \TranminusdatasizeWithFactor \Tboot \geq \frac{1}{2}
$,
and the
desired bound 
$
\displaystyle
\frac{[\Rad \upmu(s,\widetilde{u},\widetilde{\vartheta})]_+}
{\upmu(s,\widetilde{u},\widetilde{\vartheta})}
\leq C 
\leq \frac{C}{\Tboot^{1/2}}
\leq \frac{C}{\sqrt{\Tboot-s}}
\leq \mbox{RHS~\eqref{E:UNIFORMBOUNDFORMRADMUOVERMU}}
$
follows easily from 
\eqref{E:RADMUOVERMUALGEBRAICBOUND}
and the second term in the $\min$ on RHS~\eqref{E:UPMUMINLOWERBOUND}.
In the remaining case 
$
\displaystyle
\LateTimeLUnitMu > \frac{1}{4} \TranminusdatasizeWithFactor
$, 
we have 
$
\displaystyle
\frac{1}{\LateTimeLUnitMu} \leq C
$, 
and using \eqref{E:RADMUOVERMUALGEBRAICBOUND} and
the first term in the $\min$ on RHS~\eqref{E:UPMUMINLOWERBOUND}, 
we deduce that
$
\displaystyle
\frac{[\Rad \upmu(s,\widetilde{u},\widetilde{\vartheta})]_+}
{\upmu(s,\widetilde{u},\widetilde{\vartheta})}
\leq \frac{C}{\sqrt{t-s}}
$.
Since this estimate holds for all $t < \Tboot$ with a uniform constant $C$,
we conclude \eqref{E:UNIFORMBOUNDFORMRADMUOVERMU} in this case.

We now prove \eqref{E:RADMUOVERMUALGEBRAICBOUND}.
To this end, we will show that 
\begin{align} \label{E:FIRSTRADMUOVERMUALGEBRAICBOUND}
	\frac{[\Rad \upmu(s,\widetilde{u},\widetilde{\vartheta})]_+}
		{\upmu(s,\widetilde{u},\widetilde{\vartheta})}
	& \leq 
		2 H^{1/2} \frac{\sqrt{\upmu(s,\widetilde{u},\widetilde{\vartheta}) - \upmu_{(Min)}(s,\widetilde{u})}}
		{\upmu(s,\widetilde{u},\widetilde{\vartheta})}.
\end{align}
Then viewing RHS~\eqref{E:FIRSTRADMUOVERMUALGEBRAICBOUND} as a function
of the real variable $\upmu(s,\widetilde{u},\widetilde{\vartheta})$ 
(with all other parameters fixed)
on the domain $[\upmu_{(Min)}(s,\widetilde{u}),\infty)$, we
carry out a simple calculus exercise 
to find that RHS~\eqref{E:FIRSTRADMUOVERMUALGEBRAICBOUND}
$
\displaystyle
\leq H^{1/2}\frac{1}{\sqrt{\upmu_{(Min)}(s,\widetilde{u})}}
$,
which yields \eqref{E:RADMUOVERMUALGEBRAICBOUND}.

We now prove \eqref{E:FIRSTRADMUOVERMUALGEBRAICBOUND}.
Let $u_{\ast}$ be as defined just above \eqref{E:MUMIN}.
For any $u' \in [u_{\ast},\widetilde{u}]$, 
we use the mean value theorem to obtain 
\begin{align} \label{E:MVTESTIMATES}
\dot{F}(\widetilde{u}) 
- 
\dot{F}(u') 
\leq H (\widetilde{u} - u'),
	\qquad
F(\widetilde{u}) 
- 
F(u') 
\geq 
\min_{u'' \in [u',\widetilde{u}]} \dot{F}(u'')
(\widetilde{u} - u').
\end{align}
Setting 
$
\displaystyle
u_1 := 
\widetilde{u} 
- 
\frac{1}{2} \frac{\dot{F}(\widetilde{u})}{H}
$, 
we find from the first estimate in \eqref{E:MVTESTIMATES} that 
for $u' \in [u_1,\widetilde{u}]$, we have
$
\displaystyle
\dot{F}(u')
\geq 
\frac{1}{2} \dot{F}(\widetilde{u})
$.
Using also the second estimate in \eqref{E:MVTESTIMATES}, we find that
$
\displaystyle
F(\widetilde{u}) 
- 
F(u_1)
\geq 
\frac{1}{2} \dot{F}(\widetilde{u}) (\widetilde{u}-u_1)
= \frac{1}{4} \frac{\dot{F}^2(\widetilde{u})}{H}
$. 
Noting that the definition \eqref{E:MUMIN} of $\upmu_{(Min)}$ implies that
$F(u_1) 
\geq \upmu_{(Min)}(s,\widetilde{u})
$,
we deduce from the previous estimate that
\begin{align} \label{E:FDIFFERENCELOWERBOUND}
	\upmu (s,\widetilde{u},\widetilde{\vartheta})
	- 
	\upmu_{(Min)}(s,\widetilde{u})
	& \geq 
	\frac{1}{4} \frac{[\Rad \upmu(s,\widetilde{u},\widetilde{\vartheta})]_+^2}{H}.
\end{align}
Taking the square root of \eqref{E:FDIFFERENCELOWERBOUND},
rearranging,
and dividing by
$\upmu(s,\widetilde{u},\widetilde{\vartheta})$,
we conclude the desired estimate \eqref{E:FIRSTRADMUOVERMUALGEBRAICBOUND}.

It remains for us to prove \eqref{E:UPMUMINLOWERBOUND}. 
Reasoning as in the proof of 
\eqref{E:MUFIRSTLOWERBOUND}-\eqref{E:AMPLITUDEDEVIATIONFUNCTIONMUSECONDLOWERBOUND}
and using \eqref{E:LATETIMELMUTIMESTISLESSTHANONE},
we find that for $0 \leq s \leq t < \Tboot$ and $u' \in [0,u]$,
we have
$\upmu_{(Min)}(s,u') 
\geq 
\left\lbrace
	1 - C_{\mydiam} \Psiep - C \varepsilon
\right\rbrace
(1 - \LateTimeLUnitMu s)
\geq
\left\lbrace
	1 - C_{\mydiam} \Psiep - C \varepsilon
\right\rbrace
\LateTimeLUnitMu (t-s)
$.
From these two inequalities, 
we conclude \eqref{E:UPMUMINLOWERBOUND}.

\medskip

\noindent {\textbf{Proof of} \eqref{E:LOCALIZEDMUMUSTSHRINK}}:
A straightforward modification of the proof of \eqref{E:MUSTARBOUNDS},
based on equations
\eqref{E:MUAPPROXDEF}
and \eqref{E:MUSPLIT},
yields that for $0 \leq s_1 \leq s_2 \leq t < \Tboot$
and $(u',\vartheta) \in \Vminus{t}{u}$,
we have
$
\displaystyle
\frac{\upmu(s_2,u',\vartheta)}{\upmu(s_1,u',\vartheta)}
= \left\lbrace
		1 
		+ 
		\mathcal{O}(\varepsilon)
	\right\rbrace
	\left\lbrace
	\frac{1 + \left(\frac{\Lunit \upmu(t,u',\vartheta)}
									{\mathring{\upmu}(u',\vartheta) - M(0,u',\vartheta;t)} \right)
									s_2}{
									1 + \left(\frac{\Lunit \upmu(t,u',\vartheta)}
									{\mathring{\upmu}(u',\vartheta) - M(0,u',\vartheta;t)} \right)
									s_1}	
\right\rbrace
$.
The estimate \eqref{E:LOCALIZEDMUMUSTSHRINK} 
then follows as a simple consequence.

\medskip

\noindent {\textbf{Proof of}
\eqref{E:LOCALIZEDMUCANTGROWTOOFAST}, \eqref{E:KEYMUNOTDECAYBOUND}, \textbf{and} \eqref{E:KEYMUNOTDECAYINGMINUSPARTLMUOVERMUBOUND}}:
By \eqref{E:LUNITUPMUATTIMETMINUSLUNITUPMUATTIMESPOINTWISEESTIMATE},
if $(u',\vartheta) \in \Vplus{t}{u}$ and
$0 \leq s' \leq s \leq t < \Tboot$,
then
$[\Lunit \upmu]_-(s',u,\vartheta) \leq C \varepsilon$
and
$\Lunit \upmu(s',u,\vartheta) \geq - C \varepsilon$.
Integrating the latter estimate with respect to $s'$ from $0$ to $s$
and using \eqref{E:MUINITIALDATAESTIMATE},
we find that  
$\upmu(s,u',\vartheta) 
\geq 1 - C_{\mydiam} \Psiep - C \varepsilon
$.
Moreover, from \eqref{E:UPMULINFTY}, 
we have the crude bound $\upmu(s,u',\vartheta) \leq C$.
The desired bounds
\eqref{E:LOCALIZEDMUCANTGROWTOOFAST}, \eqref{E:KEYMUNOTDECAYBOUND}, and 
\eqref{E:KEYMUNOTDECAYINGMINUSPARTLMUOVERMUBOUND}
now readily follow
from these estimates.

\medskip

\noindent {\textbf{Proof of} \eqref{E:LMUPLUSNEGLIGIBLEINSIGMAMINUS}}:
By \eqref{E:LUNITUPMUATTIMETMINUSLUNITUPMUATTIMESPOINTWISEESTIMATE},
if $(u',\vartheta) \in \Vminus{t}{u}$ and
$0 \leq s \leq t < \Tboot$,
then
$[\Lunit \upmu]_+(s,u',\vartheta)
=  [\Lunit \upmu]_+(t,u',\vartheta) + \mathcal{O}(\varepsilon) = \mathcal{O}(\varepsilon)$.
The desired bound \eqref{E:LMUPLUSNEGLIGIBLEINSIGMAMINUS} thus follows.

\end{proof}

\subsection{Sharp time-integral estimates involving \texorpdfstring{$\upmu$}{the inverse foliation density}}
\label{SS:SHARPTIMEINTEGRALESTIMATES}
In deriving a priori energy estimates, we use a Gronwall argument that features
time integrals involving difficult factors of $\upmu_{\star}^{-\Contwo}$
for various constants $\Contwo > 0$.
In the next proposition, we bound these time integrals.

\begin{proposition}[\textbf{Fundamental estimates for time integrals involving $\upmu_{\star}^{-1}$}] 
\label{P:MUINVERSEINTEGRALESTIMATES}
	Let
	\begin{align*}
		1 < \Contwo \leq 100
	\end{align*}
	be a real number.\footnote{In practice, 
	to close our energy estimates, we need only to consider
	values of $\Contwo$ that are significantly less than $100$.
	At this point in the paper, we prefer to 
	allow $\Contwo$ to be as large as $100$ so that we have
	a comfortable margin of error later in the paper.} 
	The following estimates hold for $(t,u) \in [0,\Tboot) \times [0,U_0]$.

	\medskip

	\noindent \underline{\textbf{Estimates relevant for borderline top-order spacetime integrals}}.
	There exist constants $C_{\mydiam} > 0$ 
	(see Subsect.\ \ref{SS:NOTATIONANDINDEXCONVENTIONS} regarding our use of the notation $C_{\mydiam}$)
	and $C > 0$ such that 
	\begin{align} \label{E:KEYMUTOAPOWERINTEGRALBOUND}
		\int_{s=0}^t 
			\frac{\left\| [\Lunit \upmu]_- \right\|_{L^{\infty}(\Sigma_s^u)}} 
					 {\upmu_{\star}^{\Contwo}(s,u)}
		\, ds 
		& \leq \frac{1 + C_{\mydiam} \Psiep + C \varepsilon^{1/2}}{\Contwo-1} 
			\upmu_{\star}^{1-\Contwo}(t,u).
	\end{align}

	\noindent \underline{\textbf{Estimates relevant for borderline top-order hypersurface integrals}}.
	Let $\Sigmaminus{t}{t}{u}$ be the subset of $\Sigma_t$ defined in \eqref{E:SIGMAMINUS}.
	There exist constants $C_{\mydiam} > 0$ and $C > 0$ such that
	\begin{align} \label{E:KEYHYPERSURFACEMUTOAPOWERINTEGRALBOUND}
		\left\| 
			\Lunit \upmu 
		\right\|_{L^{\infty}(\Sigmaminus{t}{t}{u})} 
		\int_{s=0}^t 
			\frac{1} 
				{\upmu_{\star}^{\Contwo}(s,u)}
			\, ds 
		& \leq \frac{1 + C_{\mydiam} \Psiep + C \varepsilon^{1/2}}{\Contwo-1} 
			\upmu_{\star}^{1-\Contwo}(t,u).
	\end{align}

	\medskip

	\noindent \underline{\textbf{Estimates relevant for less dangerous top-order spacetime integrals}}.
	There exists a constant $C > 0$ 
	such that 
	\begin{align} \label{E:LOSSKEYMUINTEGRALBOUND}
		\int_{s=0}^t \frac{1} 
			{\upmu_{\star}^{\Contwo}(s,u)}
		\, ds 
		& \leq C \left\lbrace 
			1
			+ 
			\frac{1}{\Contwo-1} \right\rbrace \upmu_{\star}^{1-\Contwo}(t,u).
	\end{align} 

	\medskip

	\noindent \underline{\textbf{Estimates for integrals that lead to only $\ln \upmu_{\star}^{-1}$ degeneracy}}. 
	There exist constants $C_{\mydiam} > 0$ and $C > 0$ such that
	\begin{align} \label{E:KEYMUINVERSEINTEGRALBOUND}
		\int_{s=0}^t 
			\frac{\left\| [\Lunit \upmu]_- \right\|_{L^{\infty}(\Sigma_s^u)}} 
					 {\upmu_{\star}(s,u)}
		\, ds 
		& \leq (1 + C_{\mydiam} \Psiep + C \varepsilon^{1/2}) \ln \upmu_{\star}^{-1}(t,u) 
			+ C_{\mydiam} \Psiep
			+ C \varepsilon^{1/2},
			\\
		\int_{s=0}^t 
			\frac{1}{\upmu_{\star}(s,u)}
		\, ds 
		& \leq  C \left\lbrace \ln \upmu_{\star}^{-1}(t,u) + 1 \right\rbrace.
		\label{E:LOGLOSSMUINVERSEINTEGRALBOUND}
	\end{align}

\medskip

	\noindent \underline{\textbf{Estimates for integrals that break the $\upmu_{\star}^{-1}$ degeneracy}}. 
	There exists a constant $C > 0$ such that
	\begin{align} \label{E:LESSSINGULARTERMSMPOINTNINEINTEGRALBOUND}
		\int_{s=0}^t 
			\frac{1} 
			{\upmu_{\star}^{9/10}(s,u)}
		\, ds 
		& \leq C.
	\end{align}

\end{proposition}

\begin{proof}
See Subsect.\ \ref{SS:OFTENUSEDESTIMATES} for some comments on the analysis.

\noindent {\textbf{Proof of} \eqref{E:KEYMUTOAPOWERINTEGRALBOUND}, 
\eqref{E:KEYHYPERSURFACEMUTOAPOWERINTEGRALBOUND},
and \eqref{E:KEYMUINVERSEINTEGRALBOUND}}:
To prove \eqref{E:KEYMUTOAPOWERINTEGRALBOUND}, we 
first consider the case 
$\LateTimeLUnitMu \geq \varepsilon^{1/2}$
in \eqref{E:LUNITUPMUMINUSBOUND}.
Using \eqref{E:MUSTARBOUNDS} and \eqref{E:LUNITUPMUMINUSBOUND}, we deduce that
\begin{align} \label{E:PROOFKEYMUTOAPOWERINTEGRALBOUND}
		\int_{s=0}^t 
			\frac{\left\| [\Lunit \upmu]_- \right\|_{L^{\infty}(\Sigma_s^u)}} 
					 {\upmu_{\star}^{\Contwo}(s,u)}
		\, ds 
		& = 
		\left\lbrace
			1 + \mathcal{O}_{\mydiam}(\Psiep) + \mathcal{O}(\varepsilon^{1/2})
		\right\rbrace
			\int_{s=0}^t 
				\frac{\LateTimeLUnitMu}{(1 - \LateTimeLUnitMu s)^{\Contwo}}
			\, ds 
				\\
		& \leq \frac{1 + \mathcal{O}_{\mydiam}(\Psiep) + \mathcal{O}(\varepsilon^{1/2})}{\Contwo - 1}
				\frac{1}{(1 - \LateTimeLUnitMu t)^{\Contwo-1}}
			= 	\frac{1 + \mathcal{O}_{\mydiam}(\Psiep) + \mathcal{O}(\varepsilon^{1/2})}{\Contwo - 1}
					\upmu_{\star}^{1-\Contwo}(t,u)
					\notag
	\end{align}
	as desired. We now consider the remaining case
	$\LateTimeLUnitMu \leq \varepsilon^{1/2}$
	in \eqref{E:LUNITUPMUMINUSBOUND}.
	Using 
	\eqref{E:MUSTARBOUNDS},
	\eqref{E:LUNITUPMUMINUSBOUND},
	and the fact that $0 \leq s \leq t < \Tboot \leq 2 \TranminusdatasizeWithFactor^{-1}$,
	we see that for $\varepsilon$ sufficiently small relative to 
	$\TranminusdatasizeWithFactor$, 
	we have
	\begin{align} \label{E:SECONDCASEPROOFKEYMUTOAPOWERINTEGRALBOUND}
		\int_{s=0}^t 
			\frac{\left\| [\Lunit \upmu]_- \right\|_{L^{\infty}(\Sigma_s^u)}} 
					 {\upmu_{\star}^{\Contwo}(s,u)}
		\, ds 
		& \leq
			C \varepsilon^{1/2}
			\int_{s=0}^t 
				\frac{1}{(1 - \LateTimeLUnitMu s)^{\Contwo}}
			\, ds 
				\\
		& 
			\leq
			C \varepsilon^{1/2}
			\int_{s=0}^t 
				1
			\, ds 
			\leq 
			C \varepsilon^{1/2}
			\leq
			C \varepsilon^{1/2}
			\frac{1}{(1 - \LateTimeLUnitMu t)^{\Contwo-1}}
			\leq \frac{1}{\Contwo - 1}
			\upmu_{\star}^{1-\Contwo}(t,u)
					\notag
	\end{align}
	as desired.
	We have thus proved \eqref{E:KEYMUTOAPOWERINTEGRALBOUND}.

	Inequality \eqref{E:KEYMUINVERSEINTEGRALBOUND}
	can be proved using similar arguments and we omit the details.
	
	Inequality \eqref{E:KEYHYPERSURFACEMUTOAPOWERINTEGRALBOUND} 
	can be proved using similar arguments
	with the help of the estimate \eqref{E:HYPERSURFACELARGETIMEHARDCASEOMEGAMINUSBOUND}
	and we omit the details.

	\medskip

\noindent {\textbf{Proof of} \eqref{E:LOSSKEYMUINTEGRALBOUND},
\eqref{E:LOGLOSSMUINVERSEINTEGRALBOUND}, 
and \eqref{E:LESSSINGULARTERMSMPOINTNINEINTEGRALBOUND}}:
	To prove \eqref{E:LOSSKEYMUINTEGRALBOUND}, 
	we first use \eqref{E:MUSTARBOUNDS} to deduce
	\begin{align} \label{E:PROOFLOSSKEYMUINTEGRALBOUND}
		\int_{s=0}^t \frac{1} 
			{\upmu_{\star}^{\Contwo}(s,u)}
		\, ds 
		& \leq 
		C
		\int_{s=0}^t 
			\frac{1}{(1 - \LateTimeLUnitMu s)^{\Contwo}}
		\, ds,
	\end{align}
	where $\LateTimeLUnitMu = \LateTimeLUnitMu(t,u)$
	is defined in \eqref{E:CRUCIALLATETIMEDERIVATIVEDEF}.
	We first assume that 
	$
	\displaystyle
	\LateTimeLUnitMu \leq \frac{1}{4} \TranminusdatasizeWithFactor
	$. 
	Then since $0 \leq t < \Tboot < 2 \TranminusdatasizeWithFactor^{-1}$,
	we see from \eqref{E:MUSTARBOUNDS} that 
	$
	\displaystyle
	\upmu_{\star}(s,u) \geq \frac{1}{4}$ for $0 \leq s \leq t
	$
	and that RHS~\eqref{E:PROOFLOSSKEYMUINTEGRALBOUND} 
	$\leq C 
	\leq 
	C \upmu_{\star}^{1-\Contwo}(t,u)$
	as desired.
	In the remaining case, we have 
	$
	\displaystyle
	\LateTimeLUnitMu > \frac{1}{4} \TranminusdatasizeWithFactor
	$,
	and we can use \eqref{E:MUSTARBOUNDS},
	the estimate 
	$
	\displaystyle
	\frac{1}{\LateTimeLUnitMu} 
	\leq C$,
	and \eqref{E:LATETIMELMUTIMESTISLESSTHANONE}
	to bound RHS~\eqref{E:PROOFLOSSKEYMUINTEGRALBOUND} by
	$
	\displaystyle
	\leq
		\frac{C}{\LateTimeLUnitMu}
		\frac{1}{(\Contwo - 1)}
		\frac{1}{(1 - \LateTimeLUnitMu t)^{\Contwo-1}}
		\leq 
		\frac{C}{\Contwo - 1}
		\upmu_{\star}^{1-\Contwo}(t,u)
  $
	as desired.

	Inequalities 
	\eqref{E:LOGLOSSMUINVERSEINTEGRALBOUND}
	and
	\eqref{E:LESSSINGULARTERMSMPOINTNINEINTEGRALBOUND} can be proved in a similar
	fashion. We omit the details, aside from remarking that the last step of the proof of 
	\eqref{E:LESSSINGULARTERMSMPOINTNINEINTEGRALBOUND}
	relies on the trivial estimate $(1 - \LateTimeLUnitMu t)^{1/10} \leq 1$.

\end{proof}

\section{Pointwise estimates for the error terms}
\label{S:POINTWISEESTIMATES}
In this section, we use some estimates that we established in prior sections 
to derive pointwise estimates for the error terms that we 
encounter in our energy estimates.
Remark~\ref{R:ESTIMATESPROVEDINOTHERPAPER} especially applies in this section.

\subsection{Definition of ``harmless'' error terms}
\label{SS:HARMLESS}
Most error terms that we encounter 
are harmless in the sense that they remain negligible all the way up
to the shock. We now precisely define what we mean by ``harmless.''

\begin{definition}[\textbf{Harmless terms}]
	\label{D:HARMLESSTERMS}
	A $Harmless^{[1,N]}$ term is any term such that 
	under the data-size and bootstrap assumptions 
	of Subsects.\ \ref{SS:DATAASSUMPTIONS}-\ref{SS:PSIBOOTSTRAP}
	and the smallness assumptions of Subsect.\ \ref{SS:SMALLNESSASSUMPTIONS}, 
	the following bound holds on
	$\mathcal{M}_{\Tboot,U_0}$
	(see Subsect.\ \ref{SS:STRINGSOFCOMMUTATIONVECTORFIELDS} regarding the vectorfield operator notation):
	\begin{align} \label{E:HARMESSTERMPOINTWISEESTIMATE}
		\left| 
			Harmless^{[1,N]}
		\right|
		& \lesssim 
			\left|
				\Fullset_*^{[1,N+1];1} \Psi
			\right|
			+
			\left|
				\Fullset_*^{[1,N];1} \GdVar
			\right|
			+
			\left|
				\Tanset_*^{[1,N]} \BadVar
			\right|.
	\end{align}
	
	A $Harmless_{(Slow)}^{\leq N}$ term is any term such that 
	under the data-size and bootstrap assumptions 
	of Subsects.\ \ref{SS:DATAASSUMPTIONS}-\ref{SS:PSIBOOTSTRAP}
	and the smallness assumptions of Subsect.\ \ref{SS:SMALLNESSASSUMPTIONS}, 
	the following bound holds on
	$\mathcal{M}_{\Tboot,U_0}$:
	\begin{align} \label{E:SLOWHARMESSTERMPOINTWISEESTIMATE}
		\left| 
			Harmless_{(Slow)}^{\leq N}
		\right|
		& \lesssim 
			\left|
				\Tanset^{\leq N} \bigslow
			\right|.
	\end{align}
	
\end{definition}

\begin{remark}[\textbf{A difference compared to \cite{jSgHjLwW2016}}]
	\label{R:DIFFERENTHARMLESSTERMDEF}
	Our definition of $Harmless^{[1,N]}$ terms 
	is similar to the definition of the $Harmless^{\leq N}$ terms featured in \cite{jSgHjLwW2016},
	the difference being that here we do not allow for the presence of order $0$ terms 
	on RHS~\eqref{E:HARMESSTERMPOINTWISEESTIMATE}.
	The reason for our slightly different definition is that in this paper,
	some of the order $0$ quantities are
	controlled by the smallness parameter $\Psiep$
	rather than by $\mathring{\upepsilon}$,
	and we find it convenient to highlight that 
	(based in part on our assumptions \eqref{E:SOMENONINEARITIESARELINEAR} on the semilinear inhomogeneous terms)
	such order $0$ quantities do not appear in our energy estimates.
	Our definition of $Harmless_{(Slow)}^{\leq N}$ terms
	accounts for the harmless error terms
	corresponding to the slow wave variable $\bigslow$.
	Note that terms that are order $0$ in $\bigslow$
	\emph{are} allowed on RHS~\eqref{E:SLOWHARMESSTERMPOINTWISEESTIMATE}.
\end{remark}

\subsection{Identification of the key difficult error terms in the commuted equations}
\label{SS:KEYERRORTERMS}
As we mentioned, most error terms that arise upon commuting the wave equations
are negligible. In the next proposition, we identify those error terms that are not.

\begin{proposition}[\textbf{Identification of the key difficult error term factors}]
\label{P:IDOFKEYDIFFICULTENREGYERRORTERMS}
Recall that $\GeoAngFlatRadComponent$ is the scalar function from Lemma~\ref{L:GEOANGDECOMPOSITION}.
For $1 \leq N \leq 18$, we have the following estimates:
\begin{subequations}
\begin{align}
		\upmu \square_g (\GeoAng^{N-1} \Lunit \Psi)
	& = (\angdiffuparg{\#} \Psi) \cdot (\upmu \angdiff \GeoAng^{N-1} \mytr \upchi)
			+ Harmless^{[1,N]}
			+ Harmless_{(Slow)}^{\leq N},
			\label{E:LISTHEFIRSTCOMMUTATORIMPORTANTTERMS} \\
	\upmu \square_g (\GeoAng^N \Psi)
	& = (\Rad \Psi) \GeoAng^N \mytr \upchi
			+ \GeoAngFlatRadComponent (\angdiffuparg{\#} \Psi) \cdot (\upmu \angdiff \GeoAng^{N-1} \mytr \upchi)
				 + Harmless^{[1,N]}
				 + Harmless_{(Slow)}^{\leq N}.
				 \label{E:GEOANGANGISTHEFIRSTCOMMUTATORIMPORTANTTERMS}
\end{align}
\end{subequations}

Furthermore, if $2 \leq N \leq 18$ and $\Tanset^N$ is any $N^{th}$ order 
$\mathcal{P}_u$-tangential operator except for $\GeoAng^{N-1} \Lunit$ or $\GeoAng^N$,
then
\begin{align} \label{E:HARMLESSORDERNCOMMUTATORS}
	\upmu \square_g (\Tanset^N \Psi)
	& = Harmless^{[1,N]}	
		+
		Harmless_{(Slow)}^{\leq N}.
\end{align}

In addition, for $1 \leq N \leq 18$, we have the following estimates,
($i,j=1,2$):
\begin{subequations}
\begin{align}
	\upmu \partial_t \Tanset^N \slow_0
	& = \upmu (h^{-1})^{ab} \partial_a \Tanset^N \slow_b
		+ 
		2 \upmu (h^{-1})^{0a} \partial_a \Tanset^N \slow_0
		+ Harmless^{[1,N]}
		+ Harmless_{(Slow)}^{\leq N},
		 \label{E:SLOWTIMECOMMUTED} \\
	\upmu \partial_t \slow_i
	& = \upmu \partial_i \slow_0
		+ Harmless^{[1,N]}
	+ Harmless_{(Slow)}^{\leq N},
		\label{E:SLOWSPACECOMMUTED} \\
	\upmu \partial_t \slow
	& = \upmu \slow_0
		+ Harmless^{[1,N]}
		+ Harmless_{(Slow)}^{\leq N},
	\label{E:SLOWCOMMUTED}
		\\
	\upmu \partial_i \slow_j
	& = \upmu \partial_j \slow_i
		+ Harmless^{[1,N]}
		+ Harmless_{(Slow)}^{\leq N}.
		\label{E:SYMMETRYOFMIXEDPARTIALSCOMMUTED}
\end{align}
\end{subequations}

Finally, we have the following estimates:
\begin{subequations}
\begin{align}
	\upmu \partial_t \slow_0
	& = \upmu (h^{-1})^{ab} \partial_a \slow_b
		+ 
		2 \upmu (h^{-1})^{0a} \partial_a \slow_0
		+ Harmless^{[1,1]}
		+ Harmless_{(Slow)}^{\leq 0},
		 \label{E:SLOWTIMENOTCOMMUTED} \\
	\upmu \partial_t \slow_i
	& = \upmu \partial_i \slow_0,
		\label{E:SLOWSPACENOTCOMMUTED} \\
	\upmu \partial_t \slow
	& = \upmu \slow_0,
	\label{E:SLOWNOTCOMMUTED}
		\\
	\upmu \partial_i \slow_j
	& = \upmu \partial_j \slow_i.
		\label{E:SYMMETRYOFMIXEDPARTIALSNOTCOMMUTED}
\end{align}
\end{subequations}

\end{proposition}

\begin{proof}
	See Subsect.\ \ref{SS:OFTENUSEDESTIMATES} for some comments on the analysis.
	We first establish \eqref{E:SLOWTIMECOMMUTED}.
	From Lemmas~\ref{L:SCHEMATICDEPENDENCEOFMANYTENSORFIELDS} and \ref{L:CARTESIANVECTORFIELDSINTERMSOFGEOMETRICONES}
	and the assumptions on the semilinear inhomogeneous terms stated in \eqref{E:SOMENONINEARITIESARELINEAR},
	we see that the products of $\upmu$ and the semilinear inhomogeneous terms on the second line of
	RHS~\eqref{E:SLOW0EVOLUTION} are of the schematic form
	$
	\smoothfunction(\BadVar,\bigslow,\Rad \Psi,\Singletan \Psi) \Singletan \Psi 
	+ 
	\smoothfunction(\BadVar,\bigslow,\Rad \Psi,\Singletan \Psi) \bigslow
	$.
	Thus, from the $L^{\infty}$ estimates of Prop.~\ref{P:IMPROVEMENTOFAUX},
	we find that the $\Tanset^N$ derivatives of these terms are bounded in magnitude by
	$
	\lesssim 
	\left|
		\Fullset_*^{[1,N+1];1} \Psi
	\right|
	+
	\left|
		\Tanset_*^{[1,N]} \BadVar
	\right|
	+
	\left|
		\Tanset^{\leq N} \bigslow
	\right|
	=
	Harmless^{[1,N]}
	+ 
	Harmless_{(Slow)}^{\leq N}
	$.
	To complete the proof of \eqref{E:SLOWTIMECOMMUTED}, 
	it remains for us to bound the commutator terms
	$
	[\upmu \partial_t, \Tanset^N] \slow_0
	$,
	$
	[(h^{-1})^{ab} \upmu \partial_a, \Tanset^N] \slow_b
	$,
	and
	$
	[(h^{-1})^{0a} \upmu \partial_a, \Tanset^N] \slow_0
	$.
	We show how to bound the last one; the first two can be bounded similarly.
	From Lemmas~\ref{L:SCHEMATICDEPENDENCEOFMANYTENSORFIELDS} and \ref{L:CARTESIANVECTORFIELDSINTERMSOFGEOMETRICONES}
	and the fact that $(h^{-1})^{\alpha \beta} = (h^{-1})^{\alpha \beta}(\Psi,\bigslow)$,
	we see that obtaining the desired bounds is equivalent to showing that
	$
	[\smoothfunction(\BadVar,\bigslow) \Singletan, \Tanset^N] \slow_0
	+
	[\smoothfunction(\GdVar,\bigslow) \Rad, \Tanset^N] \slow_0
	=
	Harmless^{[1,N]}
	+ 
	Harmless_{(Slow)}^{\leq N}
	$.
	To obtain these estimates, we use the commutator estimates 
	\eqref{E:PURETANGENTIALFUNCTIONCOMMUTATORESTIMATE} and \eqref{E:ONERADIALTANGENTIALFUNCTIONCOMMUTATORESTIMATE}
	with $f = \slow_0$,
	the algebraic identity provided by Lemma~\ref{L:RADOFSLOWWAVEALGEBRAICALLYEXPRESSED}
	(which allows us to replace $\Rad \bigslow$ with $\Singletan \bigslow$ up to error terms),
	and the $L^{\infty}$ estimates of Prop.~\ref{P:IMPROVEMENTOFAUX}.
	
	To derive \eqref{E:SLOWSPACECOMMUTED}-\eqref{E:SYMMETRYOFMIXEDPARTIALSCOMMUTED},
	we use the same reasoning that we used in the previous paragraph;
	the analysis is even simpler since,
	in view of the absence of semilinear inhomogeneous terms
	on RHSs~\eqref{E:SLOWIEVOLUTION}-\eqref{E:SYMMETRYOFMIXEDPARTIALS}, 
	we encounter only commutator error terms.
	
	The estimates \eqref{E:SLOWTIMENOTCOMMUTED}-\eqref{E:SYMMETRYOFMIXEDPARTIALSNOTCOMMUTED}
	can be established via arguments similar to but simpler than the ones
	we used to derive
	\eqref{E:SYMMETRYOFMIXEDPARTIALSCOMMUTED}-\eqref{E:SYMMETRYOFMIXEDPARTIALSCOMMUTED},
	and we therefore omit the details.
	
	We now prove the estimates 
	\eqref{E:LISTHEFIRSTCOMMUTATORIMPORTANTTERMS}-\eqref{E:HARMLESSORDERNCOMMUTATORS}.
	These estimates were essentially proved in
	\cite{jSgHjLwW2016}*{Proposition 11.2}, based in part on
	estimates that are analogs of the estimates of
	Lemmas~\ref{L:POINTWISEFORRECTANGULARCOMPONENTSOFVECTORFIELDS} and \ref{L:POINTWISEESTIMATESFORGSPHEREANDITSDERIVATIVES}
	and the $L^{\infty}$ estimates of Prop.~\ref{P:IMPROVEMENTOFAUX}.
	However, the derivatives of  
	$\upmu \times$ the semilinear inhomogeneous terms on RHS~\eqref{E:FASTWAVE}
	were not treated in \cite{jSgHjLwW2016}) (because these terms were not present in that paper).
	To handle these ``new'' terms, we first use
	Lemma~\ref{L:SCHEMATICDEPENDENCEOFMANYTENSORFIELDS},
	Lemma~\ref{L:CARTESIANVECTORFIELDSINTERMSOFGEOMETRICONES},
	and the assumptions \eqref{E:SOMENONINEARITIESARELINEAR}
	to deduce that
	$\upmu \times$ the semilinear inhomogeneous terms on RHS~\eqref{E:FASTWAVE} 
	are of the schematic form
	$
	\smoothfunction(\BadVar,\bigslow,\Rad \Psi,\Singletan \Psi) \Singletan \Psi 
	+ 
	\smoothfunction(\BadVar,\bigslow,\Rad \Psi,\Singletan \Psi) \bigslow
	$.
	Thus, for the same reasons given in the first paragraph of the proof,
	the $\Tanset^N$ derivatives of these terms are
	$
	Harmless^{[1,N]}
	+ 
	Harmless_{(Slow)}^{\leq N}
	$
	as desired.
	We also clarify that the right-hand sides of the estimates
	\eqref{E:LISTHEFIRSTCOMMUTATORIMPORTANTTERMS}-\eqref{E:HARMLESSORDERNCOMMUTATORS}
	do not feature the order $0$ terms $|\Psi|$ or $|\GdVar|$.
	This is different compared to the analogous
	estimates stated in \cite{jSgHjLwW2016}*{Proposition 11.2},
	but follows from the proof given there
	and from the commutator estimates of 
	Lemmas~\ref{L:COMMUTATORESTIMATES} and
	\ref{L:TRANSVERALTANGENTIALCOMMUTATOR}
	(see also the remarks made in the discussion of the proofs of
		Lemmas~\ref{L:COMMUTATORESTIMATES} and
	\ref{L:TRANSVERALTANGENTIALCOMMUTATOR}
	and Prop.\ \ref{P:IMPROVEMENTOFAUX}).
	
\end{proof}

\subsection{Pointwise estimates for the most difficult product}
Out of all products that we encounter in the energy estimates,
the most difficult one to control is the 
first product on RHS~\eqref{E:GEOANGANGISTHEFIRSTCOMMUTATORIMPORTANTTERMS},
namely $(\Rad \Psi) \GeoAng^N \mytr \upchi$.
In the next proposition, we derive pointwise estimates for this difficult product.
We also derive pointwise estimates for the error term factor
$
\upmu \GeoAng^N \mytr \upchi
$,
which is much easier to control due to the factor of $\upmu$.
Compared to previous works, the estimates of the proposition involve
new terms stemming from the influence of
the slow wave variable $\bigslow$ on $\mytr \upchi$, 
that is, on the null mean curvature of the characteristics
corresponding to the fast wave $\Psi$.

\begin{proposition}[\textbf{The key pointwise estimate for} $(\Rad \Psi) \GeoAng^N \mytr \upchi$]
	\label{P:KEYPOINTWISEESTIMATE}
For $1 \leq N \leq 18$, let 
\begin{align} \label{E:TOPORDERMODIFIEDTRCHI}
\upchifullmodarg{\GeoAng^N}
:= \upmu \GeoAng^N \mytr \upchi 
+ \GeoAng^N 
\left\lbrace
	- G_{\Lunit \Lunit} \Rad \Psi
				- \frac{1}{2} \upmu \mytr \angG \Lunit \Psi
				- \frac{1}{2} \upmu G_{\Lunit \Lunit} \Lunit \Psi 
				+ \upmu \angGnospacemixedarg{\Lunit}{\#} \cdot \angdiff \Psi
\right\rbrace.
\end{align}
We have the following pointwise estimate:
\begin{align} \label{E:KEYPOINTWISEESTIMATE}
		\left|
			(\Rad \Psi) \GeoAng^N \mytr \upchi
		\right|
		(t,u,\vartheta)
		& \leq
			\boxed{2} 
			\frac
			{\left\|
					[\Lunit \upmu]_-
			\right\|_{L^{\infty}(\Sigma_t^u)}}
			{\upmu_{\star}(t,u)}
			\left| 
				\Rad \GeoAng^N \Psi 
			\right|
			(t,u,\vartheta)
				\\
		 &  \ \ + 
						\boxed{4}
						\frac
						{
							\left\|
								[\Lunit \upmu]_-
							\right\|_{L^{\infty}(\Sigma_t^u)}}
						{\upmu_{\star}(t,u)}
					\int_{t'=0}^t 
						\frac
						{
							\left\|
								[\Lunit \upmu]_-
							\right\|_{L^{\infty}(\Sigma_{t'}^u)}}
						{\upmu_{\star}(t',u)}
						\left| 
							\Rad \GeoAng^N \Psi 
						\right|(t',u,\vartheta)
				\, dt'
				\notag \\
	& \ \	+ \mbox{\upshape Error},
		\notag
	\end{align}
	where
	\begin{align}  \label{E:ERRORTERMKEYPOINTWISEESTIMATE}
		\left|
			\mbox{\upshape Error}
		\right|
		(t,u,\vartheta)
		& \lesssim
			\frac{1}{\upmu_{\star}(t,u)}
			\left|\upchifullmodarg{\GeoAng^N} \right|(0,u,\vartheta)
			+ 
			\left|
				\Fullset_*^{[1,N+1];1} \Psi
			\right|
			(t,u,\vartheta)
				\\
			& 
			\ \
			+
			\frac{1}{\upmu_{\star}(t,u)}
			\left|
				\Fullset_*^{[1,N];1} \Psi
			\right|
			(t,u,\vartheta)
				\notag \\
		& \ \
			+
			\frac{1}{\upmu_{\star}(t,u)}
			\left|
				\Tanset^{[1,N]} \GdVar
			\right|
			(t,u,\vartheta)
			+
			\frac{1}{\upmu_{\star}(t,u)}
			\left|
				\Tanset_*^{[1,N]} \BadVar
			\right|
			(t,u,\vartheta)
				\notag \\
		& \ \ 
			  + 
				\varepsilon
				\frac{1}{\upmu_{\star}(t,u)}
				\int_{t'=0}^t
					\frac{1}{\upmu_{\star}(t',u)}
					\left|
						\Rad \Tanset^N \Psi
					\right|
					(t',u,\vartheta)
				\, dt'
				\notag 
				\\
		& \ \ 
			  + 
				\frac{1}{\upmu_{\star}(t,u)}
				\int_{t'=0}^t
					\left|
						\Fullset_*^{[1,N+1];1} \Psi
					\right|
					(t',u,\vartheta)
				\, dt'
				\notag 
				\\
		& \ \
				+ 
				\frac{1}{\upmu_{\star}(t,u)}
				\int_{t'=0}^t
					\frac{1}{\upmu_{\star}(t',u)}
					\left\lbrace
						\left|
							\Fullset_*^{[1,N];1} \Psi
						\right|
						+
						\left|
							\Tanset^{[1,N]} \GdVar
						\right|
						+
					\left|
						\Tanset_*^{[1,N]} \BadVar
					\right|
					\right\rbrace	
					(t',u,\vartheta)
				\, dt'
				\notag
				\\
		& \ \ + 	
						\frac{1}{\upmu_{\star}(t,u)}
						\int_{t'=0}^t 
							\left|
								\Tanset^{\leq N} \bigslow
							\right|
							(t',u,\vartheta)
						\, dt'.
						\notag
\end{align}

	Furthermore, we have the following less precise pointwise estimate:
	\begin{align} \label{E:LESSPRECISEKEYPOINTWISEESTIMATE}
		& 
		\left|
			\upmu \GeoAng^N \mytr \upchi
		\right|
		(t,u,\vartheta)
			\\
		& \lesssim
		\left|\upchifullmodarg{\GeoAng^N} \right|(0,u,\vartheta)
		+
		\upmu
		\left|
			\Tanset^{N+1} \Psi
		\right|(t,u,\vartheta)
		+
		\left|
			\Rad \Tanset^N \Psi
		\right|(t,u,\vartheta)
			\notag \\
	& \ \ 
		+
		\left|
			\Fullset_*^{[1,N];1} \Psi
		\right|(t,u,\vartheta)
		+
		\left|
			\Tanset^{[1,N]} \GdVar
		\right|(t,u,\vartheta)
		+
		\left|
			\Tanset_*^{[1,N]} \BadVar
		\right|(t,u,\vartheta)
		\notag	\\
		& \ \
				+
				\int_{t'=0}^t 
						\frac{1}{\upmu_{\star}(t',u)}
						\left| 
							\Rad \Tanset^N \Psi 
						\right|
						(t',u,\vartheta)
				\, dt'
				+ 
				\int_{t'=0}^t 
					\left|
						\Fullset_*^{[1,N+1];1} \Psi
					\right|
					(t',u,\vartheta)
				\, dt'
							\notag \\
		& \ \   + 
						\int_{t'=0}^t 
							\frac{1}{\upmu_{\star}(t',u)}
							\left\lbrace
							\left|
								\Fullset_*^{[1,N];1} \Psi
							\right|
							+
							\left|
								\Tanset^{[1,N]} \GdVar
							\right|
							+
							\left|
								\Tanset_*^{[1,N]} \BadVar
							\right|
							\right\rbrace
							(t',u,\vartheta)
						\, dt'
						\notag
						\\
		& \  \
					+ 
						\int_{t'=0}^t 
							\left|
								\Tanset^{\leq N} \bigslow
							\right|
							(t',u,\vartheta)
						\, dt'.
					\notag	
		\end{align}

\end{proposition}

\begin{proof}[Proof outline]
	The estimate \eqref{E:KEYPOINTWISEESTIMATE}
	was essentially proved in
	\cite{jSgHjLwW2016}*{Proposition 11.10},
	 based in part on
	estimates that are analogs of the estimates of
	Lemmas~\ref{L:POINTWISEFORRECTANGULARCOMPONENTSOFVECTORFIELDS} and \ref{L:POINTWISEESTIMATESFORGSPHEREANDITSDERIVATIVES},
	the $L^{\infty}$ estimates of Prop.~\ref{P:IMPROVEMENTOFAUX},
	the estimate \eqref{E:RADDERIVATIVESOFGLLDIFFERENCEBOUND},
	and the estimates of Prop.~\ref{P:SHARPMU}.
	We note that on RHS~\eqref{E:ERRORTERMKEYPOINTWISEESTIMATE},
	we have corrected a typo 
	that appeared in \cite{jSgHjLwW2016}. Specifically,
	the factor 
	$\left|
		\Rad \Tanset^N \Psi
	\right|
	$
	in the term
	$
	\displaystyle
	\varepsilon
				\frac{1}{\upmu_{\star}(t,u)}
				\int_{t'=0}^t
					\frac{1}{\upmu_{\star}(t',u)}
					\left|
						\Rad \Tanset^N \Psi
					\right|
					(t',u,\vartheta)
				\, dt'
	$
	on RHS~\eqref{E:ERRORTERMKEYPOINTWISEESTIMATE} was mistakenly listed as
	$
	\left|
		\Fullset_*^{\leq N+1;1} \Psi
	\right|
	$
	in \cite{jSgHjLwW2016}*{Equation (11.33)}.
	We also clarify that RHS~\eqref{E:KEYPOINTWISEESTIMATE}
	does not feature the order $0$ terms $|\Psi|$ or $|\GdVar|$.
	This is different compared to the analogous
	estimates stated in \cite{jSgHjLwW2016}*{Proposition 11.10},
	but follows from the proof given there
	(for reasons similar to the ones
	that we described in the discussion of the proofs of
		Lemmas~\ref{L:COMMUTATORESTIMATES} and \ref{L:TRANSVERALTANGENTIALCOMMUTATOR},
		Prop.\ \ref{P:IMPROVEMENTOFAUX}, 
		and Prop.\ \ref{P:IDOFKEYDIFFICULTENREGYERRORTERMS}).
	In addition, we note that in \cite{jSgHjLwW2016}*{Proposition 11.10}, 
	the coefficient in front of the analog of the second product on RHS~\eqref{E:KEYPOINTWISEESTIMATE}
	was stated as $\boxed{4}(1 + C \varepsilon)$ 
	rather than $\boxed{4}$.
	Here, we have relegated the $C \varepsilon$
	contribution to the error term
	$
	\displaystyle
	\varepsilon
				\frac{1}{\upmu_{\star}(t,u)}
				\int_{t'=0}^t
				\cdots
	$ 
	on the fourth line of RHS~\eqref{E:ERRORTERMKEYPOINTWISEESTIMATE}; this is a minor (essentially cosmetic) change 
	that is justified by the arguments given in the proof of \cite{jSgHjLwW2016}*{Proposition 11.10}.
	
	The only new term appearing on RHS~\eqref{E:KEYPOINTWISEESTIMATE}
	compared to \cite{jSgHjLwW2016}*{Proposition 11.10}
	is the last one on RHS~\eqref{E:ERRORTERMKEYPOINTWISEESTIMATE}
	(which involves the time integral of $\left|\Tanset^{\leq N} \bigslow \right|$),
	whose origin we now explain.
	To do this, we must explain some features of the proof of \eqref{E:KEYPOINTWISEESTIMATE},
	which relies on the ``modified'' version of $\mytr \upchi$ described in Subsubsect.\ \ref{SSS:ENERGYESTIMATES},
	namely the quantity $\upchifullmodarg{\GeoAng^N}$ defined in \eqref{E:TOPORDERMODIFIEDTRCHI}.
	As we explained in the discussion below equation \eqref{E:TRCHISCHEMATICEVOLUTION}, 
	we are forced to work with $\upchifullmodarg{\GeoAng^N}$
	in order to avoid losing a derivative at the top order.
	Specifically, to prove \eqref{E:KEYPOINTWISEESTIMATE}, 
	one first derives a transport equation for $\upchifullmodarg{\GeoAng^N}$
	(see the proof of \cite{jSgHjLwW2016}*{Lemma 11.9} for more details)
	of the form
	$
	\Lunit (\iota \upchifullmodarg{\GeoAng^N})
	= \cdots
	$,
	where $\iota$ is an appropriately defined integrating factor that verifies
	$\iota(s,u,\vartheta)
	=
	\displaystyle
	\left\lbrace
		1 + \mathcal{O}(\varepsilon)
	\right\rbrace
	\frac{\upmu^2(0,u,\vartheta)}{\upmu^2(s,u,\vartheta)},
	$
	and $\cdots$ contains, among other terms, 
	$
	\displaystyle
	\frac{1}{2}
	\iota 
	\GeoAng^N 
	\left(
		\upmu 
		G_{\Lunit \Lunit}
		\times \mbox{{\upshape RHS}~\eqref{E:FASTWAVE}}
	\right)
	$;
	the terms generated by 
	$
	\displaystyle
	\frac{1}{2}
	\iota 
	\GeoAng^N 
	\left(
		\upmu 
		G_{\Lunit \Lunit}
		\times \mbox{{\upshape RHS}~\eqref{E:FASTWAVE}}
	\right)
	$
	are the new ones compared to the terms found in \cite{jSgHjLwW2016}*{Proposition 11.10}.
	Using Lemmas~\ref{L:SCHEMATICDEPENDENCEOFMANYTENSORFIELDS} and \ref{L:CARTESIANVECTORFIELDSINTERMSOFGEOMETRICONES}
	and 
	our assumptions \eqref{E:SOMENONINEARITIESARELINEAR} on the semilinear inhomogeneous terms,
	we see that
	$
	\displaystyle
	\frac{1}{2}
	\iota 
	\GeoAng^N 
	\left(
		\upmu 
		G_{\Lunit \Lunit}
		\times \mbox{{\upshape RHS}~\eqref{E:FASTWAVE}}
	\right)
	=
	\iota
	\GeoAng^N
		\left\lbrace
			\smoothfunction(\GdVar,\bigslow,\Rad \Psi,\Singletan \Psi)\Singletan \Psi 
			+ 
			\smoothfunction(\BadVar,\bigslow,\Rad \Psi,\Singletan \Psi) \bigslow
		\right\rbrace
	$.
	Therefore, with the help of the $L^{\infty}$ estimates of Prop.~\ref{P:IMPROVEMENTOFAUX},
	we can pointwise bound these new terms in magnitude by
	$
	\displaystyle
	\lesssim
	\iota
	\left|
		\Fullset_*^{[1,N+1];1} \Psi
	\right|
	+
	\iota
	\left|
		\Tanset^{\leq N} \bigslow
	\right|
	+
	\iota
	\left|
		\Tanset_*^{[1,N]} \BadVar
	\right|
	$.
	Revisiting the proofs of \cite{jSgHjLwW2016}*{Lemma 11.9} and \cite{jSgHjLwW2016}*{Lemma 11.10}, 
	which are based on integrating the evolution equation 
	$
	\Lunit 
	(\iota \upchifullmodarg{\GeoAng^N})
	= \cdots
	$
	in time,
	we obtain, using the above pointwise bounds for the new terms in $\cdots$,
	the following estimate: 
	\[
	\left|
		\upmu \GeoAng^N \mytr \upchi
	\right|(t,u,\vartheta)
	\leq 
	C
	\left\lbrace
	\sup_{0 \leq t' \leq t}
		\frac{\upmu(t,u,\vartheta)}{\upmu(t',u,\vartheta)}
	\right\rbrace^2
	\times
	\int_{t'=0}^t 
		\left|
			\Tanset^{\leq N} \bigslow
		\right|(t',u,\vartheta)
	\, dt'
	+ \cdots,
	\]
	where the factor
	$
	\displaystyle
	\left\lbrace
	\sup_{0 \leq t' \leq t}
		\frac{\upmu(t,u,\vartheta)}{\upmu(t',u,\vartheta)}
	\right\rbrace^2
	$
	is generated by the integrating factor $\iota$ and
	$\cdots$ now denotes terms of the same type that appeared in the proof of \cite{jSgHjLwW2016}*{Lemma 11.9}.
	From Def.~\ref{D:REGIONSOFDISTINCTUPMUBEHAVIOR}
	and the estimates
	\eqref{E:LOCALIZEDMUCANTGROWTOOFAST}
	and
	\eqref{E:LOCALIZEDMUMUSTSHRINK},
	we find that
	\begin{align} \label{E:CRUDEMUOVERMUBOUND}
		\sup_{0 \leq t' \leq t}
		\frac{\upmu(t,u,\vartheta)}{\upmu(t',u,\vartheta)}
		& \leq C.
	\end{align}
	It therefore follows that
	\begin{align} \label{E:KEYDIFFICULTPRODUCTPROOFESTIMATE}
	\left|
		(\Rad \Psi) \GeoAng^N \mytr \upchi 
	\right|(t,u,\vartheta)
	& 
	\leq
	C
	\frac{1}{\upmu(t,u,\vartheta)}
	\left|
		\Rad \Psi
	\right|
	(t,u,\vartheta)
	\int_{t'=0}^t 
		\left|
			\Tanset^{\leq N} \bigslow
		\right|(t',u,\vartheta)
	\, dt'
	+ \cdots
		\\
	& 
	\leq
	C
	\frac{1}{\upmu_{\star}(t,u)}
	\int_{t'=0}^t 
		\left|
			\Tanset^{\leq N} \bigslow
		\right|(t',u,\vartheta)
	\, dt'
	+ \cdots,
	\notag
	\end{align}
	where $\cdots$ again denotes terms that appear in the proof of \cite{jSgHjLwW2016}*{Lemma 11.9}
	(and thus on RHS~\eqref{E:KEYPOINTWISEESTIMATE} as well)
	and to obtain the last inequality in \eqref{E:KEYDIFFICULTPRODUCTPROOFESTIMATE},
	we used the simple bound $\| \Rad \Psi \|_{L^{\infty}(\Sigma_t^u)} \lesssim 1$
	(that is, \eqref{E:PSITRANSVERSALLINFINITYBOUNDBOOTSTRAPIMPROVED}).
	This explains the origin of the last term on RHS~\eqref{E:ERRORTERMKEYPOINTWISEESTIMATE}
	and completes our proof outline of \eqref{E:KEYPOINTWISEESTIMATE}.
	
	Similarly, the estimate \eqref{E:LESSPRECISEKEYPOINTWISEESTIMATE}
	was essentially obtained in
	\cite{jSgHjLwW2016}*{Proposition 11.10}
	using ideas similar to but simpler than the ones used in the proof of \eqref{E:KEYPOINTWISEESTIMATE}.
	The new terms mentioned in the previous paragraph also make a contribution
	to RHS~\eqref{E:LESSPRECISEKEYPOINTWISEESTIMATE}
	for essentially the same reason that they appeared on
	RHS~\eqref{E:KEYDIFFICULTPRODUCTPROOFESTIMATE}.
	Specifically, 
	they lead to the last term
	on RHS~\eqref{E:LESSPRECISEKEYPOINTWISEESTIMATE}.
	In closing, we note that RHS~\eqref{E:LESSPRECISEKEYPOINTWISEESTIMATE}
	is less singular with respect to factors of 
	$
	\displaystyle
	\frac{1}{\upmu}
	$
	compared to RHS~\eqref{E:KEYPOINTWISEESTIMATE}.
	The reason is that
	LHS~\eqref{E:LESSPRECISEKEYPOINTWISEESTIMATE}
	has an extra factor of $\upmu$ in it compared to
	LHS~\eqref{E:KEYPOINTWISEESTIMATE}.
\end{proof}

\subsection{Pointwise estimates for the remaining terms in the energy estimates}
\label{SS:POINTWISEESTIMATESREMININGTERMSENERGYESTIMATES}
We now derive pointwise estimates for the energy estimates error integrands
$\basicenergyerrorarg{\Mult}{i}[f]$ from RHS~\eqref{E:E0DIVID}
and the error integrand
$
\left\lbrace 
					1 + \upgamma \smoothfunction(\upgamma)
				\right\rbrace 
				\slowbasicenergyerror[\altbigslow]
$
from \eqref{E:SLOWENERGYID}.

\begin{lemma}[\textbf{Pointwise bounds for the remaining error terms in the energy estimates}]
\label{L:MULTIPLIERVECTORFIEDERRORTERMPOINTWISEBOUND}
	Consider the error terms 
	$\basicenergyerrorarg{\Mult}{1}[f]$,
	$\cdots$,
	$\basicenergyerrorarg{\Mult}{5}[f]$
	defined in \eqref{E:MULTERRORINTEG1}-\eqref{E:MULTERRORINTEG5}.
	Let $\varsigma > 0$ be a real number.
	Then the following pointwise estimate holds
	without any absolute value taken on the left,
	where the implicit constants are independent of $\varsigma$:
	\begin{align} \label{E:MULTIPLIERVECTORFIEDERRORTERMPOINTWISEBOUND}
		\sum_{i=1}^5 \basicenergyerrorarg{\Mult}{i}[f]
		& \lesssim
			(1 + \varsigma^{-1})(\Lunit f)^2
			+ (1 + \varsigma^{-1}) (\Rad f)^2
			+ \upmu |\angdiff f|^2
			+ \varsigma \TranminusdatasizeWithFactor |\angdiff f|^2
				\\
		& \ \
				+ \frac{1}{\sqrt{\Tboot - t}} 
				\upmu |\angdiff f|^2.
				\notag
\end{align}

	In addition, we have the following pointwise estimate for the
	error integrand 
	$
	\left\lbrace 
					1 + \upgamma \smoothfunction(\upgamma)
				\right\rbrace 
				\slowbasicenergyerror[\altbigslow]
	$
	on RHS~\eqref{E:SLOWENERGYID},
	where $\slowbasicenergyerror[\altbigslow]$ is defined by \eqref{E:SLOWWAVEBASICENERGYINTEGRAND}:
	\begin{align} \label{E:SIMPLEPOINTWISEBOUNDSLOWWAVEBASICENERGYINTEGRAND}
		\left|
			\left\lbrace 
					1 + \upgamma \smoothfunction(\upgamma)
			\right\rbrace 
			\slowbasicenergyerror[\altbigslow]		
		\right|
		& \lesssim |\altbigslow|^2.
	\end{align}

\end{lemma}

\begin{remark}	
	\label{R:NEEDTHEESTIMATEWITHTANSETNPSIINPLACEOFPSI}
	In deriving energy estimates, we will rely on the estimate
	\eqref{E:MULTIPLIERVECTORFIEDERRORTERMPOINTWISEBOUND} with
	$\Tanset^N \Psi$ in the role of $f$
	and the estimate 
	\eqref{E:SIMPLEPOINTWISEBOUNDSLOWWAVEBASICENERGYINTEGRAND}
	with
	$\Tanset^N \bigslow$ in the role of $\altbigslow$.
\end{remark}

\begin{proof}
	See Subsect.\ \ref{SS:OFTENUSEDESTIMATES} for some comments on the analysis.
	We first prove \eqref{E:MULTIPLIERVECTORFIEDERRORTERMPOINTWISEBOUND}.
	Only the term $\basicenergyerrorarg{\Mult}{3}[f]$
	is difficult to treat.
	Specifically, using 
	\eqref{E:TENSORSDEPENDINGONGOODVARIABLESGOODPSIDERIVATIVES},
	\eqref{E:TENSORSDEPENDINGONGOODVARIABLESBADDERIVATIVES},
	\eqref{E:ANGDIFFXI},
	and the $L^{\infty}$ estimates
	of Props.~\ref{P:IMPROVEMENTOFAUX} and \ref{P:IMPROVEMENTOFHIGHERTRANSVERSALBOOTSTRAP},
	it is straightforward to verify that the terms in braces on
	RHSs \eqref{E:MULTERRORINTEG1}, \eqref{E:MULTERRORINTEG2}, \eqref{E:MULTERRORINTEG4}, and
	\eqref{E:MULTERRORINTEG5} are bounded in magnitude by $\lesssim 1$.
	It follows that for $i=1,2,4,5$,
	$\left|
		\basicenergyerrorarg{\Mult}{i}[f]
	\right|
	$ 
	is $\lesssim$ the terms on the first line of
	RHS~\eqref{E:MULTIPLIERVECTORFIEDERRORTERMPOINTWISEBOUND}.
	The quantities $\varsigma$ and $\TranminusdatasizeWithFactor$
	appear on RHS~\eqref{E:MULTIPLIERVECTORFIEDERRORTERMPOINTWISEBOUND} 
	because we use Young's inequality to bound
	$\basicenergyerrorarg{\Mult}{4}[f] \lesssim |\Lunit f||\angdiff f| 
	\leq \varsigma^{-1} \TranminusdatasizeWithFactor^{-1}(\Lunit f)^2 
	+ \varsigma \TranminusdatasizeWithFactor |\angdiff f|^2
	\leq C \varsigma^{-1} (\Lunit f)^2 + \varsigma \TranminusdatasizeWithFactor |\angdiff f|^2
	$.
	Similar remarks apply to $\basicenergyerrorarg{\Mult}{5}[f]$.
		To bound
	$\basicenergyerrorarg{\Mult}{3}[f]$,
	we also use 
	\eqref{E:POSITIVEPARTOFLMUOVERMUISBOUNDED}
	and 
	\eqref{E:UNIFORMBOUNDFORMRADMUOVERMU},
	which allow us to bound the
	first two terms in braces on RHS~\eqref{E:MULTERRORINTEG3}.
	Note that since no absolute value 
	is taken on LHS~\eqref{E:MULTIPLIERVECTORFIEDERRORTERMPOINTWISEBOUND},
	we can replace the factor 
	$(\Rad \upmu)/\upmu$ 
	from RHS~\eqref{E:MULTERRORINTEG3}
	with the factor $[\Rad \upmu]_+/\upmu$,
	which is bounded by \eqref{E:UNIFORMBOUNDFORMRADMUOVERMU}.
	This completes our proof of \eqref{E:MULTIPLIERVECTORFIEDERRORTERMPOINTWISEBOUND}.
	
	To prove \eqref{E:SIMPLEPOINTWISEBOUNDSLOWWAVEBASICENERGYINTEGRAND}, we 
	first use
	Lemmas \ref{L:SCHEMATICDEPENDENCEOFMANYTENSORFIELDS} and \ref{L:CARTESIANVECTORFIELDSINTERMSOFGEOMETRICONES}
	to deduce that
	$
	\displaystyle
	\upmu \partial_{\kappa}
	\left(
		(h^{-1})^{\alpha \beta}(\Psi,\bigslow)
	\right)
	= \smoothfunction(\BadVar) \Singletan \Psi 
		+
		\smoothfunction(\GdVar) \Rad \Psi
		+
		\smoothfunction(\BadVar) \Singletan \bigslow
		+
		\smoothfunction(\GdVar) \Rad \bigslow
	$.
	From this schematic identity
	and
	the $L^{\infty}$ estimates of Prop.~\ref{P:IMPROVEMENTOFAUX},
	we obtain the bounds
	$
	\displaystyle
	\left|
		(h^{-1})^{\alpha \beta}(\Psi,\bigslow)
	\right|
	\lesssim 1
	$,
	$
	\displaystyle
	\left|
		\upmu \partial_{\kappa} 
			\left(
				(h^{-1})^{\alpha \beta}(\Psi,\bigslow)
			\right)
	\right|
	\lesssim 1
	$,
	$
	\displaystyle
	\left|
		1 + \upgamma \smoothfunction(\upgamma)
	\right|
	\lesssim 1
	$,
	and $\upmu \lesssim 1$,
	from which the desired estimate \eqref{E:SIMPLEPOINTWISEBOUNDSLOWWAVEBASICENERGYINTEGRAND} easily follows.
\end{proof}

\section{Energy estimates and improvements of the fundamental \texorpdfstring{$L^{\infty}$}{essential sup-norm} bootstrap assumptions}
\label{S:ENERGYESTIMATES}
In this section, we derive the main estimates of this article: a priori $L^2$ estimates for the solution up to top order.
As a simple corollary, we will also derive
strict improvements of the fundamental $L^{\infty}$ bootstrap assumptions
\eqref{E:PSIFUNDAMENTALC0BOUNDBOOTSTRAP}.
Remark~\ref{R:ESTIMATESPROVEDINOTHERPAPER} especially applies in this section.

\subsection{Definitions of the fundamental \texorpdfstring{$L^2$}{square integral}-controlling quantities}
\label{SS:SQUAREINTEGRALCONTNROLLINGQUANT}
In this subsection, we define the quantities that we use to control the solution in $L^2$
up to top order.

\begin{definition}[\textbf{The main coercive quantities used for controlling the solution and its derivatives in} $L^2$]
\label{D:MAINCOERCIVEQUANT}
In terms of the energies and null fluxes of Defs.~\ref{D:ENERGYFLUX} and \ref{D:SLOWWAVEENERGYFLUX},
we define
\begin{subequations}
\begin{align}
	\totTanmax{N}(t,u)
	& := \max_{|\vec{I}| = N}
		\sup_{(t',u') \in [0,t] \times [0,u]} 
		\left\lbrace
			\enzero[\Tanset^{\vec{I}} \Psi](t',u')
			+ 
			\flzero[\Tanset^{\vec{I}} \Psi](t',u')
		\right\rbrace,
			\label{E:Q0TANNDEF} 
			\\
	\totTanmax{[1,N]}(t,u)
	& := \max_{1 \leq M \leq N} \totTanmax{M}(t,u),
		\label{E:MAXEDQ0TANLEQNDEF} 
			\\
\slowtotTanmax{N}(t,u)
	& := \max_{|\vec{I}| = N}
		\sup_{(t',u') \in [0,t] \times [0,u]} 
		\left\lbrace
			\slowen[\Tanset^{\vec{I}} \bigslow](t',u')
				+ 
			\slowfl[\Tanset^{\vec{I}} \bigslow](t',u')
		\right\rbrace,
			\label{E:SLOWQ0TANNDEF} 
			\\
	\slowtotTanmax{\leq N}(t,u)
	& := \max_{M \leq N} \slowtotTanmax{M}(t,u).
		\label{E:MAXEDSLOWQ0TANLEQNDEF} 
\end{align}
\end{subequations}

\end{definition}

We use the following coercive spacetime integrals 
to control non-$\upmu$-weighted error integrals involving geometric torus derivatives.
These integrals are generated by the term
$
\displaystyle
	- 
				\frac{1}{2} 
				\int_{\mathcal{M}_{t,u}}
				[\Lunit \upmu]_- |\angdiff f|^2 
$
on the RHS of the fast wave energy identity \eqref{E:E0DIVID}.

\begin{definition}[\textbf{Key coercive spacetime integrals}]
\label{D:COERCIVEINTEGRAL}
	We associate the following integrals to $\Psi$,
	where $[\Lunit \upmu]_- = |\Lunit \upmu|$ when $\Lunit \upmu < 0$
	and $[\Lunit \upmu]_- = 0$ when $\Lunit \upmu \geq 0$:
	\begin{subequations}
	\begin{align} \label{E:COERCIVESPACETIMEDEF} 
		\coercivespacetime[\Psi](t,u)
		& :=
	 	\frac{1}{2}
	 	\int_{\mathcal{M}_{t,u}}
			[\Lunit \upmu]_-
			|\angdiff \Psi|^2
		\, d \vol, 
				\\
		\coerciveTanspacetimemax{N}(t,u) 
		& := \max_{|\vec{I}| = N} \coercivespacetime[\Tanset^{\vec{I}} \Psi](t,u),
			\\
		\coerciveTanspacetimemax{[1,N]}(t,u) 
		& := \max_{1 \leq M \leq N} \coerciveTanspacetimemax{M}(t,u).
		\label{E:MAXEDCOERCIVESPACETIMEDEF}
	\end{align}
	\end{subequations}
\end{definition}

\begin{remark}[\textbf{The energies vanish for simple plane wave solutions}]
\label{R:ENERGIESVANISHFORSIMPLEPLANEWAVE}
	Note that for simple outgoing plane wave solutions, if $N \geq 1$, then
	$\totTanmax{[1,N]}(t,u) \equiv 0$,
	$\slowtotTanmax{\leq N}(t,u) \equiv 0$,
	and
	$\coerciveTanspacetimemax{[1,N]}(t,u) \equiv 0$.
	Hence, in some sense, $\totTanmax{[1,N]}$ and $\slowtotTanmax{\leq N}$
	measure the extent to which the solution deviates from a simple outgoing plane wave.
	These facts are tied to Lemma~\ref{L:INITIALSIZEOFL2CONTROLLING} below and
	to the fact that $\mathring{\upepsilon} = 0$ for
	simple outgoing plane wave solutions. 
\end{remark}

\subsection{The coerciveness of the fundamental \texorpdfstring{$L^2$}{square integral}-controlling quantities}
\label{SS:COERCIVENESSOFL2CONTROLLING}

\subsubsection{Preliminary lemmas}
\label{SSS:COERCIVENESSPRELIMINARYLEMMAS}
In this subsection, we quantify the coerciveness of the $L^2$-controlling quantities that we defined
in Subsect.\ \ref{SS:SQUAREINTEGRALCONTNROLLINGQUANT}.
We start with a lemma that provides an identity for the time derivative of integrals
over the tori $\ell_{t,u}$.

\begin{lemma}\cite{jSgHjLwW2016}*{Lemma 3.6; \textbf{Identity for the time derivative of} $\ell_{t,u}$ \textbf{integrals}}
\label{L:LDERIVATIVEOFLINEINTEGRAL}
	The following identity holds for scalar-valued functions $f$:
	\begin{align} \label{E:LDERIVATIVEOFLINEINTEGRAL}
		\frac{\partial}{\partial t}
		\int_{\ell_{t,u}}
			f
		\, d \spherevol
		& 
		= 
		\int_{\ell_{t,u}}
			\left\lbrace
				\Lunit f
				+ 
				\mytr \upchi f
			\right\rbrace
		\, d \spherevol.
	\end{align}
\end{lemma}

In the next lemma, we derive estimates for the
metric component $\gtancomp$ defined in \eqref{E:METRICANGULARCOMPONENT}.

\begin{lemma}[\textbf{Pointwise estimates for} $\gtancomp$]
\label{L:POINTWISEESTIMATEFORGTANCOMP}
Let $\gtancomp > 0$ be the scalar function defined in \eqref{E:METRICANGULARCOMPONENT}.
The following estimates hold
(see Subsect.\ \ref{SS:NOTATIONANDINDEXCONVENTIONS} regarding our use of the notation $\mathcal{O}_{\mydiam}(\cdot)$):
\begin{align} \label{E:MOREPRECISEPOINTWISEESTIMATEFORGTANCOMP}  
	\gtancomp(t,u,\vartheta) 
		& = \left\lbrace 
					1 + \mathcal{O}(\varepsilon)
				\right\rbrace
				\gtancomp(0,u,\vartheta)
		= 1 + \mathcal{O}_{\mydiam}(\Psiep) + \mathcal{O}(\varepsilon).
\end{align}
\end{lemma}

\begin{proof}
	See Subsect.\ \ref{SS:OFTENUSEDESTIMATES} for some comments on the analysis.
	Using \eqref{E:LDERIVATIVEOFVOLUMEFORMFACTOR} and 
	\eqref{E:PURETANGENTIALCHICOMMUTEDLINFINITY}, we deduce
	$\Lunit \ln \gtancomp = \mathcal{O}(\varepsilon)$.
	Integrating in time,
	we deduce $\ln \gtancomp(t,u,\vartheta) = \ln \gtancomp(0,u,\vartheta) + \mathcal{O}(\varepsilon)$,
	which yields the first equality in \eqref{E:MOREPRECISEPOINTWISEESTIMATEFORGTANCOMP}.
	The second equality in
	\eqref{E:MOREPRECISEPOINTWISEESTIMATEFORGTANCOMP} then follows from
	the first one and the estimate
	$\gtancomp(0,u,\vartheta) = 1 + \mathcal{O}_{\mydiam}(\Psiep)$,
	which we now derive.
	To this end, we first note that by construction,
	we have $\CoordAng|_{t=0} = \partial_2$.
	Hence, from
	\eqref{E:LITTLEGDECOMPOSED}-\eqref{E:METRICPERTURBATIONFUNCTION}
	and
	\eqref{E:PSIITSELFLINFTYSMALLDATAASSUMPTIONSALONGSIGMA0},
	we conclude that
	$
	\gtancomp^2|_{t=0} 
	= g(\CoordAng,\CoordAng)|_{t=0} 
	= g_{22}|_{t=0}
	=
	1 + \mathcal{O}_{\mydiam}(\Psiep)
	$
	as desired.
	\end{proof}

In the next lemma, we compare various integrals 
that are computed with respect to forms evaluated at different times.

\begin{lemma}[\textbf{Comparison results for integrals}]
\label{L:LINEVOLUMEFORMCOMPARISON}
Let $p=p(\vartheta)$ be a non-negative function of $\vartheta$. 
Then the following estimates hold
for $(t,u) \in [0,\Tboot) \times [0,U_0]$:
\begin{align} \label{E:LINEVOLUMEFORMCOMPARISON}
	\int_{\vartheta \in \mathbb{T}} p(\vartheta) d \argspherevol{(0,u,\vartheta)}
	& 
	=
	\left\lbrace 
		1 + \mathcal{O}(\varepsilon)
	\right\rbrace
	\int_{\ell_{t,u}} p(\vartheta) d \argspherevol{(t,u,\vartheta)}.
\end{align}

Furthermore, let $p=p(u',\vartheta)$ be a non-negative function of $(u',\vartheta) \in [0,u] \times \mathbb{T}$
that \textbf{does not depend on $t$}.
Then for $s, t \in [0,\Tboot)$ and $u \in [0,U_0]$, 
we have:
\begin{align} \label{E:SIGMATVOLUMEFORMCOMPARISON}
	\int_{\Sigma_s^u} p \, d \tvol
	& =
		\left\lbrace
			1 + \mathcal{O}(\varepsilon)
		\right\rbrace
	\int_{\Sigma_t^u} p \, d \tvol.
\end{align}

\end{lemma}

\begin{proof}
	From \eqref{E:RESCALEDVOLUMEFORMS} and the first equality in \eqref{E:MOREPRECISEPOINTWISEESTIMATEFORGTANCOMP},
	we deduce that 
	$d \argspherevol{(t,u,\vartheta)} = 
		\left\lbrace 
			1 + \mathcal{O}(\varepsilon) 
		\right\rbrace d \argspherevol{(0,u,\vartheta)}$,
	which yields \eqref{E:LINEVOLUMEFORMCOMPARISON}.
	\eqref{E:SIGMATVOLUMEFORMCOMPARISON} then follows from
	\eqref{E:LINEVOLUMEFORMCOMPARISON} and the fact that
	$d \tvol(t,u',\vartheta) = d \argspherevol{(t,u',\vartheta)} du'$
	along $\Sigma_t^u$.
\end{proof}

We now provide a simple variant of Minkowski's inequality for integrals.

\begin{lemma}[\textbf{Estimate for the norm} $\| \cdot \|_{L^2(\Sigma_t^u)}$ \textbf{of time-integrated functions}] 
\label{L:L2NORMSOFTIMEINTEGRATEDFUNCTIONS}
Let $f$ be a scalar function and set
$F(t,u,\vartheta) := \int_{t'=0}^t f(t',u,\vartheta) \, dt'$.
The following estimate holds:
\begin{align} \label{E:L2NORMSOFTIMEINTEGRATEDFUNCTIONS}
	\| F \|_{L^2(\Sigma_t^u)} 
	& \leq 
		(1 + C \varepsilon)
		\int_{t'=0}^t 
			\| f \|_{L^2(\Sigma_{t'}^u)}
		\, dt'.
\end{align}
\end{lemma}

\begin{proof}
	Recall that 
	$\| F \|_{L^2(\Sigma_t^u)}
	: = 
				\left\lbrace
				\int_{u'=0}^u
					\int_{\ell_{t,u'}}
						F^2(t,u',\vartheta) 
					\, d \argspherevol{(t,u',\vartheta)}
				\, du'
				\right\rbrace^{1/2}
		$.
	Using \eqref{E:LINEVOLUMEFORMCOMPARISON},
	we deduce that for $0 \leq t' \leq t$, 
	we have
	$d \argspherevol{(t',u',\vartheta)}
	= 
	\left\lbrace 
		1 
		+
		\mathcal{O}(\varepsilon) 
	\right\rbrace 
	d \argspherevol{(0,u',\vartheta)}$. 
	\eqref{E:L2NORMSOFTIMEINTEGRATEDFUNCTIONS} follows from this estimate and
	from applying Minkowski's inequality 
	for integrals 
	(with respect to the measure $d \argspherevol{(0,u',\vartheta)} \, du'$)
	to the equation defining $F$.
\end{proof}

\subsubsection{The coerciveness of the fundamental $L^2$-controlling quantities}
\label{SSS:COERCIVENESSOFL2CONTROLLING}
We now provide the main lemma of Subsect.\ \ref{SS:COERCIVENESSOFL2CONTROLLING}.

\begin{lemma}[\textbf{The coerciveness of the fundamental} $L^2$-\textbf{controlling quantities}]
	\label{L:COERCIVENESSOFCONTROLLING}
	Let $1 \leq M \leq N \leq 18$, and let
	$\Tanset^M$ be an $M^{th}$-order $\mathcal{P}_u$-tangential vectorfield operator.
	We have the following bounds
	for $(t,u) \in [0,\Tboot) \times [0,U_0]$:
	\begin{align} \label{E:COERCIVENESSOFCONTROLLING}
			\totTanmax{[1,N]}(t,u)
			\geq 
			\max
			\Big\lbrace
				&
				\frac{1}{2}
				\left\|
					\sqrt{\upmu} \Lunit \Tanset^M \Psi
				\right\|_{L^2(\Sigma_t^u)}^2,
					\,
				\left\|
					\Rad \Tanset^M \Psi
				\right\|_{L^2(\Sigma_t^u)}^2,
					\,
				\frac{1}{2}
				\left\|
					\sqrt{\upmu} \angdiff \Tanset^M \Psi
				\right\|_{L^2(\Sigma_t^u)}^2,
				\\
			& 
				\left\|
					\Lunit \Tanset^M \Psi
				\right\|_{L^2(\mathcal{P}_u^t)}^2,
					\,
				\left\|
					\sqrt{\upmu} \angdiff \Tanset^M \Psi
				\right\|_{L^2(\mathcal{P}_u^t)}^2
			\Big\rbrace.
				\notag
	\end{align}

	Moreover, if $1 \leq M \leq N \leq 18$, then the following bounds hold:
	\begin{subequations}
	\begin{align} \label{E:PSIHIGHERORDERL2ESTIMATELOSSOFONEDERIVATIVE} 
		\left\|
			\Tanset^M \Psi
		\right\|_{L^2(\Sigma_t^u)},
			\,
		\left\|
			\Tanset^M \Psi
		\right\|_{L^2(\ell_{t,u})}
		& \leq
			C \mathring{\upepsilon}
			+ 
			C
			\totTanmax{[1,N]}^{1/2}(t,u),
				\\
		\left\|
			\Tanset^M \Psi
		\right\|_{L^2(\Sigma_t^u)}
		& \leq 
			C \mathring{\upepsilon}
			+
			C
			\int_{t'=0}^{t}
				\frac{1}{\upmu_{\star}^{1/2}(t',u)}
				\totTanmax{[1,N]}^{1/2}(t',u)
			\, dt'.
			\label{E:ANOTHERPSIHIGHERORDERL2ESTIMATELOSSOFONEDERIVATIVE}
	\end{align}
	\end{subequations}

In addition,
with $\mathbf{1}_{\lbrace \upmu \leq 1/4 \rbrace}$
denoting the characteristic function of the spacetime
subset 
$
\displaystyle
\lbrace (t,u,\vartheta) \in [0,\infty) \times [0,U_0] \times \mathbb{T}
	\ | \ 
\upmu(t,u,\vartheta) \leq 1/4 \rbrace
$, 
then for $1 \leq M \leq N \leq 18$,
we have the following bound:
	\begin{align} \label{E:KEYSPACETIMECOERCIVITY}
		\coerciveTanspacetimemax{[1,N]}(t,u) 
		& \geq 
		\frac{1}{8}
		\TranminusdatasizeWithFactor
		\int_{\mathcal{M}_{t,u}}
			\mathbf{1}_{\lbrace \upmu \leq 1/4 \rbrace}
			\left|
				\angdiff \Tanset^M \Psi
			\right|^2
		\, d \vol.
	\end{align}
	
	In addition, for $M \leq N \leq 18$, we have the following bounds:
	\begin{align} \label{E:SLOWCOERCIVENESSOFCONTROLLING}
		\slowtotTanmax{\leq N}(t,u)
		\geq
		\frac{1}{C}
		\left\|
			\sqrt{\upmu} \Tanset^M \bigslow
		\right\|_{L^2(\Sigma_t^u)}^2
		+ 
		\frac{1}{C}
		\left\|
			\Tanset^M \bigslow
		\right\|_{L^2(\mathcal{P}_u^t)}^2.
	\end{align}

	Finally, for $0 \leq M \leq N-1 \leq 17$, 
	we have the following bounds:
	\begin{subequations}
	\begin{align} \label{E:ELLTUSLOWCOERCIVENESSOFCONTROLLING}
		\left\|
			\Tanset^M \bigslow
		\right\|_{L^2(\Sigma_t^u)},
			\,
		\left\|
			\Tanset^M \bigslow
		\right\|_{L^2(\ell_{t,u})}
		& \leq
		C \mathring{\upepsilon}
		+
		C
		\slowtotTanmax{\leq N}^{1/2}(t,u),
			\\
		\left\|
			\Tanset^M \bigslow
		\right\|_{L^2(\Sigma_t^u)}
		& \leq
		C \mathring{\upepsilon}
		+
		C
			\int_{t'=0}^{t}
				\frac{1}{\upmu_{\star}^{1/2}(t',u)}
				\slowtotTanmax{\leq N}^{1/2}(t',u)
			\, dt'.
			\label{E:ANOTHERLOWCOERCIVENESSOFCONTROLLING}
	\end{align}
	\end{subequations}
	
	\end{lemma}

\begin{proof}
\eqref{E:COERCIVENESSOFCONTROLLING} follows as a straightforward consequence of
definition \eqref{E:ENERGYORDERZEROCOERCIVENESS}, Young's inequality, and
definition \eqref{E:MAXEDQ0TANLEQNDEF}.

\eqref{E:KEYSPACETIMECOERCIVITY} follows from
the estimate \eqref{E:SMALLMUIMPLIESLMUISNEGATIVE}
and definition \eqref{E:MAXEDCOERCIVESPACETIMEDEF}.

The estimate \eqref{E:SLOWCOERCIVENESSOFCONTROLLING}
is a straightforward consequence of
Lemma~\ref{L:COERCIVENESSOFSLOWWAVEENERGIESANDFLUXES}
and
definition \eqref{E:MAXEDSLOWQ0TANLEQNDEF}.

To prove \eqref{E:PSIHIGHERORDERL2ESTIMATELOSSOFONEDERIVATIVE},
we first note that it suffices to obtain the desired estimate
for 
$\left\|
	\Tanset^M \Psi
\right\|_{L^2(\ell_{t,u})}
$.
The reason is that we can integrate the corresponding estimate for
$\left\|
	\Tanset^M \Psi
\right\|_{L^2(\ell_{t,u})}^2
$
with respect to $u$ to obtain the desired bound for
$
\left\|
	\Tanset^M \Psi
\right\|_{L^2(\Sigma_t^u)}^2
$.
To obtain the desired bound \eqref{E:PSIHIGHERORDERL2ESTIMATELOSSOFONEDERIVATIVE} for 
$\left\|
	\Tanset^M \Psi
\right\|_{L^2(\ell_{t,u})}
$,
we first use \eqref{E:LDERIVATIVEOFLINEINTEGRAL}
with 
$f= \left| \Tanset^M \Psi \right|^2$,
Young's inequality, and the estimate \eqref{E:PURETANGENTIALCHICOMMUTEDLINFINITY}
to obtain
\begin{align} \label{E:FIRSTSTEPOORDER0FASTCOERCIVENESSOFCONTROLLING}
		\frac{\partial}{\partial t}
		\int_{\ell_{t,u}}
			\left| \Tanset^M \Psi \right|^2
		\, d \argspherevol{(t,u',\vartheta)}
		& 
		= 
		\int_{\ell_{t,u}}
			\left\lbrace
				\Lunit \left( \left| \Tanset^M \Psi \right|^2 \right)
				+ 
				\mytr \upchi \left| \Tanset^M \Psi \right|^2
			\right\rbrace
		\, d \argspherevol{(t,u',\vartheta)}
			\\
		& \leq 
		\int_{\ell_{t,u}}
			\left| \Lunit \Tanset^M \Psi \right|^2
		\, d \argspherevol{(t,u',\vartheta)}
		+
		C
		\int_{\ell_{t,u}}
				\left| \Tanset^M \Psi \right|^2
		\, d \argspherevol{(t,u',\vartheta)}.
		\notag
	\end{align}
Integrating \eqref{E:FIRSTSTEPOORDER0FASTCOERCIVENESSOFCONTROLLING} with respect to time 
starting from time $0$, we obtain
\begin{align} \label{E:SECONDSTEPOORDER0FASTCOERCIVENESSOFCONTROLLING}
	\left\| \Tanset^M \Psi \right\|_{L^2(\ell_{t,u})}^2
	& \leq
	\left\| \Tanset^M \Psi \right\|_{L^2(\ell_{0,u})}^2
	+
	\left\| \Lunit \Tanset^M \Psi \right\|_{L^2(\mathcal{P}_u^t)}^2
	+
	C
	\int_{t'=0}^t
		\left\| \Tanset^M \Psi \right\|_{L^2(\ell_{t',u})}^2
	\, dt'.
\end{align}
From the small-data assumption \eqref{E:PSIL2SMALLDATAASSUMPTIONSALONGL0U}, 
we see that the first term on RHS~\eqref{E:SECONDSTEPOORDER0FASTCOERCIVENESSOFCONTROLLING} is 
$\lesssim \mathring{\upepsilon}^2$, while from \eqref{E:COERCIVENESSOFCONTROLLING},
we see that the second term $\left\| \Lunit \Tanset^M \Psi \right\|_{L^2(\mathcal{P}_u^t)}^2$
is $\lesssim \totTanmax{M}(t,u) \lesssim \totTanmax{[1,N]}(t,u)$. 
Using these bounds and applying Gronwall's inequality
to \eqref{E:SECONDSTEPOORDER0FASTCOERCIVENESSOFCONTROLLING}, 
we conclude the desired estimate
$
\left\| \Tanset^M \Psi \right\|_{L^2(\ell_{t,u})}^2
\lesssim 
\mathring{\upepsilon}^2
+
\totTanmax{[1,N]}(t,u)
$.

To prove \eqref{E:ANOTHERPSIHIGHERORDERL2ESTIMATELOSSOFONEDERIVATIVE},
we first use the fundamental theorem of calculus to express
$
\Tanset^M \Psi(t,u,\vartheta)
=
\Tanset^M \Psi(0,u,\vartheta)
+
\int_{t'=0}^t
	\Lunit \Tanset^M \Psi(t',u,\vartheta)
\, dt'
$.
The desired estimate now follows from this identity,
\eqref{E:SIGMATVOLUMEFORMCOMPARISON} with $s=0$
(which implies that
$
\| 
	\Tanset^M \Psi(0,\cdot)
\|_{L^2(\Sigma_t^u)}
=
\left\lbrace
	1 + \mathcal{O}(\varepsilon)
\right\rbrace
\| 
	\Tanset^M \Psi
\|_{L^2(\Sigma_0^u)}
$),
Lemma~\ref{L:L2NORMSOFTIMEINTEGRATEDFUNCTIONS},
the small-data assumption \eqref{E:PSIL2SMALLDATAASSUMPTIONSALONGSIGMA0},
and \eqref{E:COERCIVENESSOFCONTROLLING}.
To prove \eqref{E:ANOTHERLOWCOERCIVENESSOFCONTROLLING},
we use a similar argument based on the small-data assumption 
\eqref{E:SLOWL2SMALLDATAASSUMPTIONSALONGSIGMA0}
and \eqref{E:SLOWCOERCIVENESSOFCONTROLLING}.

Finally, we note that
\eqref{E:ELLTUSLOWCOERCIVENESSOFCONTROLLING} 
follows from arguments similar to the ones that we used to prove
the estimate
\eqref{E:PSIHIGHERORDERL2ESTIMATELOSSOFONEDERIVATIVE}
for 
$\left\|
	\Tanset^M \Psi
\right\|_{L^2(\ell_{t,u})}
$
together with the small-data assumption \eqref{E:SLOWL2SMALLDATAASSUMPTIONSALONGL0U}.

\end{proof}

\subsubsection{The initial smallness of the fundamental $L^2$-controlling quantities}
\label{SSS:INITIALSMALLNESSOFL2CONTROLLING}
The next lemma shows that the fundamental $L^2$-controlling quantities
of Def.~\ref{D:MAINCOERCIVEQUANT} are initially small.

\begin{lemma}[\textbf{The fundamental controlling quantities are initially small}]
\label{L:INITIALSIZEOFL2CONTROLLING}
	Assume that $1 \leq N \leq 18$.
	Under the data-size assumptions
	of Subsect.\ \ref{SS:SIZEOFTBOOT},
	the following estimates hold for 
	$(t,u) \in [0,2 \TranminusdatasizeWithFactor^{-1}] \times [0,U_0]$
	(see also Remark~\ref{R:ENERGIESVANISHFORSIMPLEPLANEWAVE}):
	\begin{subequations}
	\begin{align} \label{E:INITIALSIZEOFL2CONTROLLING}
		\totTanmax{[1,N]}(0,u),
			\, 
		\totTanmax{[1,N]}(t,0) \lesssim \mathring{\upepsilon}^2,
			\\
		\slowtotTanmax{\leq N}(0,u), 
			\,
		\slowtotTanmax{\leq N}(t,0) \lesssim \mathring{\upepsilon}^2.
		\label{E:SLOWWAVEINITIALSIZEOFL2CONTROLLING}
	\end{align}
	\end{subequations}
\end{lemma}

\begin{proof}
	We first note that by
	\eqref{E:UPITSELFLINFINITYSIGMA0CONSEQUENCES}
	and \eqref{E:UPITSELFLINFINITYP0CONSEQUENCES},
	we have $\upmu \approx 1$ along $\Sigma_0^1$
	and along $\mathcal{P}_0^{2 \TranminusdatasizeWithFactor^{-1}}$.
	Using these estimates,
	\eqref{E:ENERGYORDERZEROCOERCIVENESS}-\eqref{E:NULLFLUXENERGYORDERZEROCOERCIVENESS},
	\eqref{E:SLOWSIGMATENERGYCOERCIVENESS}-\eqref{E:SLOWNULLFLUXCOERCIVENESS},
	and
	Def.~\ref{D:MAINCOERCIVEQUANT},
	we see that
	\begin{align}
		\totmax{[1,18]}(0,1)
		& \lesssim 
		\left\|
			\Fullset_*^{[1,19];1} \Psi
		\right\|_{L^2(\Sigma_0^1)}^2,
	&
	\slowtotTanmax{\leq 18}(0,1)
	& \lesssim 
		\left\|
			\Tanset^{\leq 18} \bigslow
		\right\|_{L^2(\Sigma_0^1)}^2,
			\label{E:CONTROLLINGQUANTITIESBOUNDEDALONGSIMGMA0} \\
	\totmax{[1,18]}(2 \TranminusdatasizeWithFactor^{-1},0)
	& \lesssim 
		\left\|
			\Tanset^{[1,19]} \Psi
		\right\|_{L^2(\mathcal{P}_0^{2 \TranminusdatasizeWithFactor^{-1}})}^2,
	&
	\slowtotTanmax{\leq 18}(2 \TranminusdatasizeWithFactor^{-1},0)
	& \lesssim 
		\left\|
			\Tanset^{\leq 18} \bigslow
		\right\|_{L^2(\mathcal{P}_0^{2 \TranminusdatasizeWithFactor^{-1}})}^2.
		\label{E:CONTROLLINGQUANTITIESBOUNDEDALONGP0}
	\end{align}
	The estimates \eqref{E:INITIALSIZEOFL2CONTROLLING}-\eqref{E:SLOWWAVEINITIALSIZEOFL2CONTROLLING}
	now follow from \eqref{E:CONTROLLINGQUANTITIESBOUNDEDALONGSIMGMA0}-\eqref{E:CONTROLLINGQUANTITIESBOUNDEDALONGP0}
	and the data assumptions 
	\eqref{E:PSIL2SMALLDATAASSUMPTIONSALONGSIGMA0},
	\eqref{E:SLOWL2SMALLDATAASSUMPTIONSALONGSIGMA0},
	\eqref{E:PSIL2SMALLDATAASSUMPTIONSALONGP0},
	and 
	\eqref{E:SLOWL2SMALLDATAASSUMPTIONSALONGP0}.
\end{proof}

\subsection{The main a priori energy estimates}
\label{SS:MAINAPRIORIENERGYESTIMATES}
In this subsection, we state our main a priori energy estimates.
The main step in their proof is deriving
$L^2$ estimates for the error terms in the
commuted equations; we carry out this technical analysis in later subsections.

\subsubsection{The system of integral inequalities verified by the energies}
\label{SSS:ENERGYINTEGRALINEQUALITIES}
We start with a proposition in which we provide the system of integral inequalities
verified by the energies. Its proof is located in 
Subsect.\ \ref{SS:PROOFOFPROPTANGENTIALENERGYINTEGRALINEQUALITIES}.

\begin{proposition}[\textbf{Integral inequalities for the fundamental $L^2$-controlling quantities}]
	\label{P:TANGENTIALENERGYINTEGRALINEQUALITIES}
	Assume that $1 \leq N \leq 18$ and let $\varsigma > 0$ be a real number.
	
	\medskip
	
	\noindent \underline{\textbf{Integral inequalities relevant for top-order energy estimates.}}
	Let $\Sigmaminus{t}{t}{u}$ be the subset of $\Sigma_t$ defined in \eqref{E:SIGMAMINUS}.
	There exists a constant $C > 0$,
	independent of $\varsigma$,
	such that following estimates hold
	for $(t,u) \in [0,\Tboot) \times [0,U_0]$,
	where the fourth-from-last product
	on RHS~\eqref{E:TOPORDERTANGENTIALENERGYINTEGRALINEQUALITIES}
	(which depends on $\totTanmax{[1,N-1]}$)
	is absent in the case $N=1$
	and we recall that we defined the notation $\mathcal{O}_{\mydiam}(\cdot)$
	in Subsect.\ \ref{SS:NOTATIONANDINDEXCONVENTIONS}:
	\begin{subequations}
	\begin{align} \label{E:TOPORDERTANGENTIALENERGYINTEGRALINEQUALITIES}
		&
		\max\left\lbrace
			\totTanmax{[1,N]}(t,u), \coerciveTanspacetimemax{[1,N]}(t,u), \slowtotTanmax{\leq N}(t,u)
		\right\rbrace
				\\
		& \leq C (1 + \varsigma^{-1}) \mathring{\upepsilon}^2 \upmu_{\star}^{-3/2}(t,u)
				\notag \\
		& \ \ 
			+ 
			\boxed{\left\lbrace 6 + \mathcal{O}_{\mydiam}(\Psiep) \right\rbrace}
			\int_{t'=0}^t
					\frac{\left\| [\Lunit \upmu]_- \right\|_{L^{\infty}(\Sigma_{t'}^u)}} 
							 {\upmu_{\star}(t',u)} 
				  \totTanmax{[1,N]}(t',u)
				\, dt'
				\notag \\
		& \ \
			+ 
			\boxed{8}
			\int_{t'=0}^t
				\frac{\left\| [\Lunit \upmu]_- \right\|_{L^{\infty}(\Sigma_{t'}^u)}} 
								 {\upmu_{\star}(t',u)} 
						\totTanmax{[1,N]}^{1/2}(t',u) 
						\int_{s=0}^{t'}
							\frac{\left\| [\Lunit \upmu]_- \right\|_{L^{\infty}(\Sigma_s^u)}} 
									{\upmu_{\star}(s,u)} 
							\totTanmax{[1,N]}^{1/2}(s,u) 
						\, ds
				\, dt'
			\notag	\\
		& \ \
			+ 
			\boxed{\left\lbrace 2 + \mathcal{O}_{\mydiam}(\Psiep) \right\rbrace}
			\frac{1}{\upmu_{\star}^{1/2}(t,u)}
			\totTanmax{[1,N]}^{1/2}(t,u)
			\left\| 
				\Lunit \upmu 
			\right\|_{L^{\infty}(\Sigmaminus{t}{t}{u})}
			\int_{t'=0}^t
				\frac{1}{\upmu_{\star}^{1/2}(t',u)} \totTanmax{[1,N]}^{1/2}(t',u)
			\, dt'
			\notag \\
		& \ \
			+ 
			C \varepsilon
			\int_{t'=0}^t
				\frac{1} {\upmu_{\star}(t',u)} 
						\totTanmax{[1,N]}^{1/2}(t',u) 
						\int_{s=0}^{t'}
							\frac{1} 
									{\upmu_{\star}(s,u)} 
							\totTanmax{[1,N]}^{1/2}(s,u) 
						\, ds
				\, dt'
			\notag	\\
		& \ \
			+ 
			C \varepsilon
			\int_{t'=0}^t
					\frac{1} 
						{\upmu_{\star}(t',u)} 
				  \totTanmax{[1,N]}(t',u)
				\, dt'
				\notag
				\\
		& \ \
			+ 
			C \varepsilon
			\frac{1}{\upmu_{\star}^{1/2}(t,u)}
			\totTanmax{[1,N]}^{1/2}(t,u)
			\int_{t'=0}^t
				\frac{1}{\upmu_{\star}^{1/2}(t',u)} \totTanmax{[1,N]}^{1/2}(t',u)
			\, dt'
				\notag \\
		& \ \
			+ 
			C
			\totTanmax{[1,N]}^{1/2}(t,u)
			\int_{t'=0}^t
				\frac{1}{\upmu_{\star}^{1/2}(t',u)} \totTanmax{[1,N]}^{1/2}(t',u)
			\, dt'
				\notag \\
		& \ \
			+
			C
			\int_{t'=0}^t
				\frac{1}{\sqrt{\Tboot - t'}} \totTanmax{[1,N]}(t',u)
			 \, dt'
			 \notag \\
		& \ \
			+ C (1 + \varsigma^{-1})
					\int_{t'=0}^t
					\frac{1} 
							 {\upmu_{\star}^{1/2}(t',u)} 
				  \totTanmax{[1,N]}(t',u)
				\, dt'
				\notag	\\
		& \ \
			+ C
					\int_{t'=0}^t
					\frac{1} 
							 {\upmu_{\star}(t',u)} 
				  \totTanmax{[1,N]}^{1/2}(t',u)
				  \int_{s = 0}^{t'}
				  	\frac{1} 
							 {\upmu_{\star}^{1/2}(s,u)} 
							 \totTanmax{[1,N]}^{1/2}(s,u)
				  \, ds
				\, dt'
				\notag	\\
		& \ \
			+ C
					\int_{t'=0}^t
					\frac{1} 
							 {\upmu_{\star}(t',u)} 
				  \totTanmax{[1,N]}^{1/2}(t',u)
				  \int_{s = 0}^{t'}
				  	\frac{1}{\upmu_{\star}(s,u)}
				  	\int_{s' = 0}^s
				  	\frac{1} 
							 {\upmu_{\star}^{1/2}(s',u)} 
							 \totTanmax{[1,N]}^{1/2}(s',u)
						\, ds'	 
				  \, ds
				\, dt'
				\notag	\\
		& \ \
			+ C (1 + \varsigma^{-1})
					\int_{u'=0}^u
						\totTanmax{[1,N]}(t,u')
					\, du'
				\notag	\\
		& \ \
			+ C \varepsilon \totTanmax{[1,N]}(t,u)
			+ C \varsigma \totTanmax{[1,N]}(t,u)
			+ C \varsigma \coerciveTanspacetimemax{[1,N]}(t,u)
				\notag \\
		& \ \
			+ C
				\int_{t'=0}^t
					\frac{1} 
							 {\upmu_{\star}^{5/2}(t',u)} 
				  \totTanmax{[1,N-1]}(t',u)
				\, dt'
			\notag 
			\\
		& \ \ 
			+
			C 
			\int_{t'=0}^t
				\frac{1}{\upmu_{\star}(t',u)}
				\totTanmax{[1,N]}^{1/2}(t',u)
				\int_{s=0}^{t'}
					\frac{1}{\upmu_{\star}^{1/2}(s,u)}
					\slowtotTanmax{\leq N}^{1/2}(s,u)
				\, ds
			\, dt'
				\notag \\
		& \ \
			+
			C 
			\int_{t'=0}^t
				\slowtotTanmax{\leq N}(t',u)
			\, dt'
			\notag \\
		& \ \
			+
			C 
			(1 + \varsigma^{-1})
			\int_{u'=0}^u
				\slowtotTanmax{\leq N}(t,u')
			\, du'.
			\notag
	\end{align}
	
	\medskip
	
	\noindent \underline{\textbf{Integral inequalities relevant for below-top-order energy estimates.}}
	Moreover, if $2 \leq N \leq 18$, then we have the following estimates:
	\begin{align} \label{E:BELOWTOPORDERTANGENTIALENERGYINTEGRALINEQUALITIES}
		&
		\max\left\lbrace
			\totTanmax{[1,N-1]}(t,u), \coerciveTanspacetimemax{[1,N-1]}(t,u), \slowtotTanmax{\leq N-1}(t,u)
		\right\rbrace
			\\
		& \leq C \mathring{\upepsilon}^2 
			\notag \\
		& \ \
			+
			C
			\int_{t'=0}^t
				\frac{1}{\upmu_{\star}^{1/2}(t',u)} 
						\totTanmax{[1,N-1]}^{1/2}(t',u) 
						\int_{s=0}^{t'}
							\frac{1}{\upmu_{\star}^{1/2}(s,u)} 
							\totTanmax{[1,N]}^{1/2}(s,u) 
						\, ds
				\, dt'
			\notag	\\
		& \ \
			+
			C
			\int_{t'=0}^t
				\frac{1}{\sqrt{\Tboot - t'}} \totTanmax{[1,N-1]}(t',u)
			 \, dt'
			 \notag \\
		& \ \
			+ C (1 + \varsigma^{-1})
					\int_{t'=0}^t
					\frac{1} 
							 {\upmu_{\star}^{1/2}(t',u)} 
				  \totTanmax{[1,N-1]}(t',u)
				\, dt'
				\notag \\
		& \ \
			+ C \mathring{\upepsilon}
					\int_{t'=0}^t
					\frac{1} 
							 {\upmu_{\star}^{1/2}(t',u)} 
				  \totTanmax{[1,N-1]}^{1/2}(t',u)
				\, dt'
				\notag \\
		& \ \
			+ C (1 + \varsigma^{-1})
					\int_{u'=0}^u
						\totTanmax{[1,N-1]}(t,u')
					\, du'
				\notag	\\
		& \ \
			+ C \varsigma \coerciveTanspacetimemax{[1,N-1]}(t,u)
				\notag
					\\
		& \ \
			+
			C 
			\int_{t'=0}^t
				\slowtotTanmax{\leq N-1}(t',u)
			\, dt'
			\notag \\
		& \ \
			+
			C (1 + \varsigma^{-1})
			\int_{u'=0}^u
				\slowtotTanmax{\leq N-1}(t,u')
			\, du'.
			\notag
	\end{align}
  \end{subequations}
 \end{proposition}

\begin{remark}[\textbf{The significance of the ``boxed-constant-involving'' integrals}]
	\label{R:BOXEDCONSTANTENERGYINTEGRALS}
	The boxed-constant-involving products on RHS~\eqref{E:TOPORDERTANGENTIALENERGYINTEGRALINEQUALITIES},
	such as
	$
	\boxed{\left\lbrace 6 + \mathcal{O}_{\mydiam}(\Psiep) \right\rbrace}
			\int_{t'=0}^t
			\cdots
	$
	are particularly important in that the boxed constants
	control the maximum possible blowup-rates of our high-order
	energies. Moreover, the maximum possible energy blowup-rates
	affect the number of derivatives that we need to close the estimates.
	These are the reasons that we carefully track the size of the boxed constants.
	See the proof of
	Prop.\ \ref{P:MAINAPRIORIENERGY}
	for further discussion.
\end{remark}

\subsubsection{The main a priori energy estimates}
\label{SSS:MAINAPRIRORIENERGY}
We now provide the main a priori energy estimates.

\begin{proposition}[\textbf{The main a priori energy estimates}]
	\label{P:MAINAPRIORIENERGY}
	There exists a constant $C > 0$ such that
	under the data-size and bootstrap assumptions 
	of Subsects.\ \ref{SS:DATAASSUMPTIONS}-\ref{SS:PSIBOOTSTRAP}
	and the smallness assumptions of Subsect.\ \ref{SS:SMALLNESSASSUMPTIONS}, 
	the following estimates hold
	for $(t,u) \in [0,\Tboot) \times [0,U_0]$:
	\begin{subequations}
	\begin{align}
		\totTanmax{[1,13+M]}^{1/2}(t,u)
		+ 
		\coerciveTanspacetimemax{[1,13+M]}^{1/2}(t,u)
		+
		\slowtotTanmax{[1,13+M]}^{1/2}(t,u)
		& \leq C \mathring{\upepsilon} \upmu_{\star}^{-(M+.9)}(t,u),
			&& (0 \leq M \leq 5),
				\label{E:MULOSSMAINAPRIORIENERGYESTIMATES} \\
		\totTanmax{[1,12]}^{1/2}(t,u)
		+ 
		\coerciveTanspacetimemax{[1,12]}^{1/2}(t,u)
		+
		\slowtotTanmax{\leq 12}^{1/2}(t,u)
		& \leq C \mathring{\upepsilon}.
		&&
			\label{E:NOMULOSSMAINAPRIORIENERGYESTIMATES}
	\end{align}
	\end{subequations}
\end{proposition}

\begin{proof}[Discussion of proof]
	Based on the inequalities of Prop.~\ref{P:TANGENTIALENERGYINTEGRALINEQUALITIES}
	and the sharp estimates of Props.\ \ref{P:SHARPMU} and \ref{P:MUINVERSEINTEGRALESTIMATES},
	the proof of \cite{jSgHjLwW2016}*{Proposition 14.1}
	applies with only very minor changes 
	that in particular account for the terms
	depending on $\slowtotTanmax{N}$, $N=1,2,\cdots,18$.
	In fact, if one views the quantities 
	$
	\displaystyle
	\max\left\lbrace
			\totTanmax{[1,N]}(t,u), \coerciveTanspacetimemax{[1,N]}(t,u), \slowtotTanmax{\leq N}(t,u)
		\right\rbrace
	$
	and
	$
	\displaystyle
	\max\left\lbrace
			\totTanmax{[1,N-1]}(t,u), \coerciveTanspacetimemax{[1,N-1]}(t,u), \slowtotTanmax{\leq N-1}(t,u)
		\right\rbrace
	$	
	to be the unknowns in the system of inequalities 	
	\eqref{E:TOPORDERTANGENTIALENERGYINTEGRALINEQUALITIES}-\eqref{E:BELOWTOPORDERTANGENTIALENERGYINTEGRALINEQUALITIES},
	then the proof of \cite{jSgHjLwW2016}*{Proposition 14.1} goes through almost verbatim.
	For this reason, we omit the details, noting only
	that the sharp estimates of Prop.~\ref{P:MUINVERSEINTEGRALESTIMATES}
	are essential for handling the ``boxed-constant-involving'' products
	on RHS~\eqref{E:TOPORDERTANGENTIALENERGYINTEGRALINEQUALITIES}
	and that the smallness of some factors of type 
	$\mathcal{O}_{\mydiam}(\Psiep)$ is important for controlling
	the size of various error term coefficients
	(such as the coefficients
	$\boxed{6 + \mathcal{O}_{\mydiam}(\Psiep)}$
	and
	$\boxed{2 + \mathcal{O}_{\mydiam}(\Psiep)}$
	on RHS~\eqref{E:TOPORDERTANGENTIALENERGYINTEGRALINEQUALITIES}
	and the factors on the right-hand sides of the estimates of Prop.~\ref{P:MUINVERSEINTEGRALESTIMATES}
	of the form $C_{\mydiam} \Psiep$).
	We also note that the proof of \cite{jSgHjLwW2016}*{Proposition 14.1}
	shows that the size of the boxed constants
	is directly tied to the energy blowup-rates on 
	RHS~\eqref{E:MULOSSMAINAPRIORIENERGYESTIMATES}.
	In particular, the boxed constants control the
	``maximum top-order energy blowup-rate'' of $\upmu_{\star}^{-5.9}(t,u)$,
	which is featured on RHS~\eqref{E:MULOSSMAINAPRIORIENERGYESTIMATES} in the top-order case $M=5$.
\end{proof}

\subsubsection{Strict improvement of the fundamental bootstrap assumptions}
\label{SSS:IMPROVEMENTOFFUNDAMENTALBOOT}
Using the energy estimates provided by Prop.~\ref{P:MAINAPRIORIENERGY},
we can derive strict improvements of the fundamental $L^{\infty}$ bootstrap assumptions
\eqref{E:PSIFUNDAMENTALC0BOUNDBOOTSTRAP}.
The main ingredient in this vein is the following simple Sobolev embedding result.

\begin{lemma}[{\textbf{Sobolev embedding along} $\ell_{t,u}$}]
	\label{L:SOBOLEV}
The following estimate holds for 
scalar-valued functions $f$ defined on $\ell_{t,u}$
for $(t,u) \in [0,\Tboot) \times [0,U_0]$:
\begin{align} \label{E:SOBOLEV}
		\left\|
			f
		\right\|_{L^{\infty}(\ell_{t,u})}
		& \leq 
			C
		\left\|
			\GeoAng^{\leq 1} f
		\right\|_{L^2(\ell_{t,u})}.
	\end{align}
\end{lemma}

\begin{proof}
	Standard Sobolev embedding on $\mathbb{T}$ yields
	$
	\left\|
		f
	\right\|_{L^{\infty}(\mathbb{T})}
	\leq 
	C
	\left\|
		\CoordAng^{\leq 1} f
	\right\|_{L^2(\mathbb{T})}
	$,
	where the integration measure defining $\| \cdot \|_{L^2(\mathbb{T})}$ is $d \vartheta$.
	Next, we note the estimate $|\GeoAng| = 1 + \mathcal{O}_{\mydiam}(\Psiep) + \mathcal{O}(\varepsilon)$,
	which follows from the proof of Lemma~\ref{L:TENSORSIZECONTROLLEDBYYCONTRACTIONS}
	and Cor.\ \ref{C:SQRTEPSILONTOCEPSILON}.
	Similarly, from
	Def.\ \ref{D:METRICANGULARCOMPONENT}
	and the estimate \eqref{E:MOREPRECISEPOINTWISEESTIMATEFORGTANCOMP},
	we deduce the estimate $|\CoordAng| = 1 + \mathcal{O}_{\mydiam}(\Psiep) + \mathcal{O}(\varepsilon)$.
	It follows that $|\CoordAng^{\leq 1} f| \leq C |\GeoAng^{\leq 1} f|$
	and hence
	$
	\left\|
		f
	\right\|_{L^{\infty}(\mathbb{T})}
	\leq 
	C
	\left\|
		\CoordAng^{\leq 1} f
	\right\|_{L^2(\mathbb{T})}
	\leq C 
	\left\|
		\GeoAng^{\leq 1} f
	\right\|_{L^2(\mathbb{T})}
	$.
	Also using the estimate \eqref{E:MOREPRECISEPOINTWISEESTIMATEFORGTANCOMP},
	and referring to definition \eqref{E:LINEINTEGRALDEF},
	we conclude \eqref{E:SOBOLEV}.
\end{proof}

We now derive strict improvements of the bootstrap assumptions \eqref{E:PSIFUNDAMENTALC0BOUNDBOOTSTRAP}.

\begin{corollary}[\textbf{Strict improvement of the fundamental $L^{\infty}$ bootstrap assumptions}] 	
\label{C:IMPROVEDFUNDAMENTALLINFTYBOOTSTRAPASSUMPTIONS}
	The fundamental bootstrap assumptions \eqref{E:PSIFUNDAMENTALC0BOUNDBOOTSTRAP}
	stated in Subsect.\ \ref{SS:PSIBOOTSTRAP}
	hold with RHS~\eqref{E:PSIFUNDAMENTALC0BOUNDBOOTSTRAP} replaced by
	$C \mathring{\upepsilon}$.
	In particular, if $C \mathring{\upepsilon} < \varepsilon$,
	then we have obtained a strict improvement of the bootstrap assumptions \eqref{E:PSIFUNDAMENTALC0BOUNDBOOTSTRAP}
	on $\mathcal{M}_{\Tboot,U_0}$.
\end{corollary}

\begin{proof}
	From \eqref{E:PSIHIGHERORDERL2ESTIMATELOSSOFONEDERIVATIVE},
	\eqref{E:ELLTUSLOWCOERCIVENESSOFCONTROLLING},
	the a priori energy estimates stated in \eqref{E:NOMULOSSMAINAPRIORIENERGYESTIMATES},
	and the Sobolev embedding result \eqref{E:SOBOLEV},
	we deduce that
	$
	\left\|
		\Tanset^{[1,11]} \Psi
	\right\|_{L^{\infty}(\ell_{t,u})}
	\lesssim
	\mathring{\upepsilon}
	+
	\totTanmax{[1,12]}^{1/2}(t,u)
	\lesssim
	\mathring{\upepsilon}
	$
	and
	$
	\left\|
		\Tanset^{\leq 10} \bigslow
	\right\|_{L^{\infty}(\ell_{t,u})}
	\lesssim
	\mathring{\upepsilon}
	+
	\slowtotTanmax{[1,12]}^{1/2}(t,u)
	\lesssim
	\mathring{\upepsilon}
	$,
	from which the desired estimates easily follow.
\end{proof}

\begin{remark}[\textbf{The main step in the article}]
	In view of Cor.\ \ref{C:IMPROVEDFUNDAMENTALLINFTYBOOTSTRAPASSUMPTIONS},
	we have justified the fundamental $L^{\infty}$ 
	bootstrap assumptions until the time of first shock
	formation. This is the main step in the article.
\end{remark}

\subsection{Estimates for the energy error integrals}
\label{SS:BOUNDSFORENERGYESTIMATEERRORINTEGRALS}
It remains for us to prove Prop.~\ref{P:TANGENTIALENERGYINTEGRALINEQUALITIES}.
The proof is located in Subsect.\ \ref{SS:PROOFOFPROPTANGENTIALENERGYINTEGRALINEQUALITIES}.
To prove the proposition, we must bound all of the error integrals
appearing in the energy-null flux identities (up to top order) of 
Props.~\ref{P:DIVTHMWITHCANCELLATIONS} and \ref{P:SLOWWAVEDIVTHM}
in terms of the fundamental $L^2$-controlling quantities.
The error integrals are generated by the inhomogeneous terms
on the RHS of the equations of Prop.~\ref{P:IDOFKEYDIFFICULTENREGYERRORTERMS}
as well as the integrands from Lemma~\ref{L:MULTIPLIERVECTORFIEDERRORTERMPOINTWISEBOUND}
(see also Remark~\ref{R:NEEDTHEESTIMATEWITHTANSETNPSIINPLACEOFPSI}).

\subsubsection{Estimates for the most difficult top-order energy estimate error term}
\label{SSS:ESTIMATESFORTHEMOSTDIFFICULT}
As an important first step,
in the next lemma, we bound the norm $\| \cdot \|_{L^2(\Sigma_t^u)}$ of the most difficult product that
we encounter, namely the product
$(\Rad \Psi) \GeoAng^N \mytr \upchi$ 
on the RHS of the wave equation \eqref{E:GEOANGANGISTHEFIRSTCOMMUTATORIMPORTANTTERMS}
verified by $\GeoAng^N \Psi$.
The proof of the lemma is based on the pointwise estimates 
of Prop.~\ref{P:KEYPOINTWISEESTIMATE}, which in turn was based
on the modified quantities described in Subsubsect.\ \ref{SSS:ENERGYESTIMATES};
we recall that the modified quantity \eqref{E:TOPORDERMODIFIEDTRCHI}
was needed in the proof of Prop.~\ref{P:KEYPOINTWISEESTIMATE}
in order to avoid the loss of a derivative at the top order.

\begin{lemma}[$L^2$ \textbf{bound for the most difficult product}]
	\label{L:DIFFICULTTERML2BOUND}	
		Assume that $1 \leq N \leq 18$.
		There exists a constant $C > 0$ such that
		the following $L^2$ estimate holds for the difficult product $(\Rad \Psi) \GeoAng^N \mytr \upchi$ 
		from Prop.~\ref{P:KEYPOINTWISEESTIMATE}:
	\begin{align} \label{E:DIFFICULTTERML2BOUND}
		\left\|
			(\Rad \Psi) \GeoAng^N \mytr \upchi
		\right\|_{L^2(\Sigma_t^u)}
		& \leq
			\boxed{2} 
			\frac{\left\| [\Lunit \upmu]_- \right\|_{L^{\infty}(\Sigma_t^u)}} 
						{\upmu_{\star}(t,u)} 
				\totTanmax{[1,N]}^{1/2}(t,u)
					\\
	& \ \ +
			\boxed{4}
			\frac{\left\| [\Lunit \upmu]_- \right\|_{L^{\infty}(\Sigma_t^u)}} 
						{\upmu_{\star}(t,u)} 
			\int_{s=0}^t
				\frac{\left\| [\Lunit \upmu]_- \right\|_{L^{\infty}(\Sigma_s^u)}} 
				{\upmu_{\star}(s,u)} 
				\totTanmax{[1,N]}^{1/2}(s,u) 
			\, ds	
			\notag \\
		& \ \ +
			C \varepsilon
			\frac{1} 
						{\upmu_{\star}(t,u)} 
			\int_{s=0}^t
				\frac{1} {\upmu_{\star}(s,u)} 
				\totTanmax{[1,N]}^{1/2}(s,u) 
			\, ds	
			\notag \\
		& \ \
			+ C
				\frac{1}{\upmu_{\star}(t,u)}
				\int_{s'=0}^t
					\frac{1}{\upmu_{\star}(s',u)}
					\int_{s=0}^{s'}
						\frac{1}{\upmu_{\star}^{1/2}(s,u)}
						\totTanmax{[1,N]}^{1/2}(s,u)
					\, ds
				\, ds'
			\notag  \\
		& \ \
			+ C
				\frac{1}{\upmu_{\star}(t,u)}
				\int_{s=0}^t
					\frac{1}{\upmu_{\star}^{1/2}(s,u)}
					\totTanmax{[1,N]}^{1/2}(s,u)
				\, ds
		\notag  \\
		& \ \
			+ 
			C \frac{1}{\upmu_{\star}^{1/2}(t,u)} \totTanmax{[1,N]}^{1/2}(t,u)
			+ 
			\underbrace{C \frac{1}{\upmu_{\star}^{3/2}(t,u)} \totTanmax{[1,N-1]}^{1/2}(t,u)}_{\mbox{\upshape Absent if $N=1$}}
				\notag  \\
		& \ \ +
			C 
			\frac{1} 
					{\upmu_{\star}(t,u)} 
			\int_{s=0}^t
				\frac{1} {\upmu_{\star}^{1/2}(s,u)} 
				\slowtotTanmax{\leq N}^{1/2}(s,u) 
			\, ds	
			\notag \\
		&  \ \
			+ C \frac{1}{\upmu_{\star}^{3/2}(t,u)} \mathring{\upepsilon}.
			\notag
	\end{align}	
	
	Furthermore, we have the following less precise estimate:
	\begin{align} \label{E:LESSPRECISEDIFFICULTTERML2BOUND}
		\left\|
			\upmu \GeoAng^N \mytr \upchi
		\right\|_{L^2(\Sigma_t^u)}
		& \lesssim
			\totTanmax{[1,N]}^{1/2}(t,u)
			+
			\int_{s=0}^t
				\frac{1} 
				{\upmu_{\star}(s,u)} 
				\totTanmax{[1,N]}^{1/2}(s,u)
				\, ds
			\\
	& \ \ 
			+
			\int_{s=0}^t
				\frac{1} 
				{\upmu_{\star}^{1/2}(s,u)} 
				\slowtotTanmax{\leq N}^{1/2}(s,u)
				\, ds
			+ 
				\mathring{\upepsilon} 
				\left\lbrace 
					\ln \upmu_{\star}^{-1}(t,u) 
					+ 
					1 
				\right\rbrace.
				\notag
	\end{align}
	
\end{lemma}

\begin{proof}[Proof outline]
	To prove \eqref{E:DIFFICULTTERML2BOUND}, we start by
	taking the norm $\| \cdot \|_{L^2(\Sigma_t^u)}$
	of both sides of inequality \eqref{E:KEYPOINTWISEESTIMATE}.
	The norms $\| \cdot \|_{L^2(\Sigma_t^u)}$
	of all terms on RHS~\eqref{E:KEYPOINTWISEESTIMATE}
	were bounded in the proof of
	\cite{jSgHjLwW2016}*{Lemma~14.8}
	up to the following five remarks:
	\textbf{i)} As we noted in our proof outline of \eqref{E:KEYPOINTWISEESTIMATE},
	the right-hand side of \eqref{E:KEYPOINTWISEESTIMATE}
	does not feature the order $0$ terms $|\Psi|$ or $|\GdVar|$,
	which is different than the analogous estimate stated
	in \cite{jSgHjLwW2016}. It is for this reason that
	RHS~\eqref{E:DIFFICULTTERML2BOUND} does not feature any term of size $\mathcal{O}_{\mydiam}(\Psiep)$.
	\textbf{ii)} The typo correction mentioned in the proof of Prop.~\ref{P:KEYPOINTWISEESTIMATE} 
	is important for the proof of \cite{jSgHjLwW2016}*{Lemma~14.8}.
	\textbf{iii)} The last 
	term
	$
	\displaystyle
	\frac{1}{\upmu_{\star}(t,u)}
						\int_{t'=0}^t 
							\left|
								\Tanset^{\leq N} \bigslow
							\right|
							(t',u,\vartheta)
						\, ds
	$
	on RHS~\eqref{E:ERRORTERMKEYPOINTWISEESTIMATE}
	was not present in \cite{jSgHjLwW2016}.
	To handle this last term, 
	we use \eqref{E:L2NORMSOFTIMEINTEGRATEDFUNCTIONS} and \eqref{E:SLOWCOERCIVENESSOFCONTROLLING}
	to bound its norm $\| \cdot \|_{L^2(\Sigma_t^u)}$ by
	\begin{align} \label{E:SLOWWAVETERMINTOPORDERACOUSTICALGEOMETRYESTIMATE}
		& \leq
			C
			\frac{1}{\upmu_{\star}(t,u)}
			\int_{s=0}^t 
				\left\| 
					\Tanset^{\leq N} \bigslow
				\right\|_{L^2(\Sigma_s^u)}
			\, ds
			\leq
			C 
			\frac{1} 
					{\upmu_{\star}(t,u)} 
			\int_{s=0}^t
				\frac{1} {\upmu_{\star}^{1/2}(s,u)} 
				\slowtotTanmax{\leq N}^{1/2}(s,u) 
			\, ds.
	\end{align}
	Clearly RHS~\eqref{E:SLOWWAVETERMINTOPORDERACOUSTICALGEOMETRYESTIMATE}
	is bounded by the next-to-last product on RHS~\eqref{E:DIFFICULTTERML2BOUND} as desired.
	\textbf{iv)} In \cite{jSgHjLwW2016}*{Lemma~14.8}, 
	the coefficient of the analog of 
	the second product on RHS~\eqref{E:DIFFICULTTERML2BOUND}
	was stated as $\boxed{4.05}$ 
	rather than $\boxed{4}$.
	We note that due to the remarks made 
	in the proof outline of Prop.\ \ref{P:KEYPOINTWISEESTIMATE},
	in the present paper, we have relegated the ``additional $.05$ contribution''
	to the error term $C \varepsilon \cdots$ on
	RHS~\eqref{E:DIFFICULTTERML2BOUND}; this is a minor (essentially cosmetic) change that we 
	described in our proof outline of Prop.\ \ref{P:KEYPOINTWISEESTIMATE}.
	\textbf{v)} In the case $N=1$, we use the estimate
	\eqref{E:ANOTHERPSIHIGHERORDERL2ESTIMATELOSSOFONEDERIVATIVE} with $M=1$
	to bound the norm $\| \cdot \|_{L^2(\Sigma_t^u)}$
	of the product
	$ 
	\frac{1}{\upmu_{\star}(t,u)}
			\left|
				\Fullset_*^{[1,N];1} \Psi
			\right|
			(t,u,\vartheta)
	$
	on RHS~\eqref{E:ERRORTERMKEYPOINTWISEESTIMATE}
	by 
	$
	\leq
	$
	the sum of the fifth product on RHS~\eqref{E:DIFFICULTTERML2BOUND}
	and the last product on RHS~\eqref{E:DIFFICULTTERML2BOUND}.
	This detail was not mentioned in \cite{jSgHjLwW2016}*{Lemma~14.8};
	it is relevant because the term
	$
		C \frac{1}{\upmu_{\star}^{3/2}(t,u)} \totTanmax{[1,N-1]}^{1/2}(t,u)
	$
	on RHS~\eqref{E:DIFFICULTTERML2BOUND},
	which can be used to help control the norm $\| \cdot \|_{L^2(\Sigma_t^u)}$
	of the term
	$ 
	\frac{1}{\upmu_{\star}(t,u)}
			\left|
				\Fullset_*^{[1,N];1} \Psi
			\right|
			(t,u,\vartheta)
	$
	when $N > 1$,
	is absent from RHS~\eqref{E:DIFFICULTTERML2BOUND} in the case $N=1$.
	This completes our proof outline of \eqref{E:DIFFICULTTERML2BOUND}.
	
	The estimate \eqref{E:LESSPRECISEDIFFICULTTERML2BOUND}
	can similarly be proved 
	by taking the norm $\| \cdot \|_{L^2(\Sigma_t^u)}$
	of both sides of inequality
	\eqref{E:LESSPRECISEKEYPOINTWISEESTIMATE}.
	All terms were handled 
	in the proof of
	\cite{jSgHjLwW2016}*{Lemma~14.8}
	(remark \textbf{i)} from the previous paragraph also applies here)
	except for the last 
	term
	$
	\displaystyle
				\int_{t'=0}^t 
							\left|
								\Tanset^{\leq N} \bigslow
							\right|
							(t',u,\vartheta)
						\, dt'
	$
	on RHS~\eqref{E:LESSPRECISEKEYPOINTWISEESTIMATE}, which,
	by the arguments given in the previous paragraph,
	can be bounded in the norm $\| \cdot \|_{L^2(\Sigma_t^u)}$
	by
	$
	\displaystyle
			\leq
			C
			\int_{s=0}^t
				\frac{1} {\upmu_{\star}^{1/2}(s,u)} 
				\slowtotTanmax{\leq N}^{1/2}(s,u) 
			\, ds
	$
	as desired.
\end{proof}

\subsubsection{\texorpdfstring{$L^2$}{Square integral} bounds for less degenerate top-order error integrals}
\label{SSS:LESSDEGENERATEENERGYESTIMATEINTEGRALS}
In the next lemma, we bound some up-to-top-order error integrals
that appear in our energy estimates.
As in the proof of Lemma~\ref{L:DIFFICULTTERML2BOUND},  
the proof relies on modified quantities, which are needed to avoid
the loss of a derivative. However, 
the estimates of the lemma are much less degenerate than those of
Lemma~\ref{L:DIFFICULTTERML2BOUND} because of the availability
of a helpful factor of $\upmu$ in the integrands.

\begin{lemma}[\textbf{Bounds for less degenerate top-order error integrals}]
	\label{L:LESSDEGENERATEENERGYESTIMATEINTEGRALS}
	Assume that $1 \leq N \leq 18$,
	let $\Tanset^N$ be any $N^{th}$ order 
	$\mathcal{P}_u$-tangential operator,
	and let $\GeoAngFlatRadComponent$ be the scalar function from Lemma~\ref{L:GEOANGDECOMPOSITION}.
	We have the following following integral estimates:
	\begin{subequations}
	\begin{align} \label{E:FIRSTLESSDEGENERATEENERGYESTIMATEINTEGRALS}
		&
		\left|
			\int_{\mathcal{M}_{t,u}}
				(\Rad \Tanset^N \Psi)	 
				\myarray[\GeoAngFlatRadComponent]
				{1}
				(\angdiffuparg{\#} \Psi)
				\cdot
				(\upmu \angdiff \GeoAng^{N-1} \mytr \upchi)
			\, d \vol
		\right|
			\\
		& \lesssim
			\int_{t'=0}^t
				\left\lbrace 
					\ln \upmu_{\star}^{-1}(t',u) 
					+ 
					1 
				\right\rbrace^2
				\totTanmax{[1,N]}(t',u)
			\, dt'
			+
			\int_{u'=0}^u
				\totTanmax{[1,N]}(t,u')
			\, du'
			\notag \\
		& \ \
			+
			\int_{t'=0}^t
				\slowtotTanmax{\leq N}(t',u)
			\, dt'
			+ 
			\mathring{\upepsilon}^2,
			\notag
			\\
	&
		\left|
			\int_{\mathcal{M}_{t,u}}
				(1 + 2 \upmu)
				(\Lunit \Tanset^N \Psi)	 
				\myarray[\GeoAngFlatRadComponent]
				{1}
				(\angdiffuparg{\#} \Psi)
				\cdot
				(\upmu \angdiff \GeoAng^{N-1} \mytr \upchi)
			\, d \vol
		\right|
			\label{E:SECONDLESSDEGENERATEENERGYESTIMATEINTEGRALS} \\
		& \lesssim
			\int_{t'=0}^t
				\left\lbrace 
					\ln \upmu_{\star}^{-1}(t',u) 
					+ 
					1 
				\right\rbrace^2
				\totTanmax{[1,N]}(t',u)
			\, dt'
			+
			\int_{u'=0}^u
				\totTanmax{[1,N]}(t,u')
			\, du'
			\notag \\
		& \ \
			+
			\int_{t'=0}^t
				\slowtotTanmax{\leq N}(t',u)
			\, dt'
			+ 
			\mathring{\upepsilon}^2.
			\notag
	\end{align}
	\end{subequations}
\end{lemma}

\begin{proof}
See Subsect.\ \ref{SS:OFTENUSEDESTIMATES} for some comments on the analysis.
To prove \eqref{E:SECONDLESSDEGENERATEENERGYESTIMATEINTEGRALS},
we use the fact that $\GeoAngFlatRadComponent = \smoothfunction(\GdVar) \GdVar$
 (see \eqref{E:LINEARLYSMALLSCALARSDEPENDINGONGOODVARIABLES}),
	the $L^{\infty}$ estimates of Prop.~\ref{P:IMPROVEMENTOFAUX},
	Cauchy--Schwarz, 
	and \eqref{E:COERCIVENESSOFCONTROLLING}
	to deduce
\begin{align} \label{E:FIRSTBOUNDFORSECONDLESSDEGENERATEENERGYESTIMATEINTEGRALS}
	\mbox{LHS~\eqref{E:SECONDLESSDEGENERATEENERGYESTIMATEINTEGRALS}}
	& \lesssim
		\int_{\mathcal{M}_{t,u}}
			\left|
				\Lunit \Tanset^N \Psi
			\right|^2 
		\, d \vol
		+
		\int_{\mathcal{M}_{t,u}}
			\left|
				\upmu \GeoAng^N \mytr \upchi
			\right|^2
		\, d \vol
		\\
		& 
		\lesssim
		\int_{u'=0}^u
			\left\|
				\Lunit \Tanset^N \Psi
			\right\|_{L^2(\mathcal{P}_{u'}^t)}^2
		\, du'
		+
			\int_{t' = 0}^t
				\left\|
					\upmu \GeoAng^N \mytr \upchi 
				\right\|_{L^2(\Sigma_{t'}^u)}^2
			\, dt'
			\notag \\
		&
		\lesssim
			\int_{u'=0}^u
				\totTanmax{[1,N]}(t,u')
			\, du'
			+
			\int_{t' = 0}^t
				\left\|
					\upmu \GeoAng^N \mytr \upchi 
				\right\|_{L^2(\Sigma_{t'}^u)}^2
			\, dt'.
		\notag
\end{align}
To complete the proof of \eqref{E:SECONDLESSDEGENERATEENERGYESTIMATEINTEGRALS},
we must handle the final integral on RHS~\eqref{E:FIRSTBOUNDFORSECONDLESSDEGENERATEENERGYESTIMATEINTEGRALS}.
To bound the integral by $\leq$ RHS~\eqref{E:SECONDLESSDEGENERATEENERGYESTIMATEINTEGRALS},
we bound the integrand
$
\displaystyle
\left\|
					\upmu \GeoAng^N \mytr \upchi 
				\right\|_{L^2(\Sigma_{t'}^u)}^2
$
by using inequality \eqref{E:LESSPRECISEDIFFICULTTERML2BOUND} 
(with $t'$ in place of $t$),
simple estimates of the form $ab \lesssim a^2 + b^2$,
and the following bounds for the two time integrals on RHS~\eqref{E:LESSPRECISEDIFFICULTTERML2BOUND}:
\[
	\int_{s=0}^{t'}
				\frac{1} 
				{\upmu_{\star}(s,u)} 
				\totTanmax{[1,N]}^{1/2}(s,u)
	\, ds
	\lesssim
	\left\lbrace 
		\ln \upmu_{\star}^{-1}(t',u) 
		+ 
		1 
	\right\rbrace
	\totTanmax{[1,N]}^{1/2}(t',u),
\]
\[
	\int_{s=0}^{t'}
				\frac{1} 
				{\upmu_{\star}^{1/2}(s,u)} 
				\slowtotTanmax{\leq N}^{1/2}(s,u)
	\, ds
	\lesssim
	\slowtotTanmax{\leq N}^{1/2}(t',u).
\]
The above two bounds follow from \eqref{E:LOGLOSSMUINVERSEINTEGRALBOUND} 
and the fact that $\totTanmax{[1,N]}$ and $\slowtotTanmax{\leq N}$ are increasing in their arguments.
These steps yield that all terms 
generated by
$
\displaystyle
\int_{t' = 0}^t
				\left\|
					\upmu \GeoAng^N \mytr \upchi 
				\right\|_{L^2(\Sigma_{t'}^u)}^2
			\, dt'
$
are $\lesssim \mbox{RHS~\eqref{E:SECONDLESSDEGENERATEENERGYESTIMATEINTEGRALS}}$,
except for the following integral generated by the last term 
on RHS~\eqref{E:LESSPRECISEDIFFICULTTERML2BOUND}:
\[
	\mathring{\upepsilon}^2
	\int_{t'=0}^t
		\left\lbrace 
					\ln \upmu_{\star}^{-1}(t',u) 
					+ 
					1 
				\right\rbrace^2
	\, dt'.
\]
Using \eqref{E:LESSSINGULARTERMSMPOINTNINEINTEGRALBOUND},
we deduce that the above term is $\lesssim \mathring{\upepsilon}^2$ as desired.
We have thus proved \eqref{E:SECONDLESSDEGENERATEENERGYESTIMATEINTEGRALS}.

The proof of \eqref{E:FIRSTLESSDEGENERATEENERGYESTIMATEINTEGRALS}
starts with the following analog of \eqref{E:FIRSTBOUNDFORSECONDLESSDEGENERATEENERGYESTIMATEINTEGRALS},
which can be proved in the same way:
\begin{align}
\mbox{LHS~\eqref{E:FIRSTLESSDEGENERATEENERGYESTIMATEINTEGRALS}}
& 
\lesssim
\int_{t' = 0}^t 
			\left\| 
				\Rad \Tanset^N \Psi 
			\right\|_{L^2(\Sigma_{t'}^u)}^2 
		\, dt'
+
\int_{t' = 0}^t 
	\left\|
		\upmu \GeoAng^N \mytr \upchi 
	\right\|_{L^2(\Sigma_{t'}^u)}^2 
		\, dt'
			\\
& 
\lesssim
\int_{t'=0}^t
	\totTanmax{[1,N]}(t',u)
\, dt'
+
\int_{t' = 0}^t 
	\left\|
		\upmu \GeoAng^N \mytr \upchi 
	\right\|_{L^2(\Sigma_{t'}^u)}^2 
		\, dt'.
	\notag		
\end{align}
The remaining details are similar to the ones that we gave above in our proof
of \eqref{E:SECONDLESSDEGENERATEENERGYESTIMATEINTEGRALS}; we therefore omit them.
\end{proof}

\subsubsection{Estimates involving simple error terms}
\label{SSS:ENERGYESTMATESHARMLESS}
In this subsubsection, we derive $L^2$ bounds for some simple error terms that we encounter in the energy estimates.
We start with a lemma in which we control the $L^2$ norms
of the ``easy derivatives'' of the eikonal function quantities.
By easy derivatives, we mean ones that we can control without using
the modified quantities described in Subsubsect.\ \ref{SSS:ENERGYESTIMATES}.

\begin{lemma}[$L^2$ \textbf{bounds for the eikonal function quantities that do not require modified quantities}]
	\label{L:EASYL2BOUNDSFOREIKONALFUNCTIONQUANTITIES}
	The following estimates hold for $1 \leq N \leq 18$
	(see Subsect.\ \ref{SS:STRINGSOFCOMMUTATIONVECTORFIELDS} regarding the vectorfield operator notation):
	\begin{subequations}
	\begin{align}
		\left\|
			\Lunit \Tanset_*^{[1,N]} \upmu
		\right\|_{L^2(\Sigma_t^u)},
			\,
		\left\|
			\Lunit \Tanset^{\leq N} \Lunit_{(Small)}^i
		\right\|_{L^2(\Sigma_t^u)},
			\,
		\left\|
			\Lunit \Tanset^{\leq N-1} 
			\mytr \upchi
		\right\|_{L^2(\Sigma_t^u)}
		& \lesssim 
				\mathring{\upepsilon}
				+
				\frac{\totTanmax{[1,N]}^{1/2}(t,u)}{\upmu_{\star}^{1/2}(t,u)},
			\label{E:LUNITTANGENGITALEIKONALINTERMSOFCONTROLLING}
				 \\
		\left\|
			\Lunit \Fullset^{\leq N;1} \Lunit_{(Small)}^i
		\right\|_{L^2(\Sigma_t^u)},
			\,
		\left\|
			\Lunit \Fullset^{\leq N-1;1} \mytr \upchi
		\right\|_{L^2(\Sigma_t^u)}
		& \lesssim 
				\mathring{\upepsilon}
				+
				\frac{\totTanmax{[1,N]}^{1/2}(t,u)}{\upmu_{\star}^{1/2}(t,u)},
			\label{E:LUNITONERADIALEIKONALINTERMSOFCONTROLLING}
				 \\
		\left\|
			\Tanset_*^{[1,N]} \upmu
		\right\|_{L^2(\Sigma_t^u)},
			\,
		\left\|
			\Tanset^{[1,N]} \Lunit_{(Small)}^i
		\right\|_{L^2(\Sigma_t^u)},
			\,
		\left\|
			\Tanset^{\leq N-1} \mytr \upchi
		\right\|_{L^2(\Sigma_t^u)}
		& \lesssim 
			\mathring{\upepsilon}
			+ 
			\int_{s=0}^t
				\frac{\totTanmax{[1,N]}^{1/2}(s,u)}{\upmu_{\star}^{1/2}(s,u)}
			\, ds,
				 \label{E:TANGENGITALEIKONALINTERMSOFCONTROLLING}
				 \\
		\left\|
			\Fullset_*^{[1,N];1} \Lunit_{(Small)}^i
		\right\|_{L^2(\Sigma_t^u)},
			\,
		\left\|
			\Fullset^{\leq N-1;1} \mytr \upchi
		\right\|_{L^2(\Sigma_t^u)}
		& \lesssim 
				\mathring{\upepsilon}
				+
				\int_{s=0}^t
				\frac{\totTanmax{[1,N]}^{1/2}(s,u)}{\upmu_{\star}^{1/2}(s,u)}
			\, ds.
			\label{E:ONERADIALEIKONALINTERMSOFCONTROLLING}
	\end{align}
	\end{subequations}

\end{lemma}

\begin{proof}
	Thanks in part to the estimates of 
	Lemmas~\ref{L:BEHAVIOROFEIKONALFUNCTIONQUANTITIESALONGSIGMA0},
	\ref{L:COMMUTATORESTIMATES},
	and \ref{L:TRANSVERALTANGENTIALCOMMUTATOR},
	Props.~\ref{P:IMPROVEMENTOFAUX}
	and \ref{P:MUINVERSEINTEGRALESTIMATES},
	and Lemma~\ref{L:L2NORMSOFTIMEINTEGRATEDFUNCTIONS},
	the proof of \cite{jSgHjLwW2016}*{Lemma~14.3}
	goes through nearly verbatim.
	In particular, these estimates do not involve the slow wave variable $\bigslow$.
	One small difference compared to
	\cite{jSgHjLwW2016}*{Lemma~14.3}
	is that we do not state estimates for
	$\left\|
		\Lunit_{(Small)}^i
	\right\|_{L^2(\Sigma_t^u)}
	$
	in Lemma~\ref{L:EASYL2BOUNDSFOREIKONALFUNCTIONQUANTITIES} since
	in the present paper,
	these quantities are
	not controlled by the data-size parameter
	$\mathring{\upepsilon}$,
	but rather by $\Psiep$
	(much like in the $L^{\infty}$ estimate \eqref{E:LUNITIITSEFLSMALLDATALINFINITYCONSEQUENCES}).
\end{proof}

We now estimate the error integrals
generated by the
terms 
$Harmless^{[1,N]}$ and $Harmless_{(Slow)}^{\leq N}$
on the RHSs of the equations of Prop.~\ref{P:IDOFKEYDIFFICULTENREGYERRORTERMS}.

\begin{lemma}[$L^2$ \textbf{bounds for error integrals involving} $Harmless^{[1,N]}$ \textbf{and} 
$Harmless_{(Slow)}^{\leq N}$ \textbf{terms}]
	\label{L:STANDARDPSISPACETIMEINTEGRALS}
	Assume that $1 \leq N \leq 18$ and $\varsigma > 0$.
	Recall that we defined terms of the form
	$Harmless^{[1,N]}$ and $Harmless_{(Slow)}^{\leq N}$
	in Def.~\ref{D:HARMLESSTERMS}.
	We have the following estimates,
	where the implicit constants are independent of $\varsigma$:
	\begin{align}  \label{E:STANDARDPSISPACETIMEINTEGRALS}
		\int_{\mathcal{M}_{t,u}}
		 	& \left|
					\threemyarray[(1 + \upmu) \Lunit \Tanset^N \Psi]
						{\Rad \Tanset^N \Psi}
						{\Tanset^{\leq N} \bigslow}
				\right|
				\left|
					\myarray[Harmless^{[1,N]}]
						{Harmless_{(Slow)}^{\leq N}}
				\right|
		 \, d \vol
		   \\
		& \lesssim
		 	(1 + \varsigma^{-1})
		 	\int_{t'=0}^t
				\totTanmax{[1,N]}(t',u)
			\, dt'
			+
			(1 + \varsigma^{-1})
			\int_{u'=0}^u
				\totTanmax{[1,N]}(t,u')
			\, du'
				\notag \\
		& \ \
			+ \varsigma
			  \coerciveTanspacetimemax{[1,N]}(t,u)
			+ 
			(1 + \varsigma^{-1})
			\int_{u'=0}^u
				\slowtotTanmax{\leq N}(t,u')
			\, du'
			+ \mathring{\upepsilon}^2.
			\notag 	
		\end{align}
\end{lemma}

\begin{proof}
See Subsect.\ \ref{SS:OFTENUSEDESTIMATES} for some comments on the analysis.
The bounds for error integrals that do not involve
$\bigslow$
can be derived by using arguments nearly identical to the ones given in the proof of
\cite{jSgHjLwW2016}*{Lemma~14.7}; we therefore omit those details.

It remains for us to prove the desired bounds involving $\bigslow$,
which means that we must estimate the spacetime integrals of various quadratic terms.
In this proof, we derive the desired estimates 
only for some representative
quadratic terms. The remaining terms can be similarly bounded
and we omit those details.
Specifically, we show how bound the integral of the product
	$\left|
			\Tanset^{\leq N} \bigslow
	 \right|
	 \left|
	 		\Fullset_*^{[1,N+1];1} \Psi
	 \right|
	$.
	We note that our proof of the desired bound
	relies on all of the main ideas needed to prove all of the estimates stated in the lemma.
	To proceed, we repeatedly use the commutator estimate 
	\eqref{E:ONERADIALTANGENTIALFUNCTIONCOMMUTATORESTIMATE}
	and the $L^{\infty}$ estimates of Prop.~\ref{P:IMPROVEMENTOFAUX}
	to commute the (at most one) factor of $\Rad$
	in the operator $\Fullset_*^{[1,N+1];1}$
	so that $\Rad$ acts last,
	which yields
	\[
	|\Fullset_*^{[1,N+1];1} \Psi|
	\lesssim
	|\Rad \Tanset^{[1,N]} \Psi|
	+ 
	|\Tanset^{[1,N+1]} \Psi|
	+
	|\Fullset_*^{[1,N];1} \GdVar|
	+ 
	|\Tanset_*^{[1,N]} \BadVar|.
	\]
	Thus, we must bound the integral of the four corresponding products
	generated by the RHS of the previous inequality.
	To bound the integral of the first product,
	we use Young's inequality 
	and Lemma~\ref{L:COERCIVENESSOFCONTROLLING}
	to obtain 
	\begin{align}
		& \int_{\mathcal{M}_{t,u}}
		 		\left|
					\Tanset^{\leq N} \bigslow
				\right|
				\left|
					\Rad \Tanset^{[1,N]} \Psi
				\right|
		 \, d \vol
		 	\label{E:FINALHARMLESSEXAMPLEINTEGRAL} \\
		& \lesssim
			\int_{u'=0}^u
					\int_{\mathcal{P}_{u'}^t}
						\left|
							\Tanset^{\leq N} \bigslow
						\right|^2
					\, d \conevol
			\, du'
			+
			\int_{t'=0}^t
		 			\int_{\Sigma_{t'}^u}
		 				\left|
							\Rad \Tanset^{[1,N]} \Psi
						\right|^2
					\, d \tvol
			\, dt'
				\notag
					\\ 
		& \lesssim
				\int_{u'=0}^u
					\slowtotTanmax{\leq N}(t,u')
				\, du'
				+
				\int_{t'=0}^t
					\totTanmax{[1,N]}(t',u)
				\, dt',
				\notag
	\end{align}
	which is $\lesssim$ RHS~\eqref{E:STANDARDPSISPACETIMEINTEGRALS} as desired.
	
	To handle the second product
	$|\Tanset^{\leq N} \bigslow|
	 |\Tanset^{[1,N+1]} \Psi|$,
	we first consider terms of the form
	$|\Tanset^{\leq N} \bigslow|
	 |\GeoAng \Tanset^{\leq N} \Psi|
	$.
	To bound the corresponding integrals, we use Young's inequality 
	and Lemma~\ref{L:COERCIVENESSOFCONTROLLING}
	and, for terms with more than one derivative on $\Psi$, 
	we consider separately the regions in which $\upmu < 1/4$
	and $\upmu > 1/4$,
	which in total leads to the following bound:
	\begin{align}
		& \int_{\mathcal{M}_{t,u}}
		 		\left|
					\Tanset^{\leq N} \bigslow
				\right|
				\left|
					\GeoAng \Tanset^{\leq N} \Psi
				\right|
		 \, d \vol
		 	\label{E:ANOTHERFINALHARMLESSEXAMPLEINTEGRAL} \\
		& \lesssim
			(1 + \varsigma^{-1})
		 	\int_{u'=0}^u
					\int_{\mathcal{P}_{u'}^t}
						\left|
							\Tanset^{\leq N} \bigslow
						\right|^2
					\, d \conevol
			\, du'
			+
			\varsigma
			\int_{\mathcal{M}_{t,u}}
				\mathbf{1}_{\lbrace \upmu \leq 1/4 \rbrace}
				\left|
					\angdiff \Tanset^{[1,N]} \Psi
				\right|^2
			\, d \vol
				\notag \\
		& \ \
				+
				\int_{t'=0}^t
		 			\int_{\Sigma_{t'}^u}
		 				\upmu
		 				\left|
							 \angdiff \Tanset^{[1,N]} \Psi
						\right|^2
					\, d \tvol
			\, dt'
			+
			\int_{t'=0}^t
		 			\int_{\Sigma_{t'}^u}
		 				\left|
							 \GeoAng \Psi
						\right|^2
					\, d \tvol
			\, dt'
				\notag
					\\ 
		& \lesssim
				(1 + \varsigma^{-1})
				\int_{u'=0}^u
					\slowtotTanmax{\leq N}(t,u')
				\, du'
				+ 
				\varsigma \coerciveTanspacetimemax{[1,N]}(t,u)
				+
				\int_{t'=0}^t
					\totTanmax{[1,N]}(t',u)
				\, dt'
			+
			\mathring{\upepsilon}^2,
				\notag
	\end{align}
	which is $\lesssim$ RHS~\eqref{E:STANDARDPSISPACETIMEINTEGRALS} as desired.
	To finish the proof of the desired bounds
	for the second product
	$|\Tanset^{\leq N} \bigslow|
	 |\Tanset^{[1,N+1]} \Psi|$,
	it remains for us to bound the spacetime integral of the product
	$|\Tanset^{\leq N} \bigslow|
	 |\Lunit \Tanset^{\leq N} \Psi|
	$.
	We can obtain the desired bound by using arguments similar to the ones
	we used in proving \eqref{E:ANOTHERFINALHARMLESSEXAMPLEINTEGRAL}, 
	except that we also rely on the bounds
	$
	\displaystyle
	\int_{\mathcal{M}_{t,u}}
		|\Lunit \Tanset^{[1,N]} \Psi|^2
	\, d \vol
	\lesssim
	\int_{u'=0}^u
		\totTanmax{[1,N]}(t,u')
	\, du'
	$
	and
	$
	\displaystyle
	\int_{\mathcal{M}_{t,u}}
		|\Lunit \Psi|^2
	\, d \vol
	\lesssim
	\mathring{\upepsilon}^2
	+
	\int_{t'=0}^t
		\totTanmax{[1,N]}(t',u)
	\, du'
	$,
	which are simple consequences of Lemma~\ref{L:COERCIVENESSOFCONTROLLING}.
	
	To bound the integral of the third product
	$	|\Tanset^{\leq N} \bigslow|
		|\Fullset_*^{[1,N];1} \GdVar|
	$
	and the 
	fourth product
	$	|\Tanset^{\leq N} \bigslow|
		|\Tanset_*^{[1,N]} \BadVar|
	$,
	we use arguments similar to the ones we used above
	as well as the estimates
	\eqref{E:TANGENGITALEIKONALINTERMSOFCONTROLLING}
	and
	\eqref{E:ONERADIALEIKONALINTERMSOFCONTROLLING}
	which, when combined
	with the estimate \eqref{E:LESSSINGULARTERMSMPOINTNINEINTEGRALBOUND}
	and the fact that $\totTanmax{[1,N]}$ is increasing in its arguments,
	yield the following spacetime integral bounds:
	\begin{align} \label{E:LASTREPHARMLESSERRORINTEGRAL}
		&
		\int_{\mathcal{M}_{t,u}}
		 		\left|
					\Fullset_*^{[1,N];1} \Lunit_{(Small)}^i
				\right|^2
		 \, d \vol
		 +
		 \int_{\mathcal{M}_{t,u}}
		 		\left|
					\Tanset_*^{[1,N]} \upmu
				\right|^2
		 \, d \vol
		 	\\
		 & \lesssim
		 	\int_{t'=0}^t
					\left\lbrace
						\int_{s=0}^{t'}
							\frac{\totTanmax{[1,N]}^{1/2}(s,u)}{\upmu_{\star}^{1/2}(s,u)}
						\, ds
					\right\rbrace^2
				\, dt'
				+ 
				\mathring{\upepsilon}^2
				\notag \\
		 & \lesssim
		 	\int_{t'=0}^t
					\totTanmax{[1,N]}(t',u)
				\, dt'
				+ 
				\mathring{\upepsilon}^2.
				\notag
	\end{align}
	Finally, we observe that
	$\mbox{\upshape RHS~\eqref{E:LASTREPHARMLESSERRORINTEGRAL}} 
	\lesssim
	\mbox{\upshape RHS~\eqref{E:STANDARDPSISPACETIMEINTEGRALS}} 
	$ 
	as desired.
	This completes our proof of the representative estimates.
\end{proof}

\subsubsection{Estimates for the error integrals \underline{not} depending on the semilinear inhomogeneous terms}
\label{SSS:ENERGYESTIMATESFORERRORTERMSNOTDEPENDINGONINHOMOGENEOUS}
In the next lemma, we derive bounds for the error integrals
whose integrands we pointwise bounded in Lemma~\ref{L:MULTIPLIERVECTORFIEDERRORTERMPOINTWISEBOUND}.

\begin{lemma}[\textbf{Estimates for the error integrals not depending on the semilinear inhomogeneous terms}]
\label{L:ENERGYESTIMATESFORSLOWWAVEERRORTERMSNOINHOMOGENEOUS}
	Assume that $1 \leq N \leq 18$. 
	Let $\varsigma > 0$ be a real number.
	We have the following estimate
	for the last term on RHS~\eqref{E:E0DIVID}
	(with $\Tanset^N \Psi$ in the role of $f$ and
	without any absolute value taken on the left),
	where the implicit constants are independent of $\varsigma$:
	\begin{align} \label{E:MULTIPLIERVECTORFIELDERRORINTEGRALS}
		\sum_{i=1}^5 
		\int_{\mathcal{M}_{t,u}}
			\basicenergyerrorarg{\Mult}{i}[\Tanset^N \Psi]
		\, d \vol
		& \lesssim
		 	\int_{t'=0}^t
				\frac{1}{\sqrt{\Tboot - t'}} \totTanmax{[1,N]}(t',u)
			 \, dt'
			 \\
		& \ \
			+
		 	(1 + \varsigma^{-1})
		 		\int_{t'=0}^t
					\totTanmax{[1,N]}(t',u)
				\, dt'
			 \notag	\\
		 & \ \ 
		 	+
			(1 + \varsigma^{-1})
			\int_{u'=0}^u
				\totTanmax{[1,N]}(t,u')
			\, du'
			+ \varsigma 
				\coerciveTanspacetimemax{[1,N]}(t,u).
				\notag
	\end{align}
	
	Moreover, for $N \leq 18$,
	we have the following estimate for the term on the next-to-last 
	line of RHS~\eqref{E:SLOWENERGYID}
	(with $\Tanset^N \bigslow$ in the role of $\altbigslow$):
	\begin{align} \label{E:ENERGYESTIMATESFORSLOWWAVEERRORTERMSNOINHOMOGENEOUS}
			\int_{\mathcal{M}_{t,u}}
				\left|
					\left\lbrace 
						1 + \upgamma \smoothfunction(\upgamma)
					\right\rbrace 
					\slowbasicenergyerror[\Tanset^N \bigslow]
				\right|
			 \, d \vol
			 & \leq
			 C
			 \int_{u'=0}^u
				\slowtotTanmax{\leq N}(t,u')
			\, du'.
	\end{align}
\end{lemma}

\begin{proof}
	See Subsect.\ \ref{SS:OFTENUSEDESTIMATES} for some comments on the analysis.
	To prove \eqref{E:MULTIPLIERVECTORFIELDERRORINTEGRALS},
	we integrate \eqref{E:MULTIPLIERVECTORFIEDERRORTERMPOINTWISEBOUND} 
	(with $\Tanset^N \Psi$ in the role of $f$)
	over $\mathcal{M}_{t,u}$ and use 
	Lemma~\ref{L:COERCIVENESSOFCONTROLLING}.
	
	To prove \eqref{E:ENERGYESTIMATESFORSLOWWAVEERRORTERMSNOINHOMOGENEOUS},
	we use the pointwise bound \eqref{E:SIMPLEPOINTWISEBOUNDSLOWWAVEBASICENERGYINTEGRAND}
	(with $\Tanset^N \bigslow$ in the role of $\altbigslow$)
	and the coerciveness estimate \eqref{E:SLOWCOERCIVENESSOFCONTROLLING}
	to deduce that 
	\begin{align}
		\mbox{{\upshape LHS}~\eqref{E:ENERGYESTIMATESFORSLOWWAVEERRORTERMSNOINHOMOGENEOUS}}
		\lesssim 
		\int_{u'=0}^u
			\| \Tanset^N \bigslow \|_{\mathcal{P}_{u'}^t}^2
		\, du'
		\lesssim
		\int_{u'=0}^u
			\slowtotTanmax{\leq N}(t,u')
		\, du'
	\end{align}
	as desired.
\end{proof}

\subsection{Proof of Prop.~\ref{P:TANGENTIALENERGYINTEGRALINEQUALITIES}}
\label{SS:PROOFOFPROPTANGENTIALENERGYINTEGRALINEQUALITIES}
Armed with the estimates of the previous subsections, we are now ready
to prove Prop.~\ref{P:TANGENTIALENERGYINTEGRALINEQUALITIES}.

\noindent \textbf{Proof of \eqref{E:TOPORDERTANGENTIALENERGYINTEGRALINEQUALITIES}}:
Assume that $1 \leq N \leq 18$
and let $\Tanset^N$ be an $N^{th}$-order 
$\mathcal{P}_u$-tangential vectorfield operator.
Using \eqref{E:E0DIVID} with $\Tanset^N \Psi$ in the role of $f$
and appealing to definition \eqref{E:COERCIVESPACETIMEDEF},
we deduce the following identity:
\begin{align} \label{E:E0DIVIDMAINESTIMATES}
	\enzero[\Tanset^N \Psi](t,u)
	& + 
	\flzero[\Tanset^N \Psi](t,u)
	+
	\coercivespacetime[\Tanset^N \Psi](t,u)
	\\
	& 
		=
		\enzero[\Tanset^N \Psi](0,u)
		+
		\enzero[\Tanset^N \Psi](t,0)
		\notag	\\
	& \ \
		- 
		\int_{\mathcal{M}_{t,u}}
			\left\lbrace
				(1 + 2 \upmu) (\Lunit \Tanset^N \Psi)
				+ 
				2 \Rad \Tanset^N \Psi 
			\right\rbrace
			\upmu \square_g(\Tanset^N \Psi) 
		\, d \vol
							\notag \\
				& \ \ 
						+
						\sum_{i=1}^5 
						\int_{\mathcal{M}_{t,u}}
							\basicenergyerrorarg{\Mult}{i}[\Tanset^N \Psi]
						\, d \vol.
							\notag
\end{align}
Similarly, if $N \leq 18$,
then by \eqref{E:SLOWENERGYID}
with $\Tanset^N \bigslow$ in the role of
$\altbigslow$, we have
\begin{align} \label{E:SLOWENERGYIDMAINESTIMATES}
	&
	\slowen[\Tanset^N \bigslow](t,u)
	+
	\slowfl[\Tanset^N \bigslow](t,u)
		\\
	& = 
		\slowen[\Tanset^N \bigslow](0,u)
		+
		\slowfl[\Tanset^N \bigslow](t,0)
			\notag \\
	& \ \
			+
		  \int_{\mathcal{M}_{t,u}}
				\left\lbrace 
					1 + \GdVar \smoothfunction(\GdVar)
				\right\rbrace 
				\slowbasicenergyerror[\Tanset^N \bigslow]
			 \, d \vol
			 \notag \\
		& \ \ 
			+ 
		\int_{\mathcal{M}_{t,u}}
				\left\lbrace 
					1 + \GdVar \smoothfunction(\GdVar)
				\right\rbrace 
				\left\lbrace
					4 (h^{-1})^{\alpha 0} (h^{-1})^{\beta 0} (\Tanset^N \slow_{\alpha}) F_{\beta}
					+
					2 (h^{-1})^{\alpha \beta} (\Tanset^N \slow_{\alpha}) F_{\beta}
				\right\rbrace
			\, d \vol
			\notag
				\\
		& \ \
		+ 
		\int_{\mathcal{M}_{t,u}}
				\left\lbrace 
					1 + \GdVar \smoothfunction(\GdVar)
				\right\rbrace 
				\left\lbrace
					2 (h^{-1})^{\alpha a} (h^{-1})^{b 0} (\Tanset^N \slow_{\alpha}) F_{ab}
					+ 
					2 (\Tanset^N \slow) F
				\right\rbrace
			\, d \vol,
			\notag
\end{align}
where $F_0$, $F_a$, $F$, and $F_{ab}$
respectively denote the terms
$Harmless^{[1,N]} + Harmless_{(Slow)}^{\leq N}$
on RHSs \eqref{E:SLOWTIMECOMMUTED}-\eqref{E:SYMMETRYOFMIXEDPARTIALSCOMMUTED}
and/or \eqref{E:SLOWTIMENOTCOMMUTED}.
We will show that 
\textbf{i)} when $1 \leq N \leq 18$,
RHS~\eqref{E:E0DIVIDMAINESTIMATES} $\leq$
RHS~\eqref{E:TOPORDERTANGENTIALENERGYINTEGRALINEQUALITIES};
\textbf{ii)} when $1 \leq N \leq 18$,
RHS~\eqref{E:SLOWENERGYIDMAINESTIMATES} 
$\leq$ RHS~\eqref{E:TOPORDERTANGENTIALENERGYINTEGRALINEQUALITIES};
and 
\textbf{iii)} when $N=0$,
RHS~\eqref{E:SLOWENERGYIDMAINESTIMATES} 
$\leq$
RHS~\eqref{E:TOPORDERTANGENTIALENERGYINTEGRALINEQUALITIES}
where $N=1$ in \eqref{E:TOPORDERTANGENTIALENERGYINTEGRALINEQUALITIES}
(that is, the order $0$ energies for $\bigslow$ are controlled by
RHS~\eqref{E:TOPORDERTANGENTIALENERGYINTEGRALINEQUALITIES} with $N=1$).
Then, in view of Defs.~\ref{D:MAINCOERCIVEQUANT} and \ref{D:COERCIVEINTEGRAL},
we take the appropriate maxes and supremums over these estimates,
thereby arriving at the desired estimate \eqref{E:TOPORDERTANGENTIALENERGYINTEGRALINEQUALITIES}.

We start by showing that 
when $1 \leq N \leq 18$, we have that
RHS~\eqref{E:E0DIVIDMAINESTIMATES}
$\leq$
RHS~\eqref{E:TOPORDERTANGENTIALENERGYINTEGRALINEQUALITIES}.
The vast majority of the terms on
RHS~\eqref{E:E0DIVIDMAINESTIMATES}
were suitably bounded 
by $\leq \mbox{\upshape RHS}~\eqref{E:TOPORDERTANGENTIALENERGYINTEGRALINEQUALITIES}$
in \cite{jSgHjLwW2016}*{Section 14.8}.
In particular, the terms on the first line of RHS~\eqref{E:E0DIVIDMAINESTIMATES} 
were bounded by $\lesssim \mathring{\upepsilon}^2$
in Lemma~\ref{L:INITIALSIZEOFL2CONTROLLING} while
the last integral on RHS~\eqref{E:E0DIVIDMAINESTIMATES} 
was bounded via the estimate \eqref{E:MULTIPLIERVECTORFIELDERRORINTEGRALS};
for these terms, these are precisely the same estimates proved in \cite{jSgHjLwW2016}.
We now explain the origin of the terms that are new compared to \cite{jSgHjLwW2016}
and how to bound them. We stress that 
\emph{all of the new terms are non-borderline} 
in the sense that they do not contribute to the 
``boxed-constant-involving'' terms on RHS~\eqref{E:TOPORDERTANGENTIALENERGYINTEGRALINEQUALITIES},
which drive the blowup-rate of the top-order energies.
The new terms are all found within the factor
$\upmu \square_g(\Tanset^N \Psi)$
in the first error integral 
$
- 
		\int_{\mathcal{M}_{t,u}}
			\left\lbrace
				(1 + 2 \upmu) (\Lunit \Tanset^N \Psi)
				+ 
				2 \Rad \Tanset^N \Psi 
			\right\rbrace
			\upmu \square_g(\Tanset^N \Psi) 
		\, d \vol
$
on RHS~\eqref{E:E0DIVIDMAINESTIMATES}.
To bound this error integral, we first use
(depending on the structure of $\Tanset^N$)
one of equations
\eqref{E:LISTHEFIRSTCOMMUTATORIMPORTANTTERMS}-\eqref{E:GEOANGANGISTHEFIRSTCOMMUTATORIMPORTANTTERMS}
or \eqref{E:HARMLESSORDERNCOMMUTATORS}
to algebraically substitute for
$\upmu \square_g(\Tanset^N \Psi)$.
The error integrals generated by the
$Harmless^{[1,N]}$ terms were 
suitably bounded in \cite{jSgHjLwW2016}*{Section 14.8}
(alternatively, see Lemma~\ref{L:STANDARDPSISPACETIMEINTEGRALS}).
To bound the error integrals generated by
the terms $Harmless_{(Slow)}^{\leq N}$, we
use Lemma~\ref{L:STANDARDPSISPACETIMEINTEGRALS}.
It remains for us to bound the error integrals
generated by the three products on RHSs
\eqref{E:LISTHEFIRSTCOMMUTATORIMPORTANTTERMS}-\eqref{E:GEOANGANGISTHEFIRSTCOMMUTATORIMPORTANTTERMS}
that are not of the form $Harmless^{[1,N]} + Harmless_{(Slow)}^{\leq N}$.
We proceed on a case by case basis, depending on the structure of the commutator
vectorfield string $\Tanset^N$.
In the case $\Tanset^N = \GeoAng^N$,
we must bound the difficult error integral
\begin{align} \label{E:RADDIFFICULTERRORINTEGRAL}
	-           2
							\int_{\mathcal{M}_{t,u}}
							(\Rad \GeoAng^N \Psi)	 
							(\Rad \Psi) \GeoAng^N \mytr \upchi
							\, d \vol,
\end{align}
which is generated by the term $(\Rad \Psi) \GeoAng^N \mytr \upchi$
on RHS~\eqref{E:GEOANGANGISTHEFIRSTCOMMUTATORIMPORTANTTERMS},
by $\leq$ RHS~\eqref{E:TOPORDERTANGENTIALENERGYINTEGRALINEQUALITIES}.
To this end, we first use Cauchy--Schwarz and \eqref{E:COERCIVENESSOFCONTROLLING} to bound 
the magnitude of \eqref{E:RADDIFFICULTERRORINTEGRAL} by
\begin{align} \label{E:FIRSTSTEPDIFFICULTINTEGRALBOUND}
	\leq 2 
		\int_{t' = 0}^t 
			\totTanmax{[1,N]}^{1/2}(t',u) 
			\left\| 
				(\Rad \Psi) \GeoAng^N \mytr \upchi 
			\right\|_{L^2(\Sigma_{t'}^u)} 
		\, dt'.
\end{align}
We now substitute the estimate \eqref{E:DIFFICULTTERML2BOUND}
(with $t$ replaced by $t'$) for the second integrand factor on RHS~\eqref{E:FIRSTSTEPDIFFICULTINTEGRALBOUND}.
In \cite{jSgHjLwW2016}*{Section 14.8},
all of the corresponding terms that arise from this substitution 
were bounded by $\leq \mbox{\upshape RHS}~\eqref{E:TOPORDERTANGENTIALENERGYINTEGRALINEQUALITIES}$
except for the term generated by the one
product on RHS~\eqref{E:DIFFICULTTERML2BOUND} involving $\slowtotTanmax{\leq N}$
(that is, the next-to-last product on RHS~\eqref{E:DIFFICULTTERML2BOUND}).
Upon substituting this remaining product into
\eqref{E:FIRSTSTEPDIFFICULTINTEGRALBOUND},
we generate the following error term, which is new compared to the terms appearing in \cite{jSgHjLwW2016}*{Section 14.8}:
\begin{align} \label{E:MAINDIFFICULTENERGYESTIMETERMSLOWWAVECONTRIBUTION}
C 
\int_{t' = 0}^t 
	\frac{1}{\upmu_{\star}(t',u)} 
	\totTanmax{[1,N]}^{1/2}(t',u) 
	\int_{s=0}^{t'}
		\frac{1} {\upmu_{\star}^{1/2}(s,u)} 
		\slowtotTanmax{\leq N}^{1/2}(s,u) 
	\, ds	
\, dt'.
\end{align}
We now observe that the term \eqref{E:MAINDIFFICULTENERGYESTIMETERMSLOWWAVECONTRIBUTION}
is $\leq$ the third-from-last product on RHS~\eqref{E:TOPORDERTANGENTIALENERGYINTEGRALINEQUALITIES}
as desired. We now clarify that in \cite{jSgHjLwW2016},
the coefficient of the analog of
the third product on RHS~\eqref{E:TOPORDERTANGENTIALENERGYINTEGRALINEQUALITIES}
was stated as $\boxed{8.1}$ rather than $\boxed{8}$. 
This is because the coefficient of the second product
on RHS~\eqref{E:DIFFICULTTERML2BOUND} is $\boxed{4}$, 
which is different than \cite{jSgHjLwW2016}*{Lemma~14.8},
where the analogous coefficient was stated as
$\boxed{4.05}$ (for reasons that we described in our proof outline of Lemma~\ref{L:DIFFICULTTERML2BOUND});
this is a minor remark that has no substantial bearing on our analysis.
Next, we note that the error integral
\[
		-         \int_{\mathcal{M}_{t,u}}
								(1 + 2 \upmu)
								(\Lunit \GeoAng^N \Psi)
								(\Rad \Psi) \GeoAng^N \mytr \upchi
							\, d \vol,
\]
which is also generated by the term $(\Rad \Psi) \GeoAng^N \mytr \upchi$
on RHS~\eqref{E:GEOANGANGISTHEFIRSTCOMMUTATORIMPORTANTTERMS},
can be bounded by $\leq$ RHS~\eqref{E:TOPORDERTANGENTIALENERGYINTEGRALINEQUALITIES}
by using arguments nearly identical to the ones
in \cite{jSgHjLwW2016}*{Section 14.8}. In particular, the 
arguments are independent of $\bigslow$ and therefore do not involve the
slow wave controlling quantities $\slowtotTanmax{M}$.
We do not repeat the proof here since it is
exactly the same and since it relies on
a lengthy integration by parts argument in which
one trades the factor of $\Lunit$ from $\Lunit \GeoAng^N \Psi$
for a factor of $\GeoAng$ from $\GeoAng^N \mytr \upchi$.
We note that the integration by parts generates some of the
``boxed-constant-involving'' terms on RHS~\eqref{E:TOPORDERTANGENTIALENERGYINTEGRALINEQUALITIES},
including the ``boundary term''
$
\displaystyle
\boxed{\left\lbrace 2 + \mathcal{O}_{\mydiam}(\Psiep) \right\rbrace}
			\frac{1}{\upmu_{\star}^{1/2}(t,u)}
			\totTanmax{[1,N]}^{1/2}(t,u)
			\left\| 
				\Lunit \upmu 
			\right\|_{L^{\infty}(\Sigmaminus{t}{t}{u})}
\cdots
$
on the fourth line of RHS~\eqref{E:TOPORDERTANGENTIALENERGYINTEGRALINEQUALITIES}
and a contribution to the term
$
\displaystyle
\boxed{\left\lbrace 6 + \mathcal{O}_{\mydiam}(\Psiep) \right\rbrace}
			\int_{t'=0}^t
					\frac{\left\| [\Lunit \upmu]_- \right\|_{L^{\infty}(\Sigma_{t'}^u)}} 
							 {\upmu_{\star}(t',u)} 
				  \totTanmax{[1,N]}(t',u)
				\, dt'
$
on the second line of RHS~\eqref{E:TOPORDERTANGENTIALENERGYINTEGRALINEQUALITIES}.
We clarify that the factors of
$
\mathcal{O}_{\mydiam}(\Psiep)
$
appearing in the above two terms
arise because in invoking the arguments given in 
\cite{jSgHjLwW2016}*{Section 14.8},
one must at one point derive pointwise control of
$|\angLap \GeoAng^{N-1} \Psi|$
in terms of $|\angdiff \GeoAng^N \Psi| + \cdots$,
where $\cdots$ denotes simpler error terms;
by \eqref{E:ANGDERIVATIVESINTERMSOFTANGENTIALCOMMUTATOR}
and Cor.\ \ref{C:SQRTEPSILONTOCEPSILON},
we in fact have the bound
$|\angLap \GeoAng^{N-1} \Psi| 
\leq 
\lbrace 1 + \mathcal{O}_{\mydiam}(\Psiep) + \mathcal{O}(\varepsilon) \rbrace |\angdiff \GeoAng^N \Psi| 
+ 
\cdots
$,
which is the source of the factor $ \mathcal{O}_{\mydiam}(\Psiep)$ in the above estimates.
We also clarify that the factor of $\mathcal{O}_{\mydiam}(\Psiep)$
does not appear in the arguments of \cite{jSgHjLwW2016} since the parameter $\Psiep$ was not
featured in that work (see also Remark~\ref{R:NEWPARAMETER}).
To bound the remaining two error integrals
generated by RHS~\eqref{E:GEOANGANGISTHEFIRSTCOMMUTATORIMPORTANTTERMS},
namely
\[
		-           2
							\int_{\mathcal{M}_{t,u}}
								(\Rad \GeoAng^N \Psi)	 
								\GeoAngFlatRadComponent
								(\angdiffuparg{\#} \Psi) 
								\cdot 
								(\upmu \angdiff \GeoAng^{N-1} \mytr \upchi)
							\, d \vol
\]
and
\[
		-          \int_{\mathcal{M}_{t,u}}
									(1 + 2 \upmu) (\Lunit \GeoAng^N \Psi) 
									\GeoAngFlatRadComponent
									(\angdiffuparg{\#} \Psi) 
									\cdot 
									(\upmu \angdiff \GeoAng^{N-1} \mytr \upchi)
							\, d \vol,
\]
by $\leq \mbox{\upshape RHS}~\eqref{E:TOPORDERTANGENTIALENERGYINTEGRALINEQUALITIES}$,
we simply use Lemma~\ref{L:LESSDEGENERATEENERGYESTIMATEINTEGRALS}.

To complete the proof that
RHS~\eqref{E:E0DIVIDMAINESTIMATES}
$\leq$
RHS~\eqref{E:TOPORDERTANGENTIALENERGYINTEGRALINEQUALITIES},
it remains only for us to bound the difficult error integrals that arise
in the case $\Tanset^N = \GeoAng^{N-1} \Lunit$,
which are generated by the terms on RHS~\eqref{E:LISTHEFIRSTCOMMUTATORIMPORTANTTERMS}
that are not of the form $Harmless^{[1,N]} + Harmless_{(Slow)}^{\leq N}$.
Specifically, we encounter the error integrals
\[
		-          2
							\int_{\mathcal{M}_{t,u}}
								(\Rad \GeoAng^{N-1} \Lunit \Psi)	 
								(\angdiffuparg{\#} \Psi) \cdot (\upmu \angdiff \GeoAng^{N-1} \mytr \upchi)
							\, d \vol,
\]
and
\[
		-         \int_{\mathcal{M}_{t,u}}
								(1 + 2 \upmu) 
								(\Lunit \GeoAng^{N-1} \Lunit \Psi)
								(\angdiffuparg{\#} \Psi) \cdot (\upmu \angdiff \GeoAng^{N-1} \mytr \upchi)
							\, d \vol,
\]
which we bounded in Lemma~\ref{L:LESSDEGENERATEENERGYESTIMATEINTEGRALS}.
We have thus shown that
RHS~\eqref{E:E0DIVIDMAINESTIMATES}
$\leq$
RHS~\eqref{E:TOPORDERTANGENTIALENERGYINTEGRALINEQUALITIES}.

To complete the proof of \eqref{E:TOPORDERTANGENTIALENERGYINTEGRALINEQUALITIES},
it remains only for us to show that 
\textbf{ii)} when $1 \leq N \leq 18$, 
RHS~\eqref{E:SLOWENERGYIDMAINESTIMATES} 
$\leq$ RHS~\eqref{E:TOPORDERTANGENTIALENERGYINTEGRALINEQUALITIES}
and \textbf{iii)} when $N=0$,
RHS~\eqref{E:SLOWENERGYIDMAINESTIMATES} $\leq$ RHS~\eqref{E:TOPORDERTANGENTIALENERGYINTEGRALINEQUALITIES}
but with $N=1$ on RHS~\eqref{E:TOPORDERTANGENTIALENERGYINTEGRALINEQUALITIES}
(that is, the order $0$ energies for $\bigslow$ are controlled by
RHS~\eqref{E:TOPORDERTANGENTIALENERGYINTEGRALINEQUALITIES} with $N=1$).
We start with case \textbf{ii)}.
To proceed, we first use Lemma~\ref{L:INITIALSIZEOFL2CONTROLLING}
to deduce that 
$	
\slowen[\Tanset^N \bigslow](0,u)
+
\slowfl[\Tanset^N \bigslow](t,0)
\lesssim
\mathring{\upepsilon}^2
$
as desired.
Next, we note that
we suitably bounded the first integral on RHS~\eqref{E:SLOWENERGYIDMAINESTIMATES}
in Lemma~\ref{L:ENERGYESTIMATESFORSLOWWAVEERRORTERMSNOINHOMOGENEOUS}.
To bound the integrals on the last two lines of RHS~\eqref{E:SLOWENERGYIDMAINESTIMATES},
we recall that, as we mentioned above,
the terms $F_0$, $F_a$, $F$, and $F_{ab}$ 
in the integrands
are of the form
$
Harmless^{[1,N]}
+ 
Harmless_{(Slow)}^{\leq N}
$.
Moreover, the $L^{\infty}$ estimates of Prop.~\ref{P:IMPROVEMENTOFAUX}
imply that $\left| 1 + \GdVar \smoothfunction(\GdVar) \right| \lesssim 1$ 
and, since $(h^{-1})^{\alpha \beta} = \smoothfunction(\GdVar,\bigslow)$,
that $\left| (h^{-1})^{\alpha \beta} \right| \lesssim 1$.
Thus, the desired bounds
follow from Lemma~\ref{L:STANDARDPSISPACETIMEINTEGRALS}.
Finally, we note that the argument is identical in case \textbf{iii)}, 
which completes the proof of \eqref{E:TOPORDERTANGENTIALENERGYINTEGRALINEQUALITIES}.
\medskip

\noindent \textbf{Proof of \eqref{E:BELOWTOPORDERTANGENTIALENERGYINTEGRALINEQUALITIES}}
The proof of \eqref{E:BELOWTOPORDERTANGENTIALENERGYINTEGRALINEQUALITIES} 
is similar to that of \eqref{E:TOPORDERTANGENTIALENERGYINTEGRALINEQUALITIES} 
but involves one key change.
To proceed, we let $\Tanset^{N-1}$ be an $(N-1)^{st}$-order 
$\mathcal{P}_u$-tangential vectorfield operator, where $2 \leq N \leq 18$. 
We then argue as above,
starting with the identity \eqref{E:E0DIVIDMAINESTIMATES} 
with $\Tanset^{N-1} \Psi$ in place of $\Tanset^N \Psi$
and the identity \eqref{E:SLOWENERGYIDMAINESTIMATES}
with $\Tanset^{N-1} \bigslow$ in place of $\Tanset^N \bigslow$.
We bound almost all error integrals in exactly the same way
as before, the key change being that we bound the two 
most difficult error integrals in a different way. 
Specifically, the two difficult integrals are
\[
	-           2
							\int_{\mathcal{M}_{t,u}}
							(\Rad \GeoAng^{N-1} \Psi)	 
							(\Rad \Psi) \GeoAng^{N-1} \mytr \upchi
							\, d \vol,
\]
which is an analog of \eqref{E:RADDIFFICULTERRORINTEGRAL},
and
\[
	-           \int_{\mathcal{M}_{t,u}}
								(1 + 2 \upmu)
								(\Lunit \GeoAng^{N-1} \Psi)	 
								(\Rad \Psi) \GeoAng^{N-1} \mytr \upchi
							\, d \vol,
\]
whose top-order analog we did not discuss in detail above
since it was treated
in \cite{jSgHjLwW2016}*{Section 14.8}
using arguments that are independent of the slow wave variable $\bigslow$.
Both of the above error integrals were
shown in \cite{jSgHjLwW2016}*{Section 14.8}
to be bounded
by $\leq$ RHS~\eqref{E:BELOWTOPORDERTANGENTIALENERGYINTEGRALINEQUALITIES}
using a simple argument
that does not involve the
slow wave controlling quantities $\slowtotTanmax{M}$;
the key point is to use the derivative-losing estimate \eqref{E:TANGENGITALEIKONALINTERMSOFCONTROLLING}
to control $\| \GeoAng^{N-1} \mytr \upchi \|_{L^2(\Sigma_t^u)}$.
We remark that the arguments in \cite{jSgHjLwW2016}*{Section 14.8}
for bounding these two error integrals
lead to the presence of the $\totTanmax{[1,N]}^{1/2}(s,u)$-involving
double time integral on the second line of RHS~\eqref{E:BELOWTOPORDERTANGENTIALENERGYINTEGRALINEQUALITIES}.
That is, these estimates lose one derivative 
(which is permissible below top order)
and are therefore coupled to the next-highest-order energy estimates;
the gain is that the resulting integrals
are less singular with respect to $\upmu_{\star}^{-1}$.
Finally, we note that the proofs of the energy estimates
for $\Tanset^{N-1} \bigslow$ in the below-top-order cases
$1 \leq N \leq 18$ are exactly the same as in the top-order case treated above.
We have therefore proved \eqref{E:BELOWTOPORDERTANGENTIALENERGYINTEGRALINEQUALITIES},
which finishes the proof of Prop.~\ref{P:TANGENTIALENERGYINTEGRALINEQUALITIES}.

$\hfill \qed$

\section{The main stable shock formation theorem}
\label{S:MAINTHEOREM}
\setcounter{equation}{0}
We now prove our main result.

\begin{theorem}[\textbf{Stable shock formation}]
\label{T:MAINTHEOREM}
Consider a solution
$\Psi,\bigslow$ to the system
\eqref{E:FASTWAVE}
+
\eqref{E:SLOW0EVOLUTION}-\eqref{E:SYMMETRYOFMIXEDPARTIALS}
with nonlinearities verifying the assumptions stated in
Subsubsects.\ \ref{SSS:STATEMENTOFEQUATIONS}-\ref{SSS:WAVESPEEDASSUMPTIONS}.
Let 
$\Psiep > 0$,
$\mathring{\upepsilon} \geq 0$,
$\mathring{\updelta} > 0$,
and $\TranminusdatasizeWithFactor > 0$
be the data-size parameters introduced in Sect.\ \ref{S:NORMSANDBOOTSTRAP}.
For each $U_0 \in [0,1]$, let $T_{(Lifespan);U_0}$
be the classical lifespan of $\Psi,\bigslow$
with respect to the \underline{Cartesian} coordinates $\lbrace x^{\alpha} \rbrace_{\alpha = 0,1,2}$
in the region that is completely determined by the non-trivial data lying in 
$\Sigma_0^{U_0}$ and the small data given along $\mathcal{P}_0^{2 \TranminusdatasizeWithFactor^{-1}}$
(see Figure~\ref{F:REGION} on pg.~\pageref{F:REGION}).
If 
$\Psiep$ is sufficiently small relative to $1$,
and if $\mathring{\upepsilon}$ is sufficiently small
relative to 
1,
small relative to $\mathring{\updelta}^{-1}$,
and small relative to
$\TranminusdatasizeWithFactor$
(in the sense explained in Subsect.\ \ref{SS:SMALLNESSASSUMPTIONS}),
then the following conclusions hold,
where all constants can be chosen to be independent of $U_0$.

\medskip

\noindent \underline{\textbf{Dichotomy of possibilities.}}
One of the following mutually disjoint possibilities must occur,
where $\upmu_{\star}(t,u) = \min \lbrace 1, \min_{\Sigma_t^u} \upmu \rbrace$.
\begin{enumerate}
	\renewcommand{\labelenumi}{\textbf{\Roman{enumi})}}
	\item $T_{(Lifespan);U_0} > 2 \TranminusdatasizeWithFactor^{-1}$. 
		In particular, the solution exists classically on the spacetime
		region $\mbox{\upshape cl} \mathcal{M}_{2 \TranminusdatasizeWithFactor^{-1},U_0}$,
		where $\mbox{\upshape cl}$ denotes closure.
		Furthermore, $\inf \lbrace \upmu_{\star}(s,U_0) \ | \ s \in [0,2 \TranminusdatasizeWithFactor^{-1}] \rbrace > 0$.
	\item $T_{(Lifespan);U_0} \leq 2 \TranminusdatasizeWithFactor^{-1}$,
		and 
		\begin{align} \label{E:MAINTHEOREMLIFESPANCRITERION}
		T_{(Lifespan);U_0} 
		= \sup 
			\left\lbrace 
			t \in [0, 2 \TranminusdatasizeWithFactor^{-1}) \ | \
				\inf \lbrace \upmu_{\star}(s,U_0) \ | \ s \in [0,t) \rbrace > 0
			\right\rbrace.
		\end{align}
\end{enumerate}
In addition, case $\textbf{II)}$ occurs when $U_0 = 1$. In this case, we have\footnote{See
Subsect.\ \ref{SS:NOTATIONANDINDEXCONVENTIONS} regarding our use of the notation $\mathcal{O}_{\mydiam}(\Psiep)$.}
\begin{align} \label{E:CLASSICALLIFESPANASYMPTOTICESTIMATE}
	T_{(Lifespan);1} 
	= 
	\left\lbrace 1 + \mathcal{O}_{\mydiam}(\Psiep) + \mathcal{O}(\mathring{\upepsilon}) \right\rbrace
	\TranminusdatasizeWithFactor^{-1}.
\end{align}

\medskip

\noindent \underline{\textbf{What happens in Case I).}}
In case $\textbf{I)}$, 
all of the bootstrap assumptions from Subsects.\ \ref{SS:SIZEOFTBOOT}-\ref{SS:PSIBOOTSTRAP},
the estimates of Props.~\ref{P:IMPROVEMENTOFAUX} and \ref{P:IMPROVEMENTOFHIGHERTRANSVERSALBOOTSTRAP},
and the energy estimates of Prop.~\ref{P:MAINAPRIORIENERGY}
hold on $\mbox{\upshape cl} \mathcal{M}_{2 \TranminusdatasizeWithFactor^{-1},U_0}$
with the factors of $\varepsilon$ on the RHSs replaced by $C \mathring{\upepsilon}$.
Moreover, for $0 \leq M \leq 5$, the following estimates hold
for $(t,u) \in [0,2 \TranminusdatasizeWithFactor^{-1}] \times [0,U_0]$
(see Subsect.\ \ref{SS:STRINGSOFCOMMUTATIONVECTORFIELDS} regarding the vectorfield operator notation):
\begin{subequations}
	\begin{align}
		\left\|
			\Tanset_*^{[1,12]} \upmu
		\right\|_{L^2(\Sigma_t^u)},
			\,
		\left\|
			\Tanset^{[1,12]} 
			\Lunit_{(Small)}^i
		\right\|_{L^2(\Sigma_t^u)},
			\,
		\left\|
			\Tanset^{\leq 11} 
			\mytr \upchi
		\right\|_{L^2(\Sigma_t^u)}
		& \leq 
			C \mathring{\upepsilon},
				 \label{E:LOWORDERTANGENTIALEIKONALL2MAINTHEOREM}
				 \\
		\left\|
			\Tanset_*^{13 + M} \upmu
		\right\|_{L^2(\Sigma_t^u)},
			\,
		\left\|
			\Tanset^{13 + M} 
			\Lunit_{(Small)}^i
		\right\|_{L^2(\Sigma_t^u)},
			\,
		\left\|
			\Tanset^{12 + M} 
			\mytr \upchi
		\right\|_{L^2(\Sigma_t^u)}
		& \leq 
			C \mathring{\upepsilon} \upmu_{\star}^{-(M+.4)}(t,u),
				 \label{E:MIDORDERTANGENTIALEIKONALL2MAINTHEOREM} 
				 	\\
		\left\|
			\Lunit \Tanset^{18} \upmu
		\right\|_{L^2(\Sigma_t^u)},
			\,
		\left\|
			\Lunit \Fullset^{18;1} \Lunit_{(Small)}^i
		\right\|_{L^2(\Sigma_t^u)},
			\,
		\left\|
			\Lunit \Fullset^{17;1} \mytr \upchi
		\right\|_{L^2(\Sigma_t^u)}
		& \leq 
			C \mathring{\upepsilon} \upmu_{\star}^{-6.4}(t,u),
				 \label{E:HIGHORDERLUNITTANGENTIALEIKONALL2MAINTHEOREM} 
		\\
		\left\|
			\upmu \GeoAng^{18} \mytr \upchi
		\right\|_{L^2(\Sigma_t^u)}
		& \leq 
			C \mathring{\upepsilon} \upmu_{\star}^{-5.9}(t,u).
				 \label{E:HIGHORDERANGULARTRCHIL2MAINTHEOREM} 
	\end{align}
\end{subequations}

\medskip

\noindent \underline{\textbf{What happens in Case II).}}
In case $\textbf{II)}$, 
all of the bootstrap assumptions from Subsects.\ \ref{SS:SIZEOFTBOOT}-\ref{SS:PSIBOOTSTRAP},
the estimates of Props.~\ref{P:IMPROVEMENTOFAUX} and \ref{P:IMPROVEMENTOFHIGHERTRANSVERSALBOOTSTRAP},
and the energy estimates of Prop.~\ref{P:MAINAPRIORIENERGY}
hold on $\mathcal{M}_{T_{(Lifespan);U_0},U_0}$
with the factors of $\varepsilon$ on the RHS replaced by $C \mathring{\upepsilon}$.
Moreover, for $0 \leq M \leq 5$, the estimates 
\eqref{E:LOWORDERTANGENTIALEIKONALL2MAINTHEOREM}-\eqref{E:HIGHORDERANGULARTRCHIL2MAINTHEOREM}
hold for $(t,u) \in [0,T_{(Lifespan);U_0}) \times [0,U_0]$.
In addition, the scalar functions
$\Fullset^{\leq 9;1} \Psi$,
$\Fullset^{\leq 3;2} \Psi$,
$\Rad \Rad \Rad \Psi$,
$\Fullset^{\leq 9;1} \bigslow$,
$\Rad \Rad \bigslow$,
$\Fullset^{\leq 9;1} \Lunit^i$,
$\Rad \Rad \Lunit^i$,
$\Tanset^{\leq 9} \upmu$,
$\Fullset^{\leq 3;1} \upmu$,
and
$\Rad \Rad \upmu$
extend to\footnote{In
\cite{jSgHjLwW2016}*{Theorem~15.1}, it was stated that
$\Fullset^{\leq 2;1} \upmu$
extend. In fact, the arguments given there imply that
$\Fullset^{\leq 3;1} \upmu$ extend, as we have stated here.
Moreover, we note that here we have stated that $\Fullset^{\leq 3;2} \Psi$
extend. This stands in contrast to \cite{jSgHjLwW2016}*{Theorem~15.1},
in which incorrect reasoning led to the conclusion that 
$\Fullset^{\leq 4;2} \Psi$ extend. The faulty reasoning started 
with the not-fully-justified estimate
$\| \Lunit \Fullset^{\leq 4;2} \Psi \|_{L^{\infty}(\Sigma_t^u)} \leq C \varepsilon$,
which was stated in \cite{jSgHjLwW2016}*{Proposition~9.2}.
The proof given there, however, works only for a subset of 
operators of type $\Lunit \Fullset^{\leq 4;2}$ in which a factor of $\Rad$ acts first,
for the same reasons that led to our proof of \eqref{E:WAVEEQNONCETRANSVERSALLYCOMMUTEDTRANSPORTPOINTOFVIEW}.
These are minor points that have no substantial 
bearing on the main results.} 
$\Sigma_{T_{(Lifespan);U_0}}^{U_0}$ 
as functions of 
the geometric coordinates $(t,u,\vartheta)$ 
belonging to the space 
$C\left([0,T_{(Lifespan);U_0}],L^{\infty}([0,U_0] \times \mathbb{T}) \right)$.
Furthermore, the Cartesian component functions
$g_{\alpha \beta}(\Psi)$ 
verify the estimate
$g_{\alpha \beta} 
= 
m_{\alpha \beta} 
+ \mathcal{O}_{\mydiam}(\Psiep)
+ \mathcal{O}(\mathring{\upepsilon})$ 
(where $m_{\alpha \beta} = \mbox{\upshape diag}(-1,1,1)$ is the standard Minkowski metric)
and have the same extension properties as $\Psi$
(in particular, the same $\Fullset$-derivatives of $g_{\alpha \beta}$ extend
as elements of $C\left([0,T_{(Lifespan);U_0}],L^{\infty}([0,U_0] \times \mathbb{T}) \right)$).

Moreover,	let $\Sigma_{T_{(Lifespan);U_0}}^{U_0;(Blowup)}$
be the subset of $\Sigma_{T_{(Lifespan);U_0}}^{U_0}$ 
defined by
\begin{align} \label{E:BLOWUPPOINTS}
	\Sigma_{T_{(Lifespan);U_0}}^{U_0;(Blowup)}
	:= \left\lbrace
			(T_{(Lifespan);U_0},u,\vartheta)
			\ | \
			\upmu(T_{(Lifespan);U_0},u,\vartheta)
			= 0
		\right\rbrace.
\end{align}
Then for each point $(T_{(Lifespan);U_0},u,\vartheta) \in \Sigma_{T_{(Lifespan);U_0}}^{U_0;(Blowup)}$,
there exists a past neighborhood\footnote{By a past neighborhood of $(T_{(Lifespan);U_0},u,\vartheta)$, 
we mean a set that is 
the intersection of the closed half-space 
$\lbrace (t,u',\vartheta') \in \mathbb{R} \times \mathbb{R} \times \mathbb{T} \ | \ t \leq T_{(Lifespan);U_0} \rbrace$
with a set containing $(T_{(Lifespan);U_0},u,\vartheta)$
that is open with respect to the standard topology corresponding to the geometric coordinates.} 
containing it such that the following lower bound holds in
the neighborhood:
\begin{align} \label{E:BLOWUPPOINTINFINITE}
	\left| \Radunit \Psi (t,u,\vartheta) \right|
	\geq \frac{\TranminusdatasizeWithFactor}{4 |G_{\Lunit_{(Flat)} \Lunit_{(Flat)}}(\Psi = 0)|} \frac{1}{\upmu(t,u,\vartheta)}.
\end{align}
In \eqref{E:BLOWUPPOINTINFINITE}, 
$
\displaystyle
\frac{\TranminusdatasizeWithFactor}{4 \left|G_{\Lunit_{(Flat)} \Lunit_{(Flat)}}(\Psi = 0) \right|}
$
is a \textbf{positive} data-dependent constant
(see \eqref{E:NONVANISHINGNONLINEARCOEFFICIENT}),
and the $\ell_{t,u}$-transversal vectorfield $\Radunit$ has near-Euclidean-unit length:
$\delta_{ab} \Radunit^a \Radunit^b = 1 + \mathcal{O}_{\mydiam}(\Psiep) + \mathcal{O}(\mathring{\upepsilon})$.
In particular, $\Radunit \Psi$ blows up like $1/\upmu$ at all points in $\Sigma_{T_{(Lifespan);U_0}}^{U_0;(Blowup)}$.
Conversely, at all points in
$(T_{(Lifespan);U_0},u,\vartheta) \in \Sigma_{T_{(Lifespan);U_0}}^{U_0} \backslash \Sigma_{T_{(Lifespan);U_0}}^{U_0;(Blowup)}$,
we have
\begin{align} \label{E:NONBLOWUPPOINTBOUND}
	\left| \Radunit \Psi (T_{(Lifespan);U_0},u,\vartheta) \right|
	< \infty.
\end{align}

\end{theorem}

\begin{proof}[Discussion of proof]
The proof of \cite{jSgHjLwW2016}*{Theorem~15.1}
applies nearly verbatim, 
except for a few of the statements regarding $\bigslow$.
The main ingredients in the proof are 
the estimates of Props.~\ref{P:IMPROVEMENTOFAUX} and \ref{P:IMPROVEMENTOFHIGHERTRANSVERSALBOOTSTRAP},
the energy estimates of Prop.~\ref{P:MAINAPRIORIENERGY},
Cor.\ \ref{C:IMPROVEDFUNDAMENTALLINFTYBOOTSTRAPASSUMPTIONS},
and \eqref{E:MUSTARBOUNDSUISONE}. We clarify that \eqref{E:MUSTARBOUNDSUISONE} easily yields
that $\upmu_{\star}(t,1)$ vanishes for the first time when $t$ is equal to
RHS~\eqref{E:CLASSICALLIFESPANASYMPTOTICESTIMATE}, that is, that
when $U_0 = 1$, a shock forms at a time $T_{(Lifespan);1}$
verifying the estimate \eqref{E:CLASSICALLIFESPANASYMPTOTICESTIMATE}.
Strictly speaking, in the proof of \cite{jSgHjLwW2016}*{Theorem~15.1}, 
a few additional estimates, beyond the main ingredients we just mentioned, 
were needed to complete the proof. For example, one needs estimates
for all of the components of the change of variables map $\Upsilon$, including
the scalar-valued functions $\NonRadialRad^i$ on RHS~\eqref{E:CHOV}.
However, in \cite{jSgHjLwW2016},
the needed estimates followed as straightforward consequences
of the main ingredients mentioned above, and we do not even bother to state them
here since their proofs are exactly the same as in \cite{jSgHjLwW2016}; in particular
the proofs of the omitted estimates do not involve $\bigslow$.

For the sake of thoroughness, 
we now prove some facts concerning $\bigslow$ that are not part of
the statement or proof of \cite{jSgHjLwW2016}*{Theorem~15.1}.
Specifically, the facts that
$\Fullset^{\leq 9;1} \bigslow$
and
$\Rad \Rad \bigslow$
extend to $\Sigma_{T_{(Lifespan);U_0}}^{U_0}$
as elements of the function space
$C\left([0,T_{(Lifespan);U_0}],L^{\infty}([0,U_0] \times \mathbb{T}) \right)$
are simple consequences of the fundamental theorem of calculus,
the completeness of the space $L^{\infty}([0,U_0] \times \mathbb{T})$,
and the uniform bounds
$
\displaystyle
\left\| 
	\Lunit \Fullset^{\leq 9;1} \bigslow
\right\|_{L^{\infty}(\Sigma_t^u)}
\lesssim
\mathring{\upepsilon}
$
and
$
\left\| 
	\Lunit \Rad \Rad \bigslow
\right\|_{L^{\infty}(\Sigma_t^u)}
\lesssim
\mathring{\upepsilon}
$,
which are special cases of the estimates
\eqref{E:SLOWWAVETRANSVERSALTANGENT} and \eqref{E:UPTOTWOTRANSVERSALDERIVATIVESOFSLOWINLINFINITY}
and Cor.\ \ref{C:IMPROVEDFUNDAMENTALLINFTYBOOTSTRAPASSUMPTIONS}
(recall that
$
\displaystyle
\Lunit 
= \frac{\partial}{\partial t}
$).

\end{proof}

\section*{Acknowledgments}
The author would like to thank the anonymous referee for their useful comments,
which helped him to clarify the exposition.

\bibliographystyle{amsalpha}
\bibliography{JBib}

\end{document}